\documentclass[10pt, reqno]{amsart}
\usepackage{og-amsart}

\usepackage{og-flowchart}\numberwithin{equation}{section}\address{Department of Mathematics, Princeton University}\email{onyxg@princeton.edu}
\author{Onyx Gautam}
\date{\today}
\title{Late-time tails for linear waves on radially symmetric stationary spacetimes of two space dimensions}
\begin{document}

\maketitle
\vspace{-4ex}
\begin{abstract}
We show that the leading-order term in the late-time asymptotics of solutions to
the linear wave equation on radially symmetric stationary perturbations of
\((2 + 1)\)-dimensional Minkowski space is proportional to \(u^{-1/2}v^{-1/2}\)
(which solves the wave equation on Minkowski space), where \(u\) and \(v\) are
double null coordinates. Our proof adapts the physical space techniques in the
work of Gajic \cite{GAJIC2023110058} on the wave equation with an inverse-square
potential on the Schwarzschild spacetime. In particular, we extend
the \(r^p\)-weighted energy estimates of Dafermos--Rodnianski \cite{rp-method} to
two space dimensions.
\end{abstract}
\setcounter{tocdepth}{2}
\tableofcontents
\newpage
\section{Introduction}
We are interested in solutions to the linear wave equation on spacetimes of two
space dimensions:
\begin{equation}\label{wave-equation-intro}
\Box{}_g\varphi{} = 0.
\end{equation}
Our main result determines the leading-order late-time asymptotics of solutions
to \cref{wave-equation-intro} when the metric \(g\) is a radially symmetric stationary
perturbation of the Minkowski metric.
\begin{theorem}[Main theorem on leading-order late-time asymptotics, rough version]
Let \(g\) be a Lorentzian metric on \(\R^{2 + 1}\) that is a small, stationary,
radially symmetric, and asymptotically flat perturbation of the Minkowski metric. Let
\((u,v,\theta{})\) be a double null coordinate system on \((\R^{2 + 1},g)\)
normalized as in Minkowski space. Then there is \(\delta{} > 0\) such that for
large values of \(u\), solutions \(\varphi{}\) of
\cref{wave-equation-intro} arising from suitably regular and decaying initial data
satisfy
\begin{equation}\label{intro-thm-estimate}
\varphi{}(u,v,\theta{}) = \mathfrak{L}[\varphi{}]u^{-1/2}v^{-1/2} + O(u^{-1/2-\delta{}}v^{-1/2}).
\end{equation}
Here \(\mathfrak{L}[\varphi{}]\) is a linear functional of
the initial data for \(\varphi{}\), and the implicit constant in the big-\(O\) notation
depends on a weighted Sobolev norm of the initial data. Moreover, the estimate
\cref{intro-thm-estimate} is stable under differentiation by \(r\partial{}_v\) and
\(uT\), where \(r\) is the area-radius function associated to the orbits of
radial symmetry and \(T\) is the Killing vector field associated to
stationarity.
\label{main-theorem-rough}
\end{theorem}
See \cref{main-theorem} for a more precise statement.
\begin{remark}[The leading-order profile]
The explicit function \(u^{-1/2}v^{-1/2}\) that appears in
\cref{main-theorem-rough} solves the wave equation on exact Minkowski space (where
\(u = t-r\) and \(v = t + r\)). For this reason \(u^{-1/2}v^{-1/2}\) is a
natural candidate for the leading-order profile of general solutions on
spacetimes that are small perturbations of Minkowski space.
\end{remark}
\begin{remark}[The representation formula on Minkowski space]
On exact Minkowski space, one could in principle deduce \cref{main-theorem-rough} from the
representation formula for solutions to the wave equation, but this method does
not apply to the more general backgrounds we consider.
\end{remark}
\begin{remark}[Contrast with the strong Huygens principle in odd space dimensions]
Solutions to the linear wave equation on \((n + 1)\)-dimensional Minkowski space
for \(n\ge 3\) odd obey the strong Huygens principle, and so exhibit no
non-trivial late-time asymptotics when their Cauchy data is compactly supported.
The setting of \cref{main-theorem-rough} (specialized to Minkowski space) is
therefore the simplest possible for a late-time tails result.
\end{remark}
\begin{remark}[A comment on the assumptions]
We must assume that the perturbation is small to prevent trapping or bound
states, which obstruct decay. The assumption of stationarity is natural for
potential applications that involve the stability of steady states. In such
problems, the background will be asymptotically stationary. The methods we use
have been applied to spacetimes settling down appropriately to the stationary
Schwarzschild spacetime in \cite{gajic2025linearnonlinearlatetimetails}, and so
our methods could be applied to backgrounds that are only asymptotically
stationary. Finally, the assumption that the background metric is radially
symmetric is a significant restriction of our proof. We use this assumption to
split the scalar field into its radially symmetric part \(\varphi_0\) and
non-radially symmetric part \(\varphi_{\ge 1}\) and construct the good scalar
field \(\Psi_0\coloneqq{}r^{1/2}\partial_r\varphi_0\). On non-symmetric
backgrounds, these quantities would satisfy suitable equations near infinity,
but not in a compact set, where many of the difficulties lie.
\end{remark}
\begin{remark}[A comment on the proof]
The main difficulties in this problem arise from an inverse-square potential of
a bad sign in the equation for \(r^{1/2}\varphi{}\). The inverse-square
potential is scale-critical and moreover appears with the critical constant
\(-1/4\).\sidenote{Here ``scale-critical'' means that the inverse-square potential in the
equation for \(r^{1/2}\varphi{}\) scales the same way as the derivative terms in
the spatial part of the operator in the equation for \(r^{1/2}\varphi{}\),
namely \(\partial{}_r^2 + r^{-2}\partial{}_\theta^2\) (on Minkowski space).
Inverse-square potentials are also critical for the behaviour of the spectrum of
the corresponding Schrödinger operator. See \cref{intro-spectral,intro-hardy} for further
discussion.} For this reason, techniques that are standard in \((3 + 1)\)
dimensions, such as integrated local energy decay estimates and \(r^p\)-weighted
energy estimates, degenerate in \((2 + 1)\) dimensions. The main new steps of
our proof are:
\begin{enumerate}
\item an extension of the \(r^p\)-weighted energy method of Dafermos--Rodnianski
\cite{rp-method} and methods of Gajic \cite{GAJIC2023110058} to \((2 + 1)\)
dimensions,
\item and a resolution of the difficulty caused by the inverse-square potential
with critical constant \(-1/4\) by introducing the commuted quantity
\(\Psi{}_0\coloneqq{}r^{1/2}\partial{}_r\varphi{}_0\). Here \(\varphi_0\) is
the radially symmetric part of \(\varphi{}\) and \(\partial_r\) is the
coordinate derivative in \((t,r)\) coordinates. The quantity \(\Psi_0\) is
good because it solves a wave equation with inverse-square potential with the
constant \(+ 3/4\). This means we can prove estimates for
\(\Psi_0\) using standard techniques. We then prove estimates for
\(\varphi_0\) by treating \(\Psi_0\) as a controlled source term.
\end{enumerate}
See \cref{intro-proof-summary} for a more detailed summary of the proof.
\label{proof-comment}
\end{remark}
\subsection{Related results}
\subsubsection{Late-time tails for linear and nonlinear waves in three space dimensions}
\label{intro-tails-3}
There is a vast literature studying the late-time asymptotics of linear and
nonlinear waves in three (and higher) space dimensions, including
\cite{hintz2023linearwavesasymptoticallyflat,luk-oh-tails,gajic2025linearnonlinearlatetimetails},
which treat general dynamical and nonlinear settings. We refer to
\cite[Sec.~1.3]{luk-oh-tails} for a more detailed overview of the literature,
including upper bound results.

Much of the literature studies the wave equation on black hole spacetimes.
Heuristics for the sharp decay rate of linear waves on the Schwarzschild spacetime of
three space dimensions go back to Price \cite{PhysRevD.5.2439}. The sharp
pointwise asymptotics on various black hole spacetimes, including the
subextremal Kerr background, were proven mathematically in
\cite{MR3725885,ANGELOPOULOS2023108939,Angelopoulos_2023} and in
\cite{MR4365146}. A Price's law result was established in a nonlinear and
spherically symmetric setting in \cite{gautam2024latetimetailsmassinflation}.
\subsubsection{Linear wave equations with an inverse-square potential in three space dimensions}
\label{inverse-square-potential} The first results, namely Strichartz and dispersive
estimates, for wave equations with an inverse-square potential in three space
dimensions are due to \cite{Tahvildar-Zadeh2002-pk,BURQ2003519,Burq20041665}.
The work of Gajic \cite{GAJIC2023110058} establishes late-time asymptotics for
solutions to the wave equation with an inverse-square potential on the
Schwarzschild spacetime (of three space dimensions). In
\cite{hintz2023linearwavesasymptoticallyflat}, similar results are obtained in
all higher dimensions. Another proof is given in \cite{Baskin_2025}, working in
the framework of \cite{baskin-wunsch-vasy-2015} (and relying on
\cite{BASKIN2022108589}), which establishes sharp decay rates for such equations
(as well as for solutions of the massless Dirac--Coulomb system). We also
mention the work \cite{gajic2024latetimetailsscaleinvariantwave}, which uses the
conformal embedding of \((3 + 1)\)-dimensional Minkowski space into
\(\mathrm{AdS}_2\times S^2\) to derive late-time asymptotics for such equations.
See \cite{Moortel2025-sg} for an extension of the \(r^p\)-weighted energy method
to linear wave equations on spherically symmetric spacetimes with small
inverse-square potentials (as well as small scale-critical first-order terms)
that may oscillate in time.

We also mention the works \cite{Donninger_2010,Costin2014-za}, which study the
sharp decay rates for the \emph{one-dimensional} wave equation with a potential \(V\)
satisfying \(\abs{V(x)}\lesssim (1+x)^{-p}\) for \(p > 2\), and the works
\cite{COSTIN20082321,rk-fourier-wave} which consider the case \(p = 2\). The
estimates in \cite{rk-fourier-wave} were used in \cite{Donninger_2016} to prove
the codimension-one stability of the catenoid of two space dimensions as a
stationary solution of the hyperbolic vanishing mean curvature equation to
perturbations that are radially symmetric and satisfy an additional discrete
symmetry. Finally, we refer to the survey article \cite{Schlag_2021} for further
references.
\subsubsection{Linear wave equations in two space dimensions}
\label{intro-2d}
We first discuss results that use physical space methods, as we do here. The recent work
\cite{ikehata2025localenergydecay2d} established that for a class of wave
equations with variable coefficients (without a symmetry condition as we impose
here), the \emph{local} energy (namely the energy restricted to a compact region of
space) of the solution decays like \(t^{-2}\log t\), as long as the data is
compactly supported; note that \cref{main-theorem-rough} provides the sharp rate
\(t^{-2}\). In \cite{wong2019smalldataglobalexistence}, Wong obtained the
almost-sharp \emph{global} decay rate \(u^{-1/2}v^{-1/2}\log ^2(uv)\) for solutions to
the wave equation on \((2 + 1)\)-dimensional Minkowski space using the conformal
(or Morawetz) multiplier \(u^2\partial_u + v^2\partial_v\), again under a
compact support assumption; the associated energy estimates are then used to
study the global existence problem for wave maps. In
\cite{metcalfe-hepditch-obstacle}, the authors use the multiplier \(r(\log
r)^q\partial_v\) for \(q\in (0,1)\) to obtain an integrated energy estimate in
the context of the Dirichlet problem for the wave equation outside an obstacle
containing the origin in two space dimensions.

There is also a large body of work using spectral methods to study linear wave
equations with a potential on \((2 + 1)\)-dimensional Minkowski space, namely
\((\Box{}_{m}-V)\varphi{} = 0\). In this setting, when \(\abs{V}\lesssim (1 +
r)^{-\beta{}}\) for some \(\beta{} > 2\), the spectrum of the stationary
(spatial) operator, namely the Hamiltonian \(-\Lapl _{\R^2}+V\), consists of
finitely many non-negative eigenvalues and the absolutely continuous spectrum
\([0,\infty)\). The negative eigenvalues lead to exponentially growing mode
solutions, and the zero eigenvalues are obstructions to decay. For this reason,
one first projects away the non-positive eigenvalues (thereby considering
solutions orthogonal to the corresponding eigenfunctions), or demands that the
zero be a ``regular point'' of the spectrum of \(H = -\Lapl _{\R^2}+V\). This
means that \(H\) has no resonances or zero-energy eigenfunctions.\sidenote{Let \(w\neq{}0\) be a distributional solution of \(Hw = (-\Lapl _{\R^2}+V)w = 0\). We say
\(w\) is an ``s-wave'' resonance if \(w\in L^\infty\) and \(w\notin L^p\) for any
\(1\le p<\infty\). If \(w\in L^p\) for all \(p > 2\) but \(w\notin L^2\), then
\(w\) is called a ``p-wave'' resonance. Finally, \(w\) is a zero-energy
eigenfunction if \(w\in L^2\).} This
assumption is \emph{not} satisfied in the setting of \cref{main-theorem-rough}, since the
constant function \(1\) is an ``s-wave'' resonance \(H = -\Lapl_{\R^2}\). When
this assumption is satisfied, one can obtain better decay results for linear
wave equations with a potential: \cite{Kopylova2010-dj,Green_2014}
(see also the high-frequency analysis of \cite{Moulin2009-uy}) show a pointwise
decay rate \(\abs{\varphi{}}\lesssim t^{-1}(\log t)^{-2}\) near \(\set{r=0}\),
which is faster than the rate given by \cref{main-theorem-rough} by \((\log t)^2\).

We also remark that \cite{BECEANU20165378} proves Strichartz estimates. The
methods in these works are related to the study of zero-energy resolvent
expansions for \(-\Lapl _{\R^2}+ V\) (see
\cite{doi:10.1142/S0129055X01000843,Erdo_an_2013}). Finally, we mention the work
\cite{FANELLI2022108333}, which studies wave equations in two dimensions with an
inverse-square \emph{electromagnetic} potential.
\subsection{Outline of the proof}
\label{proof-outline} We now outline the proof of \cref{main-theorem-rough} when the
background is Minkowski space (and when \(N = M = 0\)). Let \(\varphi{}\) be a
solution to the wave equation \(\Box{}_{m}\varphi{} = 0\). We will use the
following notation for scalar fields associated to \(\varphi{}\):
\begin{center}
\begin{tabular}{ll}
Notation & Meaning\\
\hline
\(\varphi{}\) & solution to \(\Box_m\varphi{} = 0\)\\
\(\varphi_0\) and \(\varphi_{\ge 1}\) & \(\varphi_0\) is the radially symmetric part of \(\varphi{}\), and \(\varphi_{\ge 1}\coloneqq{}\varphi{}-\varphi_0\)\\
\(\mathfrak{L}[\varphi{}]\) & a linear functional of the initial data for \(\varphi{}\) defined in \cref{L-frak-def}\\
\(\varphi_{\textnormal{mink}}\) & the leading-order profile \(u^{-1/2}v^{-1/2}\), which solves the wave equation on Minkowski space\\
\(\widehat{\varphi{}}\) & \(\widehat{\varphi{}}\coloneqq{}\varphi{}-\mathfrak{L}[\varphi{}]\varphi_{\textnormal{mink}}\) is the renormalized solution\\
\(\psi{}\) & \(\psi{}\coloneqq{}r^{1/2}\varphi{}\) achieves a finite non-zero limit at future null infinity\\
\(\psi{}_0\) and \(\psi_{\ge 1}\) & defined analogously to \(\varphi_0\) and \(\varphi_{\ge 1}\)\\
\(\Psi_0\) & \(\Psi_0\coloneqq{}r^{1/2}\partial_r\varphi_0\), where \(\partial_r\) is the coordinate derivative in \((t,r,\theta{})\) coordinates\\

\end{tabular}
\end{center}
As we will explain, the scalar fields \(\psi_{\ge 1}\) and \(\Psi_0\) are ``good,'' while
the radially symmetric scalar field \(\varphi_0\) is ``bad.'' The good quantities
\(\psi_{\ge 1}\) and \(\Psi_0\) satisfy estimates that are standard in \((3 +
1)\) dimensions and which fail for the bad quantity \(\varphi_0\). We will also
write \(\underline{L} = \partial_u\) and \(L = \partial_v\) for the coordinate
derivatives in \((u,v,\theta{})\) coordinates.
\subsubsection{A brief summary of the proof}
\label{intro-proof-summary} Our goal is to prove a late-time tails result for the
linear wave equation in \((2 + 1)\) dimensions. There are many such results
in higher dimensions, even in nonlinear settings (see \cref{intro-tails-3}). We
highlight in particular the work of Gajic \cite{GAJIC2023110058} on wave
equations with an inverse-square potential in \((3 + 1)\) dimensions, which
makes crucial use of the Dafermos--Rodnianski \(r^p\)-weighted estimates
\cite{rp-method}. However, the methods of \cite{rp-method,GAJIC2023110058} do not
apply directly to our problem.

The main challenge in two space dimensions appears at the level of the radially
symmetric part of \(\varphi{}\), which we call \(\varphi_0\). The quantity
\(r^{1/2}\varphi_0\) solves an equation with an inverse-square potential of a
\emph{bad sign} (with the critical constant \(-1/4\)), which obstructs standard
Morawetz and \(r^p\)-weighted estimates.\sidenote{In higher even spatial dimensions, the inverse-square potential
has a good sign, and in higher odd spatial dimensions, the inverse-square
potential does not appear. For this reason, one can establish Morawetz and
\(r^p\) estimates in higher dimensions.} The mechanism, which we identify
in \cref{axisymmetric-morawetz-failure}, is the scale invariance of the energy norm
\(\dot{H}^1(\R^2)\) in two space dimensions. On Minkowski space, one can
nevertheless prove \(r^p\)-weighted energy estimates for \(\varphi{}_0\), due to
an exact cancellation that arises when \(p = 1\) (see \cref{mink-zo}). This
cancellation does not persist on perturbations of Minkowski space.

We overcome this difficulty by identifying two ``good'' quantities that solve wave
equations with (effective) potentials of a favourable sign. The first is
\(r^{1/2}\varphi{}_{\ge 1}\coloneqq{}r^{1/2}(\varphi{}-\varphi{}_0)\), and the second, more important one, is \(\Psi_0 \coloneqq{}r^{1/2}\partial_r\varphi_0\).
One can close energy estimates for these good quantities using standard
techniques. Since we can decompose \(\varphi{} = \varphi_0 + \varphi_{\ge 1}\),
it remains to estimate \(\varphi_0\). We prove estimates for \(\varphi{}_0\) by
coupling them to estimates for the good quantity \(\Psi_0\). That is, once
\(\Psi_0\) is estimated, we can treat it as an acceptable error term in
estimates for \(\varphi_0\). We use this strategy to derive novel integrated
local energy decay and \(r\)-weighted estimates for \(\varphi_0\).

To turn our energy estimates into \emph{sharp} pointwise decay results, we subtract
from \(\varphi{}\) a suitable multiple of the desired leading-order profile
\(\varphi_{\textnormal{mink}}\coloneqq{}u^{-1/2}v^{-1/2}\) (which solves the
wave equation on \((2 + 1)\)-dimensional Minkowski space) and consider instead
the \emph{renormalized quantity}
\(\widehat{\varphi{}}\coloneqq{}\varphi{}-\mathfrak{L}[\varphi{}]\varphi_{\textnormal{mink}}\). The suitable multiple
\(\mathfrak{L}[\varphi{}]\) (introduced in \cref{main-theorem-rough}) is chosen so
that one can construct a \emph{time integral} of \(\widehat{\varphi{}}\), namely a
function \(T^{-1}\widehat{\varphi{}}\) such that \(TT^{-1}\widehat{\varphi{}} =
\widehat{\varphi{}}\). The reason we construct a time integral of
\(\widehat{\varphi{}}\), thereby expressing it as a time derivative of a
solution to \cref{wave-equation-intro}, is that time derivatives decay faster. The
improved decay we therefore obtain for \(\widehat{\varphi{}} =
\varphi{}-\mathfrak{L}[\varphi{}]\varphi_{\textnormal{mink}}\) means that the
leading-order asymptotics of \(\varphi{}\) are indeed given by
\(\mathfrak{L}[\varphi{}]\varphi_{\textnormal{mink}}\).

We construct the time integral by solving the wave equation from initial data
obtained by solving an elliptic problem \(\mathcal{L}[T^{-1}\widehat{\varphi{}}]
= \mathcal{F}[\widehat{\varphi{}}]\), where \(\mathcal{L}\) is the spatial part
of the wave operator and \(\mathcal{F}\) is the first-order operator such that
\(\mathcal{L}[\varphi{}] = \mathcal{F}[T\varphi{}]\) for solutions of
\cref{wave-equation-intro}. Working with the renormalized quantity
\(\widehat{\varphi{}}\) (as opposed to the original solution \(\varphi{}\)) is
necessary because the image of \(\mathcal{L}\) (acting on radially symmetric
functions) has codimension one; the quantity \(\mathfrak{L}[\varphi{}]\) is
defined so that the source \(\mathcal{F}[\widehat{\varphi{}}]\) is in the image of
\(\mathcal{L}\).

We then show that each time derivative gains two powers of \(u\) in energy decay
(where we consider an energy defined along asymptotically null hypersurfaces).
That is, while the \(T\)-energy \(E[\varphi{}]\) decays like \(u^{-1}\), the
\(T\)-energy \(E[T\varphi{}]\) decays like \(u^{-3}\). It follows that the
pointwise norm of a time derivative decays faster by one power in \(u\). Since
\(\widehat{\varphi{}} = TT^{-1}\widehat{\varphi{}}\) is a time derivative, this
shows that \(\widehat{\varphi{}}\) decays faster in \(u\) than
\(\varphi_{\textnormal{mink}}\), which means that \(\varphi{}\) itself is
described at late times by a multiple of \(\varphi_{\textnormal{mink}}\) plus an
error term that decays faster in \(u\).

These ideas are discussed in more detail in
\cref{intro-difficulties,intro-Psi,intro-pointwise,intro-energy-decay,intro-time-inversion,intro-morawetz-radially-symmetric,intro-r-radially-symmetric}.
\subsubsection{Difficulties associated to the zeroth-order term of a bad sign}
\label{intro-difficulties}
The main difficulty in \((2 + 1)\) dimensions is the presence of zeroth-order
terms of a \emph{bad sign} that form an obstruction to integrated local energy decay
(or Morawetz) and \(r\)-weighted energy estimates. By integrated local energy
decay, we mean an estimate of the form
\begin{equation}\label{intro-morawetz}
\int_0^\infty \int _{\Sigma{}(\tau{})\cap \set{r\le R}} r(\partial{}\varphi{})^2 + \varphi^2\dd{}r\dd{}\theta{}\dd{}\tau{} \lesssim _R E[\varphi{}](0),
\end{equation}
where \(\tau{}\) is a time function with level sets \(\Sigma{}(\tau{})\) (in our proof we take
\(\tau{}\) to have hyperboloidal level sets, so that energies defined on the
foliation \(\Sigma{}(\tau{})\) will decay), \(R > 0\) is arbitrary, and
\(E[\varphi{}](0)\) is the initial \(T\)-energy (where \(T = \partial_t\) is the
stationary Killing vector field). By an \(r\)-weighted energy estimate, we mean
an estimate roughly of the form
\begin{equation}
\int _{\Sigma{}(\tau{}_2)} r^p(L\psi{})^2\dd{}r\dd{}\theta{} + \int_{\tau{}_1}^{\tau{}_2} \int _{\Sigma{}(\tau{})} pr^{p-1}(L\psi{})^2 + (2-p)r^{p-3}(\psi{}^2 + (\partial{}_\theta{}\psi{})^2)\dd{}r\dd{}\theta{}\dd{}\tau{}\lesssim \int _{\Sigma{}(\tau{}_1)} r^p(L\psi{})^2\dd{}r\dd{}\theta{}
\end{equation}
for \(0\le \tau_1\le \tau_2\) and \(0\le p\le 2\). Here we have written \(\psi{} \coloneqq{}r^{1/2}\varphi{}\) and
\(L=\partial_v\) for the coordinate derivative in \((u,v,\theta{})\)
coordinates, where \(u = \frac{1}{2}(t - r)\) and \(v = \frac{1}{2}(t + r)\) are
the standard double null coordinates on Minkowski space.

The bad zeroth-order term is manifest in the equation for \(\psi{} = r^{1/2}\varphi{}\),
which has an \emph{inverse-square potential} with constant \(-1/4\):
\begin{equation}\label{intro-psi-equation}
r^{1/2}\Box{}\varphi{}= -\underline{L}L\psi{}  -\alpha{}r^{-2}\psi{} + r^{-2}\partial{}_\theta^2\psi{}\quad \textnormal{for } \alpha{} = -\frac{1}{4}.
\end{equation}
Here we write \(\underline{L} = \partial_u\) and \(L = \partial_v\) for the coordinate
derivatives in \((u,v,\theta{})\) coordinates. Previous works have studied wave
equations with an inverse-square potential of the form \(\alpha{}r^{-2}\) for
\(\alpha{} > -1/4\) (see \cref{inverse-square-potential} for further discussion),
but \cref{intro-psi-equation} corresponds to the case where the parameter
\(\alpha{}\) takes the critical value \(-1/4\).
\begin{remark}[Criticality of inverse-square potentials of the form \(\alpha r^{-2}\) with \(\alpha = -1/4\)]
Inverse-square potentials are critical for the behaviour of the spectrum of the
corresponding Schrödinger operator in the following sense. The operator
\(-\Lapl + V\) acting on \(L^2(\R^3)\) has an infinite discrete spectrum when
\(V\sim \alpha{}r^{-p}\) as \(r\to \infty\) for \(p > 2\), while the discrete
spectrum is finite if \(p < 2\). In the critical case \(p = 2\), the spectral
properties of this operator depend on the value of \(\alpha{}\): when \(\alpha{}
< -1/4\), the operator is unbounded from below and fails to be essentially
self-adjoint, and when \(\alpha{}\ge -1/4\), the discrete spectrum is finite.
See \cite[Thm.~XII.6]{Reed1981-ph} for further discussion.
\label{intro-spectral}
\end{remark}
\begin{remark}[Relation between Hardy's inequality and the critical value \(\alpha{}=-1/4\)]
In \(3\) dimensions, \(4\) is the best constant in Hardy's inequality,
namely the estimate \(\norm{r^{-1}f}_{L^2(\R^3)}^2\le
4\norm{\partial{}f}_{L^2(\R^3)}^2\) for \(f\in C_c^\infty(\R^3)\). This means
that, when \(\alpha{} > -1/4\), bad zeroth-order terms in estimates for
solutions \(\varphi{}\) to \((\Box{}-\alpha{}r^{-2})\varphi{} = 0\) can be
absorbed by the remaining good derivative terms. Put another way, the operator
\(-\Lapl + \alpha{}r^{-2}\) acting on \(C_c^\infty(\R^3)\) is positive-definite
when \(\alpha{} > -1/4\) but fails to be positive-definite when \(\alpha{} = -1/4\).

The analogue of Hardy's inequality fails in two dimensions, due to the
non-integrability of \(1/r^2\) near the origin, but a logarithmically modified
version holds.
\label{intro-hardy}
\end{remark}
The obstruction caused by the bad zeroth-order term can be overcome when
\(\varphi{}_0\equiv 0\), where \(\varphi{}_0\) is the radially symmetric part of
\(\varphi{}\) (namely the average of \(\varphi{}\) over circles of constant
\(\tau{}\) and \(r\), where \(\tau{}\) is a time function). Indeed, for such
functions, which satisfy \(\varphi{} = \varphi_{\ge 1}\), where \(\varphi_{\ge
1} \coloneqq{} \varphi{} - \varphi_0\) is the projection of \(\varphi{}\) to
non-zero angular modes, the following Poincaré inequality on \(S^1\) holds:
\begin{equation}
\int _{S^1} (\partial{}_\theta{}\varphi{})^2\dd{}\theta{}\ge \int _{S^1} \varphi{}^2\dd{}\theta{}\textnormal{ when }\varphi{}=\varphi{}_{\ge 1}.
\end{equation}
It follows that the angular derivative term can absorb the zeroth-order term of
a bad sign, and so the \emph{effective potential} in the equation for \(\psi_{\ge 1}\)
is \(\alpha{}r^{-2}\) for \(\alpha{} = 3/4\), which has a good sign. Since the
effective value of \(\alpha{}\) is positive, one can establish integrated local
energy decay estimates and \(r^p\)-weighted energy estimates in the full range
\(0\le p < 2\) for \(\psi{}_{\ge 1}\) using the same proofs that work in
\((3+1)\) dimensions (see \cref{good-scalar-field-energy-estimates}).

On the other hand, the integrated local energy decay estimate of
\cref{intro-morawetz}, which controls in particular a zeroth-order term in the bulk
by the \(T\)-energy, \emph{fails} to hold for radially symmetric scalar fields in \((2 + 1)\)
dimensions, due to the scale invariance of the energy space \(\dot{H}^1(\R^2)\)
(see \cref{axisymmetric-morawetz-failure}).
\subsubsection{A quantity that solves an equation with a zeroth-order term of a good sign}
\label{intro-Psi} A key observation is that the quantity
\(\Psi{}_0\coloneqq{}r^{1/2}\partial_r\varphi{}_0\), where \(\partial_r =
\partial{}_v - \partial{}_u\) is the coordinate derivative in \((t,r,\theta{})\)
coordinates, solves an equation with a zeroth-order term of a \emph{good sign}, namely
\begin{equation}\label{intro-Psi-equation}
0 = \underline{L}L\Psi{}_0 + \frac{3}{4}r^{-2}\Psi{}_0.
\end{equation}
This is to be contrasted with the bad sign in the equation that \(\psi_0 = r^{1/2}\varphi_0\) satisfies, namely
\begin{equation}
0 = \underline{L}L\psi{}_0 - \frac{1}{4}r^{-2}\psi{}_0.
\end{equation}
In other words, commuting with \(r^{1/2}\partial_r\) converts the bad sign in front of
the inverse-square potential to a favourable sign. For this reason, one can
establish integrated local energy decay estimates, as well as \(r^p\)-weighted
energy estimates in the full range \(0<p < 2\) for \(\Psi_0\), which take the
form
\begin{equation}\label{intro-Psi-rp}
\begin{split}
\tilde{\mathcal{E}}_p[\Psi{}_0](\tau{}_2) + \int_{\tau{}_1}^{\tau{}_2} \int _{\Sigma{}(\tau{})} r^{p-1}(L\Psi{}_0)^2 + r^{p-3}\Psi{}_0^2\dd{}r\dd{}\theta{}\dd{}\tau{} \lesssim_p \tilde{\mathcal{E}}_p[\Psi{}_0](\tau{}_1),
\end{split}
\end{equation}
where we have written
\begin{equation}\label{intro-Psi-rp-energy}
\tilde{\mathcal{E}}_p[\Phi{}](\tau{}) \coloneqq{} \int _{\Sigma{}(\tau{})} r^p(L\Phi{})^2 + h(r)r^{p-2}\Phi{}^2\dd{}r\dd{}\theta{}.
\end{equation}
Here \(h(r)\) is a function that measures the rate at which the hyperboloidal
foliation \(\Sigma{}(\tau{})\) becomes null as \(r\to \infty\) (see
\cref{hyperboloidal-foliation}). For example, we can take \(h(r)\) to equal \((1 +
r)^{-2}\) for \(r\ge 1\). Similar estimates hold for \(\psi_{\ge 1}\) (where the
bulk term in \cref{intro-Psi-rp} and energy in \cref{intro-Psi-rp-energy} include
angular terms).
\subsubsection{An integrated energy estimate for radially symmetric scalar fields}
\label{intro-morawetz-radially-symmetric}
The wave equation for a radially symmetric scalar field \(\varphi\) can be written
\begin{equation}\label{intro-wave-equation}
\Box{}_m\varphi{} = -\partial{}_u\partial{}_v\varphi{} + r^{-1}\partial{}_r\varphi{}.
\end{equation}
To obtain an integrated energy estimate that controls derivatives of \(\varphi\) in
the bulk by the \(T\)-energy, we treat \(r^{-1}\partial_r\varphi =
r^{-3/2}\Psi\) in \cref{intro-wave-equation} as a source term. This turns the
principal part of the equation into \(\partial_u\partial_v\varphi{}\). We use a
standard multiplier for this principal part and treat the terms involving
\(\Psi{}\) as an error controlled by the \(r^p\)-weighted energy estimates
established in \cref{intro-Psi}. This yields the novel estimate
\begin{equation}\label{intro-morawetz-Psi}
\begin{split}
\int_{\tau{}_1}^{\tau{}_2} \int _{\Sigma{}(\tau{})} (1+r)^{-1+\delta{}}(\partial{}\varphi)^2\dd{}r\dd{}\theta{}\dd{}\tau{} &\lesssim_{\delta{}} E[\varphi{}](\tau{}_1) + \tilde{\mathcal{E}}_{1+\delta{}}[\Psi{}](\tau{}_1).
\end{split}
\end{equation}
for \(0\le \tau_1\le \tau_2\). Thus, although integrated local energy decay controlling a
zeroth-order bulk term fails directly for radial waves (see
\cref{axisymmetric-morawetz-failure}), one still obtains a weaker integrated
estimate controlling only derivative terms after coupling the estimate to the
auxiliary quantity \(\Psi{}\).

This estimate is established in \cref{eb-radially-symmetric}.
\subsubsection{\(r\)-weighted energy estimates for radially symmetric scalar fields}
\label{intro-r-radially-symmetric} The usual proofs \cite{rp-method,moschidis-rp} of
\(r^p\)-weighted energy estimates in \((3 + 1)\) dimensions, which are very
robust, fail on general backgrounds in \((2 + 1)\) dimensions, and so we need to
introduce new techniques. Recall that, in these proofs, one uses a multiplier
\(r^p\partial_v\) (where \(p\in (0,2)\)) in a large-\(r\) region and controls
the error terms in the complementary finite-\(r\) region using an integrated
local energy decay estimate. As we show in \cref{axisymmetric-morawetz-failure},
the latter estimate (which would control a zeroth-order term in the bulk) does
not hold in \((2 + 1)\) dimensions. The obstruction to proving integrated local
energy decay and \(r^p\)-weighted estimates in \((2 + 1)\) dimensions is
precisely due to the zeroth-order term associated to \(\varphi_0\), the radially
symmetric part of \(\varphi{}\). For this reason, we assume for the rest of this
section that \(\varphi{} = \varphi_0\).

The zeroth-order term obstructing an \(r^p\)-weighted energy estimate happens to
vanish on exact Minkowski space when \(p = 1\)! One can exploit this
cancellation to establish \(r^p\)-weighted energy estimates on Minkowski space
in the full range \(p\in [1,2)\) (see \cref{mink-zo}).

On perturbed backgrounds, the cancellation that occurs on Minkowski space
disappears, producing zeroth-order bulk terms that, although small, we cannot
control. For this reason, we are \emph{unable} to establish \(r\)-weighted energy
estimates for \(\varphi{}\) itself in our setting. Nevertheless, we can close
estimates for derivatives of \(\varphi{}\), namely for \(T\varphi{}\) and
\((rL)\varphi{}\). Replacing \(\varphi{}\) with \(T\varphi{}\) turns the
dangerous zeroth-order term into a derivative bulk term that one can control
using the integrated energy estimate of \cref{intro-morawetz-radially-symmetric}.
On the other hand, computing with \((rL)\) generates bulk terms of a good sign
that compensate for the loss of the cancellation in the zeroth-order term.

We now explain how the cancellation on Minkowski space arises, why it fails on
general backgrounds, and how this failure can be circumvented using commutation.
We use \(f(r)L\psi{}\) as a \emph{global} multiplier for \cref{intro-psi-equation}, all the
way to \(\set{r=0}\). After differentiating by parts multiple times and writing
\(f(r) = rg(r)\), we obtain the identity
\begin{equation}\label{intro-rp-identity}
-2fL\psi{}r^{1/2}\Box{}_{m}\varphi{} = \underline{L}(f(L\psi{})^2) + \frac{1}{4}L((2rg' + g)\varphi^2) + f'(L\varphi{})^2 + \frac{1}{2}(-rg'' - g')\varphi^2.
\end{equation}
When \(f(r)\equiv r\), i.e.~\(g(r)\equiv 1\), the zeroth-order term in the identity
\cref{intro-rp-identity} vanishes,\sidenote{The two functions \(f(r) = rg(r)\) for which the zeroth-order term in
\cref{intro-rp-identity} vanishes are \(f(r) = r\) and \(f(r) = r\log r\). The
multipliers that interpolate between these two endpoints, namely \(f(r)L\) with
\(f(r) = r(\log r)^q\) for \(q\in{}(0,1)\), produce a zeroth-order term with a
good sign in \cref{intro-rp-identity}.
However, these logarithmically weighted multipliers are not sufficiently regular
at \(\set{r=0}\). One can remove the issue of regularity near \(\set{r=0}\) by
considering the Dirichlet problem for the wave equation outside an obstacle
containing the origin, as is done in \cite{metcalfe-hepditch-obstacle}, which
uses the logarithmically weighted multipliers.} yielding a coercive estimate. On Minkowski
space, \(p = 1\) is the unique \(p\in (0,2)\) for which the zeroth-order term in
\cref{intro-rp-identity} vanishes when \(f(r) = r^p\), rather than contributing
with the wrong sign;\sidenote{This observation is consistent with \cite{GAJIC2023110058}, where
\(r^p\)-weighted energy estimates are derived for solutions to a wave equation
(on the Schwarzschild spacetime of \((3 + 1)\) dimensions) with an inverse
square potential \(\alpha{}r^{-2}\) (for \(\alpha{} > -1/4\)) in the range
\(p\in (1-\beta{},1 + \beta{})\) for \(\beta{} = \sqrt{1 + 4\alpha{}}\). The
case \(\alpha{} = -1/4\) arises in \((2 + 1)\) dimensions, and in this case we
have \(\beta{} = 0\).} concretely we have \(-rg'' - g' = -(p-1)^2r^{p-2}\)
when \(g(r) = r^{p-1}\).

On perturbations of Minkowski space, the zeroth-order term in (the analogue of)
\cref{intro-rp-identity} does not vanish, even when \(g(r)\equiv 1\)! Because a
Morawetz estimate controlling a zeroth-order term in the bulk is not available,
we treat this term as an error, so that \cref{intro-rp-identity} produces (for
\(f(r)\equiv r\)) an estimate of the form
\begin{equation}\label{intro-rp-estimate}
\begin{split}
&\mathcal{E}_{1}[\varphi{}](\tau{}_2) + \int_{\tau{}_1}^{\tau{}_2} \int _{\Sigma{}(\tau{})} r(L\varphi{})^2\dd{}r\dd{}\theta{}\dd{}\tau{} \\
&\le A\mathcal{E}_{1}[\varphi{}](\tau{}_1) +  B\epsilon{}\int_{\tau{}_1}^{\tau{}_2} \int _{\Sigma{}(\tau{})} (1+r)^{-2}\varphi^2\dd{}r\dd{}\theta{}\dd{}\tau{} - C\int_{\tau{}_1}^{\tau{}_2} \int _{\Sigma{}(\tau{})} 2rL\psi{}r^{1/2}\Box{}_{m}\varphi{} \dd{}r\dd{}\theta{}\dd{}\tau{}.
\end{split}
\end{equation}
Here \(A,B,C > 0\) are constants, \(\epsilon{}\) measures the size of the metric
perturbation, and we have written
\begin{equation}\label{intro-E1}
\mathcal{E}_1[\varphi{}](\tau{})\coloneqq{}\int _{\Sigma{}(\tau{})} r(L\psi{})^2 + h(r)\varphi^2\dd{}r\dd{}\theta{},
\end{equation}
where \(h\) is as in \cref{intro-Psi}. Because of the zeroth-order bulk term on the
right-hand side of \cref{intro-rp-estimate}, we are not able to obtain
\(r\)-weighted energy estimates for \(\varphi{}\) itself on a general background.

We now explain how to close the estimates for \(T\varphi{}\) and \((rL)\varphi{}\). First, we
apply \cref{intro-rp-estimate} with \(T\varphi{}\) in place of \(\varphi{}\). When
\(\Box{}\varphi{} = 0\), we have \(\Box{}T\varphi{} = 0\), and so the final term
on the right-hand side of \cref{intro-rp-estimate} vanishes. The bulk term with an
\(\epsilon{}\)-factor (involving now \(T\varphi{}\)) can be controlled by the
\(T\)-energy of \(\varphi{}\) and an \(r^p\)-weighted energy of \(\Psi{}\), in
view of the discussion in \cref{intro-morawetz-radially-symmetric}. In conclusion,
we obtain an estimate
\begin{equation}\label{intro-E1-Tphi}
\mathcal{E}_1[T\varphi{}](\tau_2) + \int_{\tau_1}^{\tau_2} \int _{\Sigma{}(\tau{})} r(LT\varphi{})^2\dd{}r\dd{}\theta{}\dd{}\tau{} \lesssim \mathcal{E}_1[T\varphi{}](\tau_1) + E[\varphi{}](\tau{}_1) + \textnormal{terms involving }\Psi{}.
\end{equation}

Next, we apply \cref{intro-rp-estimate} with \((rL)\varphi{}\) in place of \(\varphi{}\). The
integrand of the final term on the right-hand side of \cref{intro-rp-estimate}
becomes \(-2rL(r^{1/2}(rL)\varphi{})\Box_m(rL)\varphi{}\). Expanding this and
commuting \((rL)\) with \(\Box_m\) produces bulk terms involving \(L(rL)^{\le
1}\varphi{}\) of a favourable sign and terms involving \(\Psi{}\). The good bulk
terms are enough to compensate for the term in \cref{intro-rp-estimate} with an
\(\epsilon{}\)-factor. After treating all terms involving \(\Psi{}\) as error,
we obtain the estimate
\begin{equation}\label{intro-rp-rL-estimate}
\mathcal{E}_1[(rL)\varphi{}](\tau{}_2) + \sum_{n=0}^1\int_{\tau{}_1}^{\tau{}_2} \int _{\Sigma{}(\tau{})} r(L(rL)^n\varphi{})^2\dd{}r\dd{}\theta{}\dd{}\tau{} \lesssim \mathcal{E}_1[(rL)\varphi{}](\tau{}) + \textnormal{terms involving }\Psi{}.
\end{equation}
We have therefore proven \(r\)-weighted estimates for derivatives of \(\varphi{}\), and
these estimates are sufficient to prove \cref{main-theorem-rough}.
\begin{remark}[Estimates with stronger \(r\)-weights]
In order to derive pointwise estimates, one also needs control of
\(r^p\)-weighted energies for \(p = 1 + \delta{}\), where \(\delta{} > 0\) is
small. Using \(f(r) = r(1 + r)^\delta{}\) in \cref{intro-rp-identity} (where
\(\delta{} > 0\)) produces an additional zeroth-order error term in the bulk.
Since this term comes with a \(\delta{}\)-factor, it can be absorbed as above
when \(\delta{}\) is small.
\end{remark}
These estimates are established in \cref{r-weighted-estimates}.
\subsubsection{Improved energy decay for time derivatives}
\label{intro-energy-decay} It is well-known since \cite{rp-method} that
\(r^p\)-weighted energy estimates of the type proven in
\cref{intro-r-radially-symmetric} (with \(p = 1\)), lead to \(\tau^{-1}\) decay for
the \(T\)-energy, via what we will call in this section a ``pigeonhole argument''.
However, the decay rate \(\tau^{-1}\) for the \(T\)-energy of \(\varphi{}\) is
too slow to prove \cref{main-theorem-rough}, since the best pointwise decay it can
imply for \(\varphi{}\) is \(\tau^{-1/2}\), while \cref{main-theorem-rough} shows
that \(\varphi{}\) decays like \(\tau^{-1}\) near \(\set{r=0}\). In \((3 +
1)\)-dimensions, one could improve the decay in \(\tau{}\) for \(\varphi{}\) by
increasing the range of \(p\) for which one applies \(r^p\)-weighted estimates.
However, as discussed in \cref{intro-r-radially-symmetric}, such an increased range
is not available in \((2 + 1)\) dimensions.

To obtain improved decay for \(\varphi{}\), we first obtain improved decay for
\(T\varphi{}\), whose \(T\)-energy decays two powers faster, like \(\tau^{-3}\).
This is helpful because we will eventually write a renormalized version of
\(\varphi{}\) as a time derivative (see \cref{intro-time-inversion}).

We first discuss the improved decay of \(T\)-derivatives of the ``good scalar
field'' \(\Psi_0\). The same argument works for \(\psi{}_{\ge 1}\). The methods
in this case are standard (see \cite[Sec.~8]{GAJIC2023110058}). Improved decay for
\(T\)-derivatives (after commuting with \((rL)\)) goes back to
\cite{Schlue_2013,moschidis-rp}, and appears in
\cite{Angelopoulos2018-dz,MR3725885,ANGELOPOULOS2023108939}. First, the
\(r^p\)-weighted energy estimate for \(\Psi_0\) with \(p = 1\) (see
\cref{intro-Psi-rp}) implies the estimate
\begin{equation}\label{intro-Psi-0-1-integrated}
\tilde{\mathcal{E}}_1[\Psi{}_0](\tau{}_2) + \int_{\tau{}_1}^{\tau{}_2} \tilde{\mathcal{E}}_0[\Psi{}_0](\tau{})\dd{}\tau{}\lesssim \tilde{\mathcal{E}}_1[\Psi{}_0](\tau{}_1)
\end{equation}
for \(\tau_1\le \tau_2\). A pigeonhole argument provides \(\tau^{-1}\) decay for
\(\tilde{\mathcal{E}}_0[\Psi_0](\tau{})\), which is at the level of the
\(T\)-energy. Then, one uses the equation for \(\Psi_0\) (see
\cref{intro-Psi-equation}) together with the expression \(T = L + \underline{L}\)
to estimate
\begin{equation}\label{intro-improved-energy-time}
r^2(LT\Psi{}_0)^2 \lesssim \sum_{n=0}^1 (L(rL)^n\Psi{}_0)^2 + r^{-2}\Psi{}_0^2.
\end{equation}
When integrated in the bulk, the left-hand side of
\cref{intro-improved-energy-time} corresponds to the \(p = 3\) energy of
\(T\Psi_0\), while the terms on the right-hand side of
\cref{intro-improved-energy-time} are controlled by the \(p = 1\) energy of
\(\Psi_0\) (see \cref{intro-Psi-rp}). In this way, one can exchange time
derivatives in an energy norm for two powers of \(r\), at the cost of one
commutation with \((rL)\). Heuristically, gaining two powers of \(r\)
corresponds to ``increasing the range of an \(r^p\)-weighted energy hierarchy by
two,'' which corresponds to gaining two powers of decay in \(\tau{}\). Indeed, we
obtain a \(\tau^{-3}\) decay rate for the \(T\)-energy of \(T\Psi_0\) (compared
to the \(\tau^{-1}\) decay rate for the \(T\)-energy of \(\Psi_0\) itself),
using only \(r^p\)-weighted energies with \(p = 1\).

We now discuss the estimates for radially symmetric scalar fields
\(\varphi{}=\varphi_0\), which do not appear in previous works. We want to show
that \(E[T\varphi{}_0]\) decays like \(\tau^{-3}\). As a first step, we can show
that an analogue of \cref{intro-Psi-0-1-integrated} holds for \(\varphi_0\):
\begin{equation}\label{intro-integrated-T}
E[(rL)^{\le 1}\varphi{}_0](\tau{}_2) + \int_{\tau{}_1}^{\tau{}_2} E[(rL)^{\le 1}\varphi{}_0](\tau{})\dd{}\tau{}\lesssim E[(rL)^{\le 1}\varphi{}_0](\tau{}_1) + \mathcal{E}_1[(rL)\varphi{}_0](\tau{}_1) + \textnormal{terms involving }\Psi{}_0
\end{equation}
for \(\tau_1\le \tau_2\). A pigeonhole argument extracts a decay rate \(\tau{}^{-1}\) for the
\(T\)-energy of \((rL)^{\le 1}\varphi_0\). From here, we cannot directly apply
the strategy used for \(\Psi_0\) and write \((LT\psi{}_0)^2 \lesssim
r^{-1}(L(rL)^{\le 1}\psi{}_0)^2 + r^{-2}\psi{}_0^2\) in analogy with
\cref{intro-improved-energy-time}. This is because we do not control a bulk term
involving \(L\psi_0\) (only one involving \(L\varphi{}\)) or, more importantly,
a zeroth-order bulk term involving \(\psi{}_0\). Instead, we estimate
\begin{equation}\label{intro-improved-energy-time-phi}
r^2(LT\psi{}_0)^2\lesssim r(L(rL)^{\le 1}\varphi{}_0)^2 + r^2(L\Psi{}_0)^2.
\end{equation}
After integrating in the bulk, the first term on the right-hand side of
\cref{intro-improved-energy-time-phi} can be controlled by the \(p=1\) energy of
\((rL)^{\le 1}\varphi_0\) (see \cref{intro-rp-rL-estimate}). Thus a gain in
\(r\)-powers, and hence a gain in \(\tau{}\)-decay, after commuting with \(T\)
holds for \(\varphi_0\) just as it does for \(\Psi_0\), as long as we can control
the term
\begin{equation}
\int_{\tau_1}^{\tau_2} \int _{\Sigma{}(\tau{})} r^2(L\Psi{}_0)^2\dd{}r\dd{}\theta{}\dd{}\tau{}.
\end{equation}
This bulk term corresponds to the \(p = 3\) energy of \(\Psi_0\). However, we would
like to only use \(p = 1\) energies, and even then, one can only control the
\(r^p\)-weighted energy of \(\Psi_0\) with \(p\in (0,2)\)! To control this term,
we write \(\Psi{}_0 = TT^{-1}\Psi{}_0\), where \(T^{-1}\Psi{}_0\) is a \emph{time
integral} of \(\Psi{}_0\) (see \cref{intro-time-inversion}), and exchange the time
derivative for two powers of \(r\)-decay by applying
\cref{intro-improved-energy-time} with \(T^{-1}\Psi_0\) in place of \(\Psi_0\). We
are then left with a term of the form
\begin{equation}\label{intro-L-Tinv-Psi}
\int_{\tau_1}^{\tau_2} \int _{\Sigma{}(\tau{})} (LT^{-1}\Psi{}_0)^2\dd{}r\dd{}\theta{}\dd{}\tau{},
\end{equation}
which is controlled by the \(p = 1\) energy of \(T^{-1}\Psi_0\).
\begin{remark}[Constructing two time integrals of \(\Psi_0\)]
When, after the discussion of \cref{intro-time-inversion}, we apply the analysis of
this section to \(T^{-1}\widehat{\varphi{}}_0\) in place of \(\varphi{}\) (where
\(\widehat{\varphi{}}\) is a suitably renormalized version of \(\varphi{}\)), we
will need to construct \emph{two} time integrals of the renormalized quantity
\(\widehat{\Psi{}}_0\). Although we cannot construct two time integrals of
\(\widehat{\varphi{}}_0\), this is possible for \(\widehat{\Psi{}}_0\) precisely
because this quantity solves an equation with an inverse-square potential of a
good sign.
\end{remark}
In the end, we can show that \(\mathcal{E}_1[(rL)T\varphi{}]\) decays like \(\tau^{-2}\)
(a gain in two powers over \(\mathcal{E}_1[(rL)\varphi{}]\), which is merely
bounded). On the other hand, since the estimate for
\(\mathcal{E}_1[T\varphi{}]\) sees \(E[\varphi{}]\) (which decays like
\(\tau^{-1}\)) on the right-hand side, we do not obtain an estimate better than
\(\tau^{-1}\) for \(\mathcal{E}_1[T\varphi{}]\).\sidenote{The \(\mathcal{E}_1[T\varphi{}_0](\tau{})\lesssim \tau^{-1}\) stands in contrast to the
estimate we obtain \(\tilde{\mathcal{E}}_1[T\Psi_0]\), which decays like
\(\tau^{-2}\). By \cref{main-theorem-rough}, our estimate for
\(\mathcal{E}_1[T\varphi_0]\) is not sharp. Note that, on exact Minkowski space,
one can indeed prove that \(\mathcal{E}_1[T\varphi{}_0]\) decays like \(\tau^{-2}\).} Nevertheless, we can
obtain the sharp rate \(E[T\varphi{}]\lesssim \tau^{-3}\), since the estimate
\cref{intro-integrated-T} applied to \(T\varphi{}\) in place of \(\varphi{}\) does
not see \(\mathcal{E}_1[T\varphi{}]\), but only
\(\mathcal{E}_1[(rL)T\varphi{}]\), which indeed decays like \(\tau^{-2}\).

We can also derive higher-order versions of the above estimates. In summary, we
obtain the following estimates for radially symmetric scalar fields (where \(M\ge 0\)):
\begin{equation}\label{intro-energy-decay-rates}
E[\varphi{}](\tau{})\lesssim \tau^{-1},\quad \mathcal{E}_1[\varphi{}]\lesssim 1,\quad E[T^{M+1}\varphi{}]\lesssim \tau^{-3-2M},\quad  \mathcal{E}_1[T^{M+1}\varphi{}]\lesssim \tau^{-1-2M},\quad \mathcal{E}_1[(rL)T^{M+1}\varphi{}]\lesssim \tau^{-2-2M}.
\end{equation}
The crucial feature of the estimates in \cref{intro-energy-decay-rates} is the gain
of two powers of \(\tau{}\) in energy decay for each time derivative.

This step is carried out in \cref{sec:energy-decay}.
\subsubsection{Time inversion and subtraction of the Minkowskian solution}
\label{intro-time-inversion} The gain of two powers of decay for time derivatives
discussed in \cref{intro-energy-decay} motivates the construction of \emph{time integrals}
(``inverse time derivatives'') of the solution, as in
\cite{MR3725885,ANGELOPOULOS2023108939,GAJIC2023110058}, thereby converting
improved decay for time derivatives into improved decay for the solution itself.
That is, \(\varphi{}\) itself decays faster whenever one can construct its time
integral, namely a function \(T^{-1}\varphi{}\) such that \(TT^{-1}\varphi{} =
\varphi{}\), since then \(\varphi{}\) is a time derivative.

In \cite{GAJIC2023110058}, the time integral \(T^{-1}\varphi{}\) is constructed by
inverting the spatial part of the wave operator. One also constructs
\(T^{-1}\varphi_{\textnormal{mink}}\) (where \(\varphi{}_{\textnormal{mink}}
\coloneqq{} u^{-1/2}v^{-1/2}\) is the desired leading-order profile) and
subtracts a suitable multiple of \(T^{-1}\varphi_{\textnormal{mink}}\) from
\(T^{-1}\varphi{}\). In our case, we must work directly with the renormalized
quantity \(\widehat{\varphi{}} \coloneqq{} \varphi{} -
\mathfrak{L}[\varphi{}]\varphi_{\textnormal{mink}}\). This is because the
elliptic operator we invert to construct the time integral has a codimension one
image when acting on radially symmetric functions;\sidenote{The obstruction is given by the s-wave resonance discussed in \cref{intro-2d}.} the constant
\(\mathfrak{L}[\varphi{}]\) is chosen precisely so that the source term
corresponding to \(\widehat{\varphi{}}\) that we seek to invert lies in this
image.

To construct a time integral of \(\widehat{\varphi{}}\), we recast the wave
equation as an elliptic problem (coming from the spatial part of the wave
operator) with a source term (coming from the part of the wave operator
involving time derivatives). That is, we write
\begin{equation}
\mathcal{L}\psi{} = \mathcal{F}[T\psi{}] + r^{1/2}\Box{}_{m}\varphi{},
\end{equation}
where \(\mathcal{L}\) is a second-order elliptic operator involving only spatial
derivatives and \(\mathcal{F}\) is a first-order operator. We first construct
the solution \(\Phi{}\) to the elliptic problem
\begin{equation}
\mathcal{L}\Phi{} = \mathcal{F}[\widehat{\psi{}}]|_{\Sigma{}(0)},
\end{equation}
Here \(\widehat{\psi{}}\coloneqq{}r^{1/2}\widehat{\varphi{}}\) has been constructed to lie in the
kernel of the functional determining the invertibility of \(\mathcal{L}\). We
then solve the wave equation with initial data
\((r^{-1/2}\Phi{},\widehat{\varphi{}})\) prescribed on \(\Sigma{}(0)\).
\begin{remark}[Time integrals outside of Minkowski space]
Outside of Minkowski space, \(\varphi_{\textnormal{mink}} = u^{-1/2}v^{-1/2}\), and hence \(\widehat{\varphi{}}\), is only
an approximate solution to the wave equation. Since \(\widehat{\varphi{}}\) is
not a solution to the wave equation, all the estimates in earlier sections must
accommodate for inhomogeneities. Furthermore, we must directly construct and
estimate the time integral \(T^{-1}\Box{}_g\varphi_{\textnormal{mink}}\) and include
it as a source term in the wave equation used to construct
\(T^{-1}\widehat{\varphi{}}\).
\end{remark}
This step is carried out in \cref{sec:time-integrals}.
\subsubsection{Pointwise estimates}
\label{intro-pointwise} Since pointwise estimates for \(\varphi_{\ge 1}\) follow from
standard techniques, we focus this discussion on the novel estimates for
radially symmetric scalar fields and assume \(\varphi{} = \varphi_0\).

The basic idea is to interpolate between \(\tau{}\)-decay in a finite-\(r\) region
with a loss depending on the size of the region with \(r\)-decay in a
large-\(r\) region. In order to control \(\widehat{\varphi{}}\) pointwise in a
finite-\(r\) region via Sobolev embedding, we need to control
\(\widehat{\varphi{}}\) in \(L^2\) near the origin. However, the weakest (and
hence fastest decaying) energy with this control is
\(\mathcal{E}_1[\widehat{\varphi{}}]\), which decays like \(\tau^{-1}\), by the
results discussed in \cref{intro-energy-decay,intro-time-inversion}. This would
lead to an estimate \(\abs{\widehat{\varphi{}}}|_{r=0}\lesssim \tau^{-1/2}\).
Near \(r=0\), the Minkowskian solution \(\varphi_{\textnormal{mink}} =
u^{-1/2}v^{-1/2}\) decays like \(\tau^{-1}\), and so we cannot show that
\(\widehat{\varphi{}}\) decays faster in \(u\) than \(\varphi_{\textnormal{mink}}\)
using this estimate.

To remedy this issue, we first derive an estimate for \((rL)\widehat{\varphi{}}\). We
then integrate this estimate in the \(L\)-direction to null infinity, where
\(\widehat{\varphi{}} = 0\), to obtain an estimate in the region \(\set{r\ge
1}\). To estimate \(\widehat{\varphi{}}\) in the region \(\set{r\le 1}\), we
integrate an estimate for \(\widehat{\Psi{}} =
r^{1/2}\partial_r\widehat{\varphi{}}\) in the \(\partial_r\)-direction, which is
possible since \(r^{-1/2}\) is integrable near \(\set{r=0}\).

We now elaborate on this procedure. First, we obtain the following estimate:
\begin{equation}
\norm{(rL)\varphi{}}_{L^\infty(\Sigma{}(\tau{})\cap \set{r\le R})}^2\lesssim R^C(\mathcal{E}_1[T\varphi{}](\tau{}) + E[(rL)^{\le 1}T^{\le 1}\varphi{}](\tau{}))\qquad (R\ge 1).
\end{equation}
Here \(C > 0\) is a constant determined by the foliation (specifically the
function \(h\) discussed in \cref{intro-Psi}). Applying this estimate to
\(\widehat{\varphi{}} = TT^{-1}\widehat{\varphi{}}\), we obtain, by
\cref{intro-energy-decay-rates}, the decay rate
\begin{equation}\label{intro-pointwise-near}
\norm{(rL)\widehat{\varphi{}}}_{L^\infty(\Sigma{}(\tau{})\cap \set{r\le R})}^2\lesssim R^C\tau^{-3}.
\end{equation}
On the other hand, a standard estimate using the \(r^p\)-weighted energy for \(p
= 1 + \delta{}\) with \(\delta{} > 0\) gives
\begin{equation}
\abs{\varphi{}}^2\lesssim_\delta{} r^{-1}(\mathcal{E}_{1+\delta{}}[\varphi{}](\tau{}) + E[\varphi{}](\tau{})) \qquad \textnormal{for }\delta{}>0\textnormal{ in }\set{r\ge 1}.
\end{equation}
Applying this estimate to \((rL)\widehat{\varphi{}} = (rL)TT^{-1}\widehat{\varphi{}}\) and recalling
\cref{intro-energy-decay-rates}, we obtain
\begin{equation}\label{intro-pointwise-far}
\abs{(rL)\widehat{\varphi{}}}^2\lesssim r^{-1}\tau^{-2}\qquad \textnormal{in } \set{r\ge 1}.
\end{equation}
Interpolating between \cref{intro-pointwise-near} in a small-\(r\) region and
\cref{intro-pointwise-far} in the large-\(r\) region, we obtain
\begin{equation}
\abs{(rL)\widehat{\varphi{}}}\lesssim u^{-1/2}v^{-1/2}u^{-\delta{}}\quad \textnormal{ for some }\delta{} > 0\textnormal{ depending on }C.
\end{equation}
Integrating this estimate in the \(L\)-direction (that is, in the outgoing null
direction) to future null infinity, where \(\widehat{\varphi{}}\) vanishes, one
obtains
\begin{equation}\label{intro-pointwise-rate}
\abs{\widehat{\varphi{}}}\lesssim u^{-1/2}v^{-1/2}u^{-\delta{}/2}\qquad \textnormal{in }\set{r\ge 1}.
\end{equation}
To control \(\widehat{\varphi{}}\) in \(\set{r\le 1}\), we obtain a pointwise estimate
\begin{equation}
\abs{\widehat{\Psi{}}}^2\lesssim \tau^{-5}\qquad \textnormal{in }\set{r\le 1},
\end{equation}
which is two powers better than \cref{intro-pointwise-near} because we can
construct one more time integral of \(\widehat{\Psi{}}\) than of
\(\widehat{\varphi{}}\). We then recall the definition \(\widehat{\Psi{}} =
r^{1/2}\partial_r\widehat{\varphi{}}\) and integrate in the
\(\partial_r\)-direction (where \(\partial{}_{r}\) is the coordinate derivative
in \((t,r,\theta{})\) coordinates) to \(\set{r=1}\), using the integrability of
\(r^{-1/2}\) near \(\set{r=0}\). This extends the estimate
\cref{intro-pointwise-rate} to all values of \(r\).

These estimates are carried out in \cref{sec:late-time-asymptotics}.
\subsubsection{Diagrammatic overview of the proof}
In \cref{flowchart}, we summarize the dependencies between the energy estimates and
pointwise estimates used to prove \cref{main-theorem-rough} for the radially
symmetric scalar field \(\varphi_0\) and record where each estimate is proved.
Since the estimates for the non-radially symmetric part \(\varphi_{\ge 1} =
\varphi{} - \varphi_0\) are similar to the estimates for \(\Psi_0 =
r^{1/2}\partial_r\varphi_0\), we suppress their role here.
\begin{figure}[!htb]
    \centering
    \usetikzlibrary{shapes.geometric, arrows}
\tikzstyle{default} = [rectangle, rounded corners,
minimum width=3cm,
text width=4cm,
minimum height=1.5cm,
text centered,
draw=black,]

\tikzstyle{arrow} = [thick, ->, >=stealth, shorten >= 4pt, shorten <= 1pt]
\noindent
\begin{tikzpicture}[trim left, node distance=3cm]

\node (iled-phi) [default] {Integrated energy estimate for \(T^{\le 1}T^{-1}\widehat{\varphi{}}_0\) \\ (\cref{T-estimate-ho})};
\node (rp-Psi) [default, below of=iled-phi] {\(p=1 + \delta{}\) \mbox{energy} of \(T^{\le 2}T^{-2}\widehat{\Psi{}}_0\) \\ (\cref{rp-general})};
\node (big-bulk-Psi) [default, below of=rp-Psi] {``\(p=3 + \delta{}\)'' bulk term associated to \(T^{\le 2}T^{-1}\widehat{\Psi{}}_0\) \\ (\cref{energy-decay-radially-symmetric})};

\node (rp-T-phi) [default, right of=iled-phi, xshift=2.5cm] {\(p=1 + \delta{}\) \mbox{energy} of \(T^{\le 2}T^{-1}\widehat{\varphi{}}_0\) \\ (\cref{r-estimate-T})};
\node (pointwise-rL-phi) [default, below of=rp-T-phi] {Pointwise estimate for \((rL)\widehat{\varphi{}}_0\) in \(\{r\ge 1\}\) \\ (\cref{late-time-pointwise})};
\node (rp-rL-phi) [default, below of=pointwise-rL-phi] {\(p=1 + \delta{}\)
  \mbox{energy} \\ of \((rL)T^{-1}\widehat{\varphi{}}_0\) \\ (\cref{r-estimate})};

\node (pointwise-phi) [default, right of=rp-T-phi, xshift=2.5cm] {Pointwise estimate for \(\widehat{\varphi{}}_0\) \\ (\cref{pointwise-decay})};
\node (Psi-pointwise) [default, right of=pointwise-rL-phi, xshift=2.5cm] {Pointwise estimate for \(\widehat{\Psi{}}_0\) in \(\{r\le 1\}\) \\ (\cref{late-time-pointwise})};
\node (p0-Psi) [default, below of=Psi-pointwise] {\(p=0\) energy of \(T^{-1}\widehat{\Psi{}}_0\)\\ (\cref{rp-general})};

\draw [arrow] (rp-Psi) -- (iled-phi);
\draw [arrow] (rp-Psi) -- (big-bulk-Psi);
\draw [arrow] (iled-phi) -- (rp-T-phi);
\draw [arrow] (big-bulk-Psi) -- (rp-rL-phi);
\draw [arrow] (rp-rL-phi) -- (pointwise-rL-phi);
\draw [arrow] (rp-T-phi) -- (pointwise-rL-phi);
\draw [arrow] (p0-Psi) -- (Psi-pointwise);
\draw [arrow] (Psi-pointwise) -- (pointwise-phi);
\draw [arrow] (pointwise-rL-phi) -- (pointwise-phi);

\end{tikzpicture}
    \caption{A summary of the proof of the pointwise estimate for \(\widehat{\varphi{}}_0\) that, together with estimates for \(\varphi_{\ge 1}\), implies that the late-time asymptotics of \(\varphi{}\) are described
by \(\varphi_{\textnormal{mink}}\).}
    \label{flowchart}
\end{figure}
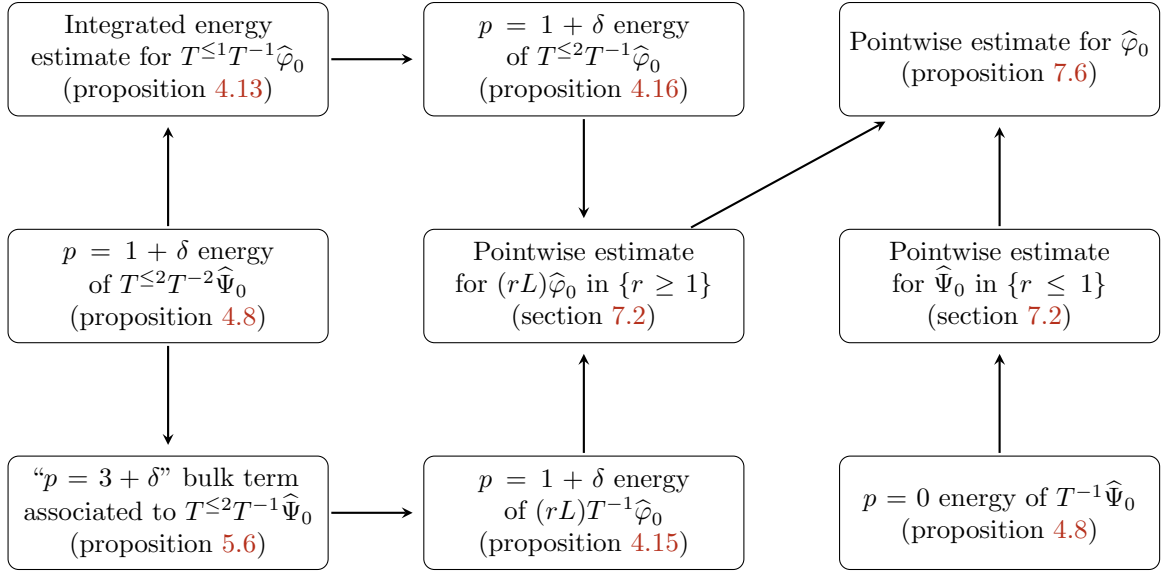
\subsection*{Acknowledgements}
We thank Dejan Gajic, Georgios Moschidis, and Igor Rodnianski for helpful
discussions. We thank Jonathan Luk and Sung-Jin Oh for pointing out references.
We thank Dejan Gajic and Igor Rodnianski for comments on the manuscript.
\section{Preliminaries}
\label{preliminaries}
\subsection{Notation and conventions for constants}
\label{sec:notation} For \(x\in \R\), we write
\(\langle{}x\rangle{}\coloneqq{}\sqrt{2 + x^2}\).

For \(g : (0,\infty)\to \R_{\ge 0}\), we write \(f(r) = O(g(r))\) to mean that \(f :
(0,\infty)\to \R\) is a smooth function such that there exists a constant \(C >
0\) for which \(\abs{f(r)}\le Cg(r)\). We write \(f(r) = \mathcal{O}(g(r))\) to
mean that \((r\partial_r)^{k}f(r) = O(g(r))\) for each \(k\ge
0\). If \(f\) depends on additional parameters \(\theta{}\), then we write
\(f(r,\theta{}) = \mathcal{O}(g(r))\) to mean that \(f(\cdot ,\theta{}) =
\mathcal{O}(g(r))\) for each choice of parameters \(\theta{}\), with the
implicit constants uniform in \(\theta{}\). In practice, \(f = f(r,\theta{})\)
will be a smooth function in polar coordinates. If the implicit constants do
depend on additional parameters, this will be indicated with the use of a
subscript (e.g.~\(f_k(r) = \mathcal{O}_k(1)\)).

Fix constants \(\eta_0 > 0\), \(\epsilon{} > 0\), and \(a>1\). The constant \(\eta_0\) will be
used to limit the range of \(p\) for which we establish \(r^p\)-weighted energy
estimates (see \cref{rp-general}) and elliptic estimates (see
\cref{good-scalar-field-tinv}). The constant \(\epsilon{}\) will measure the size
of the metric perturbation, and \(a\) will quantify the asymptotic flatness of
the metric perturbation (see \cref{sec:metric-assumptions}). We will assume without
comment that \(\epsilon{}\) is much smaller than all other constants that appear
(including \(\eta_0\)).

We write \(A\lesssim B\) (or \(B\gtrsim A\)) when there is a constant \(C\),
depending only on \(\eta_0\), \(a\), and the function \(h\) used to construct
the foliation \(\Sigma_\tau{}\) defined in \cref{hyperboloidal-foliation} such that
\(A\le CB\). We write \(A\sim B\) if \(A\lesssim B\) and \(A\gtrsim B\).
\subsection{Assumptions on the metric}
\label{sec:metric-assumptions}
The Minkowski metric on \(\R^{2 + 1}\), expressed in the global (Cartesian)
coordinate chart \((t,x^1,x^2)\), is
\begin{equation}
g_{\textnormal{mink}} \coloneqq{} -\dd{}t^2 + (\dd{}x^1)^2 + (\dd{}x^2)^2.
\end{equation}
Define \(r :\R^{2}\to \R_{\ge 0}\) by \(r(x^1,x^2)\coloneqq{}\sqrt{(x^1)^2 + (x^2)^2}\). Write
\(\theta{}\) for a coordinate on \(S^1\) (with range \((0,2\pi{})\)). Then the
standard polar coordinates \((t,r,\theta{})\) form a chart on \(\R^{2 +
1}\setminus \set{r=0}\), in which the Minkowski metric takes the form
\begin{equation}
g_{\textnormal{mink}} = -\dd{}t^2 + \dd{}r^2 + r^2\dd{}\theta^2.
\end{equation}

We consider metrics \(g\) that are asymptotically flat stationary perturbations
of the Minkowski metric, namely metrics that take the following form on \(\R^{2 + 1}\setminus \set{r=0}\):
\begin{equation}\label{metric}
g = -A(r)^2\dd{}t^2 + B(r)^2\dd{}r^2 + r^2\dd{}\theta^2,\qquad A(r),B(r) = 1 + \mathcal{O}(\epsilon{}\langle{}r\rangle^{-a}),\qquad A(0) = B(0).
\end{equation}
Here \(\epsilon{}>0\) and \(a{} > 1\) are the constants fixed in \cref{sec:notation}.
We assume that \(A(r)\) and \(B(r)\) define smooth functions on \(\R^2\), so that
the metric \(g\) is smooth on \(\R^{2 + 1}\).
\subsection{The hyperboloidal foliation and spacetime regions}
\label{hyperboloidal-foliation} Let \(h : [0,\infty)\to (0,2)\) be a smooth function such that:
\begin{enumerate}
\item \label{conditions-on-h-at-0} (regularity conditions at \(0\)) \(h(0) = 1\) and \(h^{(k)}(0) = 0\) for each \(k\ge
   1\),
\item \label{h-hyperboloidal} (decay conditions) \(h(r) = \mathcal{O}(\langle{}r\rangle{}^{-1-\eta{}_h})\) for some \(\eta{}_h>0\),
\item \label{h-derivative-small} (condition used in \cref{frakL-well-defined}) \(h(r) +
   \abs{rh'(r)}\le B_h\langle{}r\rangle^{-1-\eta{}_h}\) for some \(B_h > 0\) and  \(\abs{h'(r)}\le \epsilon_h\) for some \(\epsilon_h > 0\),
\item \label{h-lower-bound} (lower bound) and \(h(r)\ge c\langle{}r\rangle{}^{-C_h}\) for some \(c>0\) and \(C_h\ge 2\).
\end{enumerate}
\begin{remark}[Comments on the assumptions on \(h\)]
The conditions in \cref{conditions-on-h-at-0} are required to ensure that the level
sets of the function \(\tau{}\), which will be defined using the function \(h\)
in \cref{tau-def}, are smooth at \(\set{r=0}\). In \cref{h-hyperboloidal}, we used the
\(\mathcal{O}\)-notation introduced in \cref{sec:notation}. The condition
\cref{h-hyperboloidal} implies that the foliation determined by level sets of
\(\tau{}\) are asymptotically null, while \cref{h-lower-bound} implies that the
foliation does not become null too quickly. We can assume that the constant
\(\eta_h\) is sufficiently small, since if \cref{h-hyperboloidal} holds for a
particular value of \(\eta_h > 0\), it also holds for all smaller positive
values. The condition \cref{h-derivative-small} (where the existence of \(B_h\) is
implicit in \cref{h-hyperboloidal}) are used only
in the proof of \cref{frakL-well-defined}, to show that a certain quantity
associated to the foliation is non-zero, so that the quantity
\(\mathfrak{L}[\varphi{}]\) introduced in \cref{main-theorem-rough} is
well-defined. To do this, we will need to choose \(\epsilon_h\) small depending
on \(B_h\) and \(\eta_h\).
\end{remark}
Define the coordinates
\begin{equation}\label{double-null-coords}
u \coloneqq{} \frac{1}{2}(t-\tilde{G}(r)),\qquad v \coloneqq{} \frac{1}{2}(t+\tilde{G}(r)), \qquad \tilde{G}(r) \coloneqq{} \int_0^r G(s)^{-1}\dd{}s,\qquad G(r) \coloneqq{} \frac{A(r)}{B(r)}
\end{equation}
Note that \(G\circ r\) is a smooth function on \(\R^2\) with \(G(0) = 1\). Then
\((u,v,\theta{})\) are smooth coordinates on \(\R^{2 + 1}\setminus \set{r=0}\),
in which \(g\) takes the form
\begin{equation}\label{metric-double-null}
g = -4A^2\dd{}u\dd{}v + r^2\dd{}\theta^2.
\end{equation}
Evidently \(g^{-1}(\dd{}u,\dd{}u) = g^{-1}(\dd{}v,\dd{}v) = 0\), and so \(u\) and \(v\) are
double null coordinates. Define also
\begin{equation}\label{tau-def}
\tau{} \coloneqq{} u - 1 - \frac{1}{2}\int_r^\infty h(s)\dd{}s.
\end{equation}
Then \((\tau{},r,\theta{})\) are smooth coordinates on \(\R^{2 + 1}\setminus \set{r=0}\). Moreover,
\(\tau{}\) and \(h\circ r\) define smooth functions on \(\R^{2 + 1}\), in view
of \cref{axisymmetric-smoothness}, the smoothness of \(G\), and the conditions on \(h\) at \(0\) in
\cref{conditions-on-h-at-0}. The level sets of \(\tau{}\), denoted
\begin{equation}
\Sigma{}(\tau{}') \coloneqq{}\set{\tau{}=\tau{}'},
\end{equation}
are spacelike hypersurfaces, since \(g^{-1}(\dd{}\tau{},\dd{}\tau{}) = -\frac{1}{4}h(2-h) +
\mathcal{O}(\epsilon{}\langle{}r\rangle^{-a})\) is negative for \(\epsilon{}\)
sufficiently small. For \(v_0\ge 0\), define the cutoff hypersurface
\begin{equation}\label{sigma-tau-v}
\Sigma{}(\tau{},v_0)\coloneqq{}\Sigma{}(\tau{})\cap \set{v\le v_0}.
\end{equation}
We also introduce notation for the ingoing null cones
\begin{equation}
\underline{C}(v_0)\coloneqq{}\set{v=v_0}
\end{equation}
and the spacetime regions
\begin{equation}\label{spacetime-regions}
\mathcal{R}(\tau{}_1,\tau{}_2,v) = \bigcup_{\tau{}_1\le \tau{}\le \tau{}_2}^{}\Sigma{}(\tau{},v),\qquad \mathcal{R}(\tau{}_1,\infty,v)\coloneqq{}\bigcup_{\tau{}\ge \tau{}_1}^{}\mathcal{R}(\tau{}_1,\tau{},v),\qquad \mathcal{R}\coloneqq{}\bigcup_{\tau{}\ge 0}^{}\Sigma{}(\tau{}).
\end{equation}
By construction, we have \(u\ge 1\) in \(\mathcal{R}\). Finally, the spacetime volume form \(\dd{}\vol\) satisfies
\begin{equation}
\dd{}\vol = rA(r)^2G(r)^{-1}\dd{}r\dd{}\theta{}\dd{}t\sim r\dd{}r\dd{}\theta{}\dd{}t\sim r\dd{}u\dd{}v\dd{}\theta{}\sim r\dd{}r\dd{}\theta{}\dd{}\tau{},
\end{equation}
where the equivalence up to constants of these expressions holds when \(\epsilon{}\) is
sufficiently small.
\subsection{Notation for vector fields}
Define the vector field \(T\coloneqq{}\partial{}_t|_{x^1,x^2}\) on \(\R^{2 + 1}\). We introduce
the following notation for the coordinate derivatives in \((t,r,\theta{})\)
coordinates, \((u,v,\theta{})\) coordinates, and \((\tau{},r,\theta{})\) coordinates:
\begin{equation}
Z\coloneqq{}\partial{}_r|_{t,\theta{}},\qquad \underline{L}\coloneqq{}\partial{}_u|_{v,\theta{}},\qquad L\coloneqq{}\partial{}_v|_{u,\theta{}},\qquad X\coloneqq{}\partial{}_r|_{\tau{},\theta{}}.
\end{equation}
One readily derives the following relations between these vector fields on \(\R^{2 + 1}\setminus \set{r=0}\):
\begin{lemma}[Relations between coordinate derivatives]
We have
\begin{gather}
\partial{}_t|_{r,\theta{}} = T,\quad \partial{}_\tau{}|_{r,\theta{}} = 2T, \\
\underline{L} = T-GZ = (2-Gh)T - GX, \\
L = T+GZ = GhT + GX.
\end{gather}
\label{vector-field-relations}
\end{lemma}
\subsection{Projection to angular modes}
We recall the following elementary facts.
\begin{lemma}[Facts about smooth radially symmetric functions]
Suppose \(f:\R^2\to \R\) is such that \(f = F\circ r\) for some even function \(F : \R\to
\R\). Then we say \(f\) is \emph{radially symmetric}, and the following statements hold:
\begin{enumerate}
\item \label{even-smooth} \(f\) is smooth if and only if \(F\) is (if and only if \(\tilde{F}\coloneqq{}F|_{[0,\infty)}\) is
smooth and \(\tilde{F}^{(k)}(0) = 0\) for all \(k\) odd),
\item \label{whitney-part} and in this case \(F(x) = g(x^2)\) for some function \(g :
   \R\to \R\) that is smooth if and only if \(F\) is, and moreover depends smoothly
on a parameter when \(f\) does,
\item \label{axisymmetry-vanishing} and for each \(k\ge 0\), when \(f\) is smooth, \(r^{-1}\partial_r^{2k +
   1}f\) and \(\partial_r^{2k}f\), which are smooth on \(\R^2\setminus
   \set{0}\), extend to smooth functions on \(\R^2\), where we write \(\partial_r\) for
the radial derivative in polar coordinates \((r,\theta{})\).
\end{enumerate}
\label{axisymmetric-smoothness}
\end{lemma}
\begin{proof}
Part \cref{whitney-part} is a classical theorem of Whitney, established in the
(2-page!) paper \cite{whitney-even-functions} (where the smooth dependence of
\(g\) on \(F\), hence on \(f\), is evident from the proof). Part
\cref{axisymmetry-vanishing} follows from \cref{whitney-part} and the chain rule.
\end{proof}
\begin{lemma}[Vanishing at the origin for higher modes]
Let \(f : \R^2\to \R\) be smooth. Let \(f_0 = \frac{1}{2\pi{}}\int _{S^1}f\dd{}\theta{}\) denote the
radially symmetric part of \(f\), and let \(f_{\ge 1} \coloneqq{}f-f_0\) be the projection of
\(f\) to angular modes \(\ge 1\). Then \(r^{-1}f_{\ge 1}\) is bounded as \(r\to 0\).
\label{higher-modes-origin}
\end{lemma}
\begin{proof}
This follows from a Taylor expansion around the origin.
\end{proof}
\subsection{Notation for scalar fields and expressions for the wave equation}
From now on, let \(\varphi{}\) denote a smooth function on the spacetime region
\(\mathcal{R}\) defined in \cref{spacetime-regions}. We introduce the following
important rescaled quantities:
\begin{equation}
\psi{}\coloneqq{}r^{1/2}\varphi{},\qquad \Psi{}\coloneqq{}r^{1/2}GZ\varphi{} = \frac{1}{2}r^{1/2}(L-\underline{L})\varphi{}.
\end{equation}
The quantity \(\psi{}\) is natural since it attains a finite non-zero limit at null
infinity (the Friedlander radiation field). The quantity \(\Psi{}\) is important
because its radially symmetric part, which we call \(\Psi_0\), satisfies a wave
equation with a favourable potential (see \cref{equation-for-Psi}).

We write \(\varphi_0\) for the radially symmetric part of \(\varphi{}\), namely
\begin{equation}
\varphi{}_0(r)\coloneqq{}\frac{1}{2\pi{}}\int_0^{2\pi{}} \varphi{}(r,\theta{})\dd{}\theta{}.
\end{equation}
We write \(\varphi_{\ge 1}\) for the non-radially symmetric part of \(\varphi{}\), namely
\begin{equation}
\varphi_{\ge 1} \coloneqq{}\varphi{} - \varphi_0.
\end{equation}
Similarly, we write \(\psi_{\ge 1}\), \(\Psi_0\), and so on.

We will write \(\Box{}\coloneqq{}A^2\Box_g\), where \(\Box_g\) is the wave operator associated to
the metric \(g\), and will from now on refer to solutions of
\begin{equation}\label{wave-equation}
\Box{}\varphi{} = 0,
\end{equation}
which are of course equivalent to solutions of \cref{wave-equation-intro}. We now
derive various expressions for the equations satisfied by \(\varphi{}\), \(\psi{}\),
and \(\Psi{}\).
\begin{lemma}[Expressions for the wave equation]
The wave equation for \(\varphi{}\in C^\infty(\mathcal{R})\) can be written in the
following forms:
\begin{align}
\Box{}\varphi{} &= -\underline{L}L\varphi{} + \frac{1}{2}r^{-1}G(L-\underline{L})\varphi{} + r^{-2}A^2\partial{}_\theta^2\varphi{}, \label{box-equation-double-null} \\
G^{-1}\Box{}\varphi{} &= -h(2-Gh)T^2 - (2-2Gh)XT + ((Gh)' - r^{-1}(1-Gh))T \nonumber \\
&\qquad + GX^2 + r^{-1}(G + rG')X + A^2G^{-1}r^{-2}\partial{}_\theta^2. \label{box-equation-TX}
\end{align}
Moreover, \(\psi{} = r^{1/2}\varphi{}\) satisfies
\begin{align}
r^{1/2}\Box{}\varphi{} &= -\underline{L}L\psi{} + \frac{1}{4}r^{-2}G(G - 2rG')\psi{} + r^{-2}A^2\partial{}_\theta^2\psi{}, \label{equation-for-psi}
\end{align}
which can be rewritten as
\begin{equation}\label{equation-for-psi-TX}
\mathcal{L}(G^{1/2}\psi{}) = \mathcal{F}[G^{1/2}T\psi{}] + G^{-3/2}r^{1/2}\Box{}\varphi{},
\end{equation}
where
\begin{equation}\label{calL-def}
\begin{split}
\mathcal{L}&\coloneqq{}X^2 + \frac{1}{4}r^{-2}(1-2G^{-1}rG' + G^{-2}r^2G'^2 - 2G^{-1}r^2G'') + A^2G^{-2}r^{-2}\partial{}_\theta^2
\end{split}
\end{equation}
and
\begin{equation}\label{calF-def}
\mathcal{F}\coloneqq{}-G^{-1}h(2-Gh)T - 2G^{-1}(1-Gh)X + G^{-1}(G^{-1}G'(1-Gh) + (Gh)').
\end{equation}
Finally, \(\Psi{}_0 = r^{1/2}GZ\varphi{}_0\) satisfies
\begin{equation}\label{equation-for-Psi}
r^{1/2}GZ\Box{}\varphi{}_0 = -\underline{L}L\Psi{}_0 - \frac{3}{4}r^{-2}G\Bigl(G - \frac{2}{3}rG'\Bigr)\Psi{}_0.
\end{equation}
\label{wave-equation-expressions}
\end{lemma}
\begin{remark}[The effective inverse-square potential for good scalar fields]
Observe from \cref{wave-equation-expressions} that \(\Psi{}_0\) satisfies an
equation with a zeroth-order term of a good sign. To be precise, the
zeroth-order term in the equation for \(\psi_{\ge 1}\) has a bad sign, but it is
compensated for by the angular term of a good sign (via a Poincaré inequality on \(S^1\)).
\end{remark}
\begin{proof}
First, \cref{box-equation-double-null} is a direct computation (in \((u,v,\theta{})\)
coordinates) from \cref{metric-double-null} using
\begin{equation}
\Box_g\varphi{} = \abs{\Det g}^{-1/2}\partial_\alpha{}(\abs{\Det g}^{1/2}g^{\alpha{}\beta{}}\partial_\beta{}\varphi{}).
\end{equation}
Then \cref{box-equation-TX,equation-for-psi,equation-for-psi-TX,equation-for-Psi}
can be derived from \cref{box-equation-double-null} using
\cref{vector-field-relations}. For \cref{equation-for-Psi}, we note that the \(+
\frac{3}{4}\) arises from adding \(-\frac{1}{4}\) (coming from
\(r^{-1/2}\partial{}_r^2r^{1/2} = -\frac{1}{4}r^{-2}\)) and \(+ 1\) (coming from
\(-\partial{}_rr^{-1} = r^{-2}\)).
\end{proof}
\subsection{Norms}
In this section, we introduce norms that measure energy of a scalar field along
\(\Sigma{}(\tau{})\), inhomogeneities in the bulk region
\(\mathcal{R}(\tau_1,\tau_2)\), and initial data on \(\Sigma{}(0)\). Some norms
will appear with a tilde (for example \(\tilde{E}[\Psi{}_0]\) or
\(\tilde{\mathcal{A}}_{p,N}[\tilde{F}]\)), while others (such as
\(E[\varphi{}]\) or \(\mathcal{A}_{p,N}[F]\)) appear without a tilde. We will
use the tilded energies for quantities associated to the good quantities
\(\Psi{}_0 = r^{1/2}GZ\varphi_0\) and \(\psi{}_{\ge 1} = r^{1/2}\varphi{}_{\ge
1}\), which carry an \(r^{1/2}\)-weight. We will use the unadorned energies for
radially symmetric scalar fields \(\varphi_0\) (without an \(r^{1/2}\)-weight).
The notation \(\Phi{}\) will always denote an abstract good scalar field, as
defined in \cref{class-of-fields} (we will later specialize to \(\Psi_0\) and \(\psi{}_{\ge
1}\)).
\subsubsection{Index of notation}
We collect here notation for all the norms we use. We first
introduce the energy norms defined in \cref{energy-norms}.
\begin{center}
\begin{tabular}{ll}
Notation & Meaning\\
\hline
\(E[\varphi{}]\) & \(T\)-energy of a radially symmetric scalar field \(\varphi{}\)\\
\(\tilde{E}[\Phi{}]\) & \(T\)-energy of a good scalar field \(\Phi{}\)\\
\(\mathcal{E}_{1 + \delta{}}[\varphi{}]\) & \(r^p\)-weighted energy of a radially symmetric field \(\varphi{}\), with \(p = 1 + \delta{}\) for \(\delta{}\ge 0\) small\\
\(\tilde{\mathcal{E}}_p[\Phi{}]\) & \(r^p\)-weighted energy of a good scalar field \(\Phi{}\), used for \(p\in [0,2)\)\\

\end{tabular}
\end{center}
Energies with an additional subscript, such as \(E_{N}[\varphi{}]\) or
\(\tilde{\mathcal{E}}_{1,N}[\Phi{}]\) are higher-order variants of the above
energies with respect to commutation with \((rL)\) and \(\partial_\theta{}\)
(see \cref{energies-higher-order}). These energies take as arguments a time
\(\tau{}\) and a value \(v\): for example, \(E[\varphi{}](\tau{},v)\) measures
the \(T\)-energy of \(\varphi{}\) on the truncated hyperboloidal surface
\(\Sigma{}(\tau{},v)\) (see \cref{sigma-tau-v}). When we omit the parameter \(v\),
it is to be interpreted as ``\(v = \infty\).''

Next, we introduce the norms for inhomogeneities defined in
\cref{inhomogeneity-norm}.
\begin{center}
\begin{tabular}{ll}
Notation & Meaning\\
\hline
\(\mathcal{A}_{p,N}[F]\) & \(r^p\)-weighted norm of \((rL)^{n}F\), with \(p\ge 0\) and \(0\le n\le N\)\\
\(\tilde{\mathcal{A}}_{p,N}[\tilde{F}]\) & \(r^p\)-weighted norm of \((rL)^{n_1}\partial{}_\theta^{n_2}\tilde{F}\) with \(n_1 + n_2\le N\)\\

\end{tabular}
\end{center}
In practice, we will take \(F = \Box{}\varphi{}\) and \(\tilde{F} = r^{1/2}GZ\Box{}\varphi{}\) to be the
inhomogeneities associated to \(\varphi{}\) and \(\Psi{}\), respectively. These
will be non-zero when considering the renormalized quantities
\(\widehat{\varphi{}}\) and \(\widehat{\Psi{}}\), since the Minkowskian profile
\(\varphi_{\textnormal{mink}}\) is not an exact solution to the wave equation on
the perturbed background metric.

Finally, we introduce notation for the initial data norms defined in \cref{data-norm}.
\begin{center}
\begin{tabular}{ll}
Notation & Meaning\\
\hline
\(\mathfrak{L}[\varphi{}]\) & the scalar needed to construct the renormalized solution \(\widehat{\varphi{}}\coloneqq{}\varphi{}-\mathfrak{L}[\varphi{}]\varphi_{\textnormal{mink}}\)\\
\(\mathbf{D}_{N,\delta{}}[\varphi{}]\) & \(\le N\) derivatives of the data in \(r\)-weighted \(L^\infty\) and \(L^2\) norms, with corrections defined by \(\delta{}>0\)\\

\end{tabular}
\end{center}
To capture the improved decay for \(T\)-derivatives, we also use the following
norms:
\begin{center}
\begin{tabular}{ll}
Notation & Meaning\\
\hline
\(\mathcal{D}_{N,M,\delta{}}[\varphi{},\Phi{}]\) & derivatives \((rL)^{\le N}\) and \(T^{\le M}\), with \(\tau{}\)-weights; defined in \cref{energy-decay}\\
\(\tilde{\mathcal{D}}_{N,M,\delta{}}[\Phi{}]\) & analogous to \(\mathcal{D}_{N,M,\delta{}}\), used for good scalar fields; defined in \cref{energy-decay-radially-symmetric}\\

\end{tabular}
\end{center}
In practice, we will use the norm \(\mathcal{D}_{N,M,\delta{}}\) with arguments
\(\widehat{\varphi{}}_0\) and \(T^{-1}\widehat{\Psi{}}_0\), as well as
\(T^{-1}\widehat{\varphi{}}_0\) and \(T^{-2}\widehat{\Psi{}}_0\).
\subsubsection{Energy norms}
\label{energy-norms}
Define the \(T\)-energy
\begin{equation}\label{T-energy-def}
E[\varphi{}](\tau{},v)\coloneqq{}\int _{\Sigma{}(\tau{},v)}\bigl[(L\varphi{})^2 + h(r)(\underline{L}\varphi{})^2 + r^{-2}(\partial{}_\theta{}\varphi{})^2\bigr]r\dd{}r\dd{}\theta{}\sim \int _{\Sigma{}(\tau{},v)}\bigl[(X\varphi{})^2 + h(r)(T\varphi{})^2 + r^{-2}(\partial{}_\theta{}\varphi{})^2\bigr]r\dd{}r\dd{}\theta{}.
\end{equation}
We will also use the modified \(T\)-energy
\begin{equation}
\tilde{E}[\Phi{}](\tau{},v) \coloneqq{} \int _{\Sigma{}(\tau{},v)} (L\Phi{})^2 + h(r)(\underline{L}\Phi{})^2 + r^{-2}(\partial{}_\theta{}\Phi{})^2 + r^{-2}\Phi^2\dd{}r\dd{}\theta{}.
\end{equation}
The energy \(\tilde{E}\) differs from \(E\) in two ways: it lacks the factor of
\(r\) in the volume form, on account of the \(r^{1/2}\)-weight already carried
by \(\Phi{}\) (which in practice is either \(\Psi_0 = r^{1/2}GZ\varphi_0\) or
\(\psi_{\ge 1} = r^{1/2}\varphi_{\ge 1}\)), and includes the zeroth-order term
\(r^{-2}\Phi{}\), which is present because \(\Phi{}\) is a good scalar field.

Next, we introduce the following \(r\)-weighted energy for \(\delta{}\ge 0\):
\begin{equation}
\begin{split}
\mathcal{E}_{1+\delta}[\varphi{}](\tau{},v)&\coloneqq{}\int _{\Sigma{}(\tau{},v)} r\langle{}r\rangle{}^\delta{}(L\psi{})^2 + h(r)\langle{}r\rangle{}^{\delta{}}\varphi^2\dd{}r\dd{}\theta{}.
\end{split}
\end{equation}
For \(p \in \R\), we define a modified \(r\)-weighted energy:
\begin{equation}
\tilde{\mathcal{E}}_p[\Phi{}](\tau{},v) \coloneqq{} \int _{\Sigma{}(\tau{},v)}\langle{}r\rangle{}^p(L\Phi{})^2 + h(r)\langle{}r\rangle^pr^{-2}(\partial{}_\theta{}\Phi{})^2 + h(r)\langle{}r\rangle^pr^{-2}\Phi{}^2\dd{}r\dd{}\theta{}.
\end{equation}
We also define higher-order variants of the above energies (suppressing here the
arguments \((\tau{},v)\)):
\begin{equation}\label{energies-higher-order}
\begin{split}
E_N[\varphi{}]\coloneqq{}\sum_{n=0}^NE[(rL)^n\varphi{}],\qquad  \mathcal{E}_{p,N}[\varphi{}] \coloneqq{}\sum_{n=0}^{N}\mathcal{E}_p[(rL)^n\varphi{}],\qquad  \tilde{\mathcal{E}}_{p,N}[\Phi{}]\coloneqq{}\!\!\!\!\! \sum_{\substack{n_1,n_2\ge 0\\n_1+n_2\le N}} \tilde{\mathcal{E}}_{p}[\partial{}_\theta^{n_1}(rL)^{n_2}\Phi{}].
\end{split}
\end{equation}
We will not use a higher-order version of the modified \(T\)-energy. We write
\begin{equation}
E[\varphi{}](\tau{})\coloneqq{}\sup_{v\ge 0}E[\varphi{}](\tau{},v),
\end{equation}
and use similar notation for the other energies.
\subsubsection{Norms for inhomogeneities}
\label{inhomogeneity-norm}
For \(p\ge 0\), \(N\ge 0\), and \(\tau{}\ge 0\), define the norm
\begin{equation}\label{A-norm-def}
\mathcal{A}_{p,N}[F](\tau{}_1,\tau{}_2,v)\coloneqq{}\sum_{n=0}^N\int_{\tau{}_1}^{\tau{}_2} \int _{\Sigma{}(\tau{},v)} \langle{}r\rangle^p\abs{(rL)^nF}^2\dd{}r\dd{}\theta{}\dd{}\tau{}.
\end{equation}
For \(p\ge 0\), \(N\ge 0\), \(\tau_0\ge 0\), and \(v\ge 0\), define the norm
\begin{equation}\label{Anp-def}
\begin{split}
\tilde{\mathcal{A}}_{p,N}[\tilde{F}](\tau{}_1,\tau{}_2,v)&\coloneqq{}\sum_{n_1+n_2\le N}\int_{\tau{}_1}^{\tau{}_2} \int _{\Sigma{}(\tau{},v)}\langle{}r\rangle{}^{p+1}(\partial{}_\theta^{n_1}(rL)^{n_2}\tilde{F})^2\dd{}r\dd{}\theta{}\dd{}\tau{} \\
&\qquad + \sum_{n_1+n_2\le N+1}\int_{\tau{}_1}^{\tau{}_2} \int _{\Sigma{}(\tau{},v)}\langle{}r\rangle{}^{1+\eta{}_0}(\partial{}_\theta^{n_1}(rL)^{n_2}\tilde{F})^2\dd{}r\dd{}\theta{}\dd{}\tau{} \\
&\qquad + \sum_{n_1+n_2\le N}\sup_{\tau{}\in [\tau{}_1,\tau{}_2]}\int _{\Sigma{}(\tau{},v)}\langle{}r\rangle{}^{1+\eta{}_0}(\partial{}_\theta^{n_1}(rL)^{n_2}\tilde{F})^2\dd{}r\dd{}\theta{} \\
&\qquad + \sum_{n_1+n_2\le N}\int _{\underline{C}(v)\cap \set{\tau{}_1\le \tau{}\le \tau{}_2}}\langle{}r\rangle{}^{2+\eta{}_0}(\partial{}_\theta^{n_1}(rL)^{n_2}\tilde{F})^2\dd{}u\dd{}\theta{},
\end{split}
\end{equation}
where \(\eta_0\) is the small constant fixed in \cref{sec:notation}. We introduce the
constant \(\eta_0\) and the loss of derivatives in the second line (where \(N +
1\) appears instead of \(N\)) in the inhomogeneous norm \(\tilde{\mathcal{A}}\)
to avoid the anomalous degeneracy in the \(r^p\) estimates when \(p = 0\) (see
the proof of \cref{rp-general}). In practice, we will only apply these estimates to
scalar fields that satisfy a known inhomogeneity, arising from the failure of
\(\varphi_{\textnormal{mink}}\) to solve \cref{wave-equation}. Since the
inhomogeneity is known, a loss of derivatives does not present an issue. If
there is extra \(\tau{}\)-decay available, one can remove the loss of derivatives.
\subsubsection{Initial data norm}
\label{data-norm}
Define the \(L^1\)-type quantity
\begin{equation}\label{L-frak-def}
\mathfrak{L}[\varphi{}]\coloneqq{}\Bigl(\int_0^\infty r^{1/2}G^{1/2}\mathcal{F}[G^{1/2}\psi{}_{\textnormal{mink}}]|_{\Sigma{}(0)} + rG^{-1}T^{-1}\Box{}\varphi{}_{\textnormal{mink}}|_{\Sigma{}(0)}\dd{}r\Bigr)^{-1}\int_0^\infty r^{1/2}G^{1/2}\mathcal{F}[G^{1/2}\psi{}_0](r)|_{\Sigma{}(0)}\dd{}r,
\end{equation}
where \(\mathcal{F}\) is as in \cref{calF-def} and \(\varphi_{\textnormal{mink}} =
u^{-1/2}v^{-1/2}\). The well-definedness of the quantity
\(\mathfrak{L}[\varphi{}]\) (in particular the existence of
\(T^{-1}\Box{}\varphi_{\textnormal{mink}}\) and the non-vanishing of the factor
that is inverted) is established in \cref{inhomogeneous-estimates}.

For \(N\ge 0\) and \(\delta{} \ge 0\), define the initial data norm
\begin{equation}
\begin{split}
&\mathbf{D}_{N,\delta{}}[\varphi{}]\\
&\coloneqq{}\mathfrak{L}[\varphi{}]^2 +  \sum_{\substack{n,m\ge 0 \\ n+m\le N }}\norm{\langle{}r\rangle^{1/2}(rX)^nT^m\varphi{}_0}_{L^\infty(\Sigma{}(0))}^2 + \sum_{\substack{n_1,n_2,n_3\ge 0 \\ n_1+n_2+n_3\le N }}\norm{\partial{}_\theta^{n_1}(rX)^{n_2}T^{n_3}\psi{}_{\ge 1}}_{L^\infty(\Sigma{}(0))} \\
&\qquad + \sum_{m=0}^{N}(E_{N-m}[T^m\varphi{}_0](0) + \mathcal{E}_{1+\delta{},N-m}[T^m\varphi{}_0](0) + \tilde{\mathcal{E}}_{1+\delta{},N-m}[T^m\psi{}_{\ge 1}](0)) + \sum_{m=0}^{N-2} \tilde{\mathcal{E}}_{1+\delta{},N-m}[T^m\Psi{}_0](0) \\
&\qquad + \sum_{\substack{n,m\ge 0 \\ n+m\le N}} \int _{\Sigma{}(0)} \langle{}r\rangle^{3-\delta{}}(X(rX)^n\psi{}_0)^2 + \langle{}r\rangle^{2-\delta{}}(T(rX)^n\varphi{}_0)^2 + \langle{}r\rangle^{2-\delta{}}((rX)^n\varphi{}_0)^2\dd{}r\dd{}\theta{} \\
&\qquad  + \sum_{n=0}^N \int _{\Sigma{}(0,v)} \langle{}r\rangle^2(X(rX)^n\Psi{}_0)^2 + ((rX)^nT\Psi{}_0)^2 + \langle{}r\rangle^{1-2\delta{}}((rX)^n\Psi{}_0)^2 \dd{}r\dd{}\theta{} \\
&\qquad + \sum_{\substack{n_1,n_2\ge 0 \\ n_1 + n_2 \le N}} \int _{\Sigma{}(0)}\langle{}r\rangle^{3-\delta{}}(X\partial{}_\theta^{n_1}(rX)^{n_2}\psi{}_{\ge 1})^2 + \langle{}r\rangle^{1-\delta{}}(T\partial{}_\theta^{n_1}(rX)^{n_2}\psi{}_{\ge 1})^2 + \langle{}r\rangle^{1-\delta{}}(\partial{}_\theta^{n_1}(rX)^{n_2}\psi{}_{\ge 1})^2\dd{}r\dd{}\theta{}.
\end{split}
\end{equation}
We also note that the initial data norms \(\tilde{\mathcal{D}}_{N,M,\delta{}}[\Phi{}]\) and
\(\mathcal{D}_{N,M,\delta{}}[\varphi{},\Phi{}]\), which are used in
\cref{sec:energy-decay} to establish improved decay for \(T\)-derivatives, are
defined in \cref{energy-decay,energy-decay-radially-symmetric}, respectively.
\section{Precise statement of the main result}
Given the notation introduced in \cref{preliminaries}, we can state a
precise version of our main result.
\begin{theorem}[Main theorem, precise version]
Let \(h\) be the function defined in \cref{hyperboloidal-foliation} that
determines the hyperboloidal foliation \(\Sigma{}(\tau{})\), and let
\(\delta{}_h\coloneqq{} \frac{1}{4}(C_h + 1)^{-1} > 0\), where \(C_h\ge 2\)
(defined in \cref{h-lower-bound} of \cref{hyperboloidal-foliation}) determines the
polynomial rate (in the radial coordinate \(r\)) at which the hyperboloidal
foliation becomes null.

Let \(\varphi{}\in C^\infty(\mathcal{R})\) solve \cref{wave-equation}, and let
\(\varphi_{\textnormal{mink}}\coloneqq{}u^{-1/2}v^{-1/2}=
(t^2-\tilde{G}(r)^2)^{-1/2}\). If the constant \(\epsilon{} > 0\) (which
measures the size of the deviation of the metric \(g\) from the Minkowski
metric) is sufficiently small, then the following estimate holds for each \(N\ge
0\), \(M\ge 0\), \(\tau{} \ge 0\), and \(\delta{} > 0\) sufficiently small (depending
on \(h\)):
\begin{equation}\label{main-theorem-equation}
\begin{split}
&\abs{T^M(rL)^N(\varphi{} - \mathfrak{L}[\varphi{}]\varphi{}_{\textnormal{mink}})(\tau{},r,\theta{})} \\
&\lesssim_{N,M,h,\delta{}} \varphi{}_{\textnormal{mink}}(\tau{},r,\theta{})\cdot (1+\tau{})^{-M-\delta{}_{h}}\cdot \Bigl(\mathbf{D}_{\min (1,N)+M+4,\delta{}}[\varphi{}] + \sum_{k=0}^1\mathbf{D}_{\min (1,N)+M+2,\delta{}}[\partial{}_\theta^k\varphi{}]\Bigr).
\end{split}
\end{equation}
Here the linear functional \(\mathfrak{L}\) and the data norm \(\mathbf{D}\)
were defined in \cref{data-norm}.
\label{main-theorem}
\end{theorem}
\begin{proof}
This follows from \cref{pointwise-decay}.
\end{proof}
\section{Energy estimates}
\label{sec:energy-estimates}
In this section, we prove energy estimates, namely energy boundedness estimates,
integrated energy estimates, and \(r\)-weighted energy estimates. The estimates
of \cref{good-scalar-field-energy-estimates}, for scalar fields satisfying an
equation with a good zeroth-order term, are derived using standard techniques.
The estimates in \cref{axisymmetric-energy-estimates}, for radially symmetric
scalar fields, are new. There we establish estimates for \(\varphi_0\) which
include on the right-hand side quantities involving the good scalar field
\(\Psi_0\), which is estimated in \cref{good-scalar-field-energy-estimates}.
\subsection{Energy estimates for scalar fields that satisfy an equation with a good zeroth-order term}
\label{good-scalar-field-energy-estimates}
We first establish energy estimates for the quantities \(\Psi{}_0 =
r^{1/2}Z\varphi{}_0\) and \(\psi_{\ge 1} = r^{1/2}\varphi_{\ge 1}\), which
satisfy equations with (effective) zeroth-order terms of a good sign.
\subsubsection{The class of scalar fields under consideration}
\label{class-of-fields}
To treat \(\Psi_0\) and \(\psi_{\ge 1}\) uniformly (although these quantities satisfy
different equations), we formulate the class of ``good scalar fields'' for which
we prove energy estimates.
\begin{definition}[The class of good scalar fields]
We say that \(\Phi{}\in C^\infty(\mathcal{R}\setminus \set{r=0})\) satisfies an equation with a good
(effective) zeroth-order term with inhomogeneity \(\tilde{F}\in C^\infty(\mathcal{R}\setminus \set{r=0})\) if \(\Phi{}\) satisfies the boundary conditions
\begin{equation}\label{bdy-conditions-rp}
\begin{gathered}
\textnormal{for each }N_1,N_2,N_3\ge 0,\,r^{-1}\partial_\theta^{N_1}(rL)^{N_2}T^{N_3}\Phi{}\textnormal{ extends continuously to }\mathcal{R}\cap \set{r=0}\textnormal{ and vanishes there},
\end{gathered}
\end{equation}
solves the equation
\begin{equation}\label{alpha-equation}
\underline{L}L\Phi{} + \alpha{}r^{-2}(1 + f_1(r))\Phi{} - r^{-2}(1 + f_2(r))\partial{}_\theta^2\Phi{} = \tilde{F}
\end{equation}
for some \(\alpha{}\in \R\) and \(f_1,f_2=\mathcal{O}(\epsilon{}\langle{}r\rangle^{-a})\), and satisfies one of the
following conditions:
\begin{enumerate}
\item \label{alpha-big} \(\alpha{} > 0\),
\item \label{angular-big} or \(\alpha{} > -1\) and \(\Phi{} = \Phi{}_{\ge 1}\).
\end{enumerate}
\label{good-field}
\end{definition}
\begin{remark}[A coercivity property for good scalar fields]
If \(\Phi{}\) satisfies the assumptions of \cref{good-field}, then we have
\begin{equation}
\int _{S^1}(1 + f_2(r))(\partial{}_\theta{}\Phi{})^2 + \alpha{}(1+f_1(r))\Phi{}^2 \dd{}\theta{}\gtrsim _\alpha{} \int _{S^1} (\partial{}_\theta{}\Phi{})^2 + \Phi{}^2\dd{}\theta{}
\end{equation}
when \(\epsilon{}\) is sufficiently small. Indeed, this is clear when \cref{alpha-big}
holds, and if \cref{angular-big} holds, then this follows from the Poincaré
inequality
\begin{equation}
\int _{S^1}(\partial{}_\theta{}\Phi{}_{\ge 1})^2\dd{}\theta{} \ge \int _{S^1}\Phi{}^2\dd{}\theta{}.
\end{equation}
\label{effective-zo}
\end{remark}
\begin{lemma}[The quantities \(\psi_{\ge 1}\) and \(\Psi_0\) are good scalar fields]
If \(\varphi{}\in C^\infty(\mathcal{R})\), then \(\psi_{\ge 1}\) and \(\Psi_0\) satisfy the assumptions
of \cref{good-field} with inhomogeneities given by the left-hand sides of
\cref{equation-for-psi} (projected to modes \(\ge 1\)) and \cref{equation-for-Psi},
respectively.
\label{good-field-criterion}
\end{lemma}
\begin{proof}
For \(\Psi_0\), we have \(\alpha{} = 3/4\), and for \(\psi_{\ge 1}\), we have \(\alpha{} = -1/4\). The
relevant equations are satisfied by \cref{equation-for-Psi} and
\cref{equation-for-psi}. The boundary conditions are satisfied for any element of
\(r^{3/2}C^\infty(\mathcal{R})\), and we have \(\Psi_0\in
r^{3/2}C^\infty(\mathcal{R})\) by \cref{axisymmetric-smoothness} and \(\psi_{\ge 1}\in
r^{3/2}C^\infty(\mathcal{R})\) by \cref{higher-modes-origin}.
\end{proof}
\subsubsection{Energy boundedness and integrated local energy decay}
\label{good-eb}
In this section, we prove an energy boundedness estimate and integrated energy
estimate (or Morawetz estimate) for good scalar fields using the standard
multipliers \(T\) and \(f(r)Z\).
\begin{proposition}[Energy boundedness and integrated local energy decay for good scalar fields]
Suppose \(\Phi{}\in C^\infty(\mathcal{R}\setminus \set{r=0})\) satisfies the assumptions of
\cref{good-field} with inhomogeneity \(\tilde{F}\). Then for \(0\le \tau_1\le \tau_2\), \(v\ge 0\), and \(\delta{} > 0\), we have
\begin{equation}\label{modified-T-equation}
\begin{split}
&\tilde{E}[\Phi{}](\tau_2,v) + \int_{\tau_1}^{\tau_2} \int _{\Sigma{}(\tau{},v)} \langle{}r\rangle^{-1-\delta{}}((L\Phi{})^2 + (\underline{L}\Phi{})^2) + r^{-2}\langle{}r\rangle{}^{-1}(\Phi^2 + (\partial{}_\theta{}\Phi{})^2)\dd{}r\dd{}\theta{}\dd{}\tau{}\\
&\lesssim_{\delta{}} \tilde{E}[\Phi{}](\tau{}_1,v) + \tilde{\mathcal{A}}_{1+\delta{},0}[\tilde{F}](\tau{}_1,\tau{}_2,v).
\end{split}
\end{equation}
Moreover, we have
\begin{equation}\label{modified-T-prep-0}
\begin{split}
&\tilde{E}[\Phi{}](\tau_2,v) + \int_{\tau{}_1}^{\tau{}_2} \int _{\Sigma{}(\tau{},v)} \langle{}r\rangle^{-1-\delta{}}(\underline{L}\Phi{})^2 \dd{}r\dd{}\theta{}\dd{}\tau{} \\
&\lesssim_\delta{}  \tilde{E}[\Phi{}](\tau_1,v) +\int_{\tau{}_1}^{\tau{}_2} \int _{\Sigma{}(\tau{},v)} \langle{}r\rangle^{-1-\delta{}}(L\Phi{})^2 \dd{}r\dd{}\theta{}\dd{}\tau{}  \int_{\tau_1}^{\tau_2} \int _{\Sigma{}(\tau{})} \abs{T\Phi{}}\abs{\tilde{F}}\dd{}r\dd{}\theta{}\dd{}\tau{}.
\end{split}
\end{equation}
\label{modified-T-morawetz}
\end{proposition}
\begin{remark}[Higher-order estimates]
We only prove the estimates in \cref{modified-T-morawetz} for the good scalar field
\(\Phi{}\) itself, and do not prove higher-order analogues that include
commutation. Instead, we prove higher-order estimates for the \(r^p\)-weighted
energy in \cref{sec:rp-good}. In most situations, we will use the \(p=0\) energy
\(\tilde{\mathcal{E}}_0[\Phi{}]\) as a replacement for the \(T\)-energy
\(\tilde{E}[\Phi{}]\).
\end{remark}
\begin{remark}[Control of terms involving \(\underline{L}\Phi{}\)]
The importance of the estimates in \cref{modified-T-morawetz} is their control of
flux terms on \(\Sigma{}(\tau{})\) as well as bulk terms that involve
\(\underline{L}\Phi{}\), since these quantities are not controlled by the
\(r^p\)-weighted energy.
\end{remark}
\begin{proof}
\step{Step 1: Energy boundedness.} In this step we will prove \cref{modified-T-prep-0} and
also the statement that for \(0\le \tau_1\le \tau_2\)
and \(v\ge 0\), we have
\begin{equation}\label{modified-T-prep-1}
\tilde{E}[\Phi{}](\tau_2,v) + \int _{\underline{C}(v)\cap \set{\tau{}_1\le \tau{}\le \tau{}_2}}(\underline{L}\Phi{})^2 + r^{-2}\Phi^2 + r^{-2}(\partial{}_{\theta{}}\Phi{})^2 \dd{}u\dd{}\theta{}\lesssim \tilde{E}[\Phi{}](\tau_1,v) + \int_{\tau_1}^{\tau_2} \int _{\Sigma{}(\tau{})} \abs{T\Phi{}}\abs{\tilde{F}}\dd{}r\dd{}\theta{}\dd{}\tau{}.
\end{equation}
In Step 2, we will use \cref{modified-T-prep-1} to conclude
\cref{modified-T-equation}.

Let \(f:(0,\infty)\to \R_{>0}\) be a bounded non-increasing \(C^1\) function.
Multiply \cref{alpha-equation} by \(4f(r)T\Phi{}\), and then use the Leibniz rule and
\cref{vector-field-relations} to get
\begin{equation}\label{modified-T-prep}
\begin{split}
4fT\Phi{}\tilde{F}  &= T\bigl(f(r)[(2-Gh)(L\Phi{})^2 + Gh(\underline{L}\Phi{})^2 + 2\alpha{}r^{-2}(1+f_1(r))\Phi^2 + r^{-2}(1+f_2(r))(\partial{}_\theta\Phi{})^2]\bigr) \\
&\qquad + GX\bigl(f(r)[(\underline{L}\Phi{})^2 - (L\Phi{})^2]\bigr) + \partial{}_\theta{}(\cdots{}) + (-f')[(\underline{L}\Phi{})^2 - (L\Phi{})^2]
\end{split}
\end{equation}
Let \(r_0 > 0\) and \(v_0 > 0\) and integrate \cref{modified-T-prep} over the
region \(\set{\tau_1\le \tau{}\le \tau_2}\cap \set{v\le v_0}\cap \set{r\ge
r_0}\) with respect to the volume form \(G^{-1}\dd{}r\dd{}\theta{}\dd{}\tau{}\).
The boundary term at \(\set{r=r_0}\) vanishes as \(r_0\to 0\) by the boundary
conditions in \cref{bdy-conditions-rp}. By \cref{effective-zo}, the boundary
term at \(\set{v=v_0}\), with integrand \((L\Phi{})^2 +
\alpha{}r^{-2}(1+f_1(r))\Phi^2 +
r^{-2}(1+f_2(r))(\partial{}_{\theta{}}\Phi{})^2\), has a good sign since \(f(r)
> 0\), and it is comparable to the one in \cref{modified-T-prep-1} when
\(f(r)\equiv 1\). The boundary terms at \(\set{\tau{}=\tau_i}\) (\(i = 1,2\))
are comparable to \(\tilde{E}[\Phi{}](\tau_i,v)\) when \(f(r)\equiv 1\) by
\cref{effective-zo}. Thus we obtain \cref{modified-T-prep-1} by taking \(f(r)\equiv
1\). To obtain \cref{modified-T-prep-0}, take \(f(r)=1 + (1+r)^{-\delta{}}\).

\step{Step 2: Integrated local energy decay.} Let \(f:(0,\infty)\to \R_{>0}\) be a bounded
non-decreasing \(C^1\) function. Multiply \cref{alpha-equation} by
\(2f(r)(L-\underline{L})\Phi{}\), and then use the Leibniz rule to get
\begin{equation}\label{modified-T-prep-2}
\begin{split}
2f(r)F(L-\underline{L})\Phi{} &= \underline{L}(f(L\Phi{})^2 - \alpha{}f r^{-2}(1+f_1(r))\Phi^2 - f r^{-2}(1+f_2(r))(\partial{}_\theta{}\Phi{})^2) \\
&\qquad + L(-f(\underline{L}\Phi{})^2 + \alpha{}fr^{-2}(1+f_1(r))\Phi^2 + f r^{-2}(1+f_2(r))(\partial{}_\theta{}\Phi{})^2) \\
&\qquad - 2G\alpha{}(fr^{-2}(1+f_1(r)))'\Phi^2 - 2G(fr^{-2}(1+f_2(r)))'(\partial{}_\theta{}\Phi{})^2 \\
&\qquad + Gf'[(\underline{L}\Phi{})^2 + (L\Phi{})^2] + \partial{}_\theta{}(\cdots{})
\end{split}
\end{equation}
Let \(r_0 > 0\) and \(v_0 > 0\) and integrate \cref{modified-T-prep-2} over the
region \(\set{\tau_1\le \tau{}\le \tau_2}\cap \set{v\le v_0}\cap \set{r\ge
r_0}\) with respect to the volume form \(G^{-1}\dd{}r\dd{}\theta{}\dd{}\tau{}\).
The boundary term at \(\set{r=r_0}\) vanishes as \(r_0\to 0\) by the boundary
conditions in \cref{bdy-conditions-rp}. The boundary terms at
\(\set{\tau{}=\tau_i}\) and \(\set{v=v_0}\) are controlled by the left-hand side of
\cref{modified-T-prep-1}. We now consider the bulk terms. Choose \(f = 1 - (1 +
r)^{-\delta{}}\), so that \(f'\gtrsim_\delta{} \langle{}r\rangle^{-1-\delta{}}\)
and \(-(f r^{-2})' \gtrsim r^{-2}\langle{}r\rangle{}^{-1}\). Since \(f_i(r) = \mathcal{O}(\epsilon{})\), we
conclude using \cref{effective-zo} that
\begin{equation}
\begin{split}
&\tilde{E}[\Phi{}](\tau_2,v) + \int_{\tau_1}^{\tau_2} \int _{\Sigma{}(\tau{},v)} \langle{}r\rangle^{-1-\delta{}}((L\Phi{})^2 + (\underline{L}\Phi{})^2) + r^{-2}\langle{}r\rangle{}^{-1}(\Phi^2 + (\partial{}_\theta{}\Phi{})^2)\dd{}r\dd{}\theta{}\dd{}\tau{}\\
&\lesssim \tilde{E}[\Phi{}](\tau{}_1,v) + \int_{\tau{}_1}^{\tau{}_2} \int _{\Sigma{}(\tau{},v)}(\abs{L\Phi{}} + \abs{\underline{L}\Phi{}})\abs{F}\dd{}r\dd{}\theta{}\dd{}\tau{}.
\end{split}
\end{equation}
Using Young's inequality and noting that
\begin{equation}
\int_{\tau{}_1}^{\tau{}_2} \int _{\Sigma{}(\tau{},v)} \langle{}r\rangle^{1+\delta{}}\abs{F}^2\lesssim \tilde{\mathcal{A}}_{1+\delta{},0}[\tilde{F}](\tau{}_1,\tau{}_2,v)
\end{equation}
completes the proof of \cref{modified-T-equation}.
\end{proof}
\subsubsection{\(r^p\)-weighted energy estimates}
\label{sec:rp-good}
In this section, we prove \(r^p\)-weighted estimates for good scalar fields.
Since these quantities satisfy equations with a zeroth-order term of a
favourable sign, these estimates follow from standard techniques.

We first derive the equation satisfied by the commuted quantities \((rL)^N\Phi{}\).
\begin{lemma}[Equation satisfied by a good scalar field after commutation with \((rL)\)]
Suppose \(\Phi{}\) solves \cref{alpha-equation} for some \(\alpha{}\in \R\) and
\(f_1,f_2=\mathcal{O}(\epsilon{}\langle{}r\rangle^{-a})\). Then for \(N\ge 0\), we have
\begin{equation}\label{commuted-psi-equation}
\begin{split}
0 &= \underline{L}L(rL)^N\Phi{} + GNr^{-1}L(rL)^N\Phi{} + \alpha{}r^{-2}(1+f_1(r))(rL)^N\Phi{} - r^{-2}(1+f_2(r))\partial{}_\theta^2(rL)^N\Phi{} \\
&\qquad + \sum_{n=0}^{N-1}\mathcal{O}(r^{-2})(rL)^n\Phi{} + \sum_{n=0}^{N-1} \mathcal{O}(r^{-2})\partial{}_\theta^2(rL)^n\Phi{} + \sum_{n=0}^NC_{N,n}(rL)^n\tilde{F},
\end{split}
\end{equation}
for some constants \(C_{N,n}\in \R\), where the implicit constants in the
\(\mathcal{O}\)-notation depend on \(\alpha{}\), \(N\), and the functions \(f_1\) and
\(f_2\).
\label{commuted-psi}
\end{lemma}
\begin{proof}
We induct on \(N\). The \(N = 0\) case is immediate from \cref{alpha-equation}. Now
suppose \cref{commuted-psi-equation} holds for some \(N\ge 0\). Multiplying both
sides of the equation by \(r\) and using the Leibniz rule for the first term
gives
\begin{equation}\label{commuted-psi-prep}
\begin{split}
0 &= \underline{L}(rL)^{N+1}\Phi{} + G(N+1)L(rL)^N\Phi{} + \alpha{}r^{-1}(1+f_1(r))(rL)^N\Phi{} - r^{-1}(1+f_2(r))\partial{}_\theta^2(rL)^N\Phi{} \\
&\qquad + \sum_{n=0}^{N-1}\mathcal{O}(r^{-1})(rL)^n\Phi{} + \sum_{n=0}^{N-1} \mathcal{O}(r^{-1})\partial{}_\theta^2(rL)^n\Phi{} + \sum_{n=0}^{N}C_{N,n}r(rL)^n\tilde{F},
\end{split}
\end{equation}
Act with \(L\) on both sides and rewrite
\begin{equation}
LL(rL)^N\Phi{} = r^{-1}L(rL)^{N+1}\Phi{} - r^{-1}L(rL)^N\Phi{}
\end{equation}
to obtain \cref{commuted-psi-equation} with \(N + 1\) in place of \(N\). In
particular, we use the fact that \(f_i(r) =
\mathcal{O}(\langle{}r\rangle^{-a})\) for \(i = 1,2\) to obtain \(L(r^{-1}(1 +
f_i(r))) = \mathcal{O}(r^{-2})\).
\end{proof}
\begin{proposition}[\(r^p\)-weighted energy estimates]
Suppose \(\Phi{}\in C^\infty(\mathcal{R}\setminus \set{r=0})\) satisfies the assumptions of \cref{good-field} with inhomogeneity \(\tilde{F}\). Let \(p\in
[0,2-\eta_0]\) (where the small constant \(\eta_0 > 0\) was fixed in
\cref{sec:notation}). Then for \(0\le \tau_1\le \tau_2\), \(v\ge 0\), and \(N\ge 0\), we have
\begin{equation}\label{rp-general-equation}
\begin{split}
&\tilde{\mathcal{E}}_{p,N}[\Phi{}](\tau{}_2,v) + \sum_{\substack{n_1,n_2\ge 0 \\ n_1+n_2\le N}}\int_{\tau{}_1}^{\tau{}_2} \int _{\Sigma{}(\tau{},v)}(n_2+p)\langle{}r\rangle^pr^{-1}(L\partial{}_\theta^{n_1}(rL)^{n_2}\Phi{})^2 \dd{}r\dd{}\theta{}\dd{}\tau{}\\
&\qquad + \sum_{\substack{n_1,n_2\ge 0 \\ n_1+n_2\le N}}\int_{\tau{}_1}^{\tau{}_2} \int _{\Sigma{}(\tau{},v)} \langle{}r\rangle^pr^{-3}(\partial{}_{\theta{}}\partial{}_\theta^{n_1}(rL)^{n_2}\Phi{})^2 + \langle{}r\rangle^pr^{-3}(\partial{}_\theta^{n_1}\Phi{})^2\dd{}r\dd{}\theta{}\dd{}\tau{}\\
& \lesssim_{\alpha{},\eta{}_0,N} \tilde{\mathcal{E}}_{p,N}[\Phi{}](\tau{}_1,v) + \tilde{\mathcal{A}}_{p,N}[\tilde{F}](\tau{}_1,\tau{}_2,v),
\end{split}
\end{equation}
where the norm \(\mathcal{\tilde{A}}_{p,N}\) was defined in
\cref{inhomogeneity-norm}. It follows that for \(p\in (0,2-\eta{}_0]\), we have
\begin{equation}\label{rp-general-hierarchy}
\tilde{\mathcal{E}}_{p,N}[\Phi{}](\tau{}_2,v) + \int_{\tau{}_1}^{\tau{}_2} \tilde{\mathcal{E}}_{p-1,N}[\Phi{}](\tau{},v)\dd{}\tau{} \lesssim_{\alpha{},p,\eta{}_0,N} \tilde{\mathcal{E}}_{p,N}[\Phi{}](\tau{}_1,v) + \tilde{\mathcal{A}}_{p,N}[\tilde{F}](\tau{}_1,\tau{}_2,v).
\end{equation}
\label{rp-general}
\end{proposition}
\begin{proof}
\step{Step 1: Preliminary reductions.} First, note that \cref{rp-general-hierarchy}
follows immediately from \cref{rp-general-equation}, since \(h(r)\) is bounded. To
establish \cref{rp-general-equation}, it suffices to consider the case where
\(\tilde{\mathcal{E}}_{p,N}[\Phi{}](\tau_2,v)\) is replaced by
\(\tilde{\mathcal{E}}_{p}[(rL)^N\Phi{}](\tau{}_2,v)\) and \(n_1 = 0\) and \(n_2
= N\) in the sum on the left-hand side. This is because \(\partial_\theta{}\Phi{}\)
satisfies the assumptions of \cref{good-field} whenever \(\Phi{}\) does (in
particular since \(\partial_\theta{}\) commutes with equation \cref{alpha-equation}).

Moreover, it suffices to prove the estimates with \(\langle{}r\rangle^p\) in
the bulk term on the left-hand side of \cref{rp-general-equation} replaced with \(r^p\)
and with \(\tilde{\mathcal{E}}\)
replaced with \(\tilde{\mathcal{E}}^{\circ }\),
where
\begin{equation}
\tilde{\mathcal{E}}_p^{\circ }[\Phi{}](\tau{},v) \coloneqq{} \int _{\Sigma{}(\tau{},v)}r^p(L\Phi{})^2 + h(r)r^{p-2}(\partial{}_\theta{}\Phi{})^2 + h(r)r^{p-2}\Phi{}^2\dd{}r\dd{}\theta{}
\end{equation}
and the higher-order version \(\tilde{\mathcal{E}}_{p,N}^{\circ }\) is defined as in
\cref{energy-norms}. This is because the estimate for \(\tilde{\mathcal{E}}_{p,N}\) can
obtained by summing the estimates for \(\tilde{\mathcal{E}}_{0,N}^{\circ }\) and
\(\tilde{\mathcal{E}}_{{p,N}}^{\circ }\).

\step{Step 2: Zeroth-order estimate.} We first establish \cref{rp-general-equation} (with
the appropriate reductions of Step 1) when \(N = 0\).
Multiply \cref{alpha-equation} by \(2r^pL\Phi{}\) and use the Leibniz rule to get
\begin{equation}\label{rp-prelim-0}
\begin{split}
r^pL\Phi{}\cdot \tilde{F} &= \underline{L}(r^p(L\Phi{})^2) + L(r^{p-2}(1+f_2(r))(\partial{}_\theta{}\Phi{})^2 + \alpha{}r^{p-2}(1+f_1(r))\Phi{}^2) + \partial{}_\theta{}(\cdots{}) + Gpr^{p-1}(L\Phi{})^2 \\
&\qquad + G(2-p)r^{p-3}[\alpha{}(1+f_1(r) - (2-p)^{-1}rf_1'(r))\Phi{}^2 + (1+f_2(r) - (2-p)^{-1}rf_2'(r))(\partial{}_\theta{}\Phi{})^2].
\end{split}
\end{equation}
Fix \(r_0 > 0\) and \(v_0>0\). We integrate \cref{rp-prelim-0} over the region
\(\set{\tau_1\le \tau{}\le \tau_2}\cap \set{v\le v_0}\cap \set{r\ge r_0}\) with
respect to \(G^{-1}\dd{}r\dd{}\theta{}\dd{}\tau{}\). The boundary term on
\(\set{\tau{}=\tau_i}\) is proportional to
\(\tilde{\mathcal{E}}_{p}^{\circ }[\Phi{}](\tau{}_i,v)\) by \cref{effective-zo}. The bulk
term that arises is proportional to the bulk term in \cref{rp-general-equation}
(with the appropriate reductions of Step 1) when
\(N = 0\) by \cref{effective-zo}, after one chooses \(\epsilon{}\ll 2-p\le \eta{}_0\), so that
\begin{equation}
1 + f_1(r) - (2-p)^{-1}rf_1'(r) = 1 + O(\epsilon{}\max (1,(2-p)^{-1})) \ge \frac{1}{2}.
\end{equation}
The boundary
term at \(\set{r=r_0}\) is \(r^p[(L\Phi{})^2 -
r^{-2}(1+f_2(r))(\partial{}_\theta{}\Phi{})^2 - \alpha{}r^{-2}(1+f_1(r))\Phi{}^2]\). By the boundary conditions \cref{bdy-conditions-rp}, this term vanishes as
\(r_0\to 0\). Applying \cref{effective-zo} to the boundary term at
\(\set{v=v_0}\), we have
\begin{equation}\label{rp-general-N0-prep}
\begin{split}
&\tilde{\mathcal{E}}_{p}^{\circ }[\Phi{}](\tau{}_2,v) + \int_{\tau{}_1}^{\tau{}_2} \int _{\Sigma{}(\tau{},v)} pr^{p-1}(L\Phi{})^2 + r^{p-3}\Phi{}^2 + r^{p-3}(\partial{}_\theta{}\Phi{})^2\dd{}r\dd{}\theta{}\dd{}\tau{} + \int _{\underline{C}(v)\cap \set{\tau{}_1\le \tau{}\le \tau{}_2}} r^{p-2}\Phi^2\dd{}u\dd{}\theta{} \\
&\lesssim_{\alpha{}} \tilde{\mathcal{E}}_{p}^{\circ }[\Phi{}](\tau{}_1,v) + \abs[\Big]{\int_{\tau{}_1}^{\tau{}_2} \int _{\Sigma{}(\tau{},v)} r^pL\Phi{}\cdot \tilde{F}\dd{}r\dd{}\theta{}\dd{}\tau{}}.
\end{split}
\end{equation}
To control the bulk term involving the inhomogeneity \(F\), we can use Young's
inequality directly when \(p \ge  \eta{}_0\):
\begin{equation}
\abs[\Big]{\int_{\tau_1}^{\tau_2} \int _{\Sigma{}(\tau{},v)} r^pL\Phi{}\cdot \tilde{F} \dd{}r\dd{}\theta{}\dd{}\tau{}}\le \delta{}p\int_{\tau_1}^{\tau_2} \int _{\Sigma{}(\tau{},v)} r^{p-1}(L\Phi{})^2 \dd{}r\dd{}\theta{}\dd{}\tau{} + \delta{}^{-1}\eta{}_0^{-1}\int_{\tau_1}^{\tau_2} \int _{\Sigma{}(\tau{},v)} r^{p+1}\tilde{F}^2 \dd{}r\dd{}\theta{}\dd{}\tau{}.
\end{equation}
The first term on the right can be absorbed into the left-hand side of
\cref{rp-general-N0-prep} (for \(\delta{} > 0\) sufficiently small depending on
\(\alpha{}\)), and the second term on the right is controlled by the norm
\(\tilde{\mathcal{A}}_{p,0}[F](\tau_1,\tau{}_2,v)\). If \(p\le \eta_0\), then we first
integrate the \(L\)-derivative by parts and then use Young's inequality:
\begin{equation}
\begin{split}
&\abs[\Big]{\int_{\tau_1}^{\tau_2} \int _{\Sigma{}(\tau{},v)} r^pL\Phi{}\cdot \tilde{F} \dd{}r\dd{}\theta{}\dd{}\tau{}} \\
&\lesssim \delta{}\Bigl[\int_{\tau_1}^{\tau_2} \int _{\Sigma{}(\tau{},v)} r^{p-3}\Phi^2 \dd{}r\dd{}\theta{}\dd{}\tau{} + \int _{\Sigma{}(\tau{}_2,v)} h(r)r^{p-2}\Phi{}\dd{}r\dd{}\theta{} + \int _{\underline{C}(v_0)\cap \set{\tau{}_1\le \tau{}\le \tau{}_2}}r^{p-2}\Phi{}^2\dd{}u\dd{}\theta{} \Bigr] \\
&\qquad + \delta{}^{-1}\Bigl[\sum_{n=0}^1\int_{\tau_1}^{\tau_2} \int _{\Sigma{}(\tau{},v)} r^{p+1}((rL)^nF)^2 \dd{}r\dd{}\theta{}\dd{}\tau{}+ \sum_{i=1}^2 \int _{\Sigma{}(\tau{}_i,v)} h(r)r^{p+2}\tilde{F}^2\dd{}r\dd{}\theta{} \\
&\qquad + \int _{\Sigma{}(\tau{}_1,v)} h(r)r^{p-2}\Phi{}\dd{}r\dd{}\theta{} + \int _{\underline{C}(v)\cap \set{\tau{}_1\le \tau{}\le \tau{}_2}}r^{p+2}\tilde{F}^2\dd{}u\dd{}\theta{}\Bigr].
\end{split}
\end{equation}
The terms with \(\delta{}\) can be absorbed into the left-hand side of
\cref{rp-general-N0-prep}, and the terms with \(\delta^{-1}\) make up the norm
\(\tilde{\mathcal{A}}_{p,0}[F](\tau{}_1,\tau{}_2,v)\) (after noting that \(h(r)\lesssim
\langle{}r\rangle{}^{-1}\)). In either case, we obtain the \(N = 0\) case of
\cref{rp-general-equation}.

\step{Step 3: Higher-order estimates.} Now let \(N\ge 1\) and suppose
\cref{rp-general-equation} (with the appropriate reductions of Step 1) holds for
\(N-1\) in place of \(N\). Multiply \cref{commuted-psi-equation} by
\(2r^pL(rL)^N\Phi{}\) and use the Leibniz rule to get
\begin{equation}
\begin{split}
0 &= \underline{L}(r^p(L(rL)^N\Phi{})^2) + L\bigl(\alpha{}r^{p-2}(1+f_1(r))((rL)^N\Phi{})^2 + r^{p-2}(1+f_2(r))(\partial{}_\theta{}(rL)^N\Phi{})^2\bigr)  + \partial{}_\theta{}(\cdots{}) \\
&\qquad + 2G(N+p)r^{p-1}(L(rL)^N\Phi{})^2 + G\alpha{}(2-p)r^{p-3}(1 + \mathcal{O}(\epsilon{}\max (1,(2-p)^{-1})))((rL)^N\Phi{})^2 \\
&\qquad + G(2-p)r^{p-3}(1 + \mathcal{O}(\epsilon{}\max (1,(2-p)^{-1})))(\partial{}_\theta{}(rL)^N\Phi{})^2 \\
&\qquad + r^{(p-1)/2}L(rL)^N\Phi{}\cdot \Bigl[\sum_{n=0}^{N-1}\mathcal{O}(r^{(p-3)/2})(rL)^n\Phi{} + \sum_{n=0}^{N-1} \mathcal{O}(r^{(p-3)/2})(\partial{}_\theta{}\partial{}_\theta{}(rL)^n\Phi{})\Bigr] \\
&\qquad + r^pL(rL)^N\Phi{}\cdot \sum_{n=0}^NC_{N,n}(rL)^n\tilde{F},
\end{split}
\end{equation}
Fix \(r_0>0\) and \(v_0 > 0\) and integrate over \(\set{\tau_1\le \tau{}\le \tau_2}\cap
\set{v\le v_0}\cap \set{r\ge r_0}\). As in Step 2, the boundary term at
\(\set{r=r_0}\) vanishes as \(r_0\to 0\). After using
\cref{effective-zo} to show that the bulk term and the boundary terms at
\(\set{\tau{}=\tau_i}\) and \(\set{v=v_0}\) are coercive, we get
\begin{equation}\label{rp-prelim-inequality}
\begin{split}
&\tilde{\mathcal{E}}_p^{\circ }[(rL)^N\Phi{}](\tau{}_2,v) + \int_{\tau{}_1}^{\tau{}_2} \int _{\Sigma{}(\tau{},v)} r^{p-1}(L(rL)^N\Phi{})^2 + r^{p-3}(\partial{}_\theta{}(rL)^N\Phi{})^2 + r^{p-3}((rL)^N\Phi{})^2 \dd{}r\dd{}\theta{}\dd{}\tau{}\\
&\qquad + \int _{\underline{C}(v_0)\cap \set{\tau{}_1\le \tau{}\le \tau{}_2}}r^{p-2}((rL)^N\Phi)^2\dd{}u\dd{}\theta{} \\
&\lesssim \tilde{\mathcal{E}}_p^{\circ }[(rL)^N\Phi{}](\tau{}_1,v) \\
&\qquad +  \int_{\tau{}_1}^{\tau{}_2} \int _{\Sigma{}(\tau{},v)} r^{(p-1)/2}L(rL)^N\Phi{}\cdot \Bigl[\sum_{n=0}^{N-1}\mathcal{O}(r^{(p-3)/2})(rL)^n\Phi{} + \sum_{n=0}^{N-1} \mathcal{O}(r^{(p-3)/2})(\partial{}_\theta{}\partial{}_\theta{}(rL)^n\Phi{})\Bigr]\dd{}r\dd{}\theta{}\dd{}\tau{} \\
&\qquad + \sum_{n=0}^N\int_{\tau{}_1}^{\tau{}_2} \int _{\Sigma{}(\tau{},v)}r^pL(rL)^N\Phi{}\cdot C_{N,n}{(rL)}^n\tilde{F}\dd{}r\dd{}\theta{}\dd{}\tau{}.
\end{split}
\end{equation}
After using Young's inequality and adding a suitable multiple of
\cref{rp-general-equation} (with the appropriate reductions of Step 1) with \(N-1\)
in place of \(N\) (applied to both \(u\) and \(\partial_\theta{}u\)) to control
the terms on the second line of the right-hand side of \cref{rp-prelim-inequality}, we
obtain
\begin{equation}
\begin{split}
\side{LHS}{rp-prelim-inequality}\lesssim \sum_{n=0}^1\tilde{\mathcal{E}}_{p,N-1}^{\circ }[\partial{}_\theta^n\Phi{}](\tau{}_1,v) + \sum_{n=0}^N\int_{\tau{}_1}^{\tau{}_2} r^pL(rL)^N\Phi{}\cdot {(rL)}^n\tilde{F}\dd{}r\dd{}\theta{}\dd{}\tau{}.
\end{split}
\end{equation}
Controlling the bulk term on the right by integrating by parts as in Step 2
completes the proof of \cref{rp-general-equation}, in view of Step 1.
\end{proof}
\subsubsection{An integrated estimate for the modified \(T\)-energy}
The energy estimates of \cref{good-eb,sec:rp-good} allow us to quickly
prove an estimate controlling the integral of the modified
\(T\)-energy.
\begin{proposition}[Integrated estimate for the modified \(T\)-energy]
Suppose \(\Phi{}\in C^\infty(\mathcal{R}\setminus \set{r=0})\) satisfies the assumptions of
\cref{good-field} with inhomogeneity \(\tilde{F}\). Then for \(0\le \tau_1\le
\tau_2\), \(v\ge 0\), and \(\delta{} > 0\), we have
\begin{equation}
\int_{\tau{}_1}^{\tau{}_2} \tilde{E}[\Phi{}](\tau{},v)\dd{}\tau{} \lesssim_\delta{} \tilde{E}[\Phi{}](\tau{}_1,v) + \tilde{\mathcal{E}}_1[\Phi{}](\tau{}_1,v) + \tilde{\mathcal{A}}_{1+\delta{},0}[\tilde{F}](\tau{}_1,\tau{}_2,v).
\end{equation}
\label{integrated-modified-T}
\end{proposition}
\begin{proof}
Since \(h(r)\lesssim \langle{}r\rangle^{-1-\eta{}_{h}}\), we have
\begin{equation}
\int_{\tau{}_1}^{\tau{}_2} \tilde{E}[\Phi{}](\tau{},v)\dd{}\tau{}\lesssim \int_{\tau{}_1}^{\tau{}_2} \int _{\Sigma{}(\tau{},v)} (L\Phi{})^2 + r^{-2}(\partial{}_\theta{}\Phi{})^2 + r^{-2}\Phi^2 + \langle{}r\rangle^{-1-\eta{}_h}(\underline{L}\Phi{})^2  \dd{}r\dd{}\theta{}\dd{}\tau{}.
\end{equation}
To control the final term on the right-hand side, use \cref{modified-T-prep-0}
(with \(\eta_h\) in place of \(\delta{}\) there) to obtain
\begin{equation}
\int_{\tau{}_1}^{\tau{}_2} \tilde{E}[\Phi{}](\tau{},v)\dd{}\tau{}\lesssim \int_{\tau{}_1}^{\tau{}_2} \int _{\Sigma{}(\tau{},v)} (L\Phi{})^2 + r^{-2}(\partial{}_\theta{}\Phi{})^2 + r^{-2}\Phi^2  \dd{}r\dd{}\theta{}\dd{}\tau{} + \int_{\tau{}_1}^{\tau{}_2} \int _{\Sigma{}(\tau{},v)} \abs{T\Phi{}}\abs{\tilde{F}} \dd{}r\dd{}\theta{}\dd{}\tau{}.
\end{equation}
To control the first three terms on the right-hand side, use \cref{rp-general} with \(p
= 1\). For the final term, use an \(r\)-weighted Young's inequality together
with \cref{modified-T-equation}.
\end{proof}
\subsection{Energy estimates for radially symmetric scalar fields}
\label{axisymmetric-energy-estimates} We now turn our attention to radially
symmetric scalar fields, for which proving energy estimates is more difficult,
as discussed in \cref{proof-outline}.
\subsubsection{Energy boundedness and integrated energy estimates}
\label{eb-radially-symmetric}
We first record the standard \(T\)-energy estimate. For a higher-order version
of this estimate (controlling the energy \(E_N[\varphi{}]\) for \(N\ge 1\)), see
\cref{T-estimate-ho}.
\begin{proposition}[Boundedness of the \(T\)-energy]
Let \(\varphi{}\in C^\infty(\mathcal{R})\). For \(0\le \tau_1\le \tau_2\) and \(v\ge 0\), we have
\begin{equation}\label{energy-boundedness-equation}
E[\varphi{}](\tau_2,v) + \sup_{0\le v'\le v}\int _{\underline{C}(v')\cap \set{\tau_1\le \tau{}\le \tau_2}}r(\underline{L}\varphi{})^2 + r^{-1}(\partial_\theta{}\varphi{})^2\dd{}u\dd{}\theta{} \lesssim E[\varphi{}](\tau_1,v) + \int_{\tau_1}^{\tau_{2}} \int _{\Sigma{}(\tau{},v)}\abs{T\varphi{}}\abs{r\Box{}\varphi{}}\dd{}r\dd{}\theta{}\dd{}\tau{}.
\end{equation}
Moreover, for any \(\delta{} > 0\) we have
\begin{equation}\label{energy-boundedness-with-bulk}
\begin{split}
&\int_{\tau{}_1}^{\tau{}_2} \int _{\Sigma{}(\tau{},v)}\langle{}r\rangle{}^{-1-\delta{}} r(\underline{L}\varphi{})^2\dd{}r\dd{}\theta{}\dd{}\tau{} \\
&\lesssim_\delta{} E[\varphi{}](\tau{}_1,v) +  \int_{\tau{}_1}^{\tau{}_2}  \int _{\Sigma{}(\tau{},v)}\langle{}r\rangle{}^{-1-\delta{}}r(L\varphi{})^2\dd{}r\dd{}\theta{}\dd{}\tau{} + \int_{\tau_1}^{\tau_{2}} \int _{\Sigma{}(\tau{},v)}\langle{}r\rangle^{-\delta{}}\abs{T\varphi{}}\abs{r\Box{}\varphi{}}\dd{}r\dd{}\theta{}\dd{}\tau{}.
\end{split}
\end{equation}
\label{T-estimate}
\end{proposition}
\begin{proof}
Let \(f : (0,\infty)\to \R_{>0}\) be a bounded non-increasing \(C^1\) function. Multiply
\cref{box-equation-double-null} by \(-4rf(r)T\varphi{}\), and then use the Leibniz rule and
\cref{vector-field-relations} to get
\begin{equation}\label{T-energy-identity}
\begin{split}
-4rfT\varphi{}\Box{}\varphi{} &= \underline{L}(fr(L\varphi{})^2 + fr^{-1}(\partial{}_\theta{}\varphi{})^2) + L(fr(\underline{L}\varphi{})^2 + fr^{-1}(\partial{}_\theta{}\varphi{})^2) \\
&\qquad - \partial{}_\theta{}(2fr^{-2}\partial{}_\theta{}\varphi{}T\varphi{}) +  G(-f')r[(\underline{L}\varphi{})^2 - (L\varphi{})^2]\\
&=  T(fr(2-h)(L\varphi{})^2 + frh(\underline{L}\varphi{})^2 + fr^{-1}(\partial{}_\theta{}\varphi{})^2) + GX(-fr(L\varphi{})^2 + fr(\underline{L}\varphi{})^2) \\
&\qquad + \partial{}_\theta{}(\cdots{}) + G(-f')r[(\underline{L}\varphi{})^2 - (L\varphi{})^2].
\end{split}
\end{equation}
Let \(r_0> 0\) and \(v_0 > 0\) and integrate \cref{T-energy-identity}
over the region \(\set{\tau_1\le \tau{}\le \tau_2}\cap \set{v\le v_0}\cap
\set{r\ge r_0}\) with respect to the volume form \(G^{-1}\dd{}r\dd{}\theta{}\dd{}\tau{}\). The boundary term at \(\set{r=r_0}\) is \(-f r(L\varphi{})^2 +
f r(\underline{L}\varphi{})^2\), which vanishes as \(r_0\to 0\) due to the
regularity of \(\varphi{}\) at \(\set{r=0}\). To obtain
\cref{energy-boundedness-equation}, take \(f(r)\equiv 1\) and take a supremum over
\(v_0\le v\). To obtain \cref{energy-boundedness-with-bulk}, take \(f(r) = (1+r)^{-\delta{}}\).
\end{proof}
We now establish an estimate that controls a bulk term involving derivatives of
\(\varphi{}\) by the \(T\)-energy of \(\varphi{}\) and a bulk term involving
\(\Psi{}\). We need this estimate to control terms arising from commutation (see \cref{T-estimate-ho}).
\begin{proposition}[Integrated energy estimate; radially symmetric case]
Let \(\varphi{}\in C^\infty(\mathcal{R})\) be radially symmetric. For \(0\le \tau{}_1\le \tau{}_2\), \(v\ge 0\), and \(\delta{}\in (0,1)\), we have
\begin{equation}\label{morawetz-axisymmetric-equation}
\begin{split}
&E[\varphi{}](\tau{}_2,v) + \int_{\tau{}_1}^{\tau{}_2} \int _{\Sigma{}(\tau{},v)} \langle{}r\rangle{}^{-1+\delta{}}[(L\varphi{})^2 + (\underline{L}\varphi{})^2]\dd{}r\dd{}\theta{}\dd{}\tau{} \\
&\lesssim_\delta{} E[\varphi{}](\tau{}_1,v) + \int_{\tau{}_1}^{\tau{}_2} \int _{\Sigma{}(\tau{},v)} r^{-1}\langle{}r\rangle{}^{-1+\delta{}}\Psi^2 \dd{}r\dd{}\theta{}\dd{}\tau{} + \int_{\tau{}_1}^{\tau{}_2} \int _{\Sigma{}(\tau{},v)} \langle{}r\rangle{}^{-1+\delta{}}\abs{r\Box{}\varphi{}}^2 \dd{}r\dd{}\theta{}\dd{}\tau{},
\end{split}
\end{equation}
\label{morawetz-radially-symmetric}
\end{proposition}
\begin{proof}
Take \(f(r) = (1 + r)^{-1 + \delta{}}\), so that \(f\) is bounded and \((rf)'\sim _\delta{}(1 +
r)^{-1 + \delta{}}\). Rewrite \cref{box-equation-double-null} as
\begin{equation}\label{morawetz-axisymmetric-prep}
\underline{L}L\varphi{} = r^{-1}GZ\varphi{} - \Box{}\varphi{} = r^{-3/2}\Psi{} - \Box{}\varphi{}.
\end{equation}
We have reinterpreted the first-order term \(r^{-1}GZ\varphi{}\), which obstructs a
standard Morawetz estimate, as a source term involving \(\Psi{}\), which we have
already controlled. Multiply \cref{morawetz-axisymmetric-prep} by \(2rf(r)(L -
\underline{L})\varphi{} = 4f(r)r^{1/2}\Psi{}\) and use the Leibniz rule and
Young's inequality to get
\begin{equation}\label{Z-estimate-prep}
\begin{split}
\underline{L}(rf(L\varphi{})^2) - L(rf(\underline{L}\varphi{})^2) + G(rf)'[(L\varphi{})^2 + (\underline{L}\varphi{})^2] &= 4r^{-1}f\Psi^2 + 4f\cdot r^{-1/2}\Psi{}\cdot r\Box{}\varphi{} \\
&\lesssim r^{-1}f\Psi^2 + f\abs{r\Box{}\varphi{}}^2.
\end{split}
\end{equation}
Fix \(r_0>0\) and \(v_0>0\). The boundary term at \(\set{r=r_0}\) arising from
integrating \cref{Z-estimate-prep} over the region \(\set{\tau_1\le \tau{}\le
\tau_2}\cap \set{v\le v_0}\cap \set{r\ge r_0}\) is \(rf(L\varphi{})^2 -
rf(\underline{L}\varphi{})^2\), which vanishes as \(r_0\to 0\). The boundary
terms at \(\set{\tau{}=\tau{}_i}\) are controlled by the \(T\)-energy, since
\(f\) and \(G\) are bounded. After adding a suitable multiple of
\cref{energy-boundedness-equation}, we obtain \cref{morawetz-axisymmetric-equation}.
\end{proof}
Next, we turn to proving a higher-order version of
\cref{T-estimate,morawetz-radially-symmetric}. To this end, we compute the terms
that arise when commuting the above estimates with \((rL)^N\).
\begin{lemma}[Commutation with \((rL)^N\)]
Suppose \(\varphi{}\in C^\infty(\mathcal{R})\). Then for \(N\ge 0\), we have
\begin{equation}\label{Z-rL-phi}
r^{1/2}GZ(rL)^N\varphi{} = \sum_{n=0}^N \mathcal{O}(1)(rL)^n\Psi{} + \sum_{n=0}^{N-1}\mathcal{O}(r^{1/2})L(rL)^n\varphi{}
\end{equation}
and
\begin{equation}\label{L-Z-rL-phi}
L(r^{1/2}GZ(rL)^N\varphi{}) = \sum_{n=0}^N\mathcal{O}(1)L(rL)^n\Psi{} + \sum_{n=0}^N\mathcal{O}(r^{-1/2})L(rL)^n\varphi{}.
\end{equation}
Moreover, if \(\varphi{}\) is radially symmetric, we have
\begin{equation}\label{rL-commutation}
[\Box{},(rL)^N]\varphi{} = \sum_{n=0}^{N-1}\bigl[\mathcal{O}(r^{-1})L(rL)^n\varphi{} + \mathcal{O}(r^{-1/2})L(rL)^n\Psi{} + \mathcal{O}(r^{-1/2}\langle{}r\rangle^{-2})(rL)^n\Psi{}\bigr].
\end{equation}
We interpret sums from \(0\) to \(N-1\) as being empty if \(N = 0\).
\label{rL-commutation-lemma}
\end{lemma}
\begin{proof}
First, \cref{Z-rL-phi} follows from the stronger statement that for each \(N\ge 1\)
and \(0\le n\le N\), there exist constants \(C_{N,n}\) such that
\begin{equation}\label{Z-rL-phi-prep}
r^{1/2}GZ(rL)^N\varphi{} = \sum_{n=0}^NC_{N,n}(rL)^n\Psi{} + \sum_{n=0}^{N-1}\mathcal{O}(r^{1/2})L(rL)^n\varphi{}.
\end{equation}
The \(N = 1\) case of \cref{Z-rL-phi-prep} follows from the computation
\begin{equation}\label{Z-rL-phi-prep-2}
r^{1/2}GZ(rL)\varphi{} = (rL)\Psi{} - \frac{1}{2}\Psi{}  + r^{1/2}GL\varphi{},
\end{equation}
and the general case follows by induction on \(N\) after using
\cref{Z-rL-phi-prep-2} with \((rL)^N\varphi{}\) in place of \(\varphi{}\) to write
\begin{equation}
\begin{split}
r^{1/2}GZ(rL)^{N+1}\varphi{} &= r^{1/2}GZ(rL)(rL)^N\varphi{} = (rL)(r^{1/2}GZ(rL)^N\varphi{}) - \frac{1}{2}r^{1/2}GZ(rL)^N\varphi{} + r^{1/2}GL(rL)^N\varphi{}.
\end{split}
\end{equation}
Now \cref{L-Z-rL-phi} is obtained by acting with \(L\) on both sides of \cref{Z-rL-phi}.

The \(N = 1\) case of \cref{rL-commutation} follows from the computation
\begin{equation}\label{box-rL-phi-N1}
[\Box{},rL]\varphi{} = (G^2 + GrG')r^{-1}L\varphi{} + 2r^{-1/2}GL\Psi{} - r^{-1/2}GG'\Psi{},
\end{equation}
which can be obtained from \cref{box-equation-double-null}, and the inequality
\(a\ge 1\) (where \(a\) appears in \cref{sec:metric-assumptions}). The general case
follows by induction on \(N\) after writing
\begin{equation}\label{L-Z-rL-phi-prep}
[\Box{},(rL)^{N+1}\varphi{}] = (rL)[\Box{},(rL)^{N}]\varphi{} + [\Box{},rL](rL)^{N}\varphi{}
\end{equation}
and using the identities \cref{L-Z-rL-phi,Z-rL-phi} to compute the final term in
\cref{L-Z-rL-phi-prep}.
\end{proof}
\begin{proposition}[Higher-order \(T\)-energy estimate and integrated energy estimate; radially symmetric case]
Suppose \(\varphi{}\in C^\infty(\mathcal{R})\) is radially symmetric. Then for \(0\le \tau_1\le \tau_2\),
\(v\ge 0\), \(N\ge 0\), and \(\delta{} > 0\), we have
\begin{equation}\label{T-estimate-ho-equation}
\begin{split}
&E_N[\varphi{}](\tau{}_2,v) + \sum_{n=0}^N\int_{\tau{}_1}^{\tau{}_2} \int _{\Sigma{}(\tau{},v)} \langle{}r\rangle{}^{-1+\delta{}}[(L(rL)^n\varphi{})^2 + (\underline{L}(rL)^n\varphi{})^2]\dd{}r\dd{}\theta{}\dd{}\tau{} \\
&\lesssim_{N,\delta{}} E_N[\varphi{}](\tau{}_1,v) + \tilde{\mathcal{E}}_{1+\delta{},N}[\Psi{}](\tau{}_1,v) + \tilde{\mathcal{A}}_{1+\delta{},N}[\tilde{F}](\tau{}_1,\tau{}_2,v) + \mathcal{A}_{1+\delta{},N}[\Box{}\varphi{}](\tau{}_1,\tau{}_2,v),
\end{split}
\end{equation}
where \(\tilde{F} \coloneqq{} r^{1/2}GZ\Box{}\varphi{}\) is the inhomogeneity in the equation
satisfied by \(\Psi{}\) (see \cref{equation-for-Psi}) and the inhomogeneous norms
\(\mathcal{A}\) and \(\tilde{\mathcal{A}}\) were defined in
\cref{inhomogeneity-norm}.
\label{T-estimate-ho}
\end{proposition}
\begin{proof}
First, the \(N=0\) case follows from \cref{morawetz-radially-symmetric,rp-general}. Now
suppose that \(N\ge 1\) and that we have established \cref{T-estimate-ho-equation}
with \(N-1\) in place of \(N\). Apply \cref{morawetz-radially-symmetric} with \((rL)^N\varphi{}\)
in place of \(\varphi{}\) to obtain
\begin{equation}\label{T-estimate-ho-prep}
\begin{split}
&E[(rL)^N\varphi{}](\tau{}_2,v) + \int_{\tau{}_1}^{\tau{}_2} \int _{\Sigma{}(\tau{},v)} \langle{}r\rangle{}^{-1+\delta{}}[(L(rL)^N\varphi{})^2 + (\underline{L}(rL)^N\varphi{})^2]\dd{}r\dd{}\theta{}\dd{}\tau{} \\
&\lesssim_\delta{} E[(rL)^N\varphi{}](\tau{}_1,v) + \int_{\tau{}_1}^{\tau{}_2} \int _{\Sigma{}(\tau{},v)} r^{-1}\langle{}r\rangle{}^{-1+\delta{}}(r^{1/2}GZ(rL)^N\varphi{})^2 \dd{}r\dd{}\theta{}\dd{}\tau{} \\
&\qquad + \int_{\tau{}_1}^{\tau{}_2} \int _{\Sigma{}(\tau{},v)} \langle{}r\rangle{}^{-1+\delta{}}\abs{r[\Box{},(rL)^N]\varphi{}}^2 \dd{}r\dd{}\theta{}\dd{}\tau{} + \int_{\tau{}_1}^{\tau{}_2} \int _{\Sigma{}(\tau{},v)} \langle{}r\rangle{}^{-1+\delta{}}r^2\abs{(rL)^N\Box{}\varphi{}}^2 \dd{}r\dd{}\theta{}\dd{}\tau{}.
\end{split}
\end{equation}
We now control the bulk terms on the right-hand side of \cref{T-estimate-ho-prep}. Use
\cref{Z-rL-phi} and \cref{rp-general} applied to \(\Psi{}\) (with \(p = 1 + \delta{}\)) to compute
\begin{equation}\label{T-estimate-ho-prep-1}
\begin{split}
&\int_{\tau{}_1}^{\tau{}_2} \int _{\Sigma{}(\tau{},v)} r^{-1}\langle{}r\rangle{}^{-1+\delta{}}(r^{1/2}GZ(rL)^N\varphi{})^2 \dd{}r\dd{}\theta{}\dd{}\tau{}\\
&\lesssim \sum_{n=0}^N\int_{\tau{}_1}^{\tau{}_2} \int _{\Sigma{}(\tau{},v)} r^{-1}\langle{}r\rangle{}^{-1+\delta{}}((rL)^n\Psi{})^2 \dd{}r\dd{}\theta{}\dd{}\tau{} + \sum_{n=0}^{N-1}\int_{\tau{}_1}^{\tau{}_2} \int _{\Sigma{}(\tau{},v)} \langle{}r\rangle{}^{-1+\delta{}}(L(rL)^n\varphi{})^2 \dd{}r\dd{}\theta{}\dd{}\tau{} \\
&\lesssim \tilde{\mathcal{E}}_{1+\delta{},N}[\Psi{}](\tau{}_1,v) +  \tilde{\mathcal{A}}_{1+\delta{},N}[\tilde{F}](\tau{}_1,v) + \sum_{n=0}^{N-1}\int_{\tau{}_1}^{\tau{}_2} \int _{\Sigma{}(\tau{},v)} \langle{}r\rangle{}^{-1+\delta{}}(L(rL)^n\varphi{})^2 \dd{}r\dd{}\theta{}\dd{}\tau{}.
\end{split}
\end{equation}
Next, use \cref{rL-commutation} and \cref{rp-general} applied to \(\Psi{}\) (with \(p = 1 +
\delta{}\)) to compute
\begin{equation}\label{T-estimate-ho-prep-2}
\begin{split}
&\int_{\tau{}_1}^{\tau{}_2} \int _{\Sigma{}(\tau{},v)} \langle{}r\rangle{}^{-1+\delta{}}\abs{r[\Box{},(rL)^N]\varphi{}}^2 \dd{}r\dd{}\theta{}\dd{}\tau{} \\
&\lesssim \sum_{n=0}^{N-1}\int_{\tau{}_1}^{\tau{}_2} \int _{\Sigma{}(\tau{},v)} \langle{}r\rangle{}^{-1+\delta{}}r(L(rL)^n\Psi{})^2 + \langle{}r\rangle^{-4+\delta{}}((rL)^n\Psi{})^2\dd{}r\dd{}\theta{}\dd{}\tau{} \\
&\qquad +  \sum_{n=0}^{N-1}\int_{\tau{}_1}^{\tau{}_2} \int _{\Sigma{}(\tau{},v)} \langle{}r\rangle{}^{-1+\delta{}}(L(rL)^n\varphi{})^2 \dd{}r\dd{}\theta{}\dd{}\tau{} \\
&\lesssim \tilde{\mathcal{E}}_{1+\delta{},N-1}[\Psi{}](\tau{}_1,v) +  \tilde{\mathcal{A}}_{1+\delta{},N-1}[\tilde{F}](\tau{}_1,\tau{}_2,v) + \sum_{n=0}^{N-1}\int_{\tau{}_1}^{\tau{}_2} \int _{\Sigma{}(\tau{},v)} \langle{}r\rangle{}^{-1+\delta{}}(L(rL)^n\varphi{})^2 \dd{}r\dd{}\theta{}\dd{}\tau{}.
\end{split}
\end{equation}
To complete the proof of \cref{T-estimate-ho-equation}, substitute
\cref{T-estimate-ho-prep-1,T-estimate-ho-prep-2} into \cref{T-estimate-ho-prep} and
add a suitable multiple of \cref{T-estimate-ho-equation} with \(N-1\) in place of
\(N\).
\end{proof}
\subsubsection{\(r\)-weighted energy estimates}
\label{r-weighted-estimates} We now establish \(r\)-weighted energy estimates for
radially symmetric scalar fields. We establish these estimates not for
\(\varphi{}\) itself, but for \((rL)\varphi{}\) (see \cref{r-estimate}) and
\(T\varphi{}\) (see \cref{r-estimate-T}). That is, we must commute with one
derivative in order to obtain an \(r\)-weighted energy estimate.

We begin by deriving a general identity associated to multipliers of the form
\(f(r)L\).
\begin{lemma}[Identity associated to the multiplier \(f(r)L\)]
Let \(f(r) = rg(r)\) for \(g : [0,\infty)\to \R_{>0}\) a \(C^2\) function. Then for any
\(v\ge 0\) and \(0\le \tau_1\le \tau_2\), an identity of the following form
holds for radially symmetric functions \(\varphi{}\in C^\infty(\mathcal{R})\):
\begin{equation}\label{prelim-r-estimate-identity}
\begin{split}
&\int _{\Sigma{}(\tau{}_2,v)} (2-Gh)f(L\psi{})^2 + \frac{1}{4}G^2h(2rg' + (1 + O(\epsilon{}))g)\varphi^2\dd{}r\dd{}\theta{} \\
&\qquad + \int_{\tau{}_1}^{\tau{}_2} \int _{\Sigma{}(\tau{},v)} rf'(L\varphi{})^2\dd{}r\dd{}\theta{}\dd{}\tau{}  + \int _{\underline{C}(v)\cap \set{\tau{}_1\le \tau{}\le \tau{}_2}} \frac{1}{4}G^2(2rg' + (1 + O(\epsilon{}))g)\varphi^2 \dd{}u\dd{}\theta{} \\
&= \int _{\Sigma{}(\tau{}_1,v)} (2-Gh)f(L\psi{})^2 + \frac{1}{4}G^2h(2rg' + (1 + O(\epsilon{}))g)\varphi^2\dd{}r\dd{}\theta{} \\
&\qquad + \int_{\tau{}_1}^{\tau{}_2} \int _{\Sigma{}(\tau{},v)}\bigl(O(1)rg'' + O(1)g' + O(\epsilon{}\langle{}r\rangle^{-1-a})g\bigr)\varphi^2 \dd{}r\dd{}\theta{}\dd{}\tau{} \\
&\qquad - \int_{\tau{}_1}^{\tau{}_2}\int _{\Sigma{}(\tau{},v)} 2G^{-1}fL\psi{}r^{1/2}\Box{}\varphi{}\dd{}r\dd{}\theta{}\dd{}\tau{}.
\end{split}
\end{equation}
\label{r-estimate-prep}
\end{lemma}
\begin{proof}
Multiply both sides of \cref{equation-for-psi} by \(-2f(r)L\psi{}\), use the expansions
\begin{equation}\label{dvpsi-expansion}
L\psi{} = \frac{1}{2}Gr^{-1/2}\varphi{} + r^{1/2}L\varphi{},\qquad (L\psi{})^2 = r(L\varphi{})^2 + \frac{1}{2}GL(\varphi{}^2) + \frac{1}{4}G^2r^{-1}\varphi^2,
\end{equation}
and use the Leibniz rule to obtain
\begin{equation}\label{rp-initial-equation}
\begin{split}
-2fL\psi{}r^{1/2}\Box{}\varphi{} &=  \underline{L}(f(L\psi{})^2) + \frac{1}{4}L(-r^{-1}fG(G - 2rG')\varphi^2) + Grf'(L\varphi{})^2 \\
&\qquad + \frac{1}{4}G\Bigl[r(r^{-2}fG(G - 2rG'))' + G^2r^{-1}f'\Bigr]\varphi^2 + \frac{1}{2}G^2f'L(\varphi^2) \\
&= \underline{L}(f(L\psi{})^2) + \frac{1}{4}L\bigl(G^2(2rg' + g  + 2rgG^{-1}G')\varphi^2\bigr) + Grf'(L\varphi{})^2 \\
&\qquad  + \frac{1}{2}G^3\Bigl[-rg'' - (1 + 3rG^{-1}G')g' - gG^{-1}(rG'' + rG^{-1}(G')^2 + G')\Bigr]\varphi{}^2
\end{split}
\end{equation}
In passing to the final line, we expressed \(f = rg\) and differentiated the
\(L(\varphi^2)\) term by parts, which produces a term \(\frac{1}{2}(G^2f')'\varphi^2\)
and a total-\(L\)-derivative term. By the assumptions on \(G\) (and
\cref{axisymmetric-smoothness}), the zeroth-order term (in the last line of
\cref{rp-initial-equation}) has the form
\begin{equation}
\bigl(O(1)rg'' + O(1)g' + O(\epsilon{}\langle{}r\rangle^{-1-a})g\bigr)\varphi^2.
\end{equation}
Now let \(r_0>0\) and \(v_0 > 0\). Use \cref{dvpsi-expansion} and the identity \(f
= rg\) to express the integrand of the boundary term at \(\set{r=r_0}\) that
arises from integrating \cref{rp-initial-equation} on \(\set{\tau{}_1\le \tau{}\le
\tau{}_2}\cap \set{v\le v_0}\cap \set{r\ge r_0}\) with respect to the volume
form \(G^{-1}\dd{}r\dd{}\theta{}\dd{}\tau{}\) as follows:
\begin{equation}
f(L\psi{})^2 - \frac{1}{4}G^2(2rg' + g  + 2rgG^{-1}G')\varphi^2 = r^2g(L\varphi{})^2 + Gf\varphi{}L\varphi{} - \frac{1}{4}G^2(2rg' + 2rgG^{-1}G')\varphi^2.
\end{equation}
This term vanishes as \(r\to 0\) since \(g\) and \(G\) are \(C^1\) functions of
\(r\) (up to and including \(r = 0\)). We therefore obtain the desired identity
upon integration.
\end{proof}
We now prove an estimate for \((rL)\)-derivatives of the solution. We can close
the estimate because with \((rL)\) will produce bulk terms involving \(L(rL)^{\le
1}\varphi{}\) of a good sign as well as error terms involving \(\Psi{}\).
\begin{proposition}[\(r\)-weighted energy estimate for \((rL)\)-derivatives; radially symmetric case]
Suppose \(\varphi{}\in C^\infty(\mathcal{R})\) is radially symmetric. Then for \(0\le \tau_1\le \tau_2\),
\(v\ge 0\), and \(N\ge 0\), and \(\delta{}\ge 0\) sufficiently small, we have
\begin{equation}\label{r-estimate-equation}
\begin{split}
&\mathcal{E}_{1+\delta{},N}[(rL)\varphi{}](\tau{}_2,v) + \sum_{n=0}^{N+1}\int_{\tau{}_1}^{\tau{}_2} \int _{\Sigma{}(\tau{},v)} r\langle{}r\rangle^\delta{}(L(rL)^n\varphi{})^2\dd{}r\dd{}\theta{}\dd{}\tau{} \\
&\lesssim_N \mathcal{E}_{1+\delta{},N}[(rL)\varphi{}](\tau{}_1,v) + \tilde{\mathcal{E}}_{1+\delta{},N}[\Psi{}](\tau{}_1,v) + \tilde{\mathcal{A}}_{1+\delta{},N}[\tilde{F}](\tau{}_1,\tau{}_2,v) + \mathcal{A}_{3+\delta{},N+1}[\Box{}\varphi{}](\tau{}_1,\tau{}_2,v)\\
&\qquad + \sum_{n=0}^N \int_{\tau{}_1}^{\tau{}_2} \int _{\Sigma{}(\tau{},v)}r^2\langle{}r\rangle^\delta{}(L(rL)^n\Psi{})^2\dd{}r\dd{}\theta{}\dd{}\tau{},
\end{split}
\end{equation}
where \(\tilde{F} \coloneqq{} r^{1/2}GZ\Box{}\varphi{}\) is the inhomogeneity in the equation
satisfied by \(\Psi{}\) (see \cref{equation-for-Psi}) and the inhomogeneous norms
\(\mathcal{A}\) and \(\tilde{\mathcal{A}}\) were defined in
\cref{inhomogeneity-norm}.
\label{r-estimate}
\end{proposition}
\begin{proof}
\step{Step 1: The case \(N = 0\).} Set \(f(r)=r(1 + r)^\delta{}\) (i.e.~\(g(r) = (1+r)^\delta{}\)).
We explicitly compute
\begin{equation}
\begin{cases}
2rg'(r) + (1+\mathcal{O}(\epsilon{}))g(r)\sim \langle{}r\rangle^\delta{}, \\
rf'(r)\ge r\langle{}r\rangle^\delta{} \\
\abs{rg''} + \abs{g'} + \epsilon{}\langle{}r\rangle^{-1-a}\abs{g}\lesssim (\delta{} + \epsilon{}) \langle{}r\rangle^{-1+\delta{}}.
\end{cases}
\end{equation}
It follows from \cref{r-estimate-prep} (after dropping the term on
\(\underline{C}(v)\), which has a good sign) that
\begin{equation}\label{rL-phi-commuted-prep-0}
\begin{split}
&C^{-1}\mathcal{E}_{1+\delta{}}[\varphi{}](\tau{}_2,v) + C^{-1}\int_{\tau{}_1}^{\tau{}_2} \int _{\Sigma{}(\tau{},v)} r\langle{}r\rangle{}^\delta{}(L\varphi{})^2\dd{}r\dd{}\theta{}\dd{}\tau{} \\
&\le C\mathcal{E}_{1+\delta{}}[\varphi{}](\tau{}_1,v) +  C(\delta{}+\epsilon{})\int_{\tau{}_1}^{\tau{}_2} \int _{\Sigma{}(\tau{},v)} \langle{}r\rangle^{-1+\delta{}}\varphi^2\dd{}r\dd{}\theta{}\dd{}\tau{} -  \int_{\tau{}_1}^{\tau{}_2} \int _{\Sigma{}(\tau{},v)} 2G^{-1}r(1+r)^\delta{}L\psi{}r^{1/2}\Box{}\varphi{}\dd{}r\dd{}\theta{}\dd{}\tau{}
\end{split}
\end{equation}
for some constant \(C > 0\). Applying \cref{rL-phi-commuted-prep-0} to \((rL)\varphi{}\) in
place of \(\varphi{}\), we obtain
\begin{equation}\label{rL-phi-commuted-prep-1}
\begin{split}
&C^{-1}\mathcal{E}_{1+\delta{}}[(rL)\varphi{}](\tau{}_2,v) + C^{-1}\int_{\tau{}_1}^{\tau{}_2} \int _{\Sigma{}(\tau{},v)} r\langle{}r\rangle{}^\delta{}(L(rL)\varphi{})^2\dd{}r\dd{}\theta{}\dd{}\tau{} \\
&\le C\mathcal{E}_{1+\delta{}}[\varphi{}](\tau{}_1,v) +  C(\delta{}+\epsilon{})\int_{\tau{}_1}^{\tau{}_2} \int _{\Sigma{}(\tau{},v)} r\langle{}r\rangle^\delta{}(L\varphi{})^2\dd{}r\dd{}\theta{}\dd{}\tau{} \\
&\qquad -  \int_{\tau{}_1}^{\tau{}_2} \int _{\Sigma{}(\tau{},v)} 2G^{-1}r(1+r)^\delta{}L(r^{1/2}(rL)\varphi{})r^{1/2}\Box{}(rL)\varphi{}\dd{}r\dd{}\theta{}\dd{}\tau{}
\end{split}
\end{equation}

We now compute the final term on the right-hand side of \cref{rL-phi-commuted-prep-1}
and observe that it produces a bulk term (of a good sign) that controls
\(\langle{}r\rangle^{-1+\delta{}}((rL)\varphi{})^2\). The commutation formula \cref{box-rL-phi-N1} and the
identity
\begin{equation}\label{L-rL-phi-computation}
L(r^{1/2}(rL)\varphi{}) = r^{1/2}L(rL)\varphi{} + \frac{1}{2}r^{1/2}L\varphi{}
\end{equation}
imply
\begin{equation}\label{rL-phi-commuted-prep}
\begin{split}
&G^{-1}r(1+r)^\delta{}L(r^{1/2}(rL)\varphi{})r^{1/2}[\Box{},rL]\varphi{} \\
&= (1+r)^\delta{}\Bigl(r^{1/2}L(rL)\varphi{} + \frac{1}{2}r^{1/2}L\varphi{}\Bigr)((G + rG')r^{1/2}L\varphi{} + 2rL\Psi{} - rG'\Psi{}) \\
\end{split}
\end{equation}
We now expand the parentheses, producing terms proportional to
\begin{equation}
\textnormal{(I)}\coloneqq{}r(1+r)^\delta{}(L\varphi{})^2 \textnormal{ and } \textnormal{(II)}\coloneqq{}(1+r)^\delta{}\cdot r^{1/2}L(rL)\varphi{}\cdot r^{1/2}L\varphi{} = \frac{1}{2}(1+r)^\delta{}L(((rL)\varphi{})^2),
\end{equation}
with proportionality constants independent of \(\delta{}\). We also produce terms
multiplying \(L^{\le 1}\Psi{}\), which we treat as error. The key step is to
differentiate by parts term \(\textnormal{(II)}\),
producing a total-\(L\)-derivative term of a good sign and a bulk term of a bad
sign proportional to \(\delta{}r(1+r)^\delta{}(L\varphi{})^2\), which is
compensated for by the pre-existing such term of a good sign that comes without
a \(\delta{}\)-factor, namely term \(\textnormal{(I)}\). We therefore obtain
\begin{equation}
\begin{split}
&\side{LHS}{rL-phi-commuted-prep}\\
&= \frac{1}{2}L((1+r)^\delta{}(G+rG')((rL)\varphi{})^2) + \frac{1}{2}(G - rG' - r^2G'' - \delta{}(1+r)^{-1})r(1+r)^\delta{}(L\varphi{})^2 \\
&\qquad + (1+r)^\delta{}\Bigl(r^{1/2}L(rL)\varphi{} + \frac{1}{2}r^{1/2}L\varphi{}\Bigr)(2rL\Psi{} - rG'\Psi{}) \\
&\ge \frac{1}{2}L((1+O(\epsilon{}))(1+r)^\delta{}((rL)\varphi{})^2) + \frac{1}{4}r\langle{}r\rangle^\delta{}(L\varphi{})^2 - C\langle{}r\rangle{}^\delta{}(r^{1/2}\abs{L(rL)\varphi{}} + r^{1/2}\abs{L\varphi{}})(r\abs{L\Psi{}} + \langle{}r\rangle{}^{-1}\abs{\Psi{}}).
\end{split}
\end{equation}
for \(\delta{} \ge 0\) sufficiently small (say \(\delta{} < 1/8\)). We used
\(\abs{rG'}\le \langle{}r\rangle^{-1}\) to pass to the final line. Substitute
\cref{L-rL-phi-computation,rL-phi-commuted-prep} into \cref{rL-phi-commuted-prep-1}
and rearrange terms to obtain
\begin{equation}
\begin{split}
&C^{-1}\mathcal{E}_1[(rL)\varphi{}](\tau{}_2,v) + C^{-1}\int_{\tau{}_1}^{\tau{}_2} \int _{\Sigma{}(\tau{},v)} r\langle{}r\rangle^\delta{}(L(rL)\varphi{})^2 + r\langle{}r\rangle^\delta{}(L\varphi{})^2\dd{}r\dd{}\theta{}\dd{}\tau{} \\
&\qquad + \int_{\tau{}_1}^{\tau{}_2} \int _{\Sigma{}(\tau{},v)} L((1+\mathcal{O}(\epsilon{}))(1+r)^\delta{}((rL)\varphi{})^2)\dd{}r\dd{}\theta{}\dd{}\tau{}\\
&\le  C\mathcal{E}_1[(rL)\varphi{}](\tau{}_1,v) +  C(\delta{}+\epsilon{})\int_{\tau{}_1}^{\tau{}_2} \int _{\Sigma{}(\tau{},v)} r\langle{}r\rangle^\delta{}(L\varphi{})^2\dd{}r\dd{}\theta{}\dd{}\tau{} \\
&\qquad + C\int_{\tau{}_1}^{\tau{}_2} \int _{\Sigma{}(\tau{},v)}\langle{}r\rangle^\delta{}(r^{1/2}\abs{L(rL)\varphi{}} + r^{1/2}\abs{L\varphi{}})(r\abs{L\Psi{}} + \langle{}r\rangle^{-1}\abs{\Psi{}} + r^{3/2}\abs{(rL)\Box{}\varphi{}})\dd{}r\dd{}\theta{}\dd{}\tau{}.
\end{split}
\end{equation}
The final term on the left-hand side produces a term at \(\underline{C}(v)\) with a
good sign (and no boundary term at \(\set{r=0}\)). After dropping this term,
absorbing the \((\delta{} + \epsilon{})\)-term to the left-hand side (taking
\(\delta{}\) sufficiently small), and applying Young's inequality, we obtain
\begin{equation}\label{rL-phi-commuted-prep-2}
\begin{split}
&\mathcal{E}_1[(rL)\varphi{}](\tau{}_2,v) + \int_{\tau{}_1}^{\tau{}_2} \int _{\Sigma{}(\tau{},v)} r\langle{}r\rangle^\delta{}(L(rL)\varphi{})^2 + r\langle{}r\rangle^\delta{}(L\varphi{})^2\dd{}r\dd{}\theta{}\dd{}\tau{} \\
&\lesssim \mathcal{E}_1[(rL)\varphi{}](\tau{}_1,v) + \int_{\tau{}_1}^{\tau{}_2} \int _{\Sigma{}(\tau{},v)}r^2\langle{}r\rangle^\delta{}(L\Psi{})^2 + \langle{}r\rangle^{-2+\delta{}}\Psi^2 + r^3\langle{}r\rangle^\delta{}\abs{(rL)\Box{}\varphi{}}^2\dd{}r\dd{}\theta{}\dd{}\tau{}.
\end{split}
\end{equation}
Apply \cref{rp-general} to \(\Psi{}\) (with \(p = 1+\delta{}\)) to control the zeroth-order
term involving \(\Psi{}\) on the right-hand side of \cref{rL-phi-commuted-prep-2} and
obtain \cref{r-estimate-equation} when \(N = 0\).

\step{Step 2: The case \(N\ge 1\).} Suppose \(N\ge 1\) and \cref{r-estimate-equation} holds
with \(N-1\) in place of \(N\). Apply \cref{rL-phi-commuted-prep-2} with
\((rL)^N\varphi{}\) in place of \(\varphi{}\) and use \cref{rL-commutation-lemma} to obtain
\begin{equation}
\begin{split}
&\mathcal{E}_1[(rL)^N(rL)\varphi{}](\tau{}_2,v) + \int_{\tau{}_1}^{\tau{}_2} \int _{\Sigma{}(\tau{},v)} r\langle{}r\rangle^\delta{}(L(rL)^N(rL)\varphi{})^2 + r\langle{}r\rangle^\delta{}(L(rL)^N\varphi{})^2\dd{}r\dd{}\theta{}\dd{}\tau{} \\
&\lesssim \mathcal{E}_1[(rL)^N(rL)\varphi{}](\tau{}_1,v) \\
&\qquad + \int_{\tau{}_1}^{\tau{}_2} \int _{\Sigma{}(\tau{},v)}r^2\langle{}r\rangle^\delta{}(L(r^{1/2}GZ(rL)^N\varphi{}))^2 + \langle{}r\rangle^{-2+\delta{}}(r^{1/2}GZ(rL)^N\varphi{})^2 + r^3\langle{}r\rangle^\delta{}\abs{(rL)\Box{}(rL)^N\varphi{}}^2\dd{}r\dd{}\theta{}\dd{}\tau{} \\
&\lesssim \mathcal{E}_1[(rL)^N(rL)\varphi{}](\tau{}_1,v) + \sum_{n=0}^N\int_{\tau{}_1}^{\tau{}_2} \int _{\Sigma{}(\tau{},v)}r^2\langle{}r\rangle^\delta{}(L(rL)^n\Psi{})^2  + \langle{}r\rangle^{-2+\delta{}}((rL)^n\Psi)^2 + r\langle{}r\rangle^\delta{}(L(rL)^n\varphi{})^2\dd{}r\dd{}\theta{}\dd{}\tau{} \\
&\qquad + \int_{\tau{}_1}^{\tau{}_2} \int _{\Sigma{}(\tau{},v)} r^3\langle{}r\rangle^\delta{}\abs{(rL)^{N+1}\Box{}\varphi{}}^2\dd{}r\dd{}\theta{}\dd{}\tau{}
\end{split}
\end{equation}
To complete the proof of \cref{r-estimate-equation} with \(N\ge 1\), add to
\cref{rL-phi-commuted-prep-2} a multiple of \cref{r-estimate-equation} with \(\delta{}
= 0\) and \(N-1\) in place of \(N\) and a multiple of \cref{rp-general-equation}
applied to \(\Psi{}\) (with \(p = 1+\delta{}\)).
\end{proof}
We now prove an estimate for a \((rL)^NT\)-derivative of the solution. We can
close the estimate when \(N = 0\) because replacing \(\varphi{}\) with
\(T\varphi{}\) turns dangerous zeroth-order terms in the bulk into derivative
bulk terms that we can control using \cref{T-estimate-ho}. To prove the estimate
for \(N \ge 1\), we argue as in the proof of \cref{r-estimate} to commute with
\((rL)\).
\begin{proposition}[\(r\)-weighted energy estimate for a \(T\)-derivative; radially symmetric case]
Suppose \(\varphi{}\in C^\infty(\mathcal{R})\) is radially symmetric. Then for \(0\le \tau_1\le \tau_2\), \(v\ge
0\), \(N\ge 0\), and \(\delta{} > 0\) sufficiently small, we have
\begin{equation}\label{r-estimate-T-equation}
\begin{split}
&\mathcal{E}_{1+\delta{},N}[T\varphi{}](\tau{}_2,v) + \sum_{n=0}^{N}\int_{\tau{}_1}^{\tau{}_2} \int _{\Sigma{}(\tau{},v)}r\langle{}r\rangle^\delta{}(L(rL)^nT\varphi{})^2\dd{}r\dd{}\theta{}\dd{}\tau{} \\
&\lesssim_{N,\delta{}} \mathcal{E}_{1+\delta{},N}[T\varphi{}](\tau{}_1,v) + E_N[\varphi{}](\tau{}_1,v) + \tilde{\mathcal{E}}_{1+\delta{},N}[\Psi{}](\tau{}_1,v) + \tilde{\mathcal{A}}_{1+\delta{},N}[\tilde{F}](\tau{}_1,\tau{}_2,v) + \mathcal{A}_{3+\delta{},N}[T\Box{}\varphi{}](\tau{}_1,\tau{}_2,v) \\
&\qquad + \sum_{n=0}^{N-1}\int_{\tau{}_1}^{\tau{}_2}\int _{\Sigma{}(\tau{},v)}r^{2}\langle{}r\rangle^\delta{}(L(rL)^nT\Psi{})^2\dd{}r\dd{}\theta{}\dd{}\tau{},
\end{split}
\end{equation}
where \(\tilde{F} \coloneqq{} r^{1/2}GZ\Box{}\varphi{}\) is the inhomogeneity in the equation
satisfied by \(\Psi{}\) (see \cref{equation-for-Psi}) and the inhomogeneous norms
\(\mathcal{A}\) and \(\tilde{\mathcal{A}}\) were defined in
\cref{inhomogeneity-norm}.
\label{r-estimate-T}
\end{proposition}
\begin{proof}
Apply \cref{rL-phi-commuted-prep-0} with \(T\varphi{}\) in place of \(\varphi{}\) to obtain
\begin{equation}
\begin{split}
&\mathcal{E}_{1+\delta{}}[T\varphi{}](\tau{}_2,v) + \int_{\tau{}_1}^{\tau{}_2} \int _{\Sigma{}(\tau{},v)} r\langle{}r\rangle{}^\delta{}(LT\varphi{})^2\dd{}r\dd{}\theta{}\dd{}\tau{} \\
&\lesssim \mathcal{E}_{1+\delta{}}[T\varphi{}](\tau{}_1,v) +  \int_{\tau{}_1}^{\tau{}_2} \int _{\Sigma{}(\tau{},v)} \langle{}r\rangle^{-1+\delta{}}(T\varphi{})^2\dd{}r\dd{}\theta{}\dd{}\tau{} - \int_{\tau{}_1}^{\tau{}_2} \int _{\Sigma{}(\tau{},v)} r\langle{}r\rangle^\delta{}\abs{LT\psi{}}\abs{r^{1/2}T\Box{}\varphi{}}\dd{}r\dd{}\theta{}\dd{}\tau{}
\end{split}
\end{equation}
After estimating \(\abs{LT\psi{}}\lesssim r^{1/2}\abs{LT\varphi{}} + r^{-1/2}\abs{T\varphi{}}\) and applying (an \(r\)-weighted)
Young's inequality, we obtain
\begin{equation}\label{r-estimate-T-equation-prep}
\begin{split}
&\mathcal{E}_{1+\delta{}}[T\varphi{}](\tau{}_2,v) + \int_{\tau{}_1}^{\tau{}_2} \int _{\Sigma{}(\tau{},v)}r\langle{}r\rangle^\delta{}(LT\varphi{})^2\dd{}r\dd{}\theta{}\dd{}\tau{} \\
&\lesssim \mathcal{E}_{1+\delta{}}[T\varphi{}](\tau{}_1,v) + \int_{\tau{}_1}^{\tau{}_2} \int _{\Sigma{}(\tau{},v)}\langle{}r\rangle^{-1+\delta{}}(T\varphi{})^2\dd{}r\dd{}\theta{}\dd{}\tau{} + \int_{\tau{}_1}^{\tau{}_2}\int _{\Sigma{}(\tau{},v)} \langle{}r\rangle^{1+\delta{}}\abs{r\Box{}T\varphi{}}^2\dd{}r\dd{}\theta{}\dd{}\tau{}.
\end{split}
\end{equation}
Control the bulk term associated to \(T\varphi{}\) on the right-hand side of
\cref{r-estimate-T-equation-prep} using \cref{T-estimate-ho} to obtain
\cref{r-estimate-T-equation} when \(N = 0\). The case \(N\ge 1\) follows as in the
proof of \cref{T-estimate-ho}.
\end{proof}
\subsubsection{An integrated estimate for the \(T\)-energy}
We now use the energy estimates of
\cref{eb-radially-symmetric,r-weighted-estimates} to estimate the integral of the
higher-order \(T\)-energy, in preparation for an argument that establishes
energy decay using the pigeonhole principle (see \cref{sec:energy-decay}).
\begin{proposition}[Integrated estimate for the higher-order \(T\)-energy; radially symmetric case]
Suppose \(\varphi{}\in C^\infty(\mathcal{R})\) is radially symmetric. Then for \(0\le \tau_1\le \tau_2\),
\(v\ge 0\), \(N\ge 1\), and \(\delta{} > 0\) sufficiently small, we have
\begin{equation}
\begin{split}
&\int_{\tau{}_1}^{\tau{}_2} E_N[\varphi{}](\tau{},v)\dd{}\tau{}\\
&\lesssim_{N,\delta{}}  E_N[\varphi{}](\tau{}_1,v) + \mathcal{E}_{1,N-1}[(rL)\varphi{}](\tau{}_1,v) + \tilde{\mathcal{E}}_{1+\delta{},N}[\Psi{}](\tau{}_1,v) + \tilde{\mathcal{A}}_{1,N}[\tilde{F}](\tau{}_1,\tau{}_2,v) + \mathcal{A}_{3,N}[T\Box{}\varphi{}](\tau{}_1,\tau{}_2,v)\\
&\qquad + \sum_{n=0}^{N-1}\int_{\tau_1}^{\tau_{2}} \int _{\Sigma{}(\tau{},v)} r\langle{}r\rangle{}(L(rL)^n\Psi{})^2\dd{}r\dd{}\theta{}\dd{}\tau{},
\end{split}
\end{equation}
where \(\tilde{F} \coloneqq{} r^{1/2}GZ\Box{}\varphi{}\) is the inhomogeneity in the equation
satisfied by \(\Psi{}\) (see \cref{equation-for-Psi}), the inhomogeneous norms
\(\mathcal{A}\) and \(\tilde{\mathcal{A}}\) were defined in
\cref{inhomogeneity-norm}, and \(\eta_h\) is a parameter associated to the
hyperboloidal foliation (see \cref{hyperboloidal-foliation}).
\label{integrated-T-ho}
\end{proposition}
\begin{proof}
First, use \(h(r)\lesssim \langle{}r\rangle^{-1-\eta_h}\), and then use \cref{r-estimate} applied to
\(\varphi{}\) and \cref{energy-boundedness-with-bulk} (where \(\eta_h\) takes the
role of \(\delta{}\)) applied to \((rL)^n\varphi{}\) for \(0\le n\le N\) in place of \(\varphi{}\)
to obtain
\begin{equation}\label{eb-with-bulk-prep-0}
\begin{split}
&\int_{\tau{}_1}^{\tau{}_2} E_N[\varphi{}](\tau{},v)\dd{}\tau{} \\
&\lesssim \sum_{n=0}^N\int_{\tau{}_1}^{\tau{}_2} \int _{\Sigma{}(\tau{},v)}r(L(rL)^n\varphi{})^2 + \langle{}r\rangle^{-1-\eta{}_h}r(\underline{L}(rL)^n\varphi{})^2\dd{}r\dd{}\theta{}\dd{}\tau{} \\
&\lesssim \sum_{n=0}^N \int_{\tau{}_1}^{\tau{}_2}  \int _{\Sigma{}(\tau{},v)}r(L(rL)^n\varphi{})^2\dd{}r\dd{}\theta{}\dd{}\tau{} + \int_{\tau_1}^{\tau_{2}} \int _{\Sigma{}(\tau{},v)}\abs{T(rL)^n\varphi{}}\abs{r\Box{}(rL)^n\varphi{}}\dd{}r\dd{}\theta{}\dd{}\tau{}.
\end{split}
\end{equation}
We now control the final bulk term on the right-hand side of
\cref{eb-with-bulk-prep-0}. Now use an \(r\)-weighted Young's inequality and then use
\cref{rL-commutation} to estimate
\begin{equation}\label{eb-with-bulk-prep-1}
\begin{split}
&\sum_{n=0}^N\int_{\tau_1}^{\tau_{2}} \int _{\Sigma{}(\tau{},v)}\abs{T(rL)^n\varphi{}}\abs{r\Box{}(rL)^n\varphi{}}\dd{}r\dd{}\theta{}\dd{}\tau{}\\
&\lesssim \sum_{n=0}^N\int_{\tau_1}^{\tau_{2}} \int _{\Sigma{}(\tau{},v)}\langle{}r\rangle^{-1+\delta{}}(T(rL)^n\varphi{})^2 + \langle{}r\rangle^{1-\delta{}}\abs{r\Box{}(rL)^n\varphi{}}^2\dd{}r\dd{}\theta{}\dd{}\tau{}\\
&\lesssim \sum_{n=0}^N\int_{\tau_1}^{\tau_{2}} \int _{\Sigma{}(\tau{},v)} r(L(rL)^n\varphi{})^2 + \langle{}r\rangle^{-1+\delta{}}[(L(rL)^n\varphi{})^2 + (\underline{L}(rL)^n\varphi{})^2]\dd{}r\dd{}\theta{}\dd{}\tau{} \\
&\qquad +\sum_{n=0}^{N-1}\int_{\tau_1}^{\tau_{2}} \int _{\Sigma{}(\tau{},v)} r\langle{}r\rangle^{1-\delta{}}(L(rL)^n\Psi{})^2 + r\langle{}r\rangle^{-3-\delta{}}((rL)^n\Psi{})^2 + r^2\langle{}r\rangle^{1-\delta{}}\abs{(rL)^n\Box{}\varphi{}}^2\dd{}r\dd{}\theta{}\dd{}\tau{}.
\end{split}
\end{equation}
To complete the proof, substitute \cref{eb-with-bulk-prep-1} into
\cref{eb-with-bulk-prep-0} and use \cref{T-estimate-ho,r-estimate,rp-general} to
control the terms on the right.
\end{proof}
\section{Energy decay and improved energy decay for time derivatives}
\label{sec:energy-decay} The goal of this section is to show that energies
associated to time derivatives decay faster in time. In \cref{sec:time-integrals},
we will show that a suitably renormalized version of \(\varphi{}\) (namely
\(\widehat{\varphi{}} = \varphi{} -
\mathfrak{L}[\varphi{}]\varphi_{\textnormal{mink}}\)) can be written as a time
derivative. In this way we will obtain improved decay for \(\varphi{}\) itself.

In \cref{abstract-interpolation}, we prove an abstract interpolation lemma (a
version of \cite[Prop.~8.2]{GAJIC2023110058}). We then use this lemma to prove
improved decay for time derivatives of good scalar fields in
\cref{sec:energy-decay-good} and for time derivatives of radially symmetric scalar
fields in \cref{sec:energy-decay-radially-symmetric}.
\subsection{An abstract interpolation lemma}
\label{abstract-interpolation}
We formulate an abstract version of the interpolation argument used in
\cite[Prop.~8.2]{GAJIC2023110058}.
\begin{lemma}[Abstract interpolation lemma]
Let \(\alpha{},\beta{},\delta{},D\ge 0\), and let \(f:[0,\infty)\to
\R_{\ge 0}\) be an integrable function satisfying the following assumptions:
\begin{align}
f(0)&\lesssim D,\label{data-bound}\\
f(\tau{}')&\lesssim f(\tau{}) + (1+\tau{})^{-\beta{}}D\quad \textnormal{for }\tau{}'\ge \tau{}\ge 0, \label{energy-bound}\\
\int_{\tau{}}^{\infty} f(\tau{}')\dd{}\tau{}'&\lesssim R^{-\alpha{}}(1+\tau{})^\delta{}D + R(f(\tau{}) + (1+\tau{})^{-\beta{}}D)\quad \textnormal{for }1\le R\le \tau{}. \label{dyadic-bound}
\end{align}
Then for \(\tau{}\ge 0\) and \(\eta{} > 0\), we have
\begin{equation}\label{iteration-conc}
f(\tau{})\lesssim_{\alpha{},\beta{},\eta{}} (1+\tau{})^{-\min (\beta{},\alpha{}+1) + \delta{} + \eta{}}D.
\end{equation}
\label{interpolation}
\end{lemma}
\begin{remark}
In \cref{sec:energy-decay-good,sec:energy-decay-radially-symmetric}, we will apply
\cref{interpolation} when \(f(\tau{})\) is an energy defined on \(\Sigma{}(\tau{})\),
\(D\) is an initial data norm, \cref{data-bound} is the boundedness of initial
data, \cref{energy-bound} is an energy estimate, and \cref{dyadic-bound} is an
integrated energy estimate obtained by splitting the spatial integral defining
the energy \(f(\tau{})\) into the regions \(\set{r\le R}\) and \(\set{r\ge R}\).
\end{remark}
\begin{proof}
By \cref{energy-bound} and \cref{data-bound}, we have
\begin{equation}\label{energy-bound-cons}
f(\tau{})\lesssim f(0) + D\lesssim D\quad \textnormal{for }\tau{}\ge 0,
\end{equation}
and so it suffices to prove \cref{iteration-conc} for \(\tau{}\ge 1\).

\step{Step 1: The iteration argument.} Suppose we
have shown that
\begin{equation}\label{iteration-1}
f(\tau{})\lesssim \tau^{-\gamma{}+\delta{}}D\quad \textnormal{for }\tau{}\ge 1
\end{equation}
for some \(0\le \gamma{}\le \min (\beta{},\alpha{}+1)\). We will show that \cref{iteration-1} holds with
\(\min (1 + (1-\frac{1}{\alpha{}+1})\gamma{},\beta{})\) in place of \(\gamma{}\)
(with an implicit constant that is larger by a factor independent of
\(\gamma{}\)). Note that the case \(\gamma{} = 0\) of \cref{iteration-1} follows
from \cref{energy-bound-cons}.

Let \(\tau_i=2^i\) be a dyadic sequence. Combine \cref{iteration-1} with
\cref{dyadic-bound} to obtain
\begin{equation}\label{iteration-2}
\int_{\tau_i}^{\tau_{i + 1}} f(\tau{})\dd{}\tau{}\lesssim R^{-\alpha{}}\tau_i^\delta{}D + R\tau_i^{-\gamma{}}D\quad \textnormal{for }1\le R\le \tau{}_i.
\end{equation}
Since \(\gamma{}\le \alpha{} + 1\), we can take \(R = \tau_i^{\gamma{}/(\alpha{}+1)}\) in \cref{iteration-2} to obtain
\begin{equation}
\int_{\tau_i}^{\tau_{i + 1}} f(\tau{})\dd{}\tau{}\lesssim \tau{}_i^{-\gamma{}\alpha{}/(\alpha{}+1)+\delta{}}D.
\end{equation}
By the pigeonhole principle, there is \(\tau_i'\in [\tau_i,\tau_{i + 1}]\) such that
\begin{equation}\label{iteration-3}
f(\tau{}_i')\lesssim \tau{}_i^{-1-\gamma{}\alpha{}/(\alpha{}+1)+\delta{}}D.
\end{equation}
By \cref{energy-bound} and the dyadicity of the \(\tau_i\), and hence of the \(\tau_i'\),
we can extend the estimate \cref{iteration-3} to
\begin{equation}\label{iteration-4}
f(\tau{})\lesssim \tau{}^{-\min (1+\gamma{}\alpha{}/(\alpha{}+1),\beta{})+\delta{}}D\quad \textnormal{for }\tau{}\ge 1,
\end{equation}
as desired.

\step{Step 2: Completing the proof.} Define a sequence \(\gamma_n\) for \(n\ge 0\) by
\begin{equation}
\gamma{}_0=0,\qquad \gamma{}_{n+1} = \min \Bigl(1 + \Bigl(1-\frac{1}{\alpha{}+1}\Bigr)\gamma{}_n,\beta{}\Bigr).
\end{equation}
If \(\gamma_n\le \alpha{} + 1\), then so is \(\gamma_{n + 1}\) (since \(\gamma_{n + 1}\le 1 + (1-1/(\alpha{} +
1))\gamma_n\)). Since \(\gamma_0 = 0\), it follows that \(\gamma_n\le \min (\beta{},\alpha{} +
1)\) for all \(n\ge 0\). In particular, Step 1 shows that
\begin{equation}
f(\tau{})\lesssim_n \tau^{-\gamma{}_n+\delta{}}D\quad \textnormal{for }\tau{}\ge 1.
\end{equation}
If \(\beta{}\ge \alpha{} + 1\), then by induction we have
\begin{equation}
\gamma{}_{n+1} = 1 + \Bigl(1-\frac{1}{\alpha{}+1}\Bigr)\gamma{}_n\implies \gamma{}_{n+1} = \sum_{k=0}^n\Bigl(1-\frac{1}{\alpha{}+1}\Bigr)^k = \alpha{}+1-\alpha{}\Bigl(1-\frac{1}{\alpha{}+1}\Bigr)^n.
\end{equation}
If \(\beta{}<\alpha{} + 1\), then we have
\begin{equation}
\gamma{}_{n+1} = \min \Bigl(\alpha{}+1-\alpha{}\Bigl(1-\frac{1}{\alpha{}+1}\Bigr)^n,\beta{}\Bigr),
\end{equation}
and there is some \(N\ge 0\) such that \(\gamma_n=\beta{}\) for \(n\ge N\). In either case, we
establish \cref{iteration-conc} by taking \(n\) sufficiently large that
\(\alpha{}(1-1/(\alpha{} + 1))^n \le \eta{}\).
\end{proof}
\subsection{Energy decay for scalar fields that satisfy an equation with a good zeroth-order term}
\label{sec:energy-decay-good} In this section, we closely follow
\cite[Sec.~8]{GAJIC2023110058} to show that energies of time derivatives of good
scalar fields decay faster in time. We first show that we can express
\(T\)-derivatives in terms of \((rL)\)-derivatives and \(\partial_\theta{}\)-derivatives.
\begin{lemma}
If \(\Phi{}\in C^\infty(\mathcal{R}\setminus \set{r=0})\) solves \cref{alpha-equation} with
inhomogeneity \(F\in C^\infty(\mathcal{R}\setminus \set{r=0})\), then for \(N\ge 0\) and \(M\ge 1\), we have
\begin{equation}\label{commuted-Tu-equation}
\begin{split}
&L(rL)^NT^M\Phi{} \\
&= \sum_{n=0}^{N+M}\mathcal{O}(r^{-M})L(rL)^n\Phi{} \\
&\qquad + \sum_{m=0}^{M-1}\sum_{n=0}^{N+M-m-1}\Bigl(\mathcal{O}(r^{-1-M+m})(rL)^nT^m\Phi{} + \mathcal{O}(r^{-1-M+m})\partial{}_\theta^2(rL)^nT^m\Phi{} + \mathcal{O}(r^{1-M+m})(rL)^nT^mF\Bigr),
\end{split}
\end{equation}
where the implied constants in the \(\mathcal{O}\)-notation depend on \(N\),
\(M\), and the parameters of \cref{alpha-equation}.
\label{commuted-Tu}
\end{lemma}
\begin{proof}
Use \(2T = L + \underline{L}\) to write
\begin{equation}
\begin{split}
2L(rL)^NT\Phi{} &= LL(rL)^N\Phi{} + \underline{L}L(rL)^N\Phi{} = r^{-1}L(rL)^{N+1}\Phi{} - r^{-1}L(rL)^N\Phi{} + \underline{L}L(rL)^N\Phi{}.
\end{split}
\end{equation}
Rewrite the final term on the right using \cref{commuted-psi} to obtain
the \(M = 1\) case of \cref{commuted-Tu-equation}:
\begin{equation}\label{commuted-Tu-M1}
\begin{split}
L(rL)^NT\Phi{} &= \sum_{n=0}^{N+1}\mathcal{O}(r^{-1})L(rL)^n\Phi{} + \sum_{n=0}^{N}\mathcal{O}(r^{-2})(rL)^n\Phi{} + \sum_{n=0}^{N} \mathcal{O}(r^{-2})\partial{}_\theta^2(rL)^n\Phi{} + \sum_{n=0}^N\mathcal{O}(1)(rL)^nF.
\end{split}
\end{equation}
To complete the proof, induct on \(M\) and use \cref{commuted-Tu-M1} in the
induction step.
\end{proof}
\begin{proposition}[Energy decay and improved energy decay for time derivatives]
Let \(\Phi{}\in C^\infty(\mathcal{R}\setminus \set{r=0})\) satisfy the assumptions of
\cref{rp-general}. For \(p\in [0,2-\eta{}_0]\) (where \(\eta_0\) was fixed in
\cref{rp-general}), \(N\ge 0\), \(M\ge 0\), and \(v\ge 0\), define
\begin{equation}
\begin{split}
\tilde{\mathcal{D}}_{p,N,M}[\Phi{}](v)&\coloneqq{}\sum_{m=0}^M \Bigl(\tilde{\mathcal{E}}_{p,N+M-m}[T^{m}\Phi{}](0,v) + \sup_{\tau{}\ge 0}(1+\tau{})^{2m}\tilde{\mathcal{A}}_{p+1,N+M-m}[T^{m}F](\tau{},\infty,v)\Bigr),
\end{split}
\end{equation}
where \(F\in C^\infty(\mathcal{R}\setminus \set{r=0})\) is the inhomogeneity in the equation
solved by \(\Phi{}\) (see \cref{good-field}) and the norm
\(\tilde{\mathcal{A}}_{p,N}[F](\tau{})\) was defined in \cref{inhomogeneity-norm}. For \(p\in
[0,2-\eta{}_0]\), \(N\ge 0\), \(\tau{}\ge 0\), and \(v\ge 0\), we have the energy boundedness
estimate
\begin{equation}\label{improved-decay-u-0}
\tilde{\mathcal{E}}_{p,N}[\Phi{}]\lesssim_{N,M,\eta{}_0} \tilde{\mathcal{D}}_{p,N,M}[\Phi{}](v),
\end{equation}
as well as the following improved energy decay estimates for time derivatives,
where \(\eta{} > 0\) is arbitrary:
\begin{equation}\label{improved-decay-u}
\tilde{\mathcal{E}}_{p,N}[T^M\Phi{}](\tau{},v)\lesssim_{N,M,\eta{}_0,\eta{}} (1+\tau{})^{-2M+ \eta{}} \tilde{\mathcal{D}}_{p,N,M}[\Phi{}](v)\qquad (M\ge 1).
\end{equation}
\label{energy-decay}
\end{proposition}
\begin{proof}
First, \cref{improved-decay-u-0} follows from \cref{rp-general}, so we can consider
\(M\ge 1\). By \cref{rp-general}, it suffices to consider \(\tau{}\ge 1\).
Let \(1\le \tau_1\le \tau_2\) and let \(R\ge 1\). Split into the regions
\(\set{r\le R}\) and \(\set{r\ge R}\) and use \(h(r)\lesssim \langle{}r\rangle{}^{-1}\) to obtain
\begin{equation}\label{improved-T-decay-u-0}
\begin{split}
&\int_{\tau{}_1}^{\tau{}_2} \tilde{\mathcal{E}}_{p,N}[T^M\Phi{}](\tau{},v)\dd{}\tau{} \\
&= \sum_{\substack{n_1,n_2\ge 0 \\ n_1+n_2\le N}}\int_{\tau{}_1}^{\tau{}_2} \int _{\Sigma{}(\tau{},v)} r^p(L\partial{}_\theta^{n_1}(rL)^{n_2}T^M\Phi{})^2 \dd{}r\dd{}\theta{}\dd{}\tau{} \\
&\qquad + \sum_{\substack{n_1,n_2\ge 0 \\ n_1+n_2\le N}}\int_{\tau{}_1}^{\tau{}_2} h(r)r^{p-2}[(\partial{}_\theta{}\partial{}_\theta^{n_1}(rL)^{n_2}T^M\Phi{})^2 + (\partial{}_\theta^{n_1}(rL)^{n_2}T^M\Phi{})^2] \dd{}r\dd{}\theta{}\dd{}\tau{} \\
&\lesssim R\sum_{\substack{n_1,n_2\ge 0 \\ n_1+n_2\le N}} \int_{\tau{}_1}^{\tau{}_2} \int _{\Sigma{}(\tau{},v)\cap \set{r\le R}} r^{p-1}(L\partial{}_\theta^{n_1}(rL)^{n_2}T^M\Phi{})^2\dd{}r\dd{}\theta{}\dd{}\tau{} \\
&\qquad + R^{-1}\sum_{\substack{n_1,n_2\ge 0 \\ n_1+n_2\le N}}\int_{\tau{}_1}^{\tau{}_2} \int _{\Sigma{}(\tau{},v)\cap \set{r\ge R}} r^{p+1}(L\partial{}_\theta^{n_1}(rL)^{n_2}T^M\Phi{})^2\dd{}r\dd{}\theta{}\dd{}\tau{} \\
&\qquad  + \sum_{\substack{n_1,n_2\ge 0 \\ n_1+n_2\le N}}\int_{\tau{}_1}^{\tau{}_2} \int _{\Sigma{}(\tau{},v)} r^{p-3}(\partial{}_\theta{}\partial{}_\theta^{n_1}(rL)^{n_2}T^M\Phi{})^2 + r^{p-3}(\partial{}_\theta^{n_1}(rL)^{n_2}T^M\Phi{})^2 \dd{}r\dd{}\theta{}\dd{}\tau{} \\
&\lesssim R (\tilde{\mathcal{E}}_{p,N}[T^M\Phi{}](\tau{}_1,v) + \tau{}_1^{-2M} \tilde{\mathcal{D}}_{p,N,M}[\Phi{}](v))  + R^{-1}\cdot \textnormal{(I)}_{R,N,M},
\end{split}
\end{equation}
where
\begin{equation}
\textnormal{(I)}_{R,N,M}\coloneqq{}\sum_{\substack{n_1,n_2\ge 0 \\ n_1+n_2\le N}}\int_{\tau{}_1}^{\tau{}_2} \int _{\Sigma{}(\tau{},v)\cap \set{r\ge R}} r^{p+1}(L\partial{}_\theta^{n_1}(rL)^{n_2}T^M\Phi{})^2\dd{}r\dd{}\theta{}\dd{}\tau{}.
\end{equation}
To pass to the last line in \cref{improved-T-decay-u-0}, we used \cref{rp-general} and
the estimate \(\tilde{\mathcal{A}}_{p,N}[T^MF](\tau{}_1,\tau{}_2,v)\le
\tau{}_1^{-2M} \tilde{\mathcal{D}}_{p,N,M}[\Phi{}](v)\) (which follows from the
definition of the norm \(\tilde{\mathcal{D}}\)). To control the term
\(\textnormal{(I)}_{R,N,M}\), use \cref{commuted-Tu} to estimate
\begin{equation}\label{improved-T-decay-prep-0}
\begin{split}
&\textnormal{(I)}_{R,N,M}\\
&\lesssim \sum_{\substack{n_1,n_2\ge 0 \\ n_1+n_2\le N}}\int_{\tau{}_1}^{\tau{}_2} \int _{\Sigma{}(\tau{},v)\cap \set{r\ge R}} r^{-2(M-1)}\sum_{n=0}^{n_2+M}r^{p-1}(L\partial{}_\theta^{n_1}(rL)^n\Phi{})^2\dd{}r\dd{}\theta{}\dd{}\tau{} \\
&\qquad + \sum_{\substack{n_1,n_2\ge 0 \\ n_1+n_2\le N}}\sum_{m=0}^{M-1}r^{-2(M-m-1)} \sum_{n=0}^{n_2+M-m-1}\\
&\qquad \int_{\tau{}_1}^{\tau{}_2} \int _{\Sigma{}(\tau{},v)\cap \set{r\ge R}}\bigl(r^{p-3}(\partial{}_\theta\partial{}_\theta^{n_1+1}(rL)^nT^m\Phi{})^2 + r^{p-3}(\partial{}_\theta^{n_1}(rL)^nT^m\Phi{})^2 + r^{p+1}(\partial{}_\theta^{n_1}(rL)^nT^mF)^2\bigr)\dd{}r\dd{}\theta{}\dd{}\tau{}.
\end{split}
\end{equation}
Using the \(r^p\)-weighted energy estimate of \cref{rp-general} to estimate the
right-hand side of \cref{improved-T-decay-prep-0}, we obtain
\begin{equation}\label{improved-decay-iteration-pre}
\begin{split}
\textnormal{(I)}_{R,N,M}&\lesssim R^{-2(M-1)}\sum_{m=0}^{M-1} R^{2m} (\tilde{\mathcal{E}}_{p,N+M-m}[T^m\Phi{}](\tau{}_1,v) +\tau{}_1^{-2m} \tilde{\mathcal{D}}_{p,N,M}[\Phi{}](v)).
\end{split}
\end{equation}
In particular, when \(M = 1\), we have
\begin{equation}\label{improved-T-decay-u-1}
\textnormal{(I)}_{R,N,1}\lesssim \tilde{\mathcal{D}}_{p,N,1}[\Phi{}](v).
\end{equation}
Substitute \cref{improved-T-decay-u-1} into \cref{improved-T-decay-u-0} to obtain
\begin{equation}\label{iterated-estimate-0}
\int_{\tau{}_1}^{\tau{}_2} \tilde{\mathcal{E}}_{p,N}[T\Phi{}](\tau{},v)\dd{}\tau{}\lesssim R^{-1} \tilde{\mathcal{D}}_{p,N,1}[\Phi{}](v) + R(\tilde{\mathcal{E}}_{p,N}[T\Phi{}](\tau{}_1,v) + \tau{}_1^{-2} \tilde{\mathcal{D}}_{p,N,1}[\Phi{}](v)).
\end{equation}
Since \(\tau_2\ge \tau{}_1\) is arbitrary, \cref{improved-decay-u} for \(M=1\) now follows
from \cref{interpolation} (where the assumptions \cref{data-bound,energy-bound} of
\cref{interpolation} follow from \cref{rp-general} and \cref{iterated-estimate-0} plays
the role of \cref{dyadic-bound}).

Now suppose that \(M>1\) and that we have established \cref{improved-decay-u} for
\(M-1\) in place of \(M\). Then, returning to \cref{improved-decay-iteration-pre}
and using \cref{rp-general} and the induction hypothesis (applied to \(T^m\Phi{}\)
for \(0\le m\le M-1\)), we conclude that for any \(\eta{}>0\) and \(1\le R\le \tau{}_1\), we have
\begin{equation}\label{improved-decay-iteration-post}
\begin{split}
\textnormal{(I)}_{R,N,M} \lesssim_\eta{} R^{-2(M-1)}\tau{}_1^\eta{}\tilde{\mathcal{D}}_{p,N,M}[\Phi{}](v),
\end{split}
\end{equation}
where we have used the fact that (for \(N\ge 0\), \(M\ge 1\) and \(v\ge
0\))
\begin{equation}\label{data-tilde-facts}
\tilde{\mathcal{D}}_{p,N+1,M-1}[\Phi{}](v)\le \tilde{\mathcal{D}}_{p,N,M}[\Phi{}](v).
\end{equation}
Substitute \cref{improved-decay-iteration-post} into \cref{improved-T-decay-u-0} to
get
\begin{equation}\label{improved-decay-iteration-pre-2}
\int_{\tau{}_1}^{\tau{}_2} \tilde{\mathcal{E}}_{p,N}[T^M\Phi{}](\tau{},v)\dd{}\tau{}\lesssim R^{-2M+1}\tau{}_1^\eta{}\tilde{\mathcal{D}}_{p,N,M}[\Phi{}](v) + R(\tilde{\mathcal{E}}_{p,N}[T^M\Phi{}](\tau{}_1,v) + \tau{}_1^{-2M} \tilde{\mathcal{D}}_{p,N,M}[\Phi{}](v)).
\end{equation}
for \(1\le R\le \tau{}_1\). Now \cref{improved-decay-u} follows from \cref{interpolation,improved-decay-iteration-pre-2} as in
the case \(M = 1\).
\end{proof}
\subsection{Energy decay for radially symmetric scalar fields}
\label{sec:energy-decay-radially-symmetric} In this section, we partially extend the
results of \cref{sec:energy-decay-good} to radially symmetric scalar fields. We
first show that we can express \(T\)-derivatives of a radially symmetric scalar
field \(\varphi{}\) in terms of \((rL)\)-derivatives of \(\varphi{}\) and of the
quantity \(\Psi{}\).
\begin{lemma}
Suppose \(\varphi{}\in C^\infty(\mathcal{R})\). Then for \(N\ge 0\) and \(M\ge 1\), we have
\begin{equation}\label{eda-calculation-equation}
\begin{split}
L(r^{1/2}(rL)^NT^M\varphi{}) &= \mathcal{O}(r^{-M+1})\sum_{n=0}^{N+M}\mathcal{O}(r^{-1/2})L(rL)^{n}\varphi{}  \\
&\qquad + \mathcal{O}(r^{-M+1})\sum_{m=0}^{M-1}\mathcal{O}(r^{m+1})\sum_{n=0}^{N+M-m-1}\bigl(\mathcal{O}(r^{-2})T^{m}\Psi{} + \mathcal{O}(r^{-1})L(rL)^{n}T^{m}\Psi{}\bigr).
\end{split}
\end{equation}
\label{eda-calculation}
\end{lemma}
\begin{proof}
The case \(N = 0\) and \(M = 1\) follows from the computation
\begin{equation}\label{energy-decay-calc-prep-0}
LT\psi{} = L(r^{1/2}T\varphi{}) = L(r^{1/2}L\varphi{}) - L(r^{1/2}GZ\varphi{}) = r^{-1/2}L(rL)\varphi{} - \frac{1}{2}r^{-1/2}GL\varphi{} - L\Psi{}.
\end{equation}
Apply \cref{energy-decay-calc-prep-0} to \((rL)^N\varphi{}\) in place of \(\varphi{}\) and use
\cref{L-Z-rL-phi} to conclude the case \(M = 1\):
\begin{equation}\label{eda-calculation-prep}
\begin{split}
L(r^{1/2}(rL)^NT\varphi{}) &= r^{-1/2}L(rL)^{N + 1}\varphi{} - \frac{1}{2}r^{-1/2}GL(rL)^N\varphi{} - L(r^{1/2}GZ(rL)^N)\varphi{} \\
&= \sum_{n=0}^{N+1}\mathcal{O}(r^{-1/2})L(rL)^n\varphi{} + \sum_{n=0}^{N} \mathcal{O}(1)L(rL)^n\Psi{}.
\end{split}
\end{equation}
Multiply \cref{eda-calculation-prep} by \(r^{1/2}\) and
rearrange terms to obtain
\begin{equation}\label{energy-decay-calc-prep}
L(rL)^NT\varphi{} = \sum_{n=0}^{N+1}\mathcal{O}(r^{-1})L(rL)^n\varphi{} + \sum_{n=0}^N\mathcal{O}(r^{-1/2})L(rL)^n\Psi{} +  \sum_{n=0}^N\mathcal{O}(r^{-3/2})(rL)^n\Psi{}.
\end{equation}
The general formula \cref{eda-calculation-equation} follows
by induction on \(M\), where in the inductive step one uses
\cref{eda-calculation-prep,energy-decay-calc-prep} applied to
\(T^{M-1}\varphi{}\) in place of \(\varphi{}\).
\end{proof}
\begin{proposition}[Energy decay and improved energy decay for time derivatives; radially symmetric case]
Suppose \(\varphi{}\in C^\infty(\mathcal{R})\) is radially symmetric. Suppose moreover that there
exists a function \(\Phi{}\in C^\infty(\mathcal{R}\setminus \set{r=0})\) such that
\begin{enumerate}
\item \label{time-inversion-assumption} \(T\Phi{} = \Psi{} = r^{1/2}GZ\varphi{}\),
\item \label{u-rp-assumption} and \(\Phi{}\) satisfies the assumptions of \cref{good-field}
with inhomogeneity \(\tilde{F}\).
\end{enumerate}
For \(N\ge 1\), \(M\ge 0\), \(\delta{}\ge 0\), and \(v\ge 0\), define the initial data norm
\begin{equation}
\begin{split}
\mathcal{D}_{N,M,\delta{}}[\varphi{},\Phi{}](v)&\coloneqq{}\tilde{\mathcal{D}}_{1+\delta{},N,M+1}[\Phi{}](v) +  \sum_{m=0}^M \bigl(E_{N+M-m}[T^m\varphi{}](0,v) + \mathcal{E}_{1+\delta{},N+M-m}[T^m\varphi{}](0,v)\bigr) \\
&\qquad + \sum_{m=0}^M\bigl(\sup_{\tau{}\ge 0}(1+\tau{})^{2m+2}\mathcal{A}_{1+\delta{},N}[T^m\Box{}\varphi{}](\tau{},\infty,v) +\sup_{\tau{}\ge 0}(1+\tau{})^{2m}\mathcal{A}_{3+\delta{},N}[T^m\Box{}\varphi{}](\tau{},\infty,v)\bigr), \\
\end{split}
\end{equation}
where the data norm \(\tilde{\mathcal{D}}\) was defined in \cref{energy-decay} and
the norm \(\mathcal{A}\) of the inhomogeneity was defined in \cref{inhomogeneity-norm}.

Then for all \(N\ge 1\), \(\tau{}\ge 0\), and \(v\ge 0\), and \(\delta{} > 0\) sufficiently small
(that the estimates of \cref{r-weighted-estimates} hold),
we have the energy boundedness and decay estimates
\begin{align}
E_N[\varphi{}](\tau{},v)&\lesssim_{N,\delta{}} (1 + \tau{})^{-1}\mathcal{D}_{N,0,\delta{}}[\varphi{},\Phi{}](v), \\
\mathcal{E}_{1+\delta{},N-1}[(rL)\varphi{}](\tau{},v)&\lesssim_{N,\delta{}} \mathcal{D}_{N,0,\delta{}}[\varphi{},\Phi{}](v),
\end{align}
as well as the following improved energy decay estimates for time derivatives,
where \(M\ge 1\) and \(\eta{} > 0\):
\begin{align}
E_N[T^M\varphi{}](\tau{},v)&\lesssim_{N,M,\delta{},\eta{}} (1+\tau{})^{-1-2M+\delta{}+\eta{}}\mathcal{D}_{N,M,\delta{}}[\varphi{},\Phi{}](v), \label{eda-T} \\
\mathcal{E}_{1+\delta,N-1}[(rL)T^M\varphi{}](\tau{},v)&\lesssim_{N,M,\delta{},\eta{}}  (1+\tau{})^{-2M+\delta{}+\eta{}}\mathcal{D}_{N,M,\delta{}}[\varphi{},\Phi{}](v), \label{eda-weighted-rL} \\
\mathcal{E}_{1+\delta,N}[T^M\varphi{}](\tau{},v)&\lesssim_{N,M,\delta{},\eta{}}  (1+\tau{})^{-2M+ 1 + \delta{}+\eta{}}\mathcal{D}_{N,M,\delta{}}[\varphi{},\Phi{}](v). \label{eda-weighted}
\end{align}
\label{energy-decay-radially-symmetric}
\end{proposition}
\begin{remark}[Sharpness of the decay rates]
Note that, although we do not obtain the sharp decay rate (on Minkowski or given
by \cref{main-theorem}) for \(\mathcal{E}_{1 + \delta{},N}[T\varphi{}]\) (namely
\(\tau^{-2}\)), we still obtain the sharp rate \(\tau^{-3}\) for
\(E_N[T\varphi{}]\). This is crucial for the sharp energy decay and pointwise
estimates proven in \cref{sec:late-time-asymptotics}.
\end{remark}
\begin{proof}
Fix \(v\ge 0\) and \(\delta{} > 0\) sufficiently small, and let \(\eta{} >0\).

\step{Step 1: Preliminary estimates.} Let \(0\le \tau_1\le \tau_2\). First, we claim that
\begin{equation}\label{eda-1}
\sum_{n=0}^{N-1}\int_{\tau{}_1}^{\tau{}_2} \int _{\Sigma{}(\tau{},v)}r^2\langle{}r\rangle^\delta{}(L(rL)^nT^M\Psi{})^2\dd{}r\dd{}\theta{}\dd{}\tau{}\lesssim_\eta{} (1+\tau{}_1)^{-2M+\eta{}}\tilde{\mathcal{D}}_{1+\delta{},N,M+1}[\Phi{}](v).
\end{equation}
when \(N\ge 1\) and \(M\ge 0\).
We delay the proof to the end of this step. We now note several consequences of
\cref{eda-1} and the energy-estimates of \cref{axisymmetric-energy-estimates}. We have:
\begin{equation}\label{eda-energy}
\begin{split}
&E_N[T^M\varphi{}](\tau{}_2,v) + \sum_{n=0}^N\int_{\tau{}_1}^{\tau{}_2} \int _{\Sigma{}(\tau{},v)} \langle{}r\rangle^{-1+\delta{}}(T(rL)^nT^M\varphi{})^2\dd{}r\dd{}\theta{}\dd{}\tau{} \\
&\lesssim_\eta{} E_N[T^M\varphi{}](\tau{}_1,v) + (1+\tau{}_1)^{-2M-2+\eta{}}\mathcal{D}_{N,M,\delta{}}[\varphi{},\Phi{}](v)
\end{split}
\end{equation}
when \(N\ge 0\) and \(M\ge 0\), and
\begin{equation}\label{eda-int-energy}
\begin{split}
\int_{\tau{}_1}^{\tau{}_2} E_N[T^M\varphi{}](\tau{})\dd{}\tau{}\lesssim_\eta{} E_N[T^M\varphi{}](\tau{}_1,v) + \mathcal{E}_{1,N-1}[(rL)T^M\varphi{}](\tau{}_1,v) + (1+\tau{}_1)^{-2M+\eta{}}\mathcal{D}_{N,M,\delta{}}[\varphi{},\Phi{}](v)
\end{split}
\end{equation}
when \(N\ge 1\) and \(M\ge 0\), and
\begin{equation}\label{eda-rp-rL}
\begin{split}
&\mathcal{E}_{1+\delta{},N}[(rL)T^M\varphi{}](\tau{}_2,v) + \sum_{n=0}^{N+1}\int_{\tau{}_1}^{\tau{}_2} \int _{\Sigma{}(\tau{},v)} r\langle{}r\rangle^\delta{}(L(rL)^nT^M\varphi{})^2\dd{}r\dd{}\theta{}\dd{}\tau{}\\
&\lesssim_\eta{}  \mathcal{E}_{1+\delta{},N}[(rL)T^M\varphi{}](\tau{}_1,v) + (1+\tau{}_1)^{-2M+\eta{}}\mathcal{D}_{N+1,M,\delta{}}[\varphi{},\Phi{}](v)
\end{split}
\end{equation}
when \(N\ge 0\) and \(M\ge 0\), and
\begin{equation}\label{eda-rp-T}
\begin{split}
&\mathcal{E}_{1+\delta{},N}[T^M\varphi{}](\tau{}_2,v) + \sum_{n=0}^N\int_{\tau{}_1}^{\tau{}_2} \int _{\Sigma{}(\tau{},v)} r\langle{}r\rangle^\delta{}(L(rL)^nT^M\varphi{})^2\dd{}r\dd{}\theta{}\dd{}\tau{}\\
&\lesssim_\eta{}  \mathcal{E}_{1+\delta{},N}[T^M\varphi{}](\tau{}_1,v) +  E_N[T^{M-1}\varphi{}](\tau{}_1,v) + (1+\tau{}_1)^{-2M+\eta{}}\mathcal{D}_{N,M,\delta{}}[\varphi{},\Phi{}](v)
\end{split}
\end{equation}
when \(N\ge 0\) and \(M\ge 1\).
Indeed, \cref{eda-energy,eda-int-energy,eda-rp-rL,eda-rp-T} follow immediately from
\cref{T-estimate-ho,integrated-T-ho,r-estimate,r-estimate-T} combined with
\cref{energy-decay} once \cref{eda-1} has been established.

We now establish \cref{eda-1}. Use \cref{commuted-Tu} (applied to \(T^M\Phi{}\)),
\cref{rp-general} (applied to \(T^M\Phi{}\) with \(p = 1 + \delta{}\)), and \cref{energy-decay}:
\begin{equation}
\begin{split}
&\sum_{n=0}^{N-1}\int_{\tau{}_1}^{\tau{}_2} \int _{\Sigma{}(\tau{},v)}r^{2}\langle{}r\rangle^\delta{}(L(rL)^nT^M\Psi{})^2\dd{}r\dd{}\theta{}\dd{}\tau{} = \sum_{n=0}^{N-1}\int_{\tau{}_1}^{\tau{}_2} \int _{\Sigma{}(\tau{},v)}r^{2}\langle{}r\rangle^\delta{}(L(rL)^nTT^M\Phi{})^2\dd{}r\dd{}\theta{}\dd{}\tau{} \\
&\lesssim \sum_{n=0}^{N}\int_{\tau{}_1}^{\tau{}_2} \int _{\Sigma{}(\tau{},v)} \langle{}r\rangle{}^{\delta{}}(L(rL)^nT^{M}\Phi{})^2 + r^{-2}\langle{}r\rangle^\delta{}((rL)^nT^M\Phi{})^2 + r^2\langle{}r\rangle^\delta{}((rL)^nT^M\tilde{F})^2\dd{}r\dd{}\theta{}\dd{}\tau{} \\
&\lesssim \tilde{\mathcal{E}}_{0,N}[T^M\Phi{}](\tau{}_1,v) + \tilde{\mathcal{E}}_{1+\delta{},N}[T^M\Phi{}](\tau{}_1,v) + \tilde{\mathcal{A}}_{2+\delta{},N}[T^M\tilde{F}](\tau{}_1,\tau{}_2,v) \\
&\lesssim_\eta{} (1+\tau{}_1)^{-2M+\eta{}}\tilde{\mathcal{D}}_{1+\delta{},N,M}[\Phi{}](v).
\end{split}
\end{equation}

\step{Step 2: Decay for the \(T\)-energy from decay for the \(r\)-weighted energy;
proof of \cref{eda-T} from \cref{eda-weighted-rL}.} The point of this step is that the
integral of \(E_N[T^M\varphi{}]\) is controlled by
\(\mathcal{E}_{1,N-1}[(rL)T^M\varphi{}]\) (see \cref{eda-int-energy}), and so a pigeonhole argument can
obtain a decay rate \(\tau^{-1}\) faster for the former than for the latter.

Let \(N\ge 1\) and \(M\ge 0\). In this step we will show that for \(M\ge 0\),
the statement
\begin{equation}\label{eda-2-hyp}
\mathcal{E}_{1,N-1}[(rL)T^M\varphi{}](\tau{},v)\lesssim (1+\tau{})^{-\beta{}}\mathcal{D}_{N,M,\delta{}}[\varphi{},\Phi{}](v)\textnormal{ for all }\tau{}\ge 0
\end{equation}
for some \(\beta{}\le 2M\) implies the statement
\begin{equation}\label{eda-2-conc}
E_N[T^M\varphi{}](\tau{},v)\lesssim_{\eta{}}(1 + \tau{})^{-1-\beta{}+\eta{}}\mathcal{D}_{N,M,\delta{}}[\varphi{},\Phi{}](v)\textnormal{ for all }\tau{}\ge 0.
\end{equation}
In particular, \cref{eda-T} follows from
\cref{eda-weighted-rL}.

First, \cref{eda-energy} implies that it suffices to
consider \(\tau{}\ge 1\). Now suppose we have shown that for some \(k\ge 0\) and all \(\tau{}\ge
1\), we have
\begin{equation}\label{eda-2a}
E_N[T^M\varphi{}](\tau{},v)\le \tau{}^{-k}\mathcal{D}_{N,M,\delta{}}[\varphi{},\Phi{}](v).
\end{equation}
By \cref{eda-energy}, the estimate
\cref{eda-2a} holds for \(k = 0\). Let \(\tau_i = 2^i\) be a dyadic sequence. By
\cref{eda-2-hyp} and \cref{eda-2a}
together with \cref{eda-int-energy}, we have
\begin{equation}
\int_{\tau{}_i}^{\tau{}_{i+1}} E_N[T^M\varphi{}](\tau{},v)\dd{}\tau{} \lesssim_\eta{} \tau{}_i^{-\min (k,\beta{})+\eta{}}\mathcal{D}_{N,M,\delta{}}[\varphi{},\Phi{}](v).
\end{equation}
By the pigeonhole principle, there is
\(\tau{}_i'\in [\tau{}_i,\tau{}_{i+1}]\) for which
\begin{equation}
E_N[T^M\varphi{}](\tau{}_i',v)\lesssim_\eta{} \tau{}_i'^{-1-\min (k,\beta{})+\eta{}}\mathcal{D}_{N,M,\delta{}}[\varphi{},\Phi{}](v).
\end{equation}
By \cref{eda-int-energy},
\cref{eda-2-hyp}, and the dyadicity of the \(\tau_i\) (and hence of the \(\tau{}_i'\)), we
can extend this to
\begin{equation}\label{eda-2c}
E_N[T^M\varphi{}](\tau{},v)\lesssim_\eta{} \tau{}^{-1-\min (k,\beta{})+\eta{}}\mathcal{D}_{N,M,\delta{}}[\varphi{},\Phi{}](v)
\end{equation}
for all \(\tau{}\ge 1\). If \(k+1\le \beta{}\), then this implies \cref{eda-2a} for \(k + 1-\eta{}\) in
place of \(k\). We can repeat this argument until \(k \ge \beta{}\), in which
case \cref{eda-2c} implies the desired \cref{eda-2-conc} (with \(2M\eta{}\) in place of
\(\eta{}\)).

\step{Step 3: The interpolation argument.} Fix \(j\in \set{0,1}\) and \(N\ge 0\) such that
\(N + j\ge 1\), fix \(M\ge 1\), and fix \(1\le \tau_1\le \tau_2\). Let \(1\le R\le \tau{}_i\). The
goal of this step is to show that
\begin{equation}\label{eda-step-3-goal}
\begin{split}
&\int_{\tau{}_1}^{\tau{}_2} \mathcal{E}_{1+\delta{},N}[(rL)^jT^M\varphi{}](\tau{},v)\dd{}\tau{}\\
&\lesssim_\eta{} R^{-2M+1+\delta{}}\tau{}_1^\eta{}\mathcal{D}_{N+j,M,\delta{}}[\varphi{},\Phi{}](v) \\
&\qquad + R(\mathcal{E}_{1+\delta{},N}[(rL)^jT^M\varphi{}](\tau{}_1,v) + (1-j)E_{N+j}[T^{M-1}\varphi{}](\tau{}_1,v) + \tau{}_1^{-2M+\eta{}}\mathcal{D}_{N+j,M,\delta{}}[\varphi{},\Phi{}](v)).
\end{split}
\end{equation}
To produce this estimate, we will expand the integral and interpolate between
(i.e.~split the region of integration into) a large-\(r\) region (which produces
the term on the first line) and a small-\(r\) region (which produces the terms
on the last line).

Split into the regions \(\set{r\le R}\) and \(\set{r\ge R}\) and use
\(h(r)\lesssim \langle{}r\rangle{}^{-1}\) to obtain
\begin{equation}\label{eda-3-initial}
\int_{\tau{}_1}^{\tau{}_2} \mathcal{E}_{1+\delta{},N}[(rL)^jT^M\varphi{}](\tau{},v)\dd{}\tau{}\lesssim \textnormal{(I)}_{N,j,M,R} + R^{-1+\delta{}}\textnormal{(II)}_{N,j,M,R}
\end{equation}
for
\begin{equation}
\begin{split}
\textnormal{(I)}_{N,j,M,R}&\coloneqq{}\sum_{n=0}^N\int_{\tau{}_1}^{\tau{}_2}\int _{\Sigma{}(\tau{},v)} \mathbf{1}_{r\le R}r\langle{}r\rangle^\delta{}(L(r^{1/2}(rL)^{n+j}T^M\varphi{}))^2 + \langle{}r\rangle{}^{-1+\delta{}}((rL)^{n+j}T^{M}\varphi{})^2\dd{}r\dd{}\theta{}\dd{}\tau{} \\
\end{split}
\end{equation}
and
\begin{equation}
\textnormal{(II)}_{N,j,M,R}\coloneqq{}\sum_{n=0}^N  \int_{\tau{}_1}^{\tau{}_2}\int _{\Sigma{}(\tau{},v)\cap \set{r\ge R}} r^2(L(r^{1/2}(rL)^{n+j}T^M\varphi{}))^2\dd{}r\dd{}\theta{}\dd{}\tau{} \\
\end{equation}

We now control the terms on the right-hand side of \cref{eda-3-initial}. By estimating
\((L(r^{1/2}f))^2 \lesssim r(Lf)^2 + r^{-1}f^2\) and using the fact that the
first term in \(\textnormal{(I)}_{N,j,M,R}\) is only integrated over \(\set{r\le
R}\) (and \(R\ge 1\)), we obtain
\begin{equation}\label{eda-3-I-initial}
\begin{split}
\textnormal{(I)}_{N,j,M,R}\lesssim R\sum_{n=0}^N\int_{\tau{}_1}^{\tau{}_2}\int _{\Sigma{}(\tau{},v)} r\langle{}r\rangle^\delta{}(L(rL)^{n+j}T^M\varphi{})^2 + \langle{}r\rangle^{-1+\delta{}}((rL)^{n+j}T^{M}\varphi{})^2\dd{}r\dd{}\theta{}\dd{}\tau{}.
\end{split}
\end{equation}
When \(j = 0\), we have \(N\ge 1\), and so we can estimate the first term on the
left-hand side of \cref{eda-3-I-initial} using \cref{eda-rp-rL} (with \(N-1\) in place of \(N\)),
and since \(M\ge 1\), we can estimate the second term on the left-hand side of
\cref{eda-3-I-initial} by writing \((rL)^{n +j }T^M\varphi{} = T(rL)^{n}T^{M-1}\varphi{}\)
and using \cref{eda-energy} (with \(M-1\) in place of \(M\)):
\begin{equation}\label{eda-3-I-j0}
\begin{split}
\textnormal{(I)}_{N,0,M,R} &\lesssim_\eta{} R(\mathcal{E}_{1+\delta{},N-1}[(rL)T^M\varphi{}](\tau{}_1,v) + E_{N}[T^{M-1}\varphi{}](\tau{}_1,v) + \tau{}_1^{-2M+\eta{}}\mathcal{D}_{N,M,\delta{}}[\varphi{},\Phi{}](v)). \\
\end{split}
\end{equation}
When \(j = 1\), we estimate the the first term on the left-hand side of \cref{eda-3-I-initial}
using \cref{eda-rp-rL}, and we estimate the second term on the left-hand side of
\cref{eda-3-I-initial} by writing \(((rL)^{n +j}T^M\varphi{})^2 =
r^2(L(rL)^{n}T^{M}\varphi{})^2\) and using \cref{eda-rp-rL}:
\begin{equation}\label{eda-3-I-j1}
\textnormal{(I)}_{N,1,M,R}\lesssim_\eta{} R(\mathcal{E}_{1+\delta{},N}[(rL)T^M\varphi{}](\tau{}_1,v) + \tau{}_1^{-2M+\eta{}}\mathcal{D}_{N,M,\delta{}}[\varphi{},\Phi{}](v)).
\end{equation}
Combining \cref{eda-3-I-j0,eda-3-I-j1}, we obtain
\begin{equation}\label{eda-3-I}
\textnormal{(I)}_{N,1,M,R}\lesssim_\eta{}  R(\mathcal{E}_{1+\delta{},N}[(rL)^jT^M\varphi{}](\tau{}_1,v) + (1-j)E_{N+j}[T^{M-1}\varphi{}](\tau{}_1,v) + \tau{}_1^{-2M+\eta{}}\mathcal{D}_{N,M,\delta{}}[\varphi{},\Phi{}](v)).
\end{equation}

Next, we estimate \(\textnormal{(II)}_{N,j,M,R}\): the calculation in \cref{eda-calculation}, together with the fact that the
integral is over \(\set{r\ge R}\), implies that for \(R\le \tau_1\), we have
\begin{equation}\label{eda-3-II}
\begin{split}
\textnormal{(II)}_{N,j,M,R} &\lesssim R^{-2M+2}\Bigl[\sum_{n=0}^{N+M+j}\int_{\tau{}_1}^{\tau{}_2}\int _{\Sigma{}(\tau{},v)} r(L(rL)^{n}\varphi{})^2\dd{}r\dd{}\theta{}\dd{}\tau{} \\
&\qquad + \sum_{m=0}^{M-1}R^{2m + 2}\sum_{n=0}^{N+M+j-m-1}\int_{\tau{}_1}^{\tau{}_2}\int _{\Sigma{}(\tau{},v)} (L(rL)^nT^{m+1}\Phi{})^2 + r^{-2}(T^{m+1}\Phi{})^2 \dd{}r\dd{}\theta{}\dd{}\tau{}\Bigr] \\
&\lesssim  R^{-2M+2}\Bigl[\mathcal{E}_{1,N+M+j-1}[(rL)\varphi{}](0,v) + \tilde{\mathcal{D}}_{1,N+j,M+1}[\Phi{}](v) \\
&\qquad + \sum_{m=0}^{M-1}R^{2m+2} (\tilde{\mathcal{E}}_{1,N+j+M-m-1}[T^{m+1}\Phi{}] + \tilde{\mathcal{A}}_{1,N+j+M-m-1}[T^{m+1}\tilde{F}](\tau{}_1))\Bigr] \\
&\lesssim_\eta{}  R^{-2M+2}\tau{}_1^\eta{}\mathcal{D}_{N+j,M,0}[\varphi{},\Phi{}](v).
\end{split}
\end{equation}
To pass to the second last line, we applied \cref{r-estimate} (with \(N + M + j-1\)
in place of \(N\)), estimating the terms that arise on the right-hand side using
\cref{eda-1,energy-decay}, and applied \cref{rp-general} (in which the norm
\(\tilde{\mathcal{A}}\) was defined). To pass to the last line, we used \cref{energy-decay}.

Combining \cref{eda-3-I,eda-3-II,eda-3-initial}, we obtain \cref{eda-step-3-goal}.

\step{Step 4: The iteration argument: proof of \cref{eda-weighted,eda-weighted-rL}.} By
Steps 1 and 2, it suffices to prove the estimates
\cref{eda-weighted,eda-weighted-rL} for \(\tau{}\ge 1\). First, note that
\cref{eda-weighted-rL} holds for \(M = 0\) by \cref{eda-rp-rL}. Now suppose that
\(M\ge 1\) and we have established \cref{eda-weighted-rL} for \(M-1\). By Step 2,
for \(\tau{}\ge 0\) and \(N\ge 0\) we have
\begin{equation}\label{eda-4-TM}
\begin{split}
E_{N+1}[T^{M-1}\varphi{}](\tau{},v)&\lesssim_\eta{} (1+\tau{})^{-2M+1+\delta{}+\eta{}}\mathcal{D}_{N+1,M-1,\delta{}}[\varphi{},\Phi{}](v)\lesssim (1+\tau{})^{-2M+1+\delta{}+\eta{}}\mathcal{D}_{N,M,\delta{}}[\varphi{},\Phi{}](v).
\end{split}
\end{equation}
Now let \(j\in \set{0,1}\) and let \(N\ge 0\) satisfy \(N + j\ge 1\). By Step 3 and \cref{eda-4-TM}, we
have
\begin{equation}\label{eda-4-dyadic}
\begin{split}
&\int_{\tau{}_1}^{\tau{}_2} \mathcal{E}_{1+\delta{},N}[(rL)^jT^M\varphi{}](\tau{},v)\dd{}\tau{}\\
&\lesssim_\eta{} R^{-2M+1+\delta{}}\tau{}_1^\eta{}\mathcal{D}_{N+j,M,\delta{}}[\varphi{},\Phi{}](v) + R(\mathcal{E}_{1+\delta{},N}[(rL)^jT^M\varphi{}](\tau{}_1,v) + \tau{}_1^{-2M+(1-j)+\delta{}+\eta{}}\mathcal{D}_{N+j,M,\delta{}}[\varphi{},\Phi{}](v)).
\end{split}
\end{equation}
By \cref{eda-rp-rL,eda-rp-T}, we have
\begin{equation}\label{eda-4-energy-boundedness}
\mathcal{E}_{1+\delta{},N}[(rL)^jT^M\varphi{}](\tau{}_2,v)\lesssim_\eta{} \mathcal{E}_{1+\delta{},N}[(rL)^jT^M\varphi{}](\tau{}_1,v) + (1+\tau{}_1)^{-2M+(1-j)+\delta{}+\eta{}}\mathcal{D}_{N+j,M,\delta{}}[\varphi{},\Phi{}](v).
\end{equation}
It now follows from \cref{interpolation,eda-4-dyadic,eda-4-energy-boundedness} that
\begin{equation}
\mathcal{E}_{1+\delta{},N}[(rL)^jT^M\varphi{}](\tau{},v)\lesssim_\eta{} (1+\tau{})^{-2M+(1-j)+\delta{}+2\eta{}}\mathcal{D}_{N+j,M,\delta{}}[\varphi{},\Phi{}](v),
\end{equation}
which establishes \cref{eda-weighted,eda-weighted-rL} (with \(2\eta{}\) in place of
\(\eta{}\)).
\end{proof}
\section{Constructing time integrals}
\label{sec:time-integrals}
In view of the results of \cref{sec:energy-decay}, we would like to show that the
difference of \(\varphi{}\) and a suitable multiple of the claimed leading-order
profile can be written as a time derivative of a solution to \cref{wave-equation},
so that said difference decays faster than the leading-order profile itself.
That is, we aim to construct a ``time integral'' of a suitably renormalized
solution \(\widehat{\varphi{}}\) (see \cref{renormalized-solution}). Recall from \cref{wave-equation-expressions} that
\begin{equation}
\mathcal{L}(G^{1/2}\psi{}) = \mathcal{F}[G^{1/2}T\psi{}] + G^{-3/2}r^{1/2}\Box{}\varphi{},
\end{equation}
where the operators \(\mathcal{L}\) and \(\mathcal{F}\) were defined in
\cref{calL-def,calF-def}. We can therefore hope to construct a ``time integral'' of
the renormalized solution \(\widehat{\varphi{}}\) by constructing a solution \(\zeta{}\) to
the equation \(\mathcal{L}\zeta{} =
\mathcal{F}[G^{1/2}\widehat{\psi{}}]|_{\Sigma{}_0}\) on \(\Sigma_0\) and then
solving \cref{wave-equation} with initial data
\((r^{-1/2}G^{-1/2}\zeta{},\widehat{\varphi{}}|_{\Sigma{}_0})\).

In \cref{minkowskian-soln}, we introduce the Minkowskian solution
\(\varphi{}_{\textnormal{mink}}\) and the renormalized solution \(\widehat{\varphi{}}\). In
\cref{time-integral-elliptic}, we invert an elliptic operator to construct
functions which we use in \cref{time-integrals-from-data} as initial data for wave
equations to construct time integrals.
\subsection{The Minkowskian solution and the renormalized solution}
\label{minkowskian-soln}
Define the Minkowskian solution
\begin{equation}
\varphi{}_{\textnormal{mink}} \coloneqq{} u^{-1/2}v^{-1/2} = 4(t^2-\tilde{G}(r)^2)^{-1/2}.
\end{equation}
Then \(\varphi_{\textnormal{mink}}\in C^\infty(\mathcal{R})\) is radially symmetric. An explicit computation
shows that on Minkowski space (where \(\tilde{G}(r)\equiv r\)), we have
\(\Box{}_{m}\varphi{}_{\textnormal{mink}} = 0\), where
\(\Box_{m}\) is the wave operator on Minkowski space of \((2 + 1)\)
dimensions.
\subsubsection{The renormalized solution}
\label{renormalized-solution}
Let \(\varphi{}\in C^\infty(\mathcal{R})\). We recall from
\cref{L-frak-def} the linear functional \(\mathfrak{L}[\varphi{}]\) of the (radially
symmetric part of the) initial data \((\varphi{}|_{\Sigma{}(0)},T\varphi{}|_{\Sigma{}(0)})\):
\begin{equation*}
\mathfrak{L}[\varphi{}]\coloneqq{}\Bigl(\int_0^\infty r^{1/2}G^{1/2}\mathcal{F}[G^{1/2}\psi{}_{\textnormal{mink}}]|_{\Sigma{}(0)} + rG^{-1}T^{-1}\Box{}\varphi{}_{\textnormal{mink}}|_{\Sigma{}(0)}\dd{}r\Bigr)^{-1}\int_0^\infty r^{1/2}G^{1/2}\mathcal{F}[G^{1/2}\psi{}_0](r)|_{\Sigma{}(0)}\dd{}r.
\end{equation*}
Here the linear operator \(\mathcal{F}\) was defined in \cref{calF-def}. When \(\varphi{}|_{\Sigma{}(0)}\) has sufficient decay
towards infinity, the final integral defining \(\mathfrak{L}[\varphi{}]\)
converges. The quantity \(T^{-1}\Box{}\varphi_{\textnormal{mink}}\) is defined
in \cref{Tinv-mink}, and the well-definedness of the quantity
\(\mathfrak{L}[\varphi{}]\) (in particular the non-vanishing of the factor
involving \(\varphi{}_{\textnormal{mink}}\) that is inverted) is established in
\cref{frakL-well-defined}.
\begin{definition}[The renormalized solution]
Given \(\varphi{}\in C^\infty(\mathcal{R})\), we define the renormalized quantity
\begin{equation}
\widehat{\varphi{}}\coloneqq{}\varphi{}-\mathfrak{L}[\varphi{}]\varphi{}_{\textnormal{mink}}.
\end{equation}
We write
\begin{equation}
\widehat{\psi{}}\coloneqq{}r^{1/2}\widehat{\varphi{}},\qquad \widehat{\Psi{}}\coloneqq{}r^{1/2}GZ\widehat{\varphi{}}.
\end{equation}
Clearly, we have
\begin{equation}
\widehat{\varphi{}}_0 = \varphi_0 - \mathfrak{L}[\varphi{}_0]\varphi_{\textnormal{mink}},\qquad \widehat{\varphi{}}_{\ge 1} = \varphi_{\ge 1},\qquad \Box{}\varphi{} =0 \implies \Box{}\widehat{\varphi{}} = -\mathfrak{L}[\varphi{}]\Box{}\varphi{}_{\textnormal{mink}}.
\end{equation}
\end{definition}
\subsubsection{Estimates for the Minkowskian solution}
\label{inhomogeneous-estimates} In this section, we construct and estimate the
second time integral \(T^{-2}\Box{}\varphi_{\textnormal{mink}}\) of the
inhomogeneity associated to \(\varphi_{\textnormal{mink}}\). In particular, we
show that the quantity \(\mathfrak{L}[\varphi{}]\) defined in \cref{L-frak-def} is
well-defined (see \cref{frakL-well-defined}).
\begin{lemma}[Estimating \(\varphi{}_{\textnormal{mink}}\) and \(\Psi{}_{\textnormal{mink}}\)]
We have
\begin{equation}
\abs{(rL)^NT^M\varphi{}_{\textnormal{mink}}}\lesssim_{N,M} u^{-1/2-M}v^{-1/2}
\end{equation}
and
\begin{equation}
\abs{(rL)^NT^M\Psi{}_{\textnormal{mink}}}\lesssim_{N,M} r^{3/2}u^{-3/2-M}v^{-3/2}.
\end{equation}
\label{mink-energy}
\end{lemma}
\begin{proof}
This is a straightforward calculation.
\end{proof}
\begin{proposition}[Time integrals of inhomogneities associated to \(\varphi{}_{\textnormal{mink}}\)]
There exists a radially symmetric function
\(T^{-2}\Box{}\varphi{}_{\textnormal{mink}} \in C^\infty(\mathcal{R})\) such
that \(T^2T^{-2}\Box{}\varphi_{\textnormal{mink}} =
\Box{}\varphi_{\textnormal{mink}}\). Moreover, the following estimates hold for
\(N\ge 0\) and \(M\ge 0\):
\begin{align}
\abs{(rL)^NT^{-2}\Box{}\varphi{}_{\textnormal{mink}}}&\lesssim _{N} \langle{}r\rangle{}^{-1}(\tau{}+r)^{-1} \label{Tinv-inhomogeneity} \\
\abs{(rL)^NT^MT^{-1}\Box{}\varphi{}_{\textnormal{mink}}}&\lesssim_{N,M} \langle{}r\rangle{}^{-1}(1+\tau{})^{-1/2-M}(1+\tau{}+r)^{-3/2}. \label{box-inhomogeneity-M}
\end{align}
Similarly, the quantity \(\tilde{F}\coloneqq{}r^{1/2}GZT^{-2}\Box{}\varphi{}_{\textnormal{mink}}\)
is radially symmetric, satisfies \(r^{-3/2}\tilde{F}\in C^\infty(\mathcal{R})\), and
satisfies the following estimates for \(N\ge 0\) and \(M\ge 0\):
\begin{equation}\label{Tinv-Psi-inhomongeneity}
\abs{(rL)^NT^M\tilde{F}} \lesssim_{N,M} \langle{}r\rangle^{-3/2}\tau{}^{-M}(\tau{}+r)^{-1}.
\end{equation}
\label{Tinv-mink}
\end{proposition}
\begin{remark}
We write \(T^{-1}\Box{}\varphi{}_{\textnormal{mink}}
\coloneqq{}TT^{-2}\Box{}\varphi{}_{\textnormal{mink}}\).
\end{remark}
\begin{proof}
Define
\begin{equation}
T^{-2}\Box{}\varphi{}_{\textnormal{mink}} \coloneqq{} \frac{2}{\tilde{G}(r)^2}\Bigl(1 - \frac{\tilde{G}(r)}{r}G(r)\Bigr)(t-(t^2-\tilde{G}(r)^2)^{1/2}).
\end{equation}
Using \cref{box-equation-double-null,double-null-coords}, we compute
\begin{equation}
\Box{}\varphi{}_{\textnormal{mink}} = \frac{1}{4}(1 - (v-u)r^{-1}G)u^{-3/2}v^{-3/2} = \frac{1}{4}\Bigl(1-\frac{\tilde{G}(r)}{r}G(r)\Bigr)u^{-3/2}v^{-3/2} = 2\Bigl(1-\frac{\tilde{G}(r)}{r}G(r)\Bigr) \frac{1}{(t^2-\tilde{G}(r)^2)^{3/2}}.
\end{equation}
A computation now verifies that \(T^2T^{-2}\Box{}\varphi{}_{\textnormal{mink}} =
TT^{-1}\Box{}\varphi_{\textnormal{mink}} = \Box{}\varphi_{\textnormal{mink}}\).
The smoothness and radial symmetry of \(G\) (and \cref{axisymmetric-smoothness}) imply
\begin{equation}
\frac{1}{\tilde{G}(r)^2}\Bigl(1- \frac{\tilde{G}(r)}{r}G(r)\Bigr)=\mathcal{O}(\langle{}r\rangle{}^{-3}).
\end{equation}
Similar considerations show that \(T^{-2}\Box{}\varphi_{\textnormal{mink}}\in
r^2C^\infty(\mathcal{R})\). One obtains \cref{Tinv-inhomogeneity} by considering
separately the regions \(\set{r\le u/2}\) and \(\set{r\ge u/2}\). For example,
when \(N = M = 0\), we use the binomial expansion in \(\set{r\le u/2}\) and a
trivial estimate in \(\set{r\ge u/2}\) to obtain
\begin{equation}
t(1 - (1-\tilde{G}(r)^2/t^2)^{1/2})\lesssim \begin{cases}
r^2/t\sim r^2/v &\textnormal{in }\set{r\le u/2} \\
t\sim r\sim r^2/v&\textnormal{in }\set{r\ge u/2} \\
\end{cases}\implies \abs{T^{-2}\Box{}\varphi{}_{\textnormal{mink}}}\lesssim r^2\langle{}r\rangle^{-3}v^{-1},
\end{equation}
and this estimate persists under differentiation by \((rL)\) and \(T\). Note
that \(u\sim 1+\tau{}\) and \(v\sim 1+\tau{} + r\).

The smoothness of \(r^{-3/2}\tilde{F}\) and the estimates in \cref{Tinv-Psi-inhomongeneity} follow
from using \(GZ = L - T\) and \cref{axisymmetric-smoothness} and the estimates for \(T^{-2}\Box{}\varphi_{\textnormal{mink}}\).
\end{proof}
\begin{corollary}[Estimates for inhomogeneous norms associated to time integrals of \(\varphi{}_{\textnormal{mink}}\)]
Let \(F\coloneqq{}T^{-1}\Box{}\varphi_{\textnormal{mink}}\). For \(N\ge 0\), \(M\ge 0\), \(\tau{}\ge 0\), \(v\ge
0\), and \(\eta{} > 0\), we have
\begin{align}
\mathcal{A}_{1,N}[T^MF](\tau{},\infty,v)&\lesssim_{N,M,\eta{}} (1+\tau{})^{-3-2M+\eta{}}, \\
\mathcal{A}_{2-\eta{},N}[T^MF](\tau{},\infty,v)&\lesssim_{N,M,\eta{}} (1+\tau{})^{-2-2M}, \\
\mathcal{A}_{4-\eta{},N}[T^MF](\tau{},\infty,v)&\lesssim_{N,M,\eta{}} (1+\tau{})^{-2M},
\end{align}
where the norm \(\mathcal{A}\) was defined in \cref{inhomogeneity-norm}.

Let \(\tilde{F}\coloneqq{}r^{1/2}GZT^{-2}\Box{}\varphi_{\textnormal{mink}}\) be as in
\cref{Tinv-mink}. For \(p\in [0,2-\eta_0]\), \(N\ge 0\), \(M\ge 0\), \(\tau{}\ge
0\), and \(v\ge 0\),
\begin{equation}
\tilde{\mathcal{A}}_{p,N}[T^M\tilde{F}](\tau{},\infty,v)\lesssim_{N,M,\eta{}_0} (1+\tau{})^{-2M},
\end{equation}
where the norm \(\tilde{\mathcal{A}}\) was defined in \cref{inhomogeneity-norm}.
\label{A-norm-estimate}
\end{corollary}
\begin{proof}
This is an immediate consequence of \cref{Tinv-mink} and the definitions of the
norms.
\end{proof}
\begin{lemma}[Well-definedness of the quantity {\(\mathfrak{L}[\varphi{}]\)}]
Let \(\eta_h > 0\), \(B_h > 0\), and \(\epsilon_h > 0\) be the constants defined in
\cref{h-hyperboloidal} and \cref{h-derivative-small} of \cref{hyperboloidal-foliation},
respectively. If \(\epsilon_h\) is sufficiently small depending on \(B_h\) and \(\eta_h\),
then
\begin{equation}
\int_0^\infty r^{1/2}G^{1/2}\mathcal{F}[\psi{}_{\textnormal{mink}}]|_{\Sigma{}(0)} + rG^{-1}T^{-1}\Box{}\varphi{}_{\textnormal{mink}}|_{\Sigma{}(0)}\dd{}r\neq{}0,
\end{equation}
where the operator \(\mathcal{F}\) was defined in \cref{calF-def} and
\(T^{-1}\Box{}\varphi_{\textnormal{mink}}\) was defined in \cref{Tinv-mink}. In
particular, the linear functional \(\mathfrak{L}\) introduced in \cref{L-frak-def}
is well-defined.
\label{frakL-well-defined}
\end{lemma}
\begin{proof}
The value of the integral is continuous in the function \(G\) (near
\(G\equiv 1\), in the \(C^1\) topology), so it suffices to prove that the
integral has a sign when \(G\equiv 1\). We will show that the integrand (when \(G\equiv
1\)), namely
\begin{equation}
\frac{1}{4}u^{-3/2}v^{-3/2}r^{-1/2}\textnormal{(I)},\qquad \textnormal{(I)}\coloneqq{}h(r)^2r(v+u) + 4uvrh'(r) + 4u^2(1-h(r)),
\end{equation}
is positive when \(\epsilon_h\) is small enough compared to \(B_h\). By the definition
of \(\tau{}\) (see \cref{tau-def}), we have \(1\le u\le C\) on \(\Sigma{}(0)\),
where \(C\) depends only on \(B_h\). Since the first term in
\(\textnormal{(I)}\) is non-negative, the assumption \cref{h-hyperboloidal} of
\cref{hyperboloidal-foliation} implies that
\begin{equation}
\textnormal{(I)}\ge 4(1-B_h\langle{}r\rangle{}^{-1}) - 4C\langle{}r\rangle^{-\eta{}_h}.
\end{equation}
It follows that there is \(R = R(B_h,\eta{}_h) > 0\) such that in \(\set{r\ge R}\), we
have \(\textnormal{(I)}\ge 1\). Next, since \(h(0) = 1\), we have
\begin{equation}\label{1-h-small}
\abs{1-h(r)}\le \int_0^r \abs{h'(s)}\dd{}s \le \epsilon{}_hr.
\end{equation}
It follows from \cref{1-h-small}, the assumption \cref{h-derivative-small} of
\cref{hyperboloidal-foliation} and the estimate \(v = u+r\le C\langle{}r\rangle{}\) on \(\Sigma{}(0)\) (for
\(C\) depending only on \(B_h\)) that in \(\set{r\le R}\), we have
\begin{equation}
\textnormal{(I)}\ge r(2(1-\epsilon{}_hr) - C\epsilon{}_h\langle{}r\rangle{} - C\epsilon{}_h)\ge r(2 - C\epsilon{}_h\langle{}R\rangle{}).
\end{equation}
If \(\epsilon_h\ll R\), then we obtain \(\textnormal{(I)}\ge r\) in \(\set{r\le R}\).
\end{proof}
\subsection{Constructing time integral initial data by inverting an elliptic operator}
\label{time-integral-elliptic} In this section, we invert an elliptic operator to
construct functions which will be used in \cref{time-integrals-from-data} as
initial data for wave equations that define time integrals.
\subsubsection{Scalar fields that solve an equation with a good zeroth-order term}
\label{good-scalar-field-tinv}
For \(f_1(r),f_2(r)=\mathcal{O}(\epsilon{}r^2\langle{}r\rangle^{-2})\), define the operator
\begin{equation}
\tilde{\mathcal{L}}_{\alpha{},f_1,f_2}\Phi{}\coloneqq{} X^2\Phi{} - \alpha{}r^{-2}(1+f_1(r))\Phi{} + r^{-2}(1+f_2(r))\partial{}_\theta^2\Phi{}.
\end{equation}
For \(p\in \R\) and \(N\ge 0\), let \(\mathbf{H}_{p,N}\) be the closure of
\(C_c^\infty(\Sigma{}(0)\setminus \set{r=0})\) under the norm
\begin{equation}
\norm{\Phi{}}_{p,N}^2\coloneqq{}\sum_{\substack{n_1,n_2\ge 0\\n_1+n_2\le N}}\int _{\Sigma{}(0)} r^p[(X\partial{}_\theta^{n_1}(rX)^{n_2}\Phi{})^2 + r^{-2}(\partial{}_\theta^{n_1}(rX)^{n_2}\Phi{})^2 + r^{-2}(\partial{}_\theta{}\partial{}_\theta^{n_1}(rX)^{n_2}\Phi{})^2]\dd{}r\dd{}\theta{}.
\end{equation}
\begin{lemma}
For \(\Phi{}\in C_c^\infty(\Sigma{}(0)\setminus \set{r=0})\), \(p\in \R\), and \(N\ge 0\), we have
\begin{equation}
\norm{\Phi{}}_{p,N}\lesssim_{p,N} \norm{r^{p/2}\Phi{}}_{0,N}.
\end{equation}
\label{norm-comparison}
\end{lemma}
\begin{proof}
This is a straightforward calculation.
\end{proof}
\begin{lemma}[Commutation of the operator \(\tilde{\mathcal{L}}\) with \((rX)\)-derivatives]
For \(N\ge 0\), there are constants \(C_{N,n}\in \R\) such that
\begin{equation}
\tilde{\mathcal{L}}_{\alpha{},f_1,f_2}(rX)^N = \sum_{n=0}^NC_{N,n}(rX)^n\tilde{\mathcal{L}}_{\alpha{},f_1,f_2} + \sum_{n=0}^{N-1}\mathcal{O}(\langle{}r\rangle{}^{-2})(rX)^n + \sum_{n=0}^{N-1}\mathcal{O}(\langle{}r\rangle{}^{-2})\partial_\theta^2(rX)^n,
\end{equation}
where the implicit constants in the \(\mathcal{O}\)-notation depend on \(\alpha{}\),
\(f_1\), and \(f_2\).
\label{L-rX-commutation}
\end{lemma}
\begin{proof}
Induction on \(N\).
\end{proof}
\begin{lemma}[A priori estimate for the elliptic operator \(\tilde{\mathcal{L}}\)]
Suppose \(\Phi{}\in \mathbf{H}_{p,N}\) solves
\begin{equation}\label{La-equation}
\tilde{\mathcal{L}}_{\alpha{},f_1,f_2}\Phi{} = \tilde{\mathcal{F}}.
\end{equation}
Suppose moreover that one of the following assumptions holds:
\begin{enumerate}
\item \label{assumption-1} \(\alpha{} > 0\),
\item \label{assumption-2} or \(\alpha{}>-1\), \(\alpha{}\neq{}0\), and \(\Phi{}=\Phi_{\ge 1}\).
\end{enumerate}
Define \(\tilde{\alpha{}} = \alpha{}\) if \cref{assumption-1} holds and
\(\tilde{\alpha{}} = \alpha{} + 1\) if \cref{assumption-2} holds. Then for \(\eta_0
> 0\) such that \(\epsilon{} \ll \eta_0\), \(\abs{p}\le 2\sqrt{\tilde{\alpha{}}}-\eta{}_0\)
and \(N\ge 0\), we have
\begin{equation}\label{elliptic-a-priori-equation}
\begin{split}
\norm{\Phi{}}_{p,N}^2\lesssim_{p,N,\alpha{},\eta{}_0} \sum_{\substack{n_1,n_2\ge 0\\n_1+n_2\le N}}\int _{\Sigma{}(0)} r^{p+2}(\partial{}_\theta^{n_1}(rX)^{n_2}\tilde{\mathcal{F}})^2\dd{}r\dd{}\theta{}.
\end{split}
\end{equation}
\label{elliptic-a-priori}
\end{lemma}
\begin{proof}
By density, it suffices to consider \(\Phi{}\in C_c^\infty(\Sigma{}(0)\setminus \set{r=0})\).

\step{Step 1: The case \(N = 0\).} Compute
\begin{equation}
\begin{split}
-r^p\Phi{}\tilde{\mathcal{L}}_{\alpha{},f_1,f_2}\Phi{} &= -X(r^p\Phi{}X\Phi{}) - \partial{}_\theta{}((1+f_2(r))\Phi{}\partial{}_\theta{}\Phi{}) + r^p(X\Phi{})^2 + pr^{p-1}\Phi{}X\Phi{} \\
&\qquad + \alpha{}r^{p-2}(1 + \mathcal{O}(\epsilon{}))\Phi^2 + r^{p-2}(1+\mathcal{O}(\epsilon{}))(\partial{}_\theta{}\Phi{})^2.
\end{split}
\end{equation}
After integrating on \(\Sigma{}(0)\) and applying an \(r\)-weighted Young's inequality
for the last term on the first line, we obtain
\begin{equation}\label{elliptic-a-priori-prep}
\begin{split}
&\int _{\Sigma{}(0)} r^p\Bigl(1 - \frac{1}{4}\delta{}^{-1}p^2\Bigr)(X\Phi{})^2 + r^{p-2}(\alpha{}(1+\mathcal{O}(\epsilon{})) - \delta{})\Phi^2 + r^{p-2}(1+\mathcal{O}(\epsilon{}))(\partial{}_\theta{}\Phi{})^2\dd{}r\dd{}\theta{}\\
&\le \int _{\Sigma{}(0)}r^p\abs{\Phi{}}\abs{\tilde{\mathcal{L}}_{\alpha{},f_1,f_2}\Phi{}}\dd{}r\dd{}\theta{}.
\end{split}
\end{equation}
If \cref{assumption-2} holds, then we use the Poincaré inequality on \(S^1\):
\begin{equation}
\int _{S^1} (\partial{}_\theta{}\Phi{}_{\ge 1})^2\dd{}\theta{}\ge \int _{S^1} \Phi{}_{\ge 1}^2 \dd{}\theta{}.
\end{equation}
In either case, the expression on the left-hand side of \cref{elliptic-a-priori-prep} is
coercive for \(\delta{} = \alpha{} - \eta_0\) and \(\abs{p}\le 2\sqrt{\tilde{\alpha{}}} - \eta_0\), and so we obtain
\begin{equation}
\int _{\Sigma{}(0)} r^p(X\Phi{})^2 + r^{p-2}\Phi^2 + r^{p-2}(\partial{}_\theta{}\Phi{})^2\dd{}r\dd{}\theta{}\lesssim _{p,\alpha{},\eta{}_0} \int _{\Sigma{}(0)}r^p\abs{\Phi{}}\abs{\tilde{\mathcal{L}}_{\alpha{},f_1,f_2}\Phi{}}\dd{}r\dd{}\theta{}.
\end{equation}
Using Young's inequality completes the proof when \(N = 0\).

\step{Step 2: The case \(N\ge 1\).} Now assume \(N\ge 1\) and
\cref{elliptic-a-priori-equation} holds with \(N - 1\) in place of \(N\). Applying the \(N = 0\) case to \((rX)^N\Phi{}\) in place
of \(\Phi{}\) and using \cref{L-rX-commutation}, we obtain
\begin{equation}
\begin{split}
\norm{(rX)^N\Phi{}}_{p,0}\lesssim \int _{\Sigma{}(0)} r^{p+2}(\tilde{\mathcal{L}}_{\alpha{},f_1,f_2}(rX)^N\Phi{})^2 \dd{}r\dd{}\theta{}\lesssim \sum_{n=0}^N \int _{\Sigma{}(0)} r^{p+2}((rX)^N \tilde{\mathcal{F}})^2 \dd{}r\dd{}\theta{} +\sum_{n=0}^{N-1} \norm{(rX)^n\Phi{}}_{p,0}.
\end{split}
\end{equation}
Adding a suitable multiple of the \(N-1\) case, we obtain
\begin{equation}
\sum_{n=0}^N\norm{(rX)^n\Phi{}}_{p,0}\lesssim \sum_{n=0}^N \int _{\Sigma{}(0)} r^{p+2}((rX)^N \tilde{\mathcal{F}})^2 \dd{}r\dd{}\theta{}.
\end{equation}
Repeating this argument with \(\partial_\theta^n\Phi{}\) in place of \(\Phi{}\) completes the proof (since \(\partial_\theta{}\)
commutes with \(\tilde{\mathcal{L}}_{\alpha{},f_1,f_2}\)).
\end{proof}
\begin{proposition}[Inversion of the elliptic operator \(\tilde{\mathcal{L}}\) acting on good scalar fields]
Suppose that \(\tilde{\mathcal{F}}\in C^\infty(\Sigma{}(0)\setminus
\set{r=0})\) satisfies
\begin{equation}\label{F-norm}
\sum_{\substack{n_1,n_2\ge 0\\n_1+n_2\le N}}\int _{\Sigma{}(0)}r^{p+2}(\partial{}_\theta^{n_1}(rX)^{n_2} \tilde{\mathcal{F}})^2 \dd{}r\dd{}\theta{}< \infty
\end{equation}
for some \(N\ge 0\) and \(p\in \R\). Suppose moreover that either \(\alpha{} = -1/4\) and
\(\tilde{\mathcal{F}} = \tilde{\mathcal{F}}_{\ge 1}\), or \(\alpha{} = 3/4\) and
\(\tilde{\mathcal{F}}\) is radially symmetric. Let \(\tilde{\alpha{}}\) and \(\eta_0\) be as
in \cref{elliptic-a-priori}, and let \(\abs{p}\le 2\sqrt{\tilde{\alpha{}}} -
\eta{}_0\). Then there is a unique \(\Phi{}\in \mathbf{H}_{p,N}\) solving
\cref{La-equation}, which is radially symmetric (resp. supported on angular modes \(\ge
1\)) if \(\tilde{\mathcal{F}}\) is, and the following estimate holds:
\begin{equation}\label{F-norm-estimate-goal}
\norm{\Phi{}}_{p,N}^2 \lesssim_{p,N,\alpha{}} \side{LHS}{F-norm}.
\end{equation}
Moreover, if \(\mathcal{F}\in r^\beta{}C^\infty(\Sigma{}(0))\) for \(\beta{} \coloneqq{} \frac{1}{2}(1 +
\sqrt{1 + 4\alpha{}})\), then \(\Phi{}\in r^\beta{}C^\infty(\Sigma{}(0))\).
\label{good-scalar-field-Linv}
\end{proposition}
\begin{remark}
We could have formulated a version of \cref{good-scalar-field-Linv} for \(\alpha{} \notin
\set{-1/4,3/4}\), but these are the only cases that we need in order to prove
\cref{main-theorem}.
\end{remark}
\begin{proof}
As in \cite[Sec.~9.1]{GAJIC2023110058}, define the twisted operator
\begin{equation}
\tilde{\mathcal{L}}_{\alpha{},f_1,f_2}^{(p)}\Phi{}\coloneqq{}r^{p/2}\tilde{\mathcal{L}}_{\alpha{},f_1,f_2}(r^{-p/2}\Phi{}).
\end{equation}
We also introduce the associated bilinear operator \(B_{\alpha{},f_1,f_2}^{(p)}:\mathbf{H}_{0,0}\times \mathbf{H}_{0,0}\to \R\):
\begin{equation}
B_{\alpha{},f_1,f_2}^{(p)}[\Phi_1,\Phi_2]\coloneqq{}\int _{\Sigma{}(0)} X(r^{-p/2}\Phi{}_1)X(r^{p/2}\Phi{}_2) + \alpha{}r^{-2}(1+f_1(r))\Phi{}_1\Phi{}_2 + r^{-2}(1+f_2(r))\partial{}_\theta{}\Phi{}_1\partial{}_\theta{}\Phi{}_2\dd{}r\dd{}\theta{}.
\end{equation}
Note that \(B^{(p)}_{\alpha{},f_1,f_2}[\Phi{}_1,\Phi{}_2] = \langle{}
\tilde{\mathcal{L}}^{(p)}_{\alpha{},f_1,f_2}\Phi{}_1,\Phi{}_2\rangle{}_{L^2(\Sigma{}(0),\dd{}r\dd{}\theta{})}\).
One can check that for \(\abs{p}\le 2\sqrt{\tilde{\alpha{}}} - \eta_0\), we have
\begin{equation}
\abs{B_{\alpha{},f_1,f_2}^{(p)}[\Phi_1,\Phi_2]}\lesssim_{p,\alpha{}} \norm{\Phi{}_1}_{0,0}\norm{\Phi{}_2}_{0,0},\qquad \abs{B_{\alpha{},f_1,f_2}^{(p)}[\Phi,\Phi]}\gtrsim_{p,\alpha{}}\norm{\Phi{}}_{0,0}^2.
\end{equation}
By \cref{F-norm}, the functional \(\langle{}r^{p/2} \tilde{\mathcal{F}},\cdot
\rangle{}_{L^2(\Sigma{}(0),\dd{}r\dd{}\theta{})}\) is continuous on
\(\mathbf{H}_{0,0}\). By the Lax--Milgram lemma, there exists a unique
\(\Phi{}\in \mathbf{H}_{0,0}\) solving
\(\tilde{\mathcal{L}}_{\alpha{},f_1,f_2}^{(p)}\Phi{} = r^{p/2}
\tilde{\mathcal{F}}\) weakly, equivalently
\(\tilde{\Phi{}}=r^{p/2}\Phi{}\in\mathbf{H}_{0,0}\) solving
\(\tilde{\mathcal{L}}_{\alpha{},f_1,f_2}\tilde{\Phi{}} = \tilde{\mathcal{F}}\)
and, by \cref{elliptic-a-priori}, satisfying the estimate
\(\norm{\tilde{\Phi{}}}_{0,0}^2\lesssim \side{LHS}{F-norm}\). Now suppose \(N =
1\). Since \(\tilde{\mathcal{L}}\partial{}_\theta{}\tilde{\Phi{}} =
\partial{}_\theta{}\tilde{\mathcal{F}}\), the uniqueness part of the
Lax--Milgram lemma ensures that \(\partial_\theta{}\tilde{\Phi{}}\in
\mathbf{H}_{0,0}\). By \cref{L-rX-commutation} and the fact that
\(\tilde{\Phi{}},\partial_\theta{}\tilde{\Phi{}}\in \mathbf{H}_{0,0}\), we
conclude that \((rX)\tilde{\Phi{}}\) satisfies
\(r^{p/2+1}\tilde{\mathcal{L}}_{\alpha{},f_1,f_2}(rX)\tilde{\Phi{}}\in
L^2(\Sigma{}(0),\dd{}r\dd{}\theta{})\). Then a Lax--Milgram argument shows that
\((rX)\tilde{\Phi{}}\in \mathbf{H}_{0,0}\) (with an estimate). This shows that
\(\tilde{\Phi{}}\in \mathbf{H}_{0,1}\). In this way one can inductively show
that \(\tilde{\Phi{}}\in \mathbf{H}_{0,N}\) and
\(\norm{\tilde{\Phi{}}}_{0,N}^2\lesssim \side{LHS}{F-norm}\). By
\cref{norm-comparison}, we conclude that \(\Phi{} = r^{-p/2}\tilde{\Phi{}}\in
\mathbf{H}_{p,N}\) and \cref{F-norm-estimate-goal} holds.

We now deduce a smoothness property of
\(\Phi{}\) by considering the equation solved by \(\tilde{\Phi{}} \coloneqq{} r^{-\beta{}}\Phi{}\),
namely
\begin{equation}
\begin{split}
X^2\tilde{\Phi{}} + \frac{2\beta{}}{r}X\tilde{\Phi{}} - \alpha{}r^{-2}f_1(r)\tilde{\Phi{}} + r^{-2}(1+f_2(r))\partial{}_\theta^2\tilde{\Phi{}} = r^{-\beta{}}\tilde{\mathcal{F}}.
\end{split}
\end{equation}
If \(\alpha{} = -1/4\), then \(\beta{} = 1/2\), and so we have
\begin{equation}
(\Lapl _{\R^2} + r^{-2}f_2(r)\partial{}_{\theta{}}^2)\tilde{\Phi{}} + \alpha{}r^{-2}f_1(r)\tilde{\Phi{}} = r^{-\beta{}} \tilde{\mathcal{F}},
\end{equation}
where we have written \(\Lapl_{\R^2} = X^2 + r^{-1}X + r^{-2}\partial_\theta^2\). The
principal part of the operator on the left is a small and smooth perturbation of
the Laplacian in two dimensions with small and smooth potential (by the
assumptions on \(f_1\) and \(f_2\)), and the source on the right-hand side is smooth
by assumption. We conclude by standard interior elliptic regularity results that
\(\tilde{\Phi{}}\in C^\infty(\Sigma{}(0))\). If instead \(\alpha{} = 3/4\) and
\(\mathcal{F}\) is radially symmetric, then \(\Phi{}\) is radially symmetric by uniqueness.
Moreover, \(\beta{} = 3/2\), and the equation becomes
\begin{equation}
\Lapl _{\R^4}\tilde{\Phi{}} - \alpha{}r^{-2}f_1(r)\tilde{\Phi{}} = r^{-\beta{}} \tilde{\mathcal{F}},
\end{equation}
where we have written \(\Lapl _{\R^4} = X^2 + 3r^{-1}X + \Lapl _{S^3}\), and now interpret
\(\tilde{\Phi{}}\) as a spherically symmetric function on \(\R^4\). Again, elliptic
regularity for this equation is standard.
\end{proof}
\subsubsection{Radially symmetric scalar fields}
Although the elliptic theory of \cref{good-scalar-field-tinv} fails for radially
symmetric functions (because the critical value \(-1/4\) in the coefficient of
the inverse-square potential in the equation for \(\psi{}_0\) makes the operator
only degenerate elliptic rather than elliptic), for such scalar fields the
operator \(\mathcal{L}\) reduces to an ODE which we can solve via direct
integration.\sidenote{In \cite{gajic2025linearnonlinearlatetimetails}, which
studies late-time tails for linear and nonlinear waves on the Schwarzschild
spacetime of three space dimensions, an elliptic operator similar to \(\mathcal{L}\) is inverted for
low angular modes via direct integration, while elliptic theory is used to
invert the operator acting on high angular modes.}
\begin{proposition}[Inverting the operator \(\mathcal{L}\) via direct integration; radially symmetric case]
Let \(\mathcal{F}\in C^\infty((0,\infty))\). Suppose that there are constants \(B_{N,\delta{}} >
0\) such that the following estimate holds for each \(N\ge 0\) and \(\delta{} > 0\):
\begin{equation}\label{F-estimate}
\sum_{n=0}^N\int_0^\infty r^{-1}\langle{}r\rangle{}^{4-\delta{}}((rX)^n\mathcal{F})^2\dd{}r\le  B_{N,\delta{}}^2.
\end{equation}
Suppose moreover that we have the vanishing property
\begin{equation}\label{u-vanishing}
\int_0^\infty r^{1/2}G(r)^{1/2}\mathcal{F}(r)\dd{}r = 0.
\end{equation}
Then the function \(\zeta{} : (0,\infty)\to \R\) defined formally by
\begin{equation}\label{u-definition}
\zeta{}(r)\coloneqq{}-r^{1/2}G(r)^{1/2}\int_r^\infty \rho{}^{-1}G(\rho{})^{-1}\int_0^\rho{} s^{1/2}G(s)^{1/2}\mathcal{F}(s)\dd{}s\dd{}\rho{}
\end{equation}
\begin{enumerate}
\item \label{u-well-defined} is well-defined and smooth, with \(r^{-1/2}\zeta{}\) bounded as \(r\to 0\),
\item \label{u-smooth} determines an element of \(r^{1/2}C^\infty(\Sigma{}_0)\) if \(\mathcal{F}\)
does, which is moreover radially symmetric,
\item \label{u-solution} solves \(\mathcal{L}\zeta{} = \mathcal{F}\),
\item \label{u-estimate} and satisfies the following estimate for each \(N\ge
   0\) and \(\delta{} > 0\):
\end{enumerate}
\begin{equation}
\begin{split}
&\sum_{n=0}^N\int_0^\infty r(X(rX)^n(r^{-1/2}G^{-1/2}\zeta{}))^2 + r\langle{}r\rangle{}^{-2\delta{}}(X(r^{1/2}(rX)^n(r^{-1/2}G^{-1/2}\zeta{})))^2 \\
&\qquad + \langle{}r\rangle^{-2\delta{}}((rX)^n(r^{-1/2}G^{-1/2}\zeta{}))^2 \dd{}r \\
&\lesssim_{N,\delta{}} B_{N,\delta{}}^2.
\end{split}
\end{equation}
\label{L-inverse-radially-symmetric}
\end{proposition}
\begin{remark}[Purpose of the vanishing condition \cref{u-vanishing}]
Without the vanishing property \cref{u-vanishing}, the \(\rho{}\)-integral in
\cref{u-definition} would not converge.
\end{remark}
\begin{proof}
We first establish \cref{u-well-defined}. We claim that
\begin{equation}\label{u-prep-0}
\abs[\Big]{\int_0^r s^{1/2}G(s)^{1/2}\mathcal{F}(s)\dd{}s} \lesssim_\delta{} B_{0,\delta{}}r^{3/2}\langle{}r\rangle^{-2+\delta{}/2}.
\end{equation}
Indeed, using Cauchy--Schwarz and \(\abs{G(s)}\lesssim 1\), we obtain
\begin{equation}
\abs[\Big]{\int_0^r s^{1/2}G(s)^{1/2}\mathcal{F}(s)\dd{}s} \lesssim B_{0,\delta{}}\Bigl(\int_0^r s^2\langle{}s\rangle{}^{-4+\delta{}}\dd{}s\Bigr)^{1/2}\lesssim B_{0,\delta{}}r^{3/2}\langle{}r\rangle^{-3/2}.
\end{equation}
This implies \cref{u-prep-0} as \(r\to 0\). To also obtain \cref{u-prep-0} as \(r\to \infty\),
we use the vanishing condition \cref{u-vanishing} to write
\begin{equation}
\abs[\Big]{\int_0^r s^{1/2}G(s)^{1/2}\mathcal{F}(s)\dd{}s} = \abs[\Big]{\int_r^\infty s^{1/2}G(s)^{1/2}\mathcal{F}(s)\dd{}s}\lesssim B_{0,\delta{}}\Bigl(\int_r^\infty s^2\langle{}s\rangle{}^{-4+\delta{}}\dd{}s\Bigr)^{1/2}\lesssim B_{0,\delta{}}\langle{}s\rangle^{-1/2+\delta{}/2}.
\end{equation}
It follows that \(\zeta{}\) is well-defined. Moreover, \(\zeta{}\) inherits the smoothness
of \(\mathcal{F}\) on \((0,\infty)\) by the fundamental theorem of calculus. Next, we establish \cref{u-smooth}. By assumption, \(r^{-1/2}\mathcal{F}\) determines a smooth radially symmetric function on
\(\Sigma_0\), and so \(r^{1/2}\mathcal{F} \in rC^\infty(\Sigma{}_0)\). In view of
\cref{even-smooth} of \cref{axisymmetric-smoothness}, a Taylor expansion shows that
\begin{equation}
r^{-1}\int_0^r s^{1/2}G(s)^{1/2}\mathcal{F}(s)\dd{}s\in{}rC^\infty(\Sigma{}_0),
\end{equation}
and so \(r^{-1/2}\zeta{}\) determines an element of \(C^\infty(\Sigma{}_0)\), being the difference
of a constant and an element of \(r^2C^\infty(\Sigma_0)\subset C^\infty(\Sigma_0).\) Next, \cref{u-solution} follows from the fundamental theorem of calculus and the
representation
\begin{equation}
\mathcal{L}\zeta{} = r^{-1/2}G^{-1/2}X(rFX(r^{-1/2}G^{-1/2}\zeta{})).
\end{equation}
Finally, we turn to \cref{u-estimate}. Write
\begin{equation}
\textnormal{(I)}\coloneqq{}\int_r^\infty \rho{}^{-1}G(\rho{})^{-1}\int_0^\rho{} s^{1/2}G(s)^{1/2}\mathcal{F}(s)\dd{}s\dd{}\rho{},\qquad \textnormal{(II)} \coloneqq{}\int_0^r s^{1/2}G(s)^{1/2}\mathcal{F}(s)\dd{}s,
\end{equation}
so that
\begin{equation}\label{u-prep-estimates}
\abs{\textnormal{(I)}} = \abs{r^{-1/2}G^{-1/2}\zeta{}}\lesssim \langle{}r\rangle^{-1/2+\delta{}/2},\qquad \abs{\textnormal{(II)}}\lesssim_\delta{} B_{0,\delta{}}r^{3/2}\langle{}r\rangle^{-2+\delta{}/2},
\end{equation}
by \cref{u-prep-0}. By differentiating \cref{u-definition} and using \(G = \mathcal{O}(1)\) and
\(G'=\mathcal{O}(\langle{}r\rangle^{-1})\), we obtain
\begin{align}
(rX)^N(r^{-1/2}G^{-1/2}\zeta{}) &=  \mathbf{1}_{N=0}\textnormal{(I)} + \mathcal{O}(1)\textnormal{(II)} + \sum_{n=0}^{N-2}\mathcal{O}(r^{3/2})(rX)^n\mathcal{F},\label{u-prep-2} \\
X(rX)^N(r^{-1/2}G^{-1/2}\zeta{}) &= \mathcal{O}(r^{-1})\textnormal{(II)} + \sum_{n=0}^{N-1}\mathcal{O}(r^{1/2})(rX)^n\mathcal{F}, \label{u-prep-3} \\
X(r^{1/2}(rX)^N(r^{-1/2}G^{-1/2}\zeta{})) &= \mathcal{O}(r^{-1/2})\textnormal{(I)} + \mathcal{O}(r^{-1/2})\textnormal{(II)} + \sum_{n=0}^{N-2}\mathcal{O}(r)(rX)^n\mathcal{F}. \label{u-prep-4}
\end{align}
Now \cref{u-estimate} follows from
\cref{u-prep-estimates,u-prep-2,u-prep-3,u-prep-4,F-estimate}.
\end{proof}
\subsection{Constructing time integrals from initial data by solving a wave equation}
\label{time-integrals-from-data} In this section, we build time integrals of the
renormalized quantity \(\widehat{\varphi{}}\) (and of \(\widehat{\Psi{}}_0\))
using the functions constructed in \cref{time-integral-elliptic} as initial data
for wave equations. In view of the results of \cref{sec:energy-decay}, it is
important that we can construct \emph{two} time integrals of \(\widehat{\Psi{}}_0\);
this is possible because \(\widehat{\Psi{}}_0\) solves an equation with a good
zeroth-order term. We first provide a formula comparing \((rL)\)-derivatives
with \((rX)\)-derivatives.
\begin{lemma}[Comparing \((rL)\)-derivatives with \((rX)\)-derivatives]
For \(N\ge 1\), we have
\begin{equation}
(rX)^N = \sum_{\substack{n_1,n_2\ge 0 \\ n_1 + n_2 \le  N-1}}((1 + \mathcal{O}(r\langle{}r\rangle{}^{-1}))(rL)(rL)^{n_1}T^{n_2} + \mathcal{O}(r\langle{}r\rangle{}^{-1})T(rL)^{n_1}T^{n_2}),
\end{equation}
and the same formula holds with \(X\) and \(L\) swapped.
\label{rL-rX-comparison}
\end{lemma}
\begin{proof}
The case \(N = 1\) follows from the explicit formula \(rX = G^{-1}rL + rhT\) and
the fact that \(G^{-1} = \mathcal{O}(1)\) and \(rh = \mathcal{O}(1)\) (by
\cref{h-hyperboloidal} of \cref{hyperboloidal-foliation}). The general case follows by
induction on \(N\). The case when \(X\) and \(L\) are swapped is analogous.
\end{proof}
\begin{proposition}[Constructing the first time integral of \(\widehat{\varphi{}}\)]
Suppose \(\varphi{}\in C^\infty(\mathcal{R})\) solves \cref{wave-equation}. There exists
\(T^{-1}\widehat{\varphi{}}\in C^\infty(\mathcal{R})\) solving
\begin{equation}
\Box{}T^{-1}\widehat{\varphi{}} = -\mathfrak{L}[\varphi{}]T^{-1}\Box{}\varphi{}_{\textnormal{mink}}
\end{equation}
such that \(TT^{-1}\widehat{\varphi{}} = \widehat{\varphi{}}\) and the following estimates hold
for each \(N\ge 0\) and \(\delta{} > 0\):
\begin{equation}\label{Tinv-phi-estimates}
E_N[T^{-1}\widehat{\varphi{}}_0](0,v) + \mathcal{E}_{1+\delta{},N}[T^{-1}\widehat{\varphi{}}_0](0,v) + \tilde{E}[T^{-1}\widehat{\psi{}}_{\ge 1}](0,v) + \tilde{\mathcal{E}}_{1+\delta{},N}[T^{-1}\widehat{\psi{}}_{\ge 1}](0,v)\lesssim_{N,\delta{}} v^{3\delta{}}\mathbf{D}_{N,\delta{}}[\varphi{}],
\end{equation}
where the initial data norm \(\mathbf{D}_{N,\delta{}}[\varphi{}]\) was defined in \cref{data-norm}.
\label{Tinv-phi}
\end{proposition}
\begin{proof}
\step{Step 1: Construction.}
Let \(\zeta{}_0\in r^{1/2}C^\infty(\Sigma(0))\) be the solution to
\begin{equation}\label{Lphi0-solve}
\mathcal{L}\zeta{}_0 = \mathcal{F}[G^{1/2}\widehat{\psi{}}_0]|_{\Sigma{}(0)} - G^{-3/2}r^{1/2}\mathfrak{L}[\varphi{}]T^{-1}\Box{}\varphi{}_{\textnormal{mink}}|_{\Sigma{}(0)}
\end{equation}
constructed in \cref{L-inverse-radially-symmetric}. Note that the right
side of \cref{Lphi0-solve} is an element of \(r^{1/2}C^\infty(\Sigma{}(0))\) and satisfies the
vanishing condition \cref{u-vanishing} by the definition of \(\mathfrak{L}[\varphi{}]\).
Let \(\zeta{}_{\ge 1}\in r^{1/2}C^\infty(\Sigma{}(0))\) be the solution to
\begin{equation}
\mathcal{L}\zeta{}_{\ge 1} = \mathcal{F}[G^{1/2}\widehat{\psi{}}_{\ge 1}]|_{\Sigma{}(0)}
\end{equation}
constructed in \cref{good-scalar-field-Linv} (noting that the operator
\(\mathcal{L}\) is of the type considered in the proposition). Then \(\zeta{}\coloneqq{}\zeta{}_0 +
\zeta{}_{\ge 1}\) is a solution \(\zeta{}\in r^{1/2}C^\infty(\Sigma{}(0))\) to
\begin{equation}\label{Tinv-phi-prep-0}
\mathcal{L}\zeta{} = \mathcal{F}[G^{1/2}\widehat{\psi{}}]|_{\Sigma{}(0)} - G^{-3/2}r^{1/2}\mathfrak{L}[\varphi{}]T^{-1}\Box{}\varphi{}_{\textnormal{mink}}|_{\Sigma{}(0)}.
\end{equation}
Let \(T^{-1}\widehat{\varphi{}}\in C^\infty(\mathcal{R})\) be the unique solution to
\begin{equation}\label{Tinv-phi-prep}
\Box{}T^{-1}\widehat{\varphi{}} = -\mathfrak{L}[\varphi{}]T^{-1}\Box{}\varphi{}_{\textnormal{mink}},\qquad (T^{-1}\widehat{\varphi{}}|_{\Sigma{}(0)},TT^{-1}\widehat{\varphi{}}|_{\Sigma{}(0)}) = (G^{-1/2}r^{-1/2}\zeta{},\widehat{\varphi{}}|_{\Sigma{}(0)}),
\end{equation}
where \(T^{-1}\Box{}\varphi{}_{\textnormal{mink}}\) was defined in \cref{Tinv-mink}.

We now show that \(TT^{-1}\widehat{\varphi{}} = \widehat{\varphi{}}\). Since
\begin{equation}
\Box{}TT^{-1}\widehat{\varphi{}} = T\Box{}T^{-1}\widehat{\varphi{}} = \mathfrak{L}[\varphi{}]TT^{-1}\Box{}\varphi{}_{\textnormal{mink}} = \Box{}\widehat{\varphi{}},
\end{equation}
it suffices to show that \((\widehat{\varphi{}}|_{\Sigma{}(0)},T\widehat{\varphi{}}|_{\Sigma{}(0)}) =
(TT^{-1}\widehat{\varphi{}}|_{\Sigma{}(0)},TTT^{-1}\widehat{\varphi{}}|_{\Sigma{}(0)})\),
by the uniqueness of solutions to \cref{wave-equation}. The equality
\(TT^{-1}\widehat{\varphi{}}|_{\Sigma{}(0)} =
\widehat{\varphi{}}|_{\Sigma{}(0)}\) holds by construction. To see that
\(TTT^{-1}\widehat{\varphi{}}|_{\Sigma{}(0)} =
T\widehat{\varphi{}}|_{\Sigma{}(0)}\), note that, by
\cref{wave-equation-expressions,Tinv-phi-prep} we have
\begin{equation}\label{Tinv-phi-prep-1}
\mathcal{L}(G^{1/2}T^{-1}\widehat{\psi{}}) = \mathcal{F}[G^{1/2}TT^{-1}\widehat{\psi{}}] - G^{-3/2}r^{1/2}\mathfrak{L}[\varphi{}]T^{-1}\Box{}\varphi{}_{\textnormal{mink}}
\end{equation}
Restricting to \(\Sigma{}(0)\), using the prescription
\(G^{1/2}r^{1/2}T^{-1}\widehat{\varphi{}}|_{\Sigma{}(0)} = \zeta{}\), and
recalling \cref{Tinv-phi-prep-0}, we obtain
\begin{equation}\label{Tinv-phi-prep-2}
\mathcal{L}(G^{1/2}T^{-1}\widehat{\psi{}})|_{\Sigma{}(0)} = \mathcal{L}\zeta{} = \mathcal{F}[G^{1/2}\widehat{\psi{}}]|_{\Sigma{}(0)} - G^{-3/2}r^{1/2}\mathfrak{L}[\varphi{}]T^{-1}\Box{}\varphi{}_{\textnormal{mink}}|_{\Sigma{}(0)}.
\end{equation}
It follows from \cref{Tinv-phi-prep-1,Tinv-phi-prep-2} that
\begin{equation}
\mathcal{F}[G^{1/2}TT^{-1}\widehat{\psi{}}]|_{\Sigma{}(0)} = \mathcal{F}[G^{1/2}\widehat{\psi{}}]|_{\Sigma{}(0)}.
\end{equation}
Recalling the definition of \(\mathcal{F}\) (see \cref{calF-def}), we obtain the desired equality
\(TTT^{-1}\widehat{\psi{}}|_{\Sigma{}(0)} = \widehat{\psi{}}|_{\Sigma{}(0)}\).

\step{Step 2: Estimates.} We want to control energies of \(T^{-1}\widehat{\varphi{}}\) that
involve \((rL)\)-derivatives. We first use \cref{rL-rX-comparison} to replace
\((rL)\)-derivatives by \((rX)\)-derivatives and \(T\)-derivatives. We also use
the relation \(TT^{-1}\widehat{\varphi{}} = \widehat{\varphi{}}\) so that we
only estimate \(T^MT^{-1}\widehat{\varphi{}}\) for \(M = 0\) and
\(T^M\widehat{\varphi{}}\) for \(M > 0\). This is done in
\cref{Tinv-estimate-0,Tinv-estimate-1}. We then estimate the \((rX)\)-derivatives
using the results of \cref{time-integral-elliptic}, which in turn requires bounding
the source term on the right-hand side of \cref{Tinv-phi-prep-0} using
\cref{Tinv-mink}; this is done in \cref{Tinv-Linv-prep-0}.

First, by \cref{rL-rX-comparison} and \(TT^{-1}\widehat{\varphi{}}_0 =
\widehat{\varphi{}}_0\), we have
\begin{equation}\label{Tinv-estimate-0}
\begin{split}
&E_N[T^{-1}\widehat{\varphi{}}_0](0,v) \\
&\lesssim \sum_{n=0}^N\int _{\Sigma{}(0)} r(X(rX)^nT^{-1}\widehat{\varphi{}}_0)^2 + \langle{}r\rangle{}^{-1}((rX)^nT^{-1}\widehat{\varphi{}}_0)^2 \dd{}r\dd{}\theta{} \\
&\qquad + \sum_{\substack{n\ge 0,m\ge 1 \\ n+m\le N }} \int _{\Sigma{}(0)}((rX)^nT^{m-1}\widehat{\varphi{}}_0)^2 + r(X(rX)^nT^{m-1}\widehat{\varphi{}}_0)^2 + h(r)r(T(rX)^nT^{m-1}\widehat{\varphi{}}_0)^2 \dd{}r\dd{}\theta{} \\
&\lesssim  \sum_{\substack{n\ge 0,m\ge 0 \\ n+m\le N }}\norm{\langle{}r\rangle^{1/2}(rX)^nT^m\widehat{\varphi{}}_0}_{L^\infty(\Sigma{}(0))}^2 + \sum_{m=0}^{N-1}E_{N-1-m}[T^m\widehat{\varphi{}}_0](0,v) \\
&\qquad + \sum_{n=0}^N\int _{\Sigma{}(0)} r(X(rX)^nT^{-1}\widehat{\varphi{}}_0)^2 + \langle{}r\rangle{}^{-1}((rX)^nT^{-1}\widehat{\varphi{}}_0)^2\dd{}r\dd{}\theta{}.
\end{split}
\end{equation}
Similarly, we have
\begin{equation}\label{Tinv-estimate-1}
\begin{split}
&\mathcal{E}_{1+\delta{},N}[T^{-1}\widehat{\varphi{}}_{0}](0,v) \\
&\lesssim \sum_{n=0}^N \int _{\Sigma{}(0,v)}  r\langle{}r\rangle^\delta{}(X(r^{1/2}(rX)^nT^{-1}\widehat{\varphi{}}_0))^2 + \langle{}r\rangle^\delta{}((rX)^nT^{-1}\widehat{\varphi{}}_0) \dd{}r\dd{}\theta{}\\
&\qquad + \sum_{\substack{n\ge 0,m\ge 1 \\ n+m\le N }} \int _{\Sigma{}(0,v)} r\langle{}r\rangle^\delta{}(L(r^{1/2}(rL)^nT^{m-1}\widehat{\varphi{}}_0))^2 + \langle{}r\rangle^{\delta{}}((rX)^nT^{m-1}\widehat{\varphi{}}_0)^2 \dd{}r\dd{}\theta{} \\
&\lesssim  \sum_{m=0}^N\mathcal{E}_{1+\delta{},N-m}[T^m\widehat{\varphi{}}_0](0,v) + v^{3\delta{}}\sum_{\substack{n\ge 0,m\ge 0 \\ n+m\le N }}\norm{\langle{}r\rangle^{1/2}(rX)^nT^{m}\widehat{\varphi{}}_0}_{L^\infty(\Sigma{}(0))}^2\\
&\qquad + v^{3\delta{}}\sum_{n=0}^N \int _{\Sigma{}(0)}  r\langle{}r\rangle^{-2\delta{}}(X(r^{1/2}(rX)^nT^{-1}\widehat{\varphi{}}_0))^2 + \langle{}r\rangle^{-2\delta{}}((rX)^nT^{-1}\widehat{\varphi{}}_0) \dd{}r\dd{}\theta{},
\end{split}
\end{equation}
where we have used \(v\sim \langle{}r\rangle{}\) on \(\Sigma{}(0)\). By \cref{L-inverse-radially-symmetric}, we have
\begin{equation}\label{Tinv-Linv-prep-0}
\begin{split}
&\sum_{n=0}^N\int_0^\infty r(X(rX)^nT^{-1}\widehat{\varphi{}}_0)^2 + r\langle{}r\rangle{}^{-2\delta{}}(X(r^{1/2}(rX)^nT^{-1}\widehat{\varphi{}}_0))^2 + \langle{}r\rangle^{-2\delta{}}((rX)^nT^{-1}\widehat{\varphi{}}_0)^2 \dd{}r  \\
&\lesssim_{N,\delta{}} \sum_{n=0}^N \int _{\Sigma{}(0)} r^{-1}\langle{}r\rangle^{4-\delta{}}((rX)^n\side{RHS}{Tinv-phi-prep-0})^2\dd{}r\dd{}\theta{}
\end{split}
\end{equation}
By estimating the right-hand side using \cref{Tinv-mink}, combining with
\cref{Tinv-estimate-0,Tinv-estimate-1}, and using \cref{mink-energy} to replace
\(\widehat{\varphi{}}\) with \(\varphi{}\) in energy and pointwise norms, we obtain
\begin{equation}
E_N[T^{-1}\widehat{\varphi{}}_0](0,v) + \mathcal{E}_{1+\delta{},N}[T^{-1}\widehat{\varphi{}}_{0}](0,v) \lesssim v^{3\delta{}}\mathbf{D}_{N,\delta{}}[\varphi{}].
\end{equation}
The estimates for the norms involving \(\psi_{\ge 1}\) are similar, but we use
\cref{good-scalar-field-Linv} with \(p = 0\) and \(p = 1-\delta{}\) (noting that
\(\mathcal{L}\) is an operator of the form
\(\tilde{\mathcal{L}}_{\alpha{},f_1,f_2}\) considered there, with \(\alpha{} =
-1/4\)) in place of \cref{L-inverse-radially-symmetric}.
\end{proof}
We will construct time integrals of \(\widehat{\Psi{}}_0\) not by solving the wave
equation that \(\widehat{\varphi{}}\) solves, but by solving the equation that \(\widehat{\Psi{}}_0\) itself solves (see
\cref{equation-for-Psi}).
\begin{lemma}[Solvability theory for the equation solved by \(\Psi_0\)]
Suppose \(\Phi{}_0\), \(\Phi{}_1\), and \(F\) are radially symmetric elements of \(r^{3/2}C^\infty(\mathcal{R})\). Then there exists a
unique solution \(\Phi{}\in r^{3/2}C^\infty(\mathcal{R})\) to the equation
\begin{equation}\label{Psi-equation-existence-equation}
-\underline{L}L\Phi{} - \frac{3}{4}r^{-2}G\Bigl(G - \frac{2}{3}rG'\Bigr)\Phi{} = F
\end{equation}
that is radially symmetric and achieves the initial data \((\Phi{}|_{\Sigma{}_0},T\Phi{}|_{\Sigma{}_0}) =
(\Phi{}_0,\Phi{}_1)\).
\label{Psi-equation-existence}
\end{lemma}
\begin{proof}
Observe that a radially symmetric scalar field \(\Phi{}\) solves
\cref{Psi-equation-existence-equation} if and only if \(\tilde{\Phi{}} \coloneqq{}r^{-3/2}\Phi{}\)
solves
\begin{equation}
-T^2\tilde{\Phi{}} + G^2Z^2\tilde{\Phi{}} + G^2\frac{3}{r}Z\tilde{\Phi{}} + GG'Z\tilde{\Phi{}} + 2Gr^{-1}G'\tilde{f} = r^{-3/2}F.
\end{equation}
This is the equation for a spherically symmetric solution to a wave equation
associated to an asymptotically flat metric in \((4 + 1)\)-dimensions with smooth
potential and source, for which existence and uniqueness for solutions arising
from smooth data is standard. Namely, \(\Phi{}\) solves
\cref{Psi-equation-existence-equation} if and only if \(\tilde{\Phi{}} \coloneqq{}r^{-3/2}\Phi{}\)
solves
\begin{equation}\label{Psi-wave-equation-4d}
\Box{}_{\tilde{g}}\tilde{\Phi{}} + 2A^{-2}Gr^{-1}G'\tilde{\Phi{}} = r^{-3/2}F,
\end{equation}
where \(\tilde{g}\) is the spherically symmetric metric
\begin{equation}
\tilde{g} = -A(r)^2\dd{}t^2 + A(r)^2G(r)^{-2}\dd{}r^2 + r^2g_{S^3}.
\end{equation}
The reduction to \cref{Psi-wave-equation-4d} is motivated by the fact that the
\(r^3\) volume form present in four space dimensions absorbs the
\(r^{-3/2}\Phi{}\) term that would otherwise introduce a term that is too
singular at the origin to apply standard existence and uniqueness results.
\end{proof}
\begin{proposition}[Constructing the first two time integrals of \(\widehat{\Psi{}}_0\)]
There exists a radially symmetric function \(T^{-2}\widehat{\Psi{}}_0\in
r^{3/2}C^\infty(\mathcal{R})\) solving
\begin{equation}\label{Tinv-Psi-equation}
\underline{L}LT^{-2}\Psi{}_0 + \frac{3}{4}r^{-2}G\Bigl(G - \frac{2}{3}rG'\Bigr)T^{-2}\Psi{}_0 = \mathfrak{L}[\varphi{}]GZ(r^{1/2}T^{-2}\Box{}\varphi{}_{\textnormal{mink}})
\end{equation}
such that \(T^2T^{-2}\widehat{\Psi{}}_0 = \widehat{\Psi{}}_0\) and
\(T^{-1}\widehat{\Psi{}}_0\coloneqq{}-TT^{-2}\widehat{\Psi{}}_0 =
r^{1/2}GZT^{-1}\widehat{\varphi{}}_0\), where \(T^{-1}\widehat{\varphi{}}_0\)
was constructed in \cref{Tinv-phi}. Moreover, the following estimate holds for
each \(N\ge 0\) and \(\delta{} > 0\):
\begin{equation}\label{Tinv-Psi-estimate}
\tilde{\mathcal{E}}_{1+\delta{},N}[T^{-2}\widehat{\Psi{}}_0](0,v)\lesssim_{N,\delta{}} v^{3\delta{}} \mathbf{D}_{N+1,\delta{}}{}[\varphi{}],
\end{equation}
where the initial data norm \(\mathbf{D}_{N,\delta{}}[\varphi{}]\) was defined in \cref{data-norm}.
\label{Tinv-Psi}
\end{proposition}
\begin{proof}
\step{Step 1: Construction.} Note that if a radially symmetric field
 \(\Phi{}\) solves
\begin{equation}\label{Phi-equation-tinv}
\underline{L}L\Phi{} + \frac{3}{4}r^{-2}G\Bigl(G - \frac{2}{3}rG'\Bigr)\Phi{} = \tilde{F},
\end{equation}
then we have
\begin{equation}
\tilde{\mathcal{L}}(G^{1/2}\Phi{}) = \mathcal{F}[G^{1/2}T\Phi{}] - G^{-3/2}\tilde{F},
\end{equation}
where
\begin{equation}
\tilde{\mathcal{L}}\Phi{}\coloneqq{} X^2\Phi{} - \frac{3}{4}r^{-2}\Bigl(1 - \frac{2}{r}G^{-1}rG' + \frac{2}{3}G^{-1}r^2G'' - \frac{1}{3}G^{-2}(rG')^2\Bigr)\Phi{} + r^{-2}A^2G^{-2}\partial{}_\theta^2\Phi{}.
\end{equation}
In particular, \(\tilde{\mathcal{L}}\) is of the form considered in
\cref{good-scalar-field-tinv} (with \(\alpha{} = 3/4\)).

Let \(\Phi^{(1)}\) be the solution to
\begin{equation}\label{Phi-1-construction}
\tilde{\mathcal{L}}\Phi^{(1)} = \mathcal{F}[G^{1/2}\widehat{\Psi{}}_0] + \mathfrak{L}[\varphi{}]T(GZ(r^{1/2}T^{-2}\Box{}\varphi{}_{\textnormal{mink}}))
\end{equation}
constructed by \cref{good-scalar-field-Linv}. By the uniqueness of such solutions
in the Hilbert space \(\mathbf{H}_{0,N}\) considered in
\cref{good-scalar-field-Linv}, we have \(\Phi^{(1)} =
r^{1/2}GZT^{-1}\widehat{\varphi{}}_0|_{\Sigma{}(0)}\), where
\(T^{-1}\widehat{\varphi{}}_0\) was constructed in \cref{Tinv-phi}. Moreover,
\(\Phi^{(1)}\) is a radially symmetric element of \(r^{3/2}C^\infty(\Sigma{}(0))\).
Define \(T^{-1}\widehat{\Psi{}}_0\) to be the unique solution to
\cref{Psi-equation-existence-equation} with initial data
\((G^{-1/2}\Phi^{(1)},\widehat{\Psi{}}_0|_{\Sigma{}(0)})\) and inhomogeneity \(F
=
-\mathfrak{L}[\varphi{}]T(GZ(r^{1/2}T^{-2}\Box{}\varphi{}_{\textnormal{mink}}))\)
constructed in \cref{Psi-equation-existence} (where the equation is satisfied by
\cref{equation-for-Psi} applied to \(\widehat{\varphi{}}_0\)). Since
\(T^{-1}\widehat{\Psi{}}_0\) and \(r^{1/2}GZT^{-1}\widehat{\varphi{}}_0\) solve
\cref{Phi-equation-tinv} with the same initial data, they are equal by the
uniqueness statement of \cref{Psi-equation-existence}.

Next, let \(\Phi^{(2)}\) be the solution to
\begin{equation}\label{LPsi-RHS}
\tilde{\mathcal{L}}\Phi^{(2)} = \mathcal{F}[G^{1/2}T^{-1}\widehat{\Psi{}}_0] + \mathfrak{L}[\varphi{}]GZ(r^{1/2}T^{-2}\Box{}\varphi{}_{\textnormal{mink}})
\end{equation}
constructed by \cref{good-scalar-field-Linv}. Then \(\Phi{}\) is a radially symmetric element
of \(r^{3/2}C^\infty(\Sigma{}(0))\). We now let \(T^{-2}\widehat{\Psi{}}_0\) be
the unique solution to \cref{Psi-equation-existence-equation} with initial data
\((G^{-1/2}\Phi^{(2)},T^{-1}\widehat{\Psi{}}_0|_{\Sigma{}(0)})\) and
inhomogeneity \(F =
-\mathfrak{L}[\varphi{}]GZ(r^{1/2}T^{-2}\Box{}\varphi{}_{\textnormal{mink}})\)
constructed in \cref{Psi-equation-existence} (where the equation is satisfied by
\cref{equation-for-Psi} applied to \(T^{-1}\widehat{\varphi{}}_0\)). We argue as in
Step 1 of the proof of \cref{Tinv-phi} to show that \(TT^{-2}\widehat{\Psi{}}_0 =
T^{-1}\widehat{\Psi{}}_0\).

\step{Step 2: Estimates.} Arguing as in the proof of Step 2 of \cref{Tinv-phi} (using now
the estimates of \cref{good-scalar-field-Linv} with \(p = 1-2\delta{}\) in place of
those of \cref{L-inverse-radially-symmetric} and the right-hand side of \cref{LPsi-RHS} in place
of that of \cref{Tinv-phi-prep-0}, as well as \cref{Tinv-mink} to estimate the
inhomogeneity associated to \(\varphi{}_{\textnormal{mink}}\)), we obtain
\begin{equation}
\begin{split}
&\tilde{\mathcal{E}}_{1+\delta{},N}[T^{-2}\widehat{\Psi{}}_0](0,v) \\
&\lesssim \sum_{m=0}^{N-2} \tilde{\mathcal{E}}_{1+\delta{},N-m}[T^m\widehat{\Psi{}}_0](0,v)  + v^{3\delta{}}\sum_{n=0}^N \int _{\Sigma{}(0,v)} \langle{}r\rangle^{3-2\delta{}}(X(rX)^nT^{-1}\widehat{\Psi{}}_0)^2 + \langle{}r\rangle^{1-2\delta{}}((rX)^nT^{-1}\widehat{\Psi{}}_0)^2 \dd{}r\dd{}\theta{}\\
&\qquad + v^{3\delta{}}\sum_{n=0}^N \int _{\Sigma{}(0,v)} \langle{}r\rangle^{1-2\delta{}}((rX)^n\widehat{\Psi{}}_0)^2 \dd{}r\dd{}\theta{} + v^{3\delta{}}\mathfrak{L}[\varphi{}]^2\\
\end{split}
\end{equation}
By \cref{mink-energy}, we can replace \(\widehat{\Psi{}}_0\) by \(\Psi_{0}\) (at the cost
of gaining \(\mathfrak{L}[\varphi{}]^2\), which is already present) to see that all the
terms but the one involving \(T^{-1}\widehat{\Psi{}}_0\) are present in the data norm
\(\mathbf{D}_{N,\delta{}}[\varphi{}]\). On one hand, by expanding
\begin{equation}
\sum_{n=0}^NX(rX)^nT^{-1}\widehat{\Psi{}}_0 = \sum_{n=0}^NX(rX)^n(r^{1/2}GZT^{-1}\widehat{\varphi{}}_0) = \sum_{n=0}^{N+1}\mathcal{O}(r^{-1})X(r^{1/2}(rX)^{n}T^{-1}\widehat{\varphi{}}_0) + \sum_{n=0}^N\mathcal{O}(1)X(rX)^n\widehat{\psi{}}_{0},
\end{equation}
we have
\begin{equation}\label{Tinv-Psi-prep0}
\begin{split}
&\sum_{n=0}^N\int _{\Sigma{}(0,v)\cap \set{r\ge 1}} \langle{}r\rangle^{3-2\delta{}}(X(rX)^nT^{-1}\widehat{\Psi{}}_0)^2 + \langle{}r\rangle^{1-2\delta{}}((rX)^nT^{-1}\widehat{\Psi{}}_0)^2 \dd{}r\dd{}\theta{} \\
&\lesssim \sum_{n=0}^{N+1}\int _{\Sigma{}(0)} r\langle{}r\rangle^{-2\delta{}}(X(r^{1/2}(rX)^nT^{-1}\widehat{\varphi{}}_0))^2 + \langle{}r\rangle^{-2\delta{}}((rX)^nT^{-1}\widehat{\varphi{}}_0)^2 \dd{}r\dd{}\theta{} + \sum_{n=0}^{N}\int _{\Sigma{}(0)}\langle{}r\rangle^{3-2\delta{}}(X(rX)^n\widehat{\psi{}}_{0})^2 \dd{}r\dd{}\theta{} \\
&\lesssim \mathbf{D}_{N,\delta{}}[\varphi{}] + \sum_{n=0}^{N+1}\int _{\Sigma{}(0)} r^{-1}\langle{}r\rangle^{4-\delta{}}((rX)^n\side{RHS}{Tinv-phi-prep-0})^2\dd{}r\dd{}\theta{},
\end{split}
\end{equation}
We can estimate the final term on the right-hand side by \(\mathbf{D}_{N+1,\delta{}}[\varphi{}]\) as
in Step 2 of the proof of \cref{Tinv-phi}. On the other hand, by Step 1, we
can use \cref{good-scalar-field-Linv} with \(p = 0\) to estimate
\begin{equation}\label{Tinv-Psi-prep1}
\begin{split}
&\sum_{n=0}^N\int _{\Sigma{}(0,v)\cap \set{r\le 1}} \langle{}r\rangle^{3-2\delta{}}(X(rX)^nT^{-1}\widehat{\Psi{}}_0)^2 + \langle{}r\rangle^{1-2\delta{}}((rX)^nT^{-1}\widehat{\Psi{}}_0)^2 \dd{}r\dd{}\theta{} \\
&\lesssim \sum_{n=0}^N\int _{\Sigma{}(0)} \langle{}r\rangle^2((rX)^n\side{RHS}{Phi-1-construction})^2\dd{}r\dd{}\theta{} \\
&\lesssim \mathfrak{L}[\varphi{}]^2 + \sum_{n=0}^N\int _{\Sigma{}(0)} \langle{}r\rangle^2(X(rX)^n\widehat{\Psi{}}_0)^2 + ((rX)^nT\widehat{\Psi{}}_0)^2 + ((rX)^n\widehat{\Psi{}}_0)^2\dd{}r\dd{}\theta{} \lesssim \mathbf{D}_{N,\delta{}}[\varphi{}].
\end{split}
\end{equation}
Combining \cref{Tinv-Psi-prep0,Tinv-Psi-prep1} and replacing \(\widehat{\Psi{}}_0\) with
\(\Psi_0\) using \cref{mink-energy}, we obtain \cref{Tinv-Psi-estimate}.
\end{proof}
\section{Late-time asymptotics}
\label{sec:late-time-asymptotics} In this section, we combine the results of
\cref{sec:energy-decay,sec:time-integrals} to derive pointwise decay for the
renormalized quantity \(\widehat{\varphi{}}\), and hence precise late-time
asymptotics for \(\varphi{}\) itself. In \cref{late-time-energy}, we derive energy
decay estimates, and in \cref{late-time-pointwise}, we derive pointwise decay
estimates.
\subsection{Energy decay for the renormalized solution}
\label{late-time-energy}
In this section, we prove energy decay estimates for the renormalized solution
\(\widehat{\varphi{}}\). To do this, we write \(\widehat{\varphi{}} = TT^{-1}\widehat{\varphi{}}\)
using the time integral construction of \cref{sec:time-integrals}, and then apply
the improved decay results for time derivatives obtained in
\cref{sec:energy-decay}. A caveat is that the time integral \(T^{-1}\widehat{\varphi{}}\)
does not have a finite \(p=1\) energy, and so we are forced to work with a small
loss in \(v\) for the energy defined on the truncated hypersurface \(\Sigma{}(\tau{},v)\).
In \cref{energy-decay-upgrade}, we show that this loss in \(v\) can be exchanged
for a loss in \(\tau{}\)-decay for the energy defined on the full hypersurface
\(\Sigma{}(\tau{})\). The main result of this section is \cref{renormalized-energy-decay},
where we prove energy decay estimates for all the solution variables associated
to \(\widehat{\varphi{}}\).
\begin{lemma}[Estimate for the initial data of time integrals of the renormalized solution]
Let \(\varphi{}\in C^\infty(\mathcal{R})\). For \(N\ge 0\), \(M\ge
0\), \(\delta{} > 0\), and \(v\ge 0\), we have
\begin{align}
\tilde{\mathcal{D}}_{1+\delta{},N,M}[T^{-1}\widehat{\psi{}}_{\ge 1}](v) &\lesssim_{N,M,\delta{}} v^{3\delta{}}\mathbf{D}_{N+M,\delta{}}{}[\varphi{}], \label{renormalized-data-bound-prep-1}\\
\tilde{\mathcal{D}}_{1+\delta{},N,M}[T^{-2}\widehat{\Psi{}}_0](v)&\lesssim_{N,M,\delta{}}v^{3\delta{}}\mathbf{D}_{N+M+1,\delta{}}{}[\varphi{}], \\
\mathcal{D}_{N,M,\delta{}}{}[T^{-1}\widehat{\varphi{}}_0,T^{-2}\widehat{\Psi{}}_0](v) &\lesssim_{N,M,\delta{}}v^{3\delta{}}\mathbf{D}_{N+M+2,\delta{}}[\varphi{}]\qquad (N\ge 1). \label{renormalized-data-bound-prep-2}
\end{align}
\label{renormalized-data-bound}
\end{lemma}
\begin{proof}
We first establish \cref{renormalized-data-bound-prep-1} by considering
\(T^{-1}\widehat{\psi{}}_{\ge 1}\), which satisfies the assumptions of
\cref{good-field} (see \cref{good-field-criterion,equation-for-psi}) with vanishing
inhomogeneity (since \(\varphi_{\textnormal{mink}}\) is radially symmetric). It
follows from the estimates in \cref{Tinv-phi} and the definition of the norm
\(\tilde{\mathcal{D}}\) (see \cref{energy-decay}) that
\begin{equation}
\tilde{\mathcal{D}}_{1+\delta{},N,M}[\widehat{\psi{}}_{\ge 1}](v) \lesssim_{N,\delta{}} v^{3\delta{}}\mathbf{D}_{N+M,\delta{}}{}[\varphi{}] +  \sum_{m=0}^{M-1} \tilde{\mathcal{E}}_{1+\delta{},N+M-m}[T^{m}\widehat{\psi{}}_{\ge 1}](0,v).
\end{equation}

Next, we consider \(T^{-2}\widehat{\Psi{}}_0\), which satisfies the assumptions of
\cref{good-field} (see \cref{good-field-criterion,Tinv-Psi-equation}) with
inhomogeneity
\(F\coloneqq{}-r^{1/2}\mathfrak{L}[\varphi{}]GZT^{-2}\Box{}\varphi{}_{\textnormal{mink}}\).
It follows from \cref{A-norm-estimate} and the estimates in \cref{Tinv-Psi} that
\begin{equation}\label{Psi-2-data-estimate}
\tilde{\mathcal{D}}_{1+\delta{},N,M}[T^{-2}\widehat{\Psi{}}_0]\lesssim_{N,M,\delta{}} v^{3\delta{}}\mathbf{D}_{N+M+1,\delta{}}{}[\varphi{}] + \sum_{m=0}^{M} \tilde{\mathcal{E}}_{1+\delta{},N+M-m}[T^{m}\widehat{\Psi{}}_0](0,v).
\end{equation}

Finally, we establish \cref{renormalized-data-bound-prep-2} by considering
\(T^{-1}\widehat{\varphi{}}_0\), which solves
\(\Box{}T^{-1}\widehat{\varphi{}}_0 = F =
-\mathfrak{L}[\varphi{}]T^{-1}\Box{}\varphi_{\textnormal{mink}}\). It follows
from \cref{A-norm-estimate}, \cref{Psi-2-data-estimate}, and the definition of the
norm \(\mathcal{D}\) (given in \cref{energy-decay-radially-symmetric}) that
\begin{equation}
\begin{split}
\mathcal{D}_{N,M,\delta{}}[T^{-1}\widehat{\varphi{}}_{0},T^{-2}\widehat{\Psi{}}_0](v)&\lesssim_{N,M,\delta{}} v^{3\delta{}}\mathbf{D}_{N+M+2,\delta{}}{}[\varphi{}] + \sum_{m=0}^{M} \tilde{\mathcal{E}}_{1+\delta{},N+M-m}[T^{m}\widehat{\Psi{}}_0](0,v) \\
&\qquad + \sum_{m=0}^{M-1}(E_{N+M-m}[T^m\widehat{\varphi{}}_0](0,v) + \mathcal{E}_{1+\delta{},N+M-m}[T^m\widehat{\varphi{}}_0]).
\end{split}
\end{equation}
\end{proof}
\begin{lemma}[Exchanging growth in \(v\) of the truncated energy for a loss in \(\tau{}\)-decay of the full energy]
Let \(\varphi{}\in C^\infty(\mathcal{R})\). Suppose that for all \(\delta{} > 0\) small, there exists
a constant \(D_\delta{} > 0\) such that
\begin{align}
\mathcal{E}_{1+\delta{},N}[\varphi{}](\tau{},v)&\le  D_\delta{}v^{\eta{}}(1+\tau{})^{-\beta{}}
\end{align}
for some \(\eta{} > 0\) and \(\beta{}\in \R\) (independent of \(\delta{}\)). Suppose moreover that
\(D_\delta{} \le D_{\delta{}'}\) whenever \(\delta{}\le \delta{}'\). Then for \(\delta{} < \eta{}\), we have
\begin{equation}
\mathcal{E}_{1+\delta{},N}[\varphi{}](\tau{})\lesssim D_{2\delta{}}(1+\tau{})^{-\beta{}+\eta{}}.
\end{equation}
The same result holds with \(\tilde{\mathcal{E}}_{1 + \delta{},N}[\psi{}](\tau{},v)\) in place of \(\mathcal{E}_{1+\delta{},N}[\varphi{}](\tau{},v)\).
\label{energy-decay-upgrade}
\end{lemma}
\begin{proof}
We use the interpolation argument of \cite[Prop.~10.6]{GAJIC2023110058}. First,
choose \(A>0\) large enough (independently of \(\tau{}\)) that
\(\Sigma{}(\tau{})\cap \set{r\le R}\subset \Sigma{}(\tau{})\cap \set{v\le AR}\)
for all \(R\ge 1 + \tau{}\). Now let \(r_i\) be a dyadic sequence with \(r_0 =
1 + \tau{}\), and split
\begin{equation}\label{energy-decay-upgrade-prep-0}
\begin{split}
\mathcal{E}_{1+\delta{}}[\varphi{}](\tau{}) &= \int _{\Sigma{}(\tau{})} \underbrace{r(1+r)^\delta{}(L\psi{})^2 + h(r)(1+r)^\delta{}\varphi{}^2}_{\coloneqq{}(\ast{})} \dd{}r\dd{}\theta{} \\
&=  \int _{\Sigma{}(\tau{})\cap \set{r\le r_0}}(\ast{})\dd{}r\dd{}\theta{} + \sum_{i\ge 0}\int _{\Sigma{}(\tau{})\cap \set{r_i\le r\le r_{i+1}}}(\ast{})\dd{}r\dd{}\theta{}.
\end{split}
\end{equation}
By assumption, we can estimate
\begin{equation}\label{energy-decay-upgrade-prep}
\int _{\Sigma{}(\tau{})\cap \set{r\le r_0}}(\ast{})\dd{}r\dd{}\theta{} \le \int _{\Sigma{}(\tau{})\cap \set{v\le A(1+\tau{})}}(\ast{})\dd{}r\dd{}\theta{} = \mathcal{E}_{1+\delta{}}[\varphi{}](\tau{},A(1+\tau{}))\lesssim D_\delta{}(1+\tau{})^{-\beta{}+\eta{}}.
\end{equation}
Using the dyadicity of the \(r_i\), we estimate
\begin{equation}\label{energy-decay-upgrade-prep-2}
\begin{split}
\int _{\Sigma{}(\tau{})\cap \set{r_i\le r\le r_{i+1}}} (\ast{})\dd{}r\dd{}\theta{}&\le (1+r_i)^{-\delta{}}\int _{\Sigma{}(\tau{})\cap \set{v\le Ar_{i+1}}}r(1+r)^{2\delta{}}(L\psi{})^2 + h(r)(1+r)^{2\delta{}}\varphi{}^2\dd{}r\dd{}\theta{}\\
&=(1+r_i)^{-\delta{}}\mathcal{E}_{1+2\delta{}}[\varphi{}](\tau{},Ar_{i+1})\lesssim (1+r_i)^{-\delta{}+\eta{}}D_{2\delta{}}(1+\tau{})^{-\beta{}}.
\end{split}
\end{equation}
Returning to \cref{energy-decay-upgrade-prep-0} (and summing the geometric series
in \cref{energy-decay-upgrade-prep-2}), and repeating this argument for
\((rL)^n\varphi{}\) in place of \(\varphi{}\), we obtain the desired result. The
argument for \(\tilde{\mathcal{E}}_{1 + \delta{},N}[\psi{}](\tau{},v)\) is similar.
\end{proof}
\begin{proposition}[Energy estimates for the renormalized solution]
Suppose \(\varphi{}\in C^\infty(\mathcal{R})\) solves \cref{wave-equation}. For \(N\ge 0\), \(M\ge
0\), \(\tau{}\ge 0\), \(v\ge 0\), and \(\delta{} > 0\) sufficiently small, we have
\begin{align}
\tilde{E}[T^MT^{-2}\widehat{\Psi{}}_0](\tau{}) + \tilde{\mathcal{E}}_{0,N}[T^MT^{-2}\widehat{\Psi{}}_0](\tau{})&\lesssim_{N,M,\delta{}} (1+\tau{})^{-1-2M+4\delta{}}\mathbf{D}_{N+M+1,2\delta{}}[\varphi{}], \label{red-Psi}\\
\tilde{E}[T^MT^{-1}\widehat{\psi{}}_{\ge 1}](\tau{}) + \tilde{\mathcal{E}}_{0,N}[T^MT^{-1}\widehat{\psi{}}_{\ge 1}](\tau{}) &\lesssim_{N,M,\delta{}} (1+\tau{})^{-1-2M+4\delta{}}\mathbf{D}_{N+M,2\delta{}}[\varphi{}],\label{red-tilE}\\
\tilde{\mathcal{E}}_{1+\delta{},N}[T^MT^{-1}\widehat{\psi{}}_{\ge 1}](\tau{}) &\lesssim_{N,M,\delta{}} (1+\tau{})^{-2M+4\delta{}}\mathbf{D}_{N+M+2,2\delta{}}[\varphi{}], \label{red-tilcalE} \\
E_N[T^MT^{-1}\widehat{\varphi{}}_0](\tau{}) &\lesssim_{N,M,\delta{}} (1+\tau{})^{-1-2M+5\delta{}}\mathbf{D}_{N+M+2,2\delta{}}[\varphi{}], \label{red-energy}\\
\mathcal{E}_{1+\delta{},N}[T^MT^{-1}\widehat{\varphi{}}_0](\tau{}) &\lesssim_{N,M,\delta{}} (1+\tau{})^{1-2M+5\delta{}}\mathbf{D}_{N+M+2,2\delta{}}[\varphi{}]\qquad (M\ge 1). \label{red-calE}
\end{align}
\label{renormalized-energy-decay}
\end{proposition}
\begin{proof}
We first claim that
\begin{align}
\tilde{\mathcal{E}}_{1+\delta{},N}[T^MT^{-1}\widehat{\psi{}}_{\ge 1}](\tau{},v) &\lesssim_{N,M,\delta{}} v^{3\delta{}}(1+\tau{})^{-2M+\delta{}}\mathbf{D}_{N+M,\delta{}}[\varphi{}], \label{red-ge1-prep} \\
\tilde{\mathcal{E}}_{1+\delta{},N}[T^MT^{-2}\widehat{\Psi{}}_0](\tau{},v) &\lesssim_{N,M,\delta{}} v^{3\delta{}}(1+\tau{})^{-2M+\delta{}}\mathbf{D}_{N+M+1,\delta{}}[\varphi{}], \label{red-Psi-prep} \\
\mathcal{E}_{1+\delta{},N-1}[(rL)T^MT^{-1}\widehat{\varphi{}}_0](\tau{},v) &\lesssim_{N,M,\delta{}} v^{3\delta{}}(1+\tau{})^{-2M+2\delta{}}\mathbf{D}_{N+M+2,\delta{}}[\varphi{}]\qquad (N\ge 1),\label{red-0-rLE-prep} \\
\mathcal{E}_{1+\delta{},N}[T^MT^{-1}\widehat{\varphi{}}_0](\tau{},v) &\lesssim_{N,M,\delta{}} v^{3\delta{}}(1+\tau{})^{1-2M+2\delta{}}\mathbf{D}_{N+M+2,\delta{}}[\varphi{}]\qquad (M\ge 1). \label{red-0-calE-prep}
\end{align}
Indeed, these are immediate consequences of \cref{energy-decay} applied to
\(T^{-1}\widehat{\psi{}}_{\ge 1}\) and \(T^{-2}\widehat{\Psi{}}_0\) (which
satisfy the relevant assumptions by \cref{good-field-criterion}),
\cref{energy-decay-radially-symmetric} applied to \(T^{-1}\widehat{\varphi{}}_0\) (where
\(T^{-2}\widehat{\Psi{}}_0\) plays the role of \(\Phi{}\) in the theorem), and
\cref{renormalized-data-bound}. Now \cref{red-tilcalE,red-calE} follow from
\cref{energy-decay-upgrade} together with \cref{red-ge1-prep,red-0-calE-prep}, and
\cref{red-tilE} follows from \cref{red-tilcalE} and a standard pigeonhole argument
using \cref{modified-T-morawetz,integrated-modified-T}. Similarly, \cref{red-0-rLE-prep,red-Psi-prep} and
\cref{energy-decay-upgrade} imply
\begin{align}
\tilde{\mathcal{E}}_{1+\delta{},N}[T^MT^{-2}\widehat{\Psi{}}_0](\tau{}) &\lesssim_{N,M,\delta{}} (1+\tau{})^{-2M+4\delta{}}\mathbf{D}_{N+M+1,2\delta{}}[\varphi{}], \label{red-Psi-prep-2} \\
\mathcal{E}_{1+\delta{},N-1}[(rL)T^MT^{-1}\widehat{\varphi{}}_0](\tau{}) &\lesssim_{N,M,\delta{}} (1+\tau{})^{-2M+5\delta{}}\mathbf{D}_{N+M+2,2\delta{}}[\varphi{}]\qquad (N\ge 1). \label{red-rL-prep}
\end{align}
Note that \cref{red-Psi} follows from \cref{red-Psi-prep-2} and a standard pigeonhole
argument using \cref{rp-general} (see \cref{rp-general-hierarchy}).

It now remains to establish \cref{red-energy}. By \cref{T-estimate-ho},
\cref{red-Psi-prep-2}, and \cref{A-norm-estimate}, we have
\begin{equation}\label{red-prep-0}
E_N[T^MT^{-1}\widehat{\varphi{}}_0](\tau{}_2) \lesssim_{N,\delta{}} E_N[T^MT^{-1}\widehat{\varphi{}}_0](\tau{}_1) + (1+\tau{})^{-2-2M+4\delta{}}\mathbf{D}_{N+M,2\delta{}}[\varphi{}].
\end{equation}
By \cref{integrated-T-ho}, \cref{red-Psi-prep-2}, \cref{red-rL-prep},
\cref{A-norm-estimate}, and an argument as in the proof of \cref{eda-1}, we have
\begin{equation}\label{red-prep-1}
\int_{\tau{}_1}^{\tau{}_2} E_N[T^MT^{-1}\widehat{\varphi{}}_0](\tau{})\dd{}\tau{} \lesssim_{N,\delta{}} E_N[T^MT^{-1}\widehat{\varphi{}}_0](\tau{}_1,v) + (1+\tau{})^{-2M+5\delta{}}\mathbf{D}_{N+M+1,2\delta{}}[\varphi{}].
\end{equation}
A pigeonhole argument using \cref{red-prep-0,red-prep-1} (as in Step 2 of the proof
of \cref{energy-decay-radially-symmetric}) proves \cref{red-energy}.
\end{proof}
\subsection{Pointwise decay for the renormalized solution}
\label{late-time-pointwise} In this section, we complete the proof of
\cref{main-theorem} by proving sharp pointwise estimates for \(\varphi{}\), which
come from improved decay estimates for \(\widehat{\varphi{}} =
TT^{-1}\widehat{\varphi{}}\). We do not prove pointwise estimates for
\(\widehat{\varphi{}}\) directly. As usual, the obstruction comes at the level
of the radially symmetric part \(\widehat{\varphi{}}_0\). To get around this
obstruction, we first prove an estimate for \((rL)\widehat{\varphi{}}_0\) by
interpolating between the estimates of \cref{pointwise-estimate-general} applied in
appropriate large-\(r\) and small-\(r\) regions. We then integrate this estimate
in the \(L\)-direction to future null infinity (where we pick up no boundary
term). This yields an estimate in the region \(\set{r\ge 1}\). To extend the
estimate to \(\set{r\le 1}\), we integrate the estimate we have for
\(\widehat{\Psi{}}_0\sim r^{1/2}Z\widehat{\varphi{}}_0\) in the \(Z\)-direction
to \(\set{r=1}\), using the integrability of \(r^{-1/2}\) near \(\set{r=0}\). We
pursue this argument in \cref{pointwise-decay}, completing the proof of
\cref{main-theorem}.
\begin{proposition}[Pointwise estimates for general scalar fields]
Let \(\varphi{}\in C^\infty(\mathcal{R})\) and let \(\eta_h > 0\) and \(C_h\ge 2\) be the constants
defined in \cref{h-hyperboloidal} and \cref{h-lower-bound} of
\cref{hyperboloidal-foliation}. Recall the norm \(\mathcal{A}\) defined in
\cref{inhomogeneity-norm}. Then for \(R\ge 1\), \(N\ge 0\), and \(\tau{}\ge 0\), we have
\begin{equation}\label{pointwise-near}
\begin{split}
&\norm{(rL)^N\varphi{}_0}_{L^\infty(\Sigma{}(\tau{})\cap \set{r\le R})}^2 \lesssim_N R^{C_h}\Bigl(\mathcal{E}_1[T^N\varphi{}_0](\tau{}) + \!\!\!\!\sum_{\substack{0\le m\le N \\ 0\le k\le 1}}E_{N-m}[T^{m+k}\varphi{}_0](\tau{}) + \sum_{\substack{n,m\ge 0 \\ n + m \le N}} \mathcal{A}_{1,N-m}[T^m\Box{}\varphi{}](\tau{})\Bigr)
\end{split}
\end{equation}
and
\begin{equation}\label{pointwise-near-tilde}
\begin{split}
&\norm{(rL)^N\varphi{}_{\ge 1}}_{L^\infty(\Sigma{}(\tau{})\cap \set{r\le R})}\lesssim_N R^{C_h}\sum_{k=0}^1\Bigl(\tilde{\mathcal{E}}_{0,N}[\partial{}_\theta^k\psi{}_{\ge 1}](\tau{}) + \tilde{E}[\partial{}_\theta^k\psi{}_{\ge 1}](\tau{}) + \mathcal{A}_{1,N-1}[\partial{}_\theta^k\Box{}\varphi{}_{\ge 1}]\Bigr)
\end{split}
\end{equation}
and
\begin{equation}\label{pointwise-near-Psi}
\norm{\Psi{}_0}_{L^\infty(\Sigma{}(\tau{})\cap \set{r\le R})}^2\lesssim R^{C_h+2}(\tilde{\mathcal{E}}_0[\Psi{}_0](\tau{}) + \tilde{E}[\Psi{}_0](\tau{})).
\end{equation}
For \(N\ge 0\), \(\tau{}\ge 0\), \(0 < \delta{} \le \eta_h\), and \(r\ge 1\), we have
\begin{equation}\label{pointwise-far}
\abs{r^{1/2}(rL)^N\varphi{}_0(\tau{},r,\theta{})}^2 \lesssim_{N,\delta{}}\mathcal{E}_{1+\delta{}}[(rL)^N\varphi{}_0](\tau{}) + E[(rL)^N\varphi{}_0](\tau{})
\end{equation}
and
\begin{equation}\label{pointwise-far-tilde}
\begin{split}
&\abs{r^{1/2}(rL)^N\varphi{}_{\ge 1}(\tau{},r,\theta{})}^2\lesssim_{N,\delta{}} \sum_{k=0}^1\Bigl(\tilde{\mathcal{E}}_{1+\delta{},N}[\partial{}_\theta^k\psi{}_{\ge 1}](\tau{}) + \tilde{E}[\partial{}_\theta^k\psi{}_{\ge 1}](\tau{}) + \mathcal{A}_{2,N-1}[\partial{}_\theta^k\Box{}\varphi{}_{\ge 1}]\Bigr).
\end{split}
\end{equation}
\label{pointwise-estimate-general}
\end{proposition}
\begin{remark}
We could of course formulate an \((rL)\)-commuted version of \cref{pointwise-near-Psi}
(or an analogue for general scalar fields satisfying \cref{alpha-equation}), as
well as a corresponding version of \cref{pointwise-far-tilde}, but we do not need
such estimates for the proof of \cref{main-theorem}.
\end{remark}
\begin{proof}
\step{Step 1: Proof of \cref{pointwise-near}.}
Recall the following estimate for \(f\in C_c^\infty(\R^2)\):
\begin{equation}\label{Ccinfinity-estimate}
\norm{f}_{L^\infty(\R^2)}^2\lesssim \norm{f}_{L^2(\R^2)}\norm{\Lapl f}_{L^2(\R^2)}.
\end{equation}
Let \(\chi{}(r)\) be a cutoff function that is supported on \(\set{r\le 2R}\),
is identically \(1\) on \(\set{r\le R}\), and satisfies \(r\abs{\chi{}'(r)} +
r^2\abs{\chi{}''(r)}\lesssim 1\). By applying \cref{Ccinfinity-estimate} with \(\chi{}\varphi{}\)
(considered as a function on \(\Sigma{}(\tau{})\)) in place of \(f\) and using the expression
\(\Lapl =X^2 + r^{-1}X + r^{-2}\partial_\theta^2\) and the identity
\begin{equation}
\Lapl (\chi{}\varphi{}) = \chi{}\Lapl \varphi{} + (\Lapl \chi{})\varphi{} + 2\chi{}'X\varphi{}\implies \abs{\Lapl (\chi{}\varphi{})}\lesssim \abs{\Lapl \varphi{}} + R^{-2}\abs{\varphi{}} + R^{-1}\abs{X\varphi{}},
\end{equation}
we obtain
\begin{equation}\label{Cc-estimate}
\begin{split}
&\norm{\varphi{}}_{L^\infty(\Sigma{}(\tau{})\cap \set{r\le R})}^2 \\
&\lesssim R^{-2}\norm{\varphi{}}_{L^2(\Sigma{}(\tau{})\cap \set{r\le 2R})}^2 + \norm{X\varphi{}}_{L^2(\Sigma{}(\tau{})\cap \set{r\le 2R})}^2 + \norm{\varphi{}}_{L^2(\Sigma{}(\tau{})\cap \set{r\le 2R})}\norm{\Lapl \varphi{}}_{L^2(\Sigma{}(\tau{})\cap \set{r\le 2R})}.
\end{split}
\end{equation}
where the \(L^2\) norms are with respect to the volume form \(r\dd{}r\dd{}\theta{}\).

\step{Step 1a: Proof of \cref{pointwise-near} when \(N=0\).} In this step we will show that
\begin{equation}\label{pointwise-3a}
\begin{split}
\norm{\varphi{}}_{L^\infty(\Sigma{}(\tau{})\cap \set{r\le R})}^2 &\lesssim R^{C_h}(\mathcal{E}_1[\varphi{}](\tau{}) + E[\varphi{}](\tau{}) + E[T\varphi{}](\tau{}) + E[\partial{}_\theta{}\varphi{}](\tau{}))+ \int _{\Sigma{}(\tau{})\cap \set{r\le 2R}}r\abs{\Box{}\varphi{}}^2\dd{}r\dd{}\theta{}.
\end{split}
\end{equation}
First, we have
\begin{equation}\label{pointwise-3a-1}
\begin{split}
\norm{\varphi{}}_{L^2(\Sigma{}(\tau{})\cap \set{r\le 2R})}^2 + \norm{X\varphi{}}_{L^2(\Sigma{}(\tau{})\cap \set{r\le 2R})}^2 &=  \int _{\Sigma{}(\tau{})\cap \set{r\le 2R}}\varphi^2r + r(X\varphi{})^2\dd{}r\dd{}\theta{}\\
&\lesssim R^{C_h+1}\mathcal{E}_1[\varphi{}](\tau{},v_R(\tau{})) + E[\varphi{}](\tau{},v_R(\tau{})).
\end{split}
\end{equation}
We now estimate the Laplacian by the spatial part of the wave operator in
\((\tau{},r)\) coordinates, namely
\begin{equation}
\mathcal{P}\coloneqq{}X^2 + r^{-1}(1 + rG^{-1}G')X + A^2G^{-2}r^{-2}\partial{}_\theta^2,
\end{equation}
which satisfies
\begin{equation}
G^{-2}\Box{}\varphi{} = \mathcal{P}\varphi{}  -hG^{-1}(2-Gh)T^2\varphi{} - G^{-1}(2-2Gh)XT\varphi{} + G^{-1}((Gh)' - r^{-1}(1-Gh))T\varphi{},
\end{equation}
by \cref{box-equation-TX}. In view of the assumptions on \(h\) and \(G\) and the
expression \(\Lapl = X^2 + r^{-1}X + r^{-2}\partial_\theta^2\) in \((r,\theta{})\) coordinates, we have
\begin{equation}\label{laplacian-expression}
\Lapl =\mathcal{O}(1)\Box{}+\mathcal{O}(1)XT\varphi{} + \mathcal{O}(1)X\varphi{} + \mathcal{O}(1)\partial{}_\theta^2 + \mathcal{O}(1)hT^2\varphi{} + \mathcal{O}(\langle{}r\rangle{}^{-1})T\varphi{}.
\end{equation}
It follows that
\begin{equation}\label{pointwise-3a-2}
\begin{split}
&\norm{\Lapl \varphi{}}_{L^2(\Sigma{}(\tau{})\cap \set{r\le 2R})}^2 \lesssim R^{C_h-2}(E[\varphi{}](\tau{}) + E[T\varphi{}](\tau{}) + E[\partial{}_\theta{}\varphi{}](\tau{})) + \int _{\Sigma{}(\tau{})\cap \set{r\le 2R}}r\abs{\Box{}\varphi{}}^2\dd{}r\dd{}\theta{}.
\end{split}
\end{equation}
Combining \cref{pointwise-3a-1,pointwise-3a-2,Cc-estimate} we get \cref{pointwise-3a}.

\step{Step 1b: Applying \cref{Cc-estimate} with \((rX)^N\varphi{}\) for \(N\ge 1\) in
place of \(\varphi{}\).} Let \(N\ge 1\). The goal of this step is to show that
\begin{equation}\label{pointwise-near-step-goal}
\begin{split}
&\norm{(rX)^N\varphi{}}_{L^\infty(\Sigma{}(\tau{})\cap \set{r\le R})}^2 \lesssim R^{C_h}\!\!\!\!\!\sum_{\substack{0\le m\le N \\ k_1,k_2\ge 0\\k_1+k_2\le 1}}E_{N-m}[T^{k_1}\partial{}_\theta^{k_2}T^m\varphi{}](\tau{}) + \sum_{\substack{n,m\ge 0 \\ n + m\le N}} \int _{\Sigma{}(\tau{})\cap \set{r\le 2R}} r\abs{(rL)^{n}T^{m}\Box{}\varphi{}}^2\dd{}r\dd{}\theta{}.
\end{split}
\end{equation}
We first use \cref{rL-rX-comparison} and the definition of the constant \(C_{h}\ge
2\) to estimate
\begin{equation}\label{pointwise-near-prep-1}
\begin{split}
\norm{(rX)^N\varphi{}}_{L^2(\Sigma{}(\tau{})\cap \set{r\le 2R})}^2 &\lesssim  \sum_{\substack{n_1,n_2\ge 0 \\ n_1 + n_2 \le  N-1}}\int_{\Sigma{}(\tau{})\cap \set{r\le 2R}} [r^2(L(rL)^{n_1}T^{n_2}\varphi{})^2 + \langle{}r\rangle^{C_h}h(T(rL)^{n_1}T^{n_2}\varphi{})^2]r\dd{}r\dd{}\theta{} \\
&\lesssim R^{C_h}\sum_{m=0}^{N-1}E_{N-m-1}[T^{m}\varphi{}](\tau{}).
\end{split}
\end{equation}
Next, we use \cref{rL-rX-comparison} to obtain
\begin{equation}\label{pointwise-near-prep-2}
\begin{split}
\norm{X(rX)^N\varphi{}}_{L^2(\Sigma{}(\tau{})\cap \set{r\le 2R})}^2 &\lesssim \sum_{\substack{n_1,n_2\ge 0 \\ n_1 + n_2 \le  N-1}}\int_{\Sigma{}(\tau{})\cap \set{r\le 2R}} [(X(rL)^{n_1+1}T^{n_2}\varphi{})^2 + (X(rL)^{n_1}T^{n_2+1}\varphi{})^2]r\dd{}r\dd{}\theta{} \\
&\lesssim \sum_{m=0}^NE_{N-m}[T^m\varphi{}](\tau{}).
\end{split}
\end{equation}
From the commutation formula \([\Lapl ,rX] = 2\Lapl\), we obtain
\begin{equation}\label{laplacian-rX-commuted}
\Lapl (rX)^N\varphi{} = \sum_{n=0}^NC_n(rX)^n\Lapl \varphi{}
\end{equation}
for some integers \(C_n\). It follows from
\cref{laplacian-rX-commuted,laplacian-expression,rL-rX-comparison} that
\begin{equation}\label{pointwise-near-prep-3}
\begin{split}
\norm{\Lapl (rX)^N\varphi{}}_{L^2(\Sigma{}(\tau{})\cap \set{r\le 2R})}^2 &\lesssim \sum_{n=0}^N\int _{\Sigma{}(\tau{})\cap \set{r\le 2R}} [(X(rX)^n\varphi{})^2 + (X(rX)^nT\varphi{})^2 + (\partial{}_\theta{}\partial{}_\theta{}(rX)^n\varphi{})^2 \\
&\qquad + \langle{}r\rangle{}^{-2}(T(rX)^nT\varphi{})^2 + \abs{(rX)^n\Box{}\varphi{}}^2]r\dd{}r\dd{}\theta{} \\
&\lesssim R^{C_h-2}\sum_{\substack{0\le k_1,k_2\le 1 \\ k_1+k_2\le 1 }}E_N[T^{k_1}\partial{}_\theta^{k_2}\varphi{}](\tau{}) + \sum_{n=0}^N\int _{\Sigma{}(\tau{})\cap \set{r\le 2R}} r\abs{(rX)^n\Box{}\varphi{}}^2\dd{}r\dd{}\theta{} \\
\end{split}
\end{equation}
Substituting
\cref{pointwise-near-prep-1,pointwise-near-prep-2,pointwise-near-prep-3} into
\cref{Cc-estimate}, we obtain \cref{pointwise-near-step-goal}.

\step{Step 1c: Completing the proof.} We established \cref{pointwise-near} when \(N = 0\)
in Step 1a. Let \(N\ge 1\). \Cref{rL-rX-comparison} gives
\begin{equation}
\begin{split}
\abs{(rL)^N\varphi{}} &\lesssim \abs{T^N\varphi{}} +  \sum_{\substack{n_1\ge 1,n_2\ge 0 \\ n_1+n_2\le N}} \abs{(rX)^{n_1}T^{n_2}\varphi{}}.
\end{split}
\end{equation}
Combining this estimate with \cref{pointwise-near-step-goal} and \cref{pointwise-3a} applied to \(T^N\varphi{}\) in place of
\(\varphi{}\), we obtain \cref{pointwise-near} when \(N\ge 1\).

\step{Step 2: Proof of \cref{pointwise-near-tilde}.} Let \(\chi{}\) be as in Step 1. First, estimate
\begin{equation}\label{pointwise-near-tilde-0}
\begin{split}
\abs{X((\chi{}\varphi{})^2)} &\le  2\chi^2\abs{\varphi{}X\varphi{}} + 2\abs{\chi{}\chi{}'}\varphi^2\lesssim r(X\varphi{})^2 + r^{-1}\varphi^2\lesssim (X\psi{})^2 + r^{-2}\psi{}^2\lesssim  (X\psi{})^2 + \langle{}r\rangle^{C_h}h(r)r^{-2}\psi^2 \\
&\lesssim (L\psi{})^2 + h^2(\underline{L}\psi{})^2 + \langle{}r\rangle^{C_h}h(r)r^{-2}\psi^2.
\end{split}
\end{equation}
Since there are constants \(C_{N,n}\) such that
\begin{equation}
r^{1/2}(rL)^N\varphi{} = \sum_{n=0}^NC_{N,n}(rL)^n\psi{},
\end{equation}
we obtain from \cref{pointwise-near-tilde-0} (applied to \((rL)^N\varphi{}\) in place of
\(\varphi{}\)) the estimate
\begin{equation}\label{pointwise-near-tilde-1}
X((\chi{}(rL)^N\varphi{})^2)\lesssim_N \sum_{n=0}^N[(L(rL)^n\psi{})^2 + h(r)^2(\underline{L}(rL)^n\psi{})^2 + \langle{}r\rangle^{C_h}h(r)r^{-2}((rL)^n\psi{})^2].
\end{equation}
Use \cref{commuted-psi-prep,equation-for-psi,pointwise-near-tilde-0} to estimate
\begin{equation}\label{pointwise-near-tilde-2}
\sum_{n=0}^N(\underline{L}(rL)^n\psi)^2\lesssim (\underline{L}\psi{})^2 + \sum_{n=0}^{N-1} [(L(rL)^n\psi{})^2 + r^{-2}((rL)^n\psi{})^2 + r^{-2}(\partial{}_\theta{}\partial{}_\theta{}(rL)^n\psi{}) + r^3\abs{(rL)^n\Box{}\varphi{}}^2].
\end{equation}
Using the fundamental theorem of calculus in the \(X\)-direction together with
\cref{pointwise-near-tilde-1,pointwise-near-tilde-2}, we find that for \(r\le R\), we
have
\begin{equation}\label{pointwise-near-tilde-prep}
\begin{split}
\int _{S^1} ((rL)^N\varphi{})^2(\tau{},r,\theta{})\dd{}\theta{}&\le \int _{\Sigma{}(\tau{})\cap \set{r\le 2R}} \abs{X((\chi{}\varphi{})^2)}\dd{}r\dd{}\theta{}\\
&\lesssim R^{C_h}\tilde{\mathcal{E}}_{0,N}[\psi{}](\tau{}) + \tilde{E}[\psi{}](\tau{}) +\sum_{n=0}^{N-1} \int _{\Sigma{}(\tau{})\cap \set{r\le 2R}} r^3\abs{(rL)^n\Box{}\varphi{}}^2\dd{}r\dd{}\theta{}
\end{split}
\end{equation}
Sobolev embedding on the circle (and \(C_h\ge 2\)) completes the proof of \cref{pointwise-near-tilde}.

\step{Step 3: Proof of \cref{pointwise-near-Psi}.} Let \(\chi{}\) be as in Step 1. As in Step
2, we estimate
\begin{equation}\label{pointwise-near-Psi-prep}
\abs{X((\chi{}\Psi{}_0)^2)}\lesssim (L\Psi{}_0)^2 + h(r)(\underline{L}\Psi{}_0)^2 + \langle{}r\rangle^{C_h+2}h(r)r^{-2}\Psi{}_0^2.
\end{equation}
Using the fundamental theorem of calculus in the \(X\)-direction together with
\cref{pointwise-near-Psi-prep} (and the
radial symmetry of \(\Psi_0\)), we obtain \cref{pointwise-near-Psi}.

\step{Step 4: Proof of \cref{pointwise-far}.} Let \(r_0\in [1/2,1]\). For \(R\ge 1\), the
fundamental theorem of calculus in the \(X\)-direction gives
\begin{equation}\label{pointwise-far-prep-1}
\begin{split}
\int _{S^1}\abs{\psi{}(\tau{},R,\theta{})}\dd{}\theta{}&\le \int _{S^1}\abs{\psi{}(\tau{},r_0,\theta{})}\dd{}\theta{} + \int _{\Sigma{}(\tau{})\cap \set{1/2\le r\le R}} \abs{X\psi{}}\dd{}r\dd{}\theta{}.
\end{split}
\end{equation}
Since \(r\) and \(h(r)\) are comparable to \(1\) in \(\set{1/2\le r\le 1}\),
Cauchy--Schwarz gives
\begin{equation}\label{pointwise-far-prep-0}
\Bigl(\int _{\Sigma{}(\tau{})\cap \set{1/2\le r\le 1}}\abs{\psi{}}\dd{}r\dd{}\theta{}\Bigr)^2\lesssim \int _{\Sigma{}(\tau{})\cap \set{1/2\le r\le 1}}r\langle{}r\rangle^{C_h}h\varphi^2\dd{}r\dd{}\theta{}\lesssim \mathcal{E}_1[\varphi{}](\tau{}),
\end{equation}
since \(R\ge 1\). Use Cauchy--Schwarz, the integrability of
\(r\langle{}r\rangle^\delta{}\) on \([r_0,\infty)\) for \(\delta{} > 0\), and
the identity \(X = G^{-1}L - hT\) to estimate
\begin{equation}\label{pointwise-far-prep-2}
\begin{split}
\Bigl(\int _{\Sigma{}(\tau{})\cap \set{1/2\le r\le R}} \abs{X\psi{}}\dd{}r\dd{}\theta{}\Bigr)^2&\le \int _{\Sigma{}(\tau{})\cap \set{1/2\le r\le R}} r\langle{}r\rangle^\delta{}(X\psi{})^2\dd{}r\dd{}\theta{}\\
&\lesssim \int _{\Sigma{}(\tau{})} r\langle{}r\rangle^\delta{}(L\psi{})^2\dd{}r\dd{}\theta{} + \int _{\Sigma{}(\tau{})} h(r)r\langle{}r\rangle^\delta{}\cdot h(r)(T\varphi{})^2r\dd{}r\dd{}\theta{} \\
&\lesssim \mathcal{E}_{1+\delta{}}[\varphi{}](\tau{}) + E[\varphi{}](\tau{}),
\end{split}
\end{equation}
where in the last line we used \(\delta{} \le \eta_h\). Averaging \cref{pointwise-far-prep-1}
over \(r_0\in [1/2,1]\) and substituting
\cref{pointwise-far-prep-0,pointwise-far-prep-2}, we obtain
\begin{equation}\label{pointwise-far-prep-3}
\Bigl(\int _{S^1}\abs{\psi{}(\tau{},r,\theta{})}\dd{}\theta{} \Bigr)^2\lesssim \mathcal{E}_{1+\delta{}}[\varphi{}](\tau{}) + E[\varphi{}](\tau{}).
\end{equation}
The Sobolev embedding \(W^{1,1}(S^1)\hookrightarrow{}L^\infty(S^1)\) (which follows from the
fundamental theorem of calculus) combined with \cref{pointwise-far-prep-3} applied
to \((rL)^N\varphi{}\) in place of \(\varphi{}\) completes the proof.

\step{Step 5: Proof of \cref{pointwise-far-tilde}.} Let \(R\ge 1\). Arguing as in Step 3, we
obtain the estimate
\begin{equation}
\begin{split}
\Bigl(\int _{S^1}\abs{\psi{}(\tau{},R,\theta{})}\dd{}\theta{}\Bigr)^2 &\lesssim \int _{\Sigma{}(\tau{})} r\langle{}r\rangle^\delta{}(L\psi{})^2 + h(r)(\underline{L}\psi{})^2 + h(r)r^{-2}\psi^2\dd{}r\dd{}\theta{}\\
&\lesssim \tilde{\mathcal{E}}_{1+\delta{}}[\psi{}](\tau{}) + \int _{\Sigma{}(\tau{})} h(r)(\underline{L}\psi{})^2\dd{}r\dd{}\theta{}
\end{split}
\end{equation}
Combining this estimate with \cref{pointwise-near-tilde-2}, we find that for \(N\ge
0\), we have
\begin{equation}
\Bigl(\int _{S^1}\abs{(rL)^N\psi{}(\tau{},R,\theta{})}\dd{}\theta{}\Bigr)^2\lesssim \tilde{\mathcal{E}}_{1+\delta{},N}[\psi{}](\tau{}) + \tilde{E}[\psi{}](\tau{}) + \sum_{n=0}^{N-1}  \int _{\Sigma{}(\tau{})} r^3\langle{}r\rangle{}^{-1}\abs{(rL)^n\Box{}\varphi{}}^2\dd{}r\dd{}\theta{}.
\end{equation}
To complete the proof, commute the \((rL)^N\) on the left inside the
\(r^{1/2}\)-weight and use the Sobolev embedding
\(W^{1,1}(S^1)\hookrightarrow{}L^\infty(S^1)\) (which follows from the
fundamental theorem of calculus).
\end{proof}
\begin{proposition}[Pointwise estimates for the renormalized solution]
Let \(\varphi{}\in C^\infty(\mathcal{R})\) solve \cref{wave-equation}, and let \(\widehat{\varphi{}}\) be
the corresponding renormalized solution defined in \cref{renormalized-solution}.
Let \(\gamma{}\coloneqq{}(C_h + 1)^{-1}\), where \(C_h\ge 2\) (defined in
\cref{h-lower-bound} of \cref{hyperboloidal-foliation}) determines the polynomial rate
(in the radial coordinate \(r\)) at which the hyperboloidal foliation
\(\Sigma{}(\tau{})\) becomes null. For \(N\ge 0\) and \(M\ge 0\) and \(\delta{}
> 0\) sufficiently small, we have
\begin{equation}\label{pointwise-decay-equation}
\abs{(rL)^NT^M\widehat{\varphi{}}(\tau{},r,\theta{})}\lesssim \varphi{}_{\textnormal{mink}}(\tau{},r,\theta{})\cdot (1+\tau{})^{-M-\gamma{}/4}\Bigl(\mathbf{D}_{\min (1,N)+M+4,\delta{}}[\varphi{}] + \sum_{k=0}^1\mathbf{D}_{\min (1,N)+M+2,\delta{}}[\partial{}_\theta^k\varphi{}]\Bigr).
\end{equation}
\label{pointwise-decay}
\end{proposition}
\begin{proof}
Let \(R\ge 1\), and let \(\delta{} > 0\) be sufficiently small. Fix \(N\ge 0\) and \(M\ge 0\)
and write
\begin{equation}
\mathbf{D} \coloneqq{}\mathbf{D}_{N+M+4,\delta{}}[\varphi{}] + \sum_{k=0}^1\mathbf{D}_{N+M+2,\delta{}}[\partial{}_\theta^k\varphi{}].
\end{equation}

\step{Step 1: The case \(N\ge 1\).} Writing \(\widehat{\varphi{}} =
TT^{-1}\widehat{\varphi{}}_0 + TT^{-1}\widehat{\varphi{}}_{\ge 1}\) and applying
the estimates in \cref{pointwise-estimate-general} (with \(\delta{}/5\) in place of
\(\delta{}\)) and \cref{renormalized-energy-decay} and estimating the inhomogeneous
terms that arise using \cref{A-norm-estimate}, we obtain
\begin{equation}\label{pointwise-decay-equation-prep}
\abs{(rL)^NT^M\widehat{\varphi{}}(\tau{},r,\theta{})}^2\lesssim_{N,M,\delta{}} \begin{cases}
R^{C_h}(1+\tau{})^{-\min (3,2N+1)-2M+\delta{}}\mathbf{D} & \set{r\le R},\\
r^{-1}(1+\tau{})^{-2-2M+\delta{}}\mathbf{D} &\set{r\ge 1}.\
\end{cases}
\end{equation}
Taking \(R = (1+\tau{})^\gamma{}\) for \(\gamma{} \coloneqq{} (C_h + 1)^{-1}<1\) and using the first
estimate in \(\set{r\le R}\) and the second estimate in \(\set{r\ge R}\), we
obtain the estimate
\begin{equation}\label{pointwise-decay-prep}
\abs{(rL)^NT^M\widehat{\varphi{}}(\tau{},r,\theta{})}^2\lesssim_{N,M,\delta{}} u^{-1}v^{-1}(1+\tau)^{-\gamma{}-2M+\delta{}}\mathbf{D}\quad(N\ge 1),
\end{equation}
which proves \cref{pointwise-decay-equation} when \(N\ge 1\) (after taking \(\delta{} <
\gamma{}/2\)).

\step{Step 2: The case \(N = 0\).} We now consider the case \(N = 0\). There are two
reasons we do not argue directly as in the case \(N\ge 1\). First, the
right-hand sides of the estimates in \cref{pointwise-estimate-general} do not have
enough \(\tau{}\)-decay. For example, \cref{pointwise-near} includes
\(E[\widehat{\varphi{}}_0]\) when \(N = 0\), which decays like \(\tau^{-2}\),
and so this estimate concludes \(\tau^{-1}\) decay for \(\widehat{\varphi{}}_0\)
near the origin. However, we need to prove that \(\widehat{\varphi{}}_0\) decays
faster than \(\varphi{}_{\textnormal{mink}}\), which itself decays like
\(\tau^{-1}\) near the origin. More seriously, the right-hand side of
\cref{pointwise-near} when \(N = 0\) includes
\(\mathcal{E}_1[\widehat{\varphi{}}_0]\), which we do not even control (recall
that we only control \((rL)\)-derivatives and \(T\)-derivatives of the scalar
field in the \(\mathcal{E}_1\) norm, but not the scalar field itself). For this
reason, we use the case \(N = 0\) to control \(\widehat{\varphi{}}_0\) in the
region \(\set{r\ge 1}\) by integrating the estimate for
\((rL)\widehat{\varphi{}}_0\) in the \(L\)-direction to null infinity, and
integrate an estimate for \(\widehat{\Psi{}}_0 \sim r^{1/2}Z\widehat{\varphi{}}_0\) in the
\(Z\)-direction to obtain control in \(\set{r\le 1}\).

From \cref{pointwise-decay-prep}, we have
\begin{equation}\label{pointwise-decay-prep-L}
\abs{LT^M\widehat{\varphi{}}(\tau{},r,\theta{})}\lesssim_{M,\delta{}} r^{-1}u^{-1/2}v^{-1/2}(1+\tau)^{-\gamma{}/2-M+\delta{}/2}\mathbf{D}|_{N=1}.
\end{equation}
From \cref{pointwise-decay-equation-prep}, we know \(T^M\widehat{\varphi{}}|_{r=\infty} = 0\),
and so we can integrate \cref{pointwise-decay-prep-L} in the \(L\)-direction to
null infinity (and use \(1 + \tau{}\sim u\) and \(v\sim u + r\)) to obtain (in \((u,v)\) coordinates)
\begin{equation}\label{pointwise-decay-prep-phi-far}
\abs{T^M\widehat{\varphi{}}(u_0,v_0,\theta{})}\lesssim_{M,\delta{}} u_0^{-1/2-\gamma{}/2-M+\delta{}/2}\mathbf{D}|_{N=1}\cdot \int_{r_0}^\infty r^{-1}v^{-1/2}\dd{}r,
\end{equation}
where the integral is over a curve of constant \(u_0\). Let
\(r_0\coloneqq{}r(u_0,v_0)\ge 1\). We split the integral into the regions
\(\set{r_0\le r\le v_0/2}\) (where \(v\sim u\)) and \(\set{r\ge v_0/2}\) (where
\(v\sim r\)). When the first piece exists, we have \(u_0\sim v_0\), and so this
piece contributes logarithmically only in \(u_0\). The second piece always
contributes like \(v_0^{-1/2}\). Thus we have
\begin{equation}\label{pointwise-decay-prep-integral}
\begin{split}
\int_{r_0}^\infty r^{-1}v^{-1/2}\dd{}r &\lesssim \mathbf{1}_{r_0\le v_0/2}u_0^{-1/2}\int_{r_0}^{v_0/2} r^{-1}\dd{}r + \int_{v_0/2}^\infty r^{-3/2}\dd{}r \lesssim v_0^{-1/2}\log u_0.
\end{split}
\end{equation}
Combining \cref{pointwise-decay-prep-phi-far,pointwise-decay-prep-integral} we
obtain
\begin{equation}
\abs{T^M\widehat{\varphi{}}(u,v,\theta{})}\lesssim_\eta{} u^{-1/2}v^{-1/2}\cdot u^{-\gamma{}/2-M+\delta{}/2+\eta{}}\mathbf{D}|_{N=1}\quad (r(u,v)\ge 1)
\end{equation}
for any \(\eta{} > 0\), which implies \cref{pointwise-decay-equation} in the region
\(\set{r\ge 1}\).

We now consider the region \(\set{r\le 1}\). From \cref{pointwise-near-Psi} in
\cref{pointwise-estimate-general} (with \(R = 1\)) and \cref{red-Psi}, we obtain
\begin{equation}\label{pointwise-decay-prep-Z}
\abs{T^M\widehat{\Psi{}}_0}\lesssim (1+\tau{})^{-5+4\delta{}}\mathbf{D}|_{N=0}\implies \abs{ZT^M\widehat{\varphi{}}_0}\lesssim r^{-1/2}(1+\tau{})^{-5+4\delta{}}\quad \textnormal{in }\set{r\le 1}.
\end{equation}
Since \(r^{-1/2}\) is integrable near \(\set{r=0}\), we can integrate
\cref{pointwise-decay-prep-Z} in the \(Z\)-direction to \(\set{r\le 1}\) (which is
possible when \(\tau{}\gg 1\)) and use \cref{pointwise-decay-equation} for \(N =
0\) in the region \(\set{r\ge 1}\) (which we have already obtained) to obtain
\cref{pointwise-decay-equation} for \(N = 0\) in the region \(\set{r\le 1}\cap
\set{\tau{}\gg 1}\). The result in the region \(\set{\tau{}\lesssim 1}\) follows
from \cref{pointwise-decay-equation-prep}.
\end{proof}
\appendix
\section{Failure of integrated local energy decay for radially symmetric scalar fields}
\label{axisymmetric-morawetz-failure} In this section, we explain why, unlike in
\((3 + 1)\) dimensions, in two space dimensions, one cannot control a
zeroth-order bulk term by the \(T\)-energy. For this reason, the standard
strategy for proving Morawetz estimates in \((3 + 1)\) dimensions, which
provides control of a zeroth-order bulk term, must fail in \((2 + 1)\)
dimensions. In particular, in \((2 + 1)\) dimensions one cannot hope to directly
apply the method of Dafermos--Rodnianski \cite{rp-method} to prove
\(r^p\)-weighted estimates, since this strategy requires as input an integrated
local energy decay estimate.

The obstruction lies at the level of the radially symmetric part of the scalar
field. The mechanism is the scale invariance of the energy norm in two space
dimensions (namely \(k=1\) and \(p = 0\) in \cref{failure-of-iled}). As a
consequence, one can only hope to control a zeroth-order term in the bulk if one
allows on the right-hand side an energy with an \(r^p\)-weight with \(p\ge 1\)
(as we do in \cref{r-weighted-estimates}).

Due to the necessary restriction \(p\ge 1\), one can only hope to prove
\(r^p\)-weighted estimates on their own, without invoking an integrated local
energy decay statement that controls \(\varphi{}\) itself by the \(T\)-energy.
However, as we show in \cref{morawetz-radially-symmetric}, we can establish an
integrated estimate for \emph{derivatives of \(\varphi{}\)}, provided that one is
willing to include an \(r^p\)-weighted energy of
\(\Psi{}\coloneqq{}r^{1/2}\partial_r\varphi{}\) on the right-hand side. This is
the reason that our estimates for radially symmetric scalar fields \(\varphi{}\)
are coupled to the estimates for \(\Psi{}\), which is the main innovation in our
argument (see \cref{proof-outline} and in particular \cref{intro-r-radially-symmetric}
for further discussion).

To illustrate the argument, it is convenient to use energies defined on surfaces
of constant \(t\).
\begin{proposition}
For each \(T\ge 1\), there exists a radially symmetric function \(\varphi_T\in C^\infty(\R^{2 +
1})\) solving the linear wave equation on Minkowski space \((\R^{2 + 1},m)\) and
arising from compactly supported data on \(\set{t=0}\) such that for any \(p\in
[0,1)\) and \(k\ge 1\), the following estimate holds:
\begin{equation}
\int_0^T \int _{\Sigma{}_t\cap \set{r\le 1}}\varphi{}_T^2\dd{}r\dd{}\theta{}\dd{}t \gtrsim_k T^{2k-1-p}\int _{\R^2}(1+r)^p(\partial^k\varphi{}_T)^2|_{\set{t=0}}r\dd{}r\dd{}\theta{},
\end{equation}
where we write \(\Sigma_{t_0}\coloneqq{}\set{t=t_0}\) and
\(\partial^k\varphi_T\coloneqq{}\sum_{j=0}^k\partial_t^j\partial_r^{k-j}\varphi_T\). In particular, there is no uniform
estimate of the form
\begin{equation}
\int_0^\infty \int _{\Sigma{}_t\cap \set{r\le 1}}\varphi^2\dd{}r\dd{}\theta{}\dd{}t\lesssim \int _{\Sigma{}_0}(\partial{}\varphi{})^2 r\dd{}r\dd{}\theta{}
\end{equation}
for solutions \(\varphi{}\) to the wave equation on \((\R^{2+1},m)\).
\label{failure-of-iled}
\end{proposition}
\begin{remark}
This proposition says that no uniform integrated local energy decay statement
can hold for radially symmetric scalar fields in two space dimensions, even if
one allows small growing weights in time and/or a loss of derivatives on the
right-hand side. As we use in \cref{good-scalar-field-energy-estimates}, this
obstruction does not occur for scalar fields with vanishing radial part, because
there is additional coercivity available from the angular term in the energy, by
way of a Poincaré inequality on \(S^1\).
\end{remark}
\begin{proof}
The proof of this proposition was communicated to the author by Georgios
Moschidis \cite{moschidis-personal-communication}. Fix \(\chi{}:[0,\infty)\to
[0,1]\) such that \(\chi{}\equiv 1\) on \([0,1]\) and \(\chi{}\equiv 0\) on
\([2,\infty)\). Let \(\varphi_T\) be the solution to the wave equation arising
from initial data
\((\varphi{}_T|_{\set{t=0}},\partial{}_t\varphi{}_T|_{\set{t=0}}) =
(\chi{}(r/2T),0)\). Since \(\chi{}(r/2T)\equiv 1\) on \(\set{r\le 2T}\), the
domain of dependence property for solutions to the wave equation implies that
\(\varphi_T\equiv 1\) on \(\set{r\le 1}\cap \set{0\le t\le T}\subset \set{r\le 2T-\abs{t}}\), and so
\begin{equation}
\int_0^T \int _{\Sigma{}_t\cap \set{r\le 1}}\varphi{}_T^2\dd{}r\dd{}\theta{}\dd{}t \sim T.
\end{equation}
On the other hand, we can compute (using the wave equation \((-\partial_t^2 + \partial_r^2 +
r^{-1}\partial_r)\varphi_T = 0\) to express
\(\partial_t^2\varphi_T|_{\set{t=0}}\) in terms of
\(\partial_r^2\varphi_T|_{\set{t=0}}\) and
\(r^{-1}\partial_r\varphi_T|_{\set{t=0}}\), noting that
\(\partial_t\varphi_T|_{\set{t=0}} = 0\), and using \(\Supp \chi{}'\subset [1,2]\)):
\begin{equation}
\begin{split}
\sum_{j=0}^k\int _{\R^2}(1+r)^p(\partial_t^j\partial_r^{k-j}\varphi_T)^2|_{\set{t=0}}r\dd{}r\dd{}\theta{}&\lesssim_k \sum_{j=0}^{k-1}\int _{\R^2}(1+r)^pr^{-2j}(\partial_r^{k-j}\varphi_T|_{\set{t=0}})^2r\dd{}r\dd{}\theta{} \\
&\lesssim_k \sum_{j=0}^{k-1}\int_{2T}^{4T} (1+r)^pr^{-2j}(\partial{}_r^{k-j}(\chi{}(r/2T)))^2r\dd{}r\\
&\lesssim_k T^{-2k+2} \sum_{j=0}^{k-1}\int_1^2 (1+2Tx)^p\abs{\chi^{(k-j)}(x)}^2\dd{}x \\
&\lesssim_k T^{-2k+2+p}.
\end{split}
\end{equation}
In passing from the second line to the third line, we have used the fact that
\(r\sim T\) in the region of integration and that each \(r\)-derivative of
\(\chi{}(r/2T)\) produces a power of \(T^{-1}\).
\end{proof}
\section{An exceptional cancellation on Minkowski space and control of a zeroth-order bulk term}
\label{mink-zo} In this section, we derive estimates that control a zeroth-order
bulk term associated to a radially symmetric scalar field on exact Minkowski
space \((\R^{2+1},m)\). We do not use these estimates in the proof of
\cref{main-theorem}, which concerns perturbations of Minkowski space; we aim here
to highlight the special structure in the estimates present on exact Minkowski
space.

As discussed in \cref{axisymmetric-morawetz-failure}, an estimate controlling a
zeroth-order bulk term on Minkowski space must include an \(r^p\)-weighted
energy on the right hand side with \(p\ge 1\). The starting point is to prove a
\(p = 1\) estimate (see \cref{mink-rp1}). This is possible on Minkowski space due
to a crucial cancellation in the zeroth-order bulk term, which contributes with
a bad sign for all \(p\in (0,2)\setminus \set{1}\). The \(p=1\) estimate
controls a flux term along ingoing null cones, which we then use to extend the
\(r^p\)-weighted estimates to the full range \(p\in [1,2)\), albeit with growing
weights in time on the right-hand side for \(p > 1\) (see
\cref{r1+delta-estimate}).

On perturbations of Minkowski space, the cancellation that occurs on Minkowski
space for \(p = 1\) is broken, and so we must treat the zeroth-order term in the
bulk as an error term on the right-hand side. For this reason, we do not prove
\(r^p\)-weighted estimates for \(\varphi{}\) itself. However, we can prove
\(r^p\)-weighted estimates for derivatives of \(\varphi{}\), and this is enough
to close the argument. See \cref{intro-r-radially-symmetric} for further
discussion.
\begin{proposition}[\(r^p\)-weighted energy estimate for \(p=1\) on Minkowski space]
Let \(f(r) = rg(r)\) for \(g : [0,\infty)\to \R_{>0}\) a \(C^2\) function. Then for any
\(v\ge 0\) and \(0\le \tau_1\le \tau_2\), an identity of the following form
holds for radially symmetric functions \(\varphi{}\in C^\infty(\mathcal{R})\):
\begin{equation}\label{mink-rp1-equation}
\begin{split}
&\mathcal{E}_1[\varphi{}](\tau{}_2,v) + \int_{\tau{}_1}^{\tau{}_2} \int _{\Sigma{}(\tau{},v)} r(L\varphi{})^2\dd{}r\dd{}\theta{}\dd{}\tau{}  + \int _{\underline{C}(v)\cap \set{\tau{}_1\le \tau{}\le \tau{}_2}} \varphi^2 \dd{}u\dd{}\theta{} \\
&\lesssim \mathcal{E}_1[\varphi{}](\tau{}_1,v) + \int_{\tau{}_1}^{\tau{}_2}\int _{\Sigma{}(\tau{},v)} \abs{rL\psi{}}\abs{r^{1/2}\Box{}_m\varphi{}}\dd{}r\dd{}\theta{}\dd{}\tau{}.
\end{split}
\end{equation}
\label{mink-rp1}
\end{proposition}
\begin{proof}
This follows immediately from the multiplier identity in \cref{r-estimate-prep}
with \(f(r)\equiv r\) (so that \(g(r)\equiv 1\)). On exact Minkowski space, we
have \(G(r) \equiv 1\), and so the zeroth-order term in \cref{r-estimate-prep}
vanishes.
\end{proof}
We now use Hardy's inequality to prove an estimate that controls a zeroth-order
term in the bulk by the \(p=1\) energy, using in particular the control of the
flux term on ingoing null cones in \cref{mink-rp1}. However, this estimate includes
a logarithmic loss in time.
\begin{lemma}[1D Hardy inequality]
Fix \(a,b\in \R\) with \(a<b\), and let \(f,W\in C^1([a,b])\). If \(W\) is strictly
increasing, then
\begin{equation}\label{hardy-1}
\int_a^b W'(x)f^2(x)\dd{}x \le 2[W(b)f^2(b) - W(a)f^2(a)] + 4\int_a^b \frac{W(x)^2}{W'(x)}f'(x)^2\dd{}x.
\end{equation}
\label{1d-hardy}
\end{lemma}
\begin{remark}
If \(W\) is strictly decreasing, then we can apply \cref{1d-hardy} to \(-W\).
\label{hardy-negative}
\end{remark}
\begin{lemma}[Estimate for a zeroth-order bulk term in a compact-\(r\) region on Minkowski space]
Let \(\varphi{}\in C^\infty(\mathcal{R})\) be radially symmetric. Fix \(0<r_1\le r_2\). Then for \(0\le \tau_1\le
\tau_2\) and \(v\ge 0\), we have
\begin{equation}
\begin{split}
\int_{\tau{}_1}^{\tau{}_2} \int _{\Sigma{}(\tau{},v)\cap \set{r_1\le r\le r_2}} \varphi^2\dd{}r\dd{}\theta{}\dd{}\tau{} &\lesssim \log^2\langle{}\tau{}_2-\tau{}_1\rangle{}\Bigl[\mathcal{E}_1[\varphi{}](\tau{}_1,v) + \int_{\tau{}_1}^{\tau{}_2} \int _{\Sigma{}(\tau{},v)} \abs{rL\psi{}}\abs{r^{1/2}\Box{}_m\varphi{}}\dd{}r\dd{}\theta{}\dd{}\tau{}\Bigr].
\end{split}
\end{equation}
\label{preliminary-zo-control}
\end{lemma}
\begin{proof}
We first introduce some notation. Fix \(v_0\ge 0\) and \(0\le \tau_1\le \tau_2\). Let
\(v_\star{}\) be the minimum of \(v_0\) and the value of \(v\) on the circle
\(\Sigma(\tau{}_2)\cap \set{r=r_2}\). By construction, we have
\begin{equation}\label{t-set-containment}
\mathcal{R}(\tau{}_1,\tau{}_2,v_0)\cap \set{r_1\le r\le r_2}\subset \mathcal{R}(\tau{}_1,\tau{}_2,v_\star{})\cap \set{r\ge r_1}.
\end{equation}
Write \(v_{\textnormal{max}}(u_0)\) for the maximum \(v\)-value in
\(\mathcal{R}(\tau{}_1,\tau{}_2,v_\star{})\cap \set{u=u_0}\). Write
\(r_{\textnormal{max}}\) for the largest \(r\)-value in \(\mathcal{R}(\tau{}_1,\tau{}_2,v_\star{})\cap \set{r\ge r_1}\). By the definition of \(\tau{}\), we have
\begin{equation}\label{r-max-estimate}
r_{\textnormal{max}}\lesssim_{R}\tau_2-\tau_1.
\end{equation}
Finally, define the truncated outgoing null cones
\begin{equation}
C(u_0)\coloneqq{}\mathcal{R}(\tau{}_1,\tau{}_2,v_\star{})\cap \set{u=u_0}\cap \set{r\ge r_1},
\end{equation}
where we omit the fixed parameters \(\tau_1\), \(\tau_2\), \(v_\star\), and \(r_1\) from the notation.

We now apply \cref{1d-hardy} in the \(L\)-direction with the increasing weight \(W(r) = \log (2 + r)\) to obtain
\begin{equation}\label{hardy-cons-1}
\begin{split}
&\int _{C(u_0)} \varphi^2\dd{}v\dd{}\theta{} \lesssim \int _{C(u_0)} \frac{\varphi^2}{2+r}\dd{}v\dd{}\theta{} \\
&\lesssim \int _{S^1} \log (2+r)\varphi^2|_{u=u_0,v=v_{\textnormal{max}}(u_0)}\dd{}\theta{} +  \int _{C(u_0)} (2+r)\log^2 (2+r)(L\varphi{})^2\dd{}v\dd{}\theta{} \\
&\lesssim \log^2(2+r_{\textnormal{max}})\Bigl[\int _{S^1} \varphi^2|_{u=u_0,v_{\textnormal{max}}(u_0)}\dd{}\theta{} +  \int _{C(u_0)} r(L\varphi{})^2\dd{}v\dd{}\theta{}\Bigr],
\end{split}
\end{equation}
where the implicit constants depend on \(r_1\) and \(r_2\). Integrate
\cref{hardy-cons-1} in \(u\) (noting that the curves \(C(u_0)\) partition the
region \(\mathcal{R}(\tau{}_1,\tau{}_2,v_\star{})\cap \set{r\ge r_1}\)) to get
\begin{equation}\label{hardy-compact-r-prep}
\begin{split}
&\int_{\tau{}_1}^{\tau{}_2}\int _{\Sigma{}(\tau{},v_0)\cap \set{r_1\le r\le r_2}}\varphi^2\dd{}r\dd{}\theta{}\dd{}\tau{}\le \iint_{\mathcal{R}(\tau{}_1,\tau{}_2,v\star{})\cap \set{r\ge r_1}}\varphi^2\dd{}u\dd{}v\dd{}\theta{} \\
&\lesssim \log ^2(2+r_{\textnormal{max}})\Bigl[\int _{\underline{C}(v_{{\star{}}})\set{\tau{}_1\le \tau{}\le \tau{}_2}} \varphi^2\dd{}u\dd{}\theta{} + \int _{\Sigma{}(\tau{}_2,v_{\star{}})\cap \set{r\le r_2}} \varphi^2\dd{}u\dd{}\theta{} + \int _{\mathcal{R}(\tau{}_1,\tau{}_2,v_\star{})\cap \set{r\ge r_1}} r(L\varphi{})^2\dd{}u\dd{}v\dd{}\theta{}\Bigr] \\
&\lesssim \log ^2(2+r_{\textnormal{max}})\Bigl[\int _{\underline{C}(v_{\star{}})\set{\tau{}_1\le \tau{}\le \tau{}_2}} \varphi^2\dd{}u\dd{}\theta{} + \int _{\Sigma{}(\tau{}_2,v_{\star{}})} h(r)\varphi^2\dd{}r\dd{}\theta{} + \int_{\tau{}_1}^{\tau{}_2} \int _{\Sigma{}(\tau{},v_0)} r(L\varphi{})^2\dd{}r\dd{}\theta{}\dd{}\tau{}\Bigr].
\end{split}
\end{equation}
In the first and last inequalities we used the equivalence of
\(\dd{}u\dd{}v\dd{}\theta{}\) and \(\dd{}r\dd{}\theta{}\dd{}\tau{}\) as
spacetime volume forms and of \(\dd{}u\dd{}\theta{}\) and
\(h(r)\dd{}r\dd{}\theta{}\) as volume forms on \(\Sigma{}(\tau_2)\). Moreover,
in the first inequality we also used observation \cref{t-set-containment}. To
complete the proof, substitute \cref{r-max-estimate} into \cref{hardy-compact-r-prep}
and use \cref{mink-rp1} to control the terms on the right-hand side.
\end{proof}
In view of \cref{preliminary-zo-control}, in order to establish an analogue of
\cref{mink-rp1-equation} with control of a zeroth-order bulk term, it suffices to
construct a multiplier of the form \(f(r)L\) which produces (via
\cref{r-estimate-prep}) a zeroth-order bulk term that is positive for sufficiently
small and sufficiently large \(r\). We carry out this strategy in the following
proposition, using a similar multiplier to the one used in
\cite{metcalfe-hepditch-obstacle} (in the study of an exterior obstacle problem
in two space dimensions). To state the proposition, we introduce a ``\(p=1 +\)''
energy, which has a logarithmically stronger \(r\)-weight than the \(p=1\)
energy:
\begin{equation}\label{E1+def}
\mathcal{E}_{1+}[\varphi{}](\tau{},v)\coloneqq{}\int _{\Sigma{}(\tau{},v)} r\log \langle{}r\rangle{}(L\psi{})^2 + h(r)\log \langle{}r\rangle{}\varphi^2\dd{}r\dd{}\theta{}.
\end{equation}
\begin{proposition}[Estimate for a zeroth-order bulk term on Minkowski space]
Let \(\varphi{}\in C^\infty(\mathcal{R})\) be radially symmetric. For \(0\le \tau_1\le \tau_2\) and \(v\ge 0\),
we have
\begin{equation}\label{rlogr-estimate-equation}
\begin{split}
&\mathcal{E}_{1+}[\varphi{}](\tau{}_2,v) + \int_{\tau_1}^{\tau_2} \int _{\Sigma(\tau{},v)} r\log\langle{}r\rangle{}(L\varphi{})^2 + \langle{}r\rangle{}^{-1}\log^{-2}\langle{}r\rangle{}\varphi^2\dd{}r\dd{}\theta{}\dd{}\tau{} \\
&\lesssim \log ^2\langle{}\tau{}_2-\tau{}_1\rangle{}\Bigl[\mathcal{E}_{1+}[\varphi{}](\tau{}_1,v) + \int_{\tau{}_1}^{\tau{}_2}\int _{\Sigma{}(\tau{},v)} \abs{r\log \langle{}r\rangle{}L\psi{}}\abs{r^{1/2}\Box{}_m\varphi{}}
\dd{}r\dd{}\theta{}\dd{}\tau{}\Bigr].
\end{split}
\end{equation}
\label{rlogr-estimate}
\end{proposition}
\begin{remark}
The \(r\)-weight in the zeroth-order bulk term controlled on the left-hand side of
\cref{rlogr-estimate-equation} is weaker than would be expected (namely two powers
of \(r\) less than the \(r\)-weight on \((L\varphi{})^2\)) by \(\log
^3\langle{}r\rangle{}\). As the proof shows, the loss of \(\log
^3\langle{}r\rangle{}\) can be improved to a loss of \(\log ^{2 +
\delta{}}\langle{}r\rangle{}\) for any \(\delta{} > 0\). However, our proof
strategy cannot avoid the ``loss in \(r\)-weights.''
\end{remark}
\begin{proof}
Let \(0 < r_0 < 1\) be a number to be chosen. Let \(g_1,g_2 : (0,\infty)\to \R_{>0}\)
be smooth functions such that
\begin{equation}
g_1(r) = \begin{cases}
(1+r)^{-1}&\textnormal{in }\set{r\le r_0},\\
\sqrt{\log r}&\textnormal{in }\set{r\ge 2},\\
\end{cases}
\qquad g_2(r) = \begin{cases}
0&\textnormal{in }\set{r\le r_0}, \\
\log r&\textnormal{in }\set{r\ge 2},
\end{cases}
\end{equation}
with the further requirement that \(g_1\) be positive, which is possible because
\(g_1(0)\) and \(g_1(2)\) are positive. We now compute the relevant quantities
to obtain a coercive estimate from \cref{r-estimate-prep}. Since
\begin{equation}\label{0-limit-computation}
(-rg_1''(r) - g_1'(r))|_{r=0} = 1 > 0,
\end{equation}
one can choose \(r_0\) sufficiently small so that \((-rg_1''-g_1')(r)\ge
1/2> 0\) in \(\set{r\le r_0}\). By \cref{0-limit-computation} and further explicit computations, we conclude that
\begin{equation}\label{boundary-bulk-positivity}
\begin{cases}
(-rg_1'' - g_1')(r)\sim \langle{}r\rangle{}^{-1}\log ^{-3/2}\langle{}r\rangle{}&\textnormal{in }\set{r\ge 2}, \\
(-rg_2'' - g_2')(r) = 0&\textnormal{outside }\set{r_0\le r\le 2},\\
2rg_i'(r) + g_i(r) \sim g_i(r)&\textnormal{in }\set{r\ge 2},\textnormal{ where }i=1,2.
\end{cases}
\end{equation}
We now set \(f(r) = rg(r)\) for \(g \coloneqq{}A + g_1 + g_2\), where the constant \(A >
0\) is chosen sufficiently large that \(f'(r)\ge 1\). It follows from
\cref{boundary-bulk-positivity} that
\begin{equation}
\begin{cases}
f(r)\sim r\log \langle{}r\rangle{}, \\
rf'(r)\sim r\log \langle{}r\rangle{}, \\
2rg'(r) + g(r)\sim \log\langle{}r\rangle{}, \\
-rg''(r) - g'(r)\gtrsim \langle{}r\rangle{}^{-1}\log^{-2}\langle{}r\rangle{}&\textnormal{ outside }\set{r_0\le r\le 2}, \\
\abs{-rg''(r) - g'(r)}\lesssim 1&\textnormal{ in }\set{r_0\le r\le 2}. \\
\end{cases}
\end{equation}
Since \(g\) extends to a smooth function on \([0,\infty)\), we can apply
\cref{r-estimate-prep} with \(f(r)\) as above to obtain
\begin{equation}\label{r-estimate-new-prep}
\begin{split}
&\mathcal{E}_{1+}[\varphi{}](\tau{}_2,v) + \int_{\tau{}_1}^{\tau{}_2} \int _{\Sigma{}(\tau{},v)} r\log \langle{}r\rangle{}(L\varphi{})^2 + r^{-1}\log^{-2}\langle{}r\rangle{}\varphi^2\dd{}r\dd{}\theta{}\dd{}\tau{} \\
&\lesssim \mathcal{E}_{1+}[\varphi{}](\tau{}_1,v) + \int_{\tau{}_1}^{\tau{}_2} \int _{\Sigma{}(\tau{},t)\cap \set{r_0\le r\le 2}}\varphi^2\dd{}r\dd{}\theta{}\dd{}\tau{} + \int_{\tau{}_1}^{\tau{}_2}\int _{\Sigma{}(\tau{},v)} \abs{r\log \langle{}r\rangle{}L\psi{}}\abs{r^{1/2}\Box{}_{\textnormal{m}}}\varphi{}
\dd{}r\dd{}\theta{}\dd{}\tau{}.
\end{split}
\end{equation}
To complete the proof, use \cref{preliminary-zo-control} to control the second term
on the right-hand side of \cref{r-estimate-new-prep}.
\end{proof}
The control of the zeroth-order bulk term in \cref{rlogr-estimate} allows
us to use \cref{T-estimate} to control a bulk term involving \(\underline{L}\psi{}\).
\begin{corollary}[Estimate for a bulk term involving \(\underline{L}\psi\) on Minkowski space]
Let \(\varphi{}\in C^\infty(\mathcal{R})\) be radially symmetric. For \(0\le \tau_1\le \tau_2\) and \(\delta{} > 0\),
we have
\begin{equation}\label{Lbar-psi-bulk-estimate-equation}
\begin{split}
&\int_{\tau{}_1}^{\tau{}_2} \int _{\Sigma{}(\tau{},v)} r\langle{}r\rangle{}^{-2+\delta{}}(\underline{L}\psi{})^2\dd{}r\dd{}\theta{}\dd{}\tau{} \\
&\lesssim \log ^2\langle{}\tau{}_2-\tau{}_1\rangle{}\Bigl[E[\varphi{}](\tau{}_1,v) + \mathcal{E}_{1+}[\varphi{}](\tau{}_1,v)+ \int_{\tau{}_1}^{\tau{}_2} \int _{\Sigma{}(\tau{},v)} (\abs{T\varphi{}} +  \abs{r^{-1/2}r\log \langle{}r\rangle{}L\psi{}})\abs{r\Box{}_m\varphi{}} \dd{}r\dd{}\theta{}\dd{}\tau{}\Bigr],
\end{split}
\end{equation}
where the energy \(\mathcal{E}_{1 + }\) was defined in \cref{E1+def}.
\label{Lbar-psi-bulk-estimate}
\end{corollary}
\begin{proof}
Since \((\underline{L}\psi{})^2\lesssim r(\underline{L}\varphi{})^2 + r^{-1}\varphi^2\), this follows from
\cref{energy-boundedness-with-bulk} and \cref{rlogr-estimate}.
\end{proof}
Given the control of the zeroth-order bulk term in \cref{rlogr-estimate} and the
bulk term involving \(\underline{L}\psi{}\) in \cref{Lbar-psi-bulk-estimate}, we
can use an interpolation argument to control energies with stronger
\(r\)-weights, at the cost of adding stronger weights in \(\tau{}\) on the right
side.
\begin{proposition}[Estimate for a zeroth-order bulk term on Minkowski space with stronger \(r\)-weights]
Suppose \(\varphi{}\in C^\infty(\mathcal{R})\) is radially symmetric. For \(0\le \tau_1\le
\tau_2\), \(v\ge 0\), and \(\delta{}\in (0,1)\) we have
\begin{equation}\label{r1+delta-estimate-equation}
\begin{split}
&E[\varphi{}](\tau{}_2,v) + \mathcal{E}_{1+\delta}[\varphi{}](\tau{}_2,v) + \int_{\tau{}_1}^{\tau{}_2} \int _{\Sigma{}(\tau{},v)}r\langle{}r\rangle{}^\delta{}(L\varphi{})^2 + \langle{}r\rangle{}^{-1+\delta{}}\varphi{}^2\dd{}r\dd{}\theta{}\dd{}\tau{} \\
&\lesssim_\delta{} \langle{}\tau{}_2-\tau{}_1\rangle{}^{2\delta{}}\Bigl[E[\varphi{}](\tau{}_1,v) + \mathcal{E}_{1+\delta}[\varphi{}](\tau{}_1,v) + \int_{\tau{}_1}^{\tau{}_2} \int _{\Sigma{}(\tau{},v)} (\abs{T\varphi{}} +  \abs{r^{-1/2}r\langle{}r\rangle{}^\delta{}L\psi{}})\abs{r\Box{}_m\varphi{}} \dd{}r\dd{}\theta{}\dd{}\tau{}\Bigr].
\end{split}
\end{equation}
\label{r1+delta-estimate}
\end{proposition}
\begin{proof}
\step{Step 1: The multiplier estimate.} Set \(f(r) = rg(r)\) for \(g(r) = (1 + r)^\delta{}\).
We explicitly compute
\begin{equation}
\begin{cases}
2rg'(r) + g(r)\sim \langle{}r\rangle{}^\delta{},\\
rf'(r)\gtrsim r\langle{}r\rangle{}^{\delta{}}, \\
\abs{-rg''(r) - g'(r)}\lesssim \langle{}r\rangle{}^{-1+\delta{}}, \\
\end{cases}
\end{equation}
By \cref{r-estimate-prep} (whose assumptions \(f\) clearly satisfies), we get
(after dropping the boundary term on \(\underline{C}(v)\))
\begin{equation}\label{r1+delta-estimate-prelim}
\begin{split}
&\mathcal{E}_{1+\delta}[\varphi{}](\tau{}_2,v) + \int_{\tau{}_1}^{\tau{}_2} \int _{\Sigma{}{(\tau{},v)}}r\langle{}r\rangle{}^\delta{}(L\varphi{})^2 \dd{}r\dd{}\theta{}\dd{}\tau{}  \\
&\lesssim \mathcal{E}_{1+\delta}[\varphi{}](\tau{}_1,v) + \int_{\tau{}_1}^{\tau{}_2} \int _{\Sigma{}(\tau{},v)}\abs{r\langle{}r\rangle{}^\delta{}L\psi{}}\abs{r^{1/2}\Box{}_m\varphi{}}\dd{}r\dd{}\theta{}\dd{}\tau{} + \int_{\tau{}_1}^{\tau{}_2} \int _{\Sigma{}(\tau{},v)}\langle{}r\rangle{}^{-1+\delta{}}\varphi^2 \dd{}r\dd{}\theta{}\dd{}\tau{}.
\end{split}
\end{equation}

\step{Step 2: The interpolation argument.} We now control the
final term on the right-hand side of \cref{r1+delta-estimate-prelim} using
an interpolation argument. Treating \(\tau{}_1\), \(\tau_2\), and \(v\) as fixed, define
\begin{equation}
B\coloneqq{}E[\varphi{}](\tau{}_1,v) + \mathcal{E}_{1+}[\varphi{}](\tau{}_1,v) + \int_{\tau{}_1}^{\tau{}_2} \int _{\Sigma{}(\tau{},v)} (\abs{T\varphi{}} +  \abs{r^{-1/2}r\log \langle{}r\rangle{}L\psi{}})\abs{r\Box{}_m\varphi{}} \dd{}r\dd{}\theta{}\dd{}\tau{},
\end{equation}
where the energy \(\mathcal{E}_{1 + }\) was defined in \cref{E1+def}. It suffices to
show
\begin{equation}\label{interpolation-cons}
\begin{split}
&\int_{\tau{}_1}^{\tau{}_2} \int _{\Sigma{}(\tau{},v)}\langle{}r\rangle{}^{-1+\delta{}}\varphi^2\dd{}r\dd{}\theta{}\dd{}\tau{} \lesssim_\delta{} \langle{}\tau{}_2-\tau{}_1\rangle{}^{2\delta{}}B,
\end{split}
\end{equation}
since one can add a large multiple of \cref{interpolation-cons} to \cref{r1+delta-estimate-prelim}
to establish \cref{r1+delta-estimate-equation}.

We now show \cref{interpolation-cons}. On one hand, \cref{rlogr-estimate} gives
\begin{equation}\label{interpolation-input-1}
\begin{split}
&\int_{\tau{}_1}^{\tau{}_2} \int _{\Sigma{}(\tau{},v)}\langle{}r\rangle{}^{-1}\log ^{-2}\langle{}r\rangle{}\varphi^2\dd{}r\dd{}\theta{}\dd{}\tau{} \lesssim \log^2\langle{}\tau{}_2-\tau{}_1\rangle{}B.
\end{split}
\end{equation}
On the other hand, we have
\begin{equation}\label{interpolation-input-2}
\begin{split}
&\int_{\tau{}_1}^{\tau{}_2} \int _{\Sigma{}(\tau{},v)} \langle{}r\rangle{}^{-\delta{}}\varphi^2\dd{}r\dd{}\theta{}\dd{}\tau{}\lesssim \langle{}\tau{}_2-\tau{}_1\rangle{}B.
\end{split}
\end{equation}
We delay the proof of this claim to Step 3. We obtain \cref{interpolation-cons} by
interpolating between \cref{interpolation-input-1,interpolation-input-2}. By this
we mean splitting the integral into the regions \(\set{r\le R}\) and \(\set{r\ge
R}\), using \cref{interpolation-input-1} for the first region and
\cref{interpolation-input-2} for the second region, and then choosing \(R\sim
\tau{}_2-\tau{}_1\) appropriately.

\step{Step 3: Proof of \cref{interpolation-input-2}.} First, we claim that by \cref{1d-hardy} and an averaging
argument, we have
\begin{equation}\label{hardy-averaging-cons}
\int_{\tau{}_1}^{\tau{}_2} \int _{\Sigma{}(\tau{},v)} (1+r)^{-\delta{}}\varphi^2\dd{}r\dd{}\theta{}\dd{}\tau{}\lesssim \int_{\tau{}_1}^{\tau{}_2} \int _{\Sigma{}(\tau{},v)\cap \set{r\le 2}} \varphi^2\dd{}r\dd{}\theta{}\dd{}\tau{} + \int_{\tau{}_1}^{\tau{}_2}\int _{\Sigma{}(\tau{},v)} r(X\psi{})^2\dd{}r\dd{}\theta{}\dd{}\tau{} =: \textnormal{(I)} + \textnormal{(II)}. \\
\end{equation}
To see this, fix \(R\in [1,2]\) and use \cref{1d-hardy} in the \(X\)-direction with
the decreasing weight \((1 + r)^{-\delta{}}\) (see \cref{hardy-negative}) to obtain
\begin{equation}\label{hardy-averaging-cons-prep}
\begin{split}
&\int _{\Sigma{}(\tau{},v)\cap \set{r\ge 2}}(1+r)^{-\delta{}}\varphi^2\dd{}r\dd{}\theta{}\lesssim \int _{\Sigma{}(\tau{},v)\cap \set{r\ge R}}(1+r)^{-1-\delta{}}\psi{}^2\dd{}r\dd{}\theta{} \\
&\lesssim \int _{S^1}r(1+r)^{-\delta{}}\varphi^2|_{\Sigma{}_\tau{}\cap \set{r=R}}\dd{}\theta{} + \int _{\Sigma{}(\tau{},v)\cap \set{r\ge 1}}(1+r)^{1-\delta{}}(X\psi{})^2\dd{}r\dd{}\theta{}.
\end{split}
\end{equation}
Now average \cref{hardy-averaging-cons-prep} over \(R\in [1,2]\), integrate in
\(\tau{}\), and add a zeroth-order bulk term in the region \(\set{0\le r\le 2}\)
to obtain \cref{hardy-averaging-cons}.

It remains to control terms \(\textnormal{(I)}\) and \(\textnormal{(II)}\) on
the right-hand side of \cref{hardy-averaging-cons}. Control term \(\textnormal{(I)}\)
using \cref{rlogr-estimate}. For term \(\textnormal{(II)}\), use
\cref{vector-field-relations}, the estimate \(h = O(\langle{}r\rangle^{-1})\),
\cref{r-estimate}, and \cref{Lbar-psi-bulk-estimate} to get
\begin{equation}\label{interpolation-input-2-prep}
\begin{split}
\textnormal{(II)} &\lesssim \int_{\tau{}_1}^{\tau{}_2}\int _{\Sigma{}(\tau{},v)} r(L\psi{})^2\dd{}r\dd{}\theta{}\dd{}\tau{} + \int_{\tau{}_1}^{\tau{}_2}\int _{\Sigma{}(\tau{},v)}r(1+r)^{-2}(\underline{L}\psi{})^2\dd{}r\dd{}\theta{}\dd{}\tau{} \lesssim \langle{}\tau{}_2-\tau{}_1\rangle{}B.
\end{split}
\end{equation}
This concludes the proof of \cref{interpolation-input-2}.
\end{proof}
\hypersetup{urlcolor=Black}
\printbibliography
\end{document}